\documentclass[reqno,twoside,10pt,english]{amsart}
\usepackage{amsmath,amscd,amsfonts,amssymb,amsthm,epsfig,amsxtra,esint,geometry}
\newcommand{\RNum}[1]{\uppercase\expandafter{\romannumeral #1\relax}}
\usepackage{comment}

\usepackage{color}

\usepackage[final,colorlinks,linkcolor=blue,anchorcolor=red,citecolor=blue]{hyperref} 
\usepackage{graphicx}
\usepackage{titletoc,epsf}
\usepackage{amsmath,amsfonts,latexsym,amsthm,amsxtra,amssymb,bbm}
\allowdisplaybreaks 
\usepackage{graphicx,float}
\usepackage{footmisc}

\allowdisplaybreaks

\theoremstyle{plain}

\newtheorem{conjecture}{Conjecture}[section]
\newtheorem{corollary}{Corollary}[section]

\newtheorem{definition}{Definition}[section]

\newtheorem{lemma}{Lemma}[section]

\newtheorem{proposition}{Proposition}[section]

\newtheorem{remark}{Remark}[section]

\newtheorem{theorem}{Theorem}[section]

\numberwithin{equation}{section}

\newcommand{\RR}{\mathbb R}
\newcommand{\R}{\mathbb R}
\newcommand{\ZZ}{\mathbb Z}
\newcommand{\Z}{\mathbb Z}
\newcommand{\CC}{\mathbb C}

\numberwithin{equation}{section}

\numberwithin{equation}{section}

\begin{document}

\title[Classification of the minimal-mass blowup solutions to NLSS]{Classification of the minimal-mass blowup solutions to the two dimensional focusing cubic nonlinear Schr\"odinger system}

\author{Xing Cheng, Zuyu Ma, and Jiqiang Zheng}

\address[X. Cheng]{School of mathematics, Hohai University, Nanjing 210098, Jiangsu, China}
\email{{\tt chengx@hhu.edu.cn}}

\address[Z. Ma]{The Graduate School of China Academy of Engineering Physics, Beijing 100088, China}
\email{{\tt  mazuyu23@gscaep.ac.cn
}}

\address[J. Zheng]{Institute of Applied Physics and Computational Mathematics, Beijing, 100088 and National Key Laboratory of Computational Physics, Beijing 100088, China}
\email{{\tt  zheng\_jiqiang@iapcm.ac.cn}}

\begin{abstract}
In this article, we study the two dimensional focusing finitely and infinitely coupled cubic nonlinear Schr\"odinger system when the mass is equal to the scattering threshold. 
For the focusing finitely coupled cubic nonlinear Schr\"odinger system, we present a complete classification of minimal-mass blowup solutions. Specifically, we demonstrate that all such solutions must be either solitons or their pseudo-conformal transformations. To prove this result, we develop a modulation analysis that accounts for multi-component interactions to overcome the multiply phase transformations caused by the multi-component. A long time Strichartz estimate for vector-valued solutions is established to solve the difficulty posed by the Galilean transformation  and spatial translation, where a new vector-valued bilinear estimate is proven to address the challenges caused by the coupled nonlinear interaction. 
For the infinitely coupled focusing nonlinear Schr\"odinger system when the mass is equal or slightly above the scattering threshold in \cite{CGHY}, we show that scattering is the only dynamical behavior of the solutions to the infinitely coupled system. 

\bigskip

\end{abstract}

\maketitle

\section{Introduction}

\subsection{Setting of the problem.}
In this article, we consider the cubic focusing nonlinear Schr\"odinger system (NLSS) on $\mathbb{R}^2$:
	\begin{equation}\label{1.1}
		\begin{cases}
			i\partial_t \mathbf{u}   + \Delta_{\mathbb{R}^2 } \mathbf{u}  = - \mathbf{F} (\mathbf{u}) ,\\
			\mathbf{u}(0) = \mathbf{u}_{0},
		\end{cases}
	\end{equation}
where $\mathbf{u} = \{ u_j:\RR\times\RR^2\to\CC\}_{j\in \mathbb{Z}_N }$, $\mathbf{u}_0 = \{u_{0,j}:\RR^2\to\CC \}_{j\in \mathbb{Z}_N }$, and the nonlinear term $\mathbf{F} (\mathbf{u}) = \left\{F_j(\mathbf{u}) \right\}_{j \in \mathbb{Z}_N }$ is given by 	\begin{equation*}
		F_j(\mathbf{u})
		: =\sum\limits_{( j_1,j_2,j_3) \in \mathcal{R}_N(j) } u_{j_1} \bar{u}_{j_2} u_{j_3}
		= 2 \sum\limits_{k \in \mathbb{Z}_N } |u_k|^2   u_j - |u_j|^2 u_j,
	\end{equation*}
with
	\[
		\mathcal{R}_N (j) =
 		\left\{ (j_1,j_2,j_3) \in \mathbb{Z}_N^3 : j_1-j_2+j_3= j, \  j_1^2 - j_2^2 + j_3^2 = j^2 \right\}.
	\]
Here, $\mathbb{Z}_N := \{1, \cdots, N \}$ for any positive integer $N$ with the convention that $\mathbb{Z}_\infty = \mathbb{Z}$. In particular, when $N=\infty$, we let 
	\[
		\mathcal{R}(j): = \mathcal{R}_\infty (j) =
 		\left\{ (j_1,j_2,j_3) \in \mathbb{Z}^3 : j_1-j_2+j_3= j, \  j_1^2 - j_2^2 + j_3^2 = j^2 \right\}.
	\]
The system \eqref{1.1} admits the following invariant groups:
\begin{enumerate}
\item scaling symmetry
\begin{align}\label{sym1}
\lambda \mathbf{u} \left( \lambda^2 t, \lambda x \right), \ \forall\, \lambda > 0,
\end{align}
\item translation in space and time
\begin{align}\label{sym3}
\mathbf{u} (t- t_0, x- x_0), \  \forall\, t_0 \in \mathbb{R}, \  x_0 \in \mathbb{R}^2,
\end{align}
\item phase transformation
\begin{align}\label{sym4}
e^{i \gamma_j } {u}_j (t,x), \  \forall\, \gamma_j  \in [0,2\pi], j \in \mathbb{Z}_N,
\end{align}
\item Galilean transformation
\begin{align}\label{sym5}
e^{i{\xi_0 } \cdot ( x-  { \xi_0} t) }  \mathbf{u} (t, x- 2\xi_0 t), \ \forall\, \xi_0 \in \mathbb{R}^2, 
\end{align}
\item pseudo-conformal transformation
\begin{align}\label{pseudotransform}
\frac1{t} \overline{  \mathbf{u} \left( \frac1t, \frac{x}t \right) } e^{i \frac{|x|^2}{4t}}.
\end{align}

\end{enumerate}
For any positive integer $N$, the system \eqref{1.1} can be regarded as a special class of coupled Gross-Pitaevskii equations, and thus has  significant applications in  physical problems. For example, in nonlinear optics,  it serves as a good approximation for describing the propagation of self-trapped mutually incoherent wave packets. Here, $u_j$ ($j=1$, $\cdots$, $N$) represents the $j$-th component of the beam, we refer to  \cite{AA} and the references therein. Additionally, the system also has application in Bose-Einstein condensates (see \cite{ZLQYY,ZL} and the references therein). 
When $N=2$, \eqref{1.1} reduces to the strongly coupled nonlinear Schr\"odinger system:
	\begin{equation*}
		\begin{cases}
			i\partial_t u_1 + \Delta_{\mathbb{R}^2} u_1 = -  |u_1|^2 u_1 -   2|u_2|^2 u_1,\\
			i\partial_t u_2 + \Delta_{\mathbb{R}^2} u_2 = - 2|u_1|^2 u_2 - |u_2|^2 u_2,
		\end{cases}
	\end{equation*}
with initial data $u_j(0,x) = u_{0,j}(x)$ for $j=1,2$, which arises in the Hartree-Fock theory for a double condensate, i.e., a binary mixture of Bose-Einstein condensates in two different hyperfine states, we refer to \cite{E, Timmer}. Moreover, this system also arises as the non-relativistic limit of the complex-valued cubic focusing nonlinear Klein-Gordon equation in $\mathbb{R}^2$, see \cite{MN} for further details.  For $N=\infty$, the system \eqref{1.1} appears in the nonlinear approximate of the cubic focusing nonlinear Schr\"odinger equation on the cylinder $\mathbb{R}^2\times \mathbb{T}$ in \cite{CGYZ}, we use the transformation of the solution of this nonlinear Schr\"odinger system to approximate to large scale profile. The system in the defocusing case was studied in \cite{YZ}, while the focusing case has been studied in \cite{CGHY}. 

\subsubsection{Finitely coupled NLSS}

For any positive integer $N$,  \eqref{1.1} possesses  the following conservation laws:
	\begin{itemize}
		\item Mass: $M(\mathbf{u})(t) := \sum\limits_{j  = 1}^{N} \int_{\mathbb{R}^2}  \left|u_j(t,x) \right|^2 \,\mathrm{d}x$, 
		\item Energy:  	$\begin{aligned}[t]
			 E(\mathbf{u})(t) :=  & \frac12 \sum\limits_{j=1}^{N} \int_{\mathbb{R}^2} \left|\nabla u_j(t,x) \right|^2 \,\mathrm{d}x  - \frac14 \sum_{j=1 }^{N }    \int_{\mathbb{R}^2}  \bigg(  \left|u_j(t,x) \right|^4 + 2
 \sum_{ \substack{k= 1\\ k\ne j} }^{N}  \left|u_k(t,x) \right|^2  \left|u_j(t,x) \right|^2 \bigg)  \,\mathrm{d}x.
				\end{aligned}$
	\end{itemize}

The system \eqref{1.1} admits a soliton solution of the form  $\mathbf{u}=e^{it\Delta}\mathbf{Q}=
\left(e^{it\Delta}Q,\cdots,e^{it\Delta}Q \right)$,
where $Q$ is the unique radial  positive ground state to the following elliptic equation:
\begin{align*}
	\Delta Q - Q = - (2N-1)Q^{3}.
\end{align*}

In \cite{CGYZ}, the first author and collaborators 
proved that if  $M(\mathbf{u})<M(\mathbf{Q}) $, then the system \eqref{1.1} is globally well-posed and scatters  in the following sense: there exist  $\mathbf{u}_{\pm}\in L_x^2 l^2(\R^2\times\Z_N)$\footnote{Here $L_x^2l^2(\mathbb{R}^2 \times \mathbb{Z}_N  ) $ is defined in Subsection \ref{subse1.6v33}. For $N < \infty$, we see the space 
\begin{align*}
 L^2_x  l^2(\mathbb{R}^2 \times \mathbb{Z}_N )  
  = \overbrace{L_x^2(\mathbb{R}^2) \times L_x^2(\mathbb{R}^2) \times \cdots \times L_x^2(\mathbb{R}^2)}^{N-\text{copies}},
\end{align*}
 with the natural norm $$\|\mathbf{u}\|_{ L_x^2 l^2(\mathbb{R}^2 \times \mathbb{Z}_N  ) } := \bigg( \sum\limits_{j=1}^{N} \|u_j\|_{L_x^2(\RR^2)}^2\bigg)^\frac{1}{2}.$$
} such that
\begin{equation*} 
\lim\limits_{t\to\infty}\|\mathbf{u}(t)-e^{it\Delta}\mathbf{u}_{+}\|_{L_x^2l^2(\mathbb{R}^2 \times \mathbb{Z}_N  ) }=\lim\limits_{t\to-\infty}\|\mathbf{u}(t)-e^{it\Delta}\mathbf{u}_{-}\|_{L_x^2l^2(\mathbb{R}^2 \times \mathbb{Z}_N  ) }=0.
\end{equation*}
However, if $M(\mathbf{u})=M(\mathbf{Q})$, then there exist explicit examples of non-scattering solutions: the soliton $e^{it\Delta}\mathbf{Q}$  up to symmetries \eqref{sym1}-\eqref{sym5} and the pseudo-conformal transformation \eqref{pseudotransform}. Notably, the existence of the pseudo-conformal transformation \eqref{pseudotransform} allows us to directly construct finite time blowup solutions. In this paper, we investigate the dynamics of \eqref{1.1} when the mass $M(\mathbf{u})=M(\mathbf{Q})
$.

\subsubsection{Infinitely coupled NLSS}

When $N = \infty$, the system \eqref{1.1} enjoys the following conservation laws:
	\begin{enumerate}
		\item Mass:
        ${M}_{a,b,c}(\mathbf{u})(t) = \int_{\mathbb{R}^2  } \sum\limits_{j\in \mathbb{Z}} \left(a+ bj +  c j^2\right) \left|u_j(t,x )\right|^2\,\mathrm{d}x, \text{ where } a,b,c \in \mathbb{R}$.
        Especially, when $a=1$, and $b=c=0$, we denote $
        {M}(\mathbf{u}):=
        {M}_{1,0,0}(\mathbf{u})$.
		\item Energy:
        $\begin{aligned}[t]
		        {E}(\mathbf{u})(t)
        & = \int_{\mathbb{R}^2  } \sum_{j\in \mathbb{Z}} \left( \frac12 \left|\nabla u_j(t,x)\right|^2  -  \frac14\left(  F_j (\mathbf{u}) \bar{u}_j\right) (t,x) \right)   \,\mathrm{d}x\\
		& = \int_{\mathbb{R}^2   } \sum_{j\in \mathbb{Z}}\frac12 |\nabla u_j(t,x) |^2 \, \mathrm{d} x  - \frac14  \sum_{ \substack{ j\in \mathbb{Z},\\ n\in \mathbb{N}} } \bigg|\sum_{\substack{j_1-j_2=j,\\  j_1^2 - j_2^2   = n}} (  u_{j_1} \bar{u}_{j_2} )(t,x) \bigg|^2 \,\mathrm{d}x.
			\end{aligned}$
	\end{enumerate}

In \cite{CGYZ}, the first author and collaborators 
proved that if  $M(\mathbf{u})< \frac12 \|Q_0\|_{L_x^2}^2
$, where $Q_0$ is the unique positive radial solution of the elliptic equation:
\begin{align}\label{eq1.9v59}
\Delta Q_0 - Q_0 = - Q_0^{3},
\end{align}
then the system \eqref{1.1} is globally well-posed and scatters  in the following sense: there exist  $\mathbf{u}_{\pm}\in L_x^2h^1(\R^2\times\Z)$\footnote{Here $L_x^2 h^1(\R^2\times\Z)$ is defined in in Subsection \ref{subse1.6v33}.
} such that
\begin{equation*} 
		\left\|\mathbf{u}-e^{it\Delta_{\RR^2}}\mathbf{u}^{\,\pm} \right\|_{L^2_x h^1(\RR^2\times\ZZ)}
		:=\bigg\| \bigg( \sum\limits_{j\in \mathbb{Z}} \langle j\rangle^2 \big|  u_j(t) - e^{it\Delta_{\mathbb{R}^2}} u_j^{\pm  } \big|^2 \bigg)^\frac12  \bigg\|_{L_x^2(\mathbb{R}^2 )} \to 0, \text{ as } t\to \pm \infty.
\end{equation*}
In this paper, we investigate the dynamics of \eqref{1.1} for $N = \infty$ when the mass $M(\mathbf{u})
$ is equal or slightly above  $\frac12 \|Q_0\|^2_{L_x^2}
$.

Recall that $\mathbf{u}$ scatters in both time directions if and only if $\|\mathbf{u}\|_{L_{t,x}^4l^2(\R\times\R^2\times\Z_N)}<\infty$. This motivates the following definition of blowup solutions: 
\begin{definition}
We say that a solution $\mathbf{u} \in C_t^0L^2_x l^2(I \times \mathbb{R}^2 \times  \mathbb{Z}_N) $  to \eqref{1.1} blows up forward in time if there exists a time $t_0 \in I$ such that
\begin{align*}
\| \mathbf{u} \|_{L_{t,x}^4 l^2 ( [ t_0, \sup I ) \times \mathbb{R}^2 \times \mathbb{Z}_N)} = \infty,
\end{align*}
and that $\mathbf{u}$ blows up backward in time if there exists a time $t_0 \in I$ such that
\begin{align*}
\| \mathbf{u} \|_{L_{t,x}^4 l^2 ( ( \inf I, t_0 ]  \times \mathbb{R}^2 \times \mathbb{Z}_N)} = \infty.
\end{align*}

\end{definition}

\subsection{On the mass critical nonlinear Schr\"odinger equation (NLS)}\label{se2v53}

When $N=1$, the system $\eqref{1.1}$ reduces to the focusing mass-critical NLS:
\begin{align*}
	\begin{cases}
		i \partial_t u + \Delta u = - |u|^2 u, \\
		u(0) = u_0\in L_x^2(\R^2). 
	\end{cases}
\end{align*}
For a self-contained discussion, we consider the $d-$dimensional mass-critical focusing NLS:
\begin{align}\label{eq1.6v59}
   \begin{cases}
    i \partial_t u + \Delta u = - |u|^\frac4d u, \ (t,x) \in \mathbb{R} \times \mathbb{R}^d , \\
    u(0) = u_0\in L_x^2(\R^d)
    \end{cases}
\end{align}
for $d \ge 1$. This equation conserves mass $M_0(u) $ and energy $E_0(u)$, defined as:
\begin{align*}
 M_0(u)=\int_{\R^d} |u(t,x)|^2 \,\mathrm{d}x,\quad E_0(u)=\frac{1}{2}\int_{\R^d} |\nabla u(t,x)|^2 \,\mathrm{d}x-\frac{d}{2(d+2)}\int_{\R^d} |u(t,x)|^{\frac{2(d+2)}{d}}\,\mathrm{d}x.
\end{align*}
 Moreover, it enjoys the four  symmetries \eqref{sym1}-\eqref{sym5} and pseudo-conformal transformation \eqref{pseudotransform}.
In \cite{W}, M. I. Weinstein proved the sharp Gagliardo-Nirenberg inequality:
\begin{align*}
\|f\|_{L_x^\frac{2(d+2)}d}^\frac{2(d+2)}d \le \frac{d+2}d \left( \frac{ \|f \|_{L_x^2} }{ \|Q_0 \|_{L_x^2}} \right)^\frac4d \|\nabla f \|_{L_x^2}^2.
\end{align*}
where $Q_0$\footnote{Here we abuse the notation that this $Q_0$ is general dimensional extension of $Q_0$ defined in \eqref{eq1.9v59}, which is only used in subsection \ref{se2v53}.
} is the unique positive radial ground state of the elliptic equation:
\begin{align}
\Delta Q_0 - Q_0 = - Q_0^{1 + \frac4d}.
\end{align}
As a consequence, M. I. Weinstein proved that for any initial data $u_0\in H^{1}_{x}$ with $M_0(u_0)<M_0(Q_0)$, the solution exists globally. However, since the global solution is constructed by the local theory, no information is obtained about the long-time dynamic behavior of the solution $u$. In particular, scattering is not proved. The global well-posedness and scattering in $L_x^2$ for $M_0(u_0) < M_0(Q_0 )$ has been extensively studied and are now complete; see \cite{D7,D6,D66,KTV11,KVZ,TVZ1,TVZ2} and references therein. 
These results imply that non-scattering solutions can occur when $M_0(u_0) = M_0(Q_0)$. As we mentioned earlier, the soliton $e^{it}Q_0$ up to symmetries \eqref{sym1}-\eqref{sym4} and the pseudo-conformal transformation \eqref{pseudotransform} are two examples of minimal-mass blowup solutions among all non-scattering solutions. This naturally raises the question of  whether there are other such examples. 

The characterization of minimal-mass blowup solutions was initiated by M. I. Weinstein \cite{weinstein:charact}, who showed the following:
Let $u$ be an $H_x^1$ solution of \eqref{eq1.6v59} with the mass $M_0(Q_0)$ that blows up in finite time. Then there exist functions $\theta(t), x(t)$ and $ \lambda(t)$
so that:
\begin{align}\label{weinsteinre's result}
	\lambda(t)^{\frac d2} e^{i\theta(t)} u\bigl(t, \lambda(t)x+x(t)\bigr) \to Q_0 (x)  \quad \text{in} \quad H^{1}_{x}
\end{align} 
as $t$ approaches the blowup time.  In fact, convergence holds along any sequence of times for which the kinetic energy diverges. For an $H_x^1$ solution that blows up in finite time, any sequence converging to the blowup time satisfies this property. 

{For initial data in $L_x^2$, C. Fan \cite{F} established  \eqref{weinsteinre's result} in $L_x^2$ along some sequence of times that approaches the (not necessary finite) blowup time for radial initial data when $d \ge 1$. This result has been extended by B. Dodson \cite{D4,D3} to the non-radial case by incorporating an additional 
Galilean boost $\xi(t)$, i.e.
 \begin{align*}
	\lambda(t_n)^{\frac d2} e^{i\theta(t_n)} e^{ix\cdot\xi(t_n)}u\bigl(t_n, \lambda(t_n)x+x(t_n)\bigr) \to Q_0 (x)  \quad \text{in} \quad L^2_{x}
\end{align*} 
holds for 
some sequence of times $t_n$ that approaches the blowup time.} However, these results are only sequential convergence results. A more precise result can be formulated as the following conjecture: 
\begin{conjecture}[Classification of minimal-mass blowup solutions]\label{co1v59}
Let $u_0 \in L_x^2(\mathbb{R}^d) $ satisfy $M_0(u_0)=M_0(Q_0)$.
	Suppose $u$ is a maximal-lifespan solution of \eqref{eq1.6v59} with maximal lifespan $I$, and it blows up in the sense that $\|u \|_{L_{t,x}^\frac{2(d+2)}d (I \times \mathbb{R}^d)} = \infty$. Then only the following two cases occur:
\begin{enumerate}
	\item $I = (- \infty, \infty)$ and $u$ must be a soliton solution, i.e. there exist $\lambda>0, \gamma\in [0,2\pi]$, and $ \left(\tilde{x},\xi \right)\in \mathbb{R}^d \times\mathbb{R}^d $ such that:
	\begin{align*}
		u(t,x) =e^{i\gamma-it |\xi|^2}e^{i\lambda^2t}e^{ix\cdot\xi}\lambda^\frac{d}2    Q_0
		\left(\lambda(x-2t\xi)-\tilde{x} \right).
	\end{align*}	
\item 	$I = (- \infty, T)$ or $(T, \infty)$ for some $T\in\R$, and ${u}$ must be a pseudo-conformal transformation of the soliton solution, i.e. there exist $\lambda>0, \gamma\in [0,2\pi]$, and $ \left(\tilde{x},\xi \right)\in \mathbb{R}^d \times\mathbb{R}^d $ such that:
	\begin{align*}
		u(t,x) = \frac{\lambda^\frac{d}2 }{|T-t|^\frac{d}2  }e^{i\gamma}e^{\frac{i
				|x-\xi|^2}{4(t- T )}}e^{ i \frac{\lambda^2}{t- T }}  Q_0\left(\frac{\lambda(x-\xi)-( T -t)\tilde{x}}{ T -t} \right).
	\end{align*}
\end{enumerate}

\end{conjecture}

This conjecture implies a trichotomy dynamics of the mass-critical focusing nonlinear Schr\"odinger equation when $M_0(u_0) = M_0(Q_0)$,
which has been extensively studied over the past few decades.
The first part of Conjecture \ref{co1v59} has been studied first. F. Merle \cite{M} showed that any $H_x^1$ solution with minimal mass that blows up in finite time
must coincide with the pseudo-conformal soliton \eqref{pseudotransform} modulo symmetries of the equation. The proof, which was later simplified by T. Hmidi and S. Keraani
\cite{HHK}, relies heavily on the assumption that the blowup time is finite. Thanks to the pseudo-conformal symmetry, this result also implies that solutions belonging to $\Sigma:=
\left\{f\in H^{1}_{x} : \, |x|f\in L_x^2 \right\}$ that have $M_0(u)=M_0(Q_0)$
and blow up in infinite time must be the solitary wave $e^{it\Delta}Q_0$ modulo  symmetries of the equation.
For the second part of Conjecture \ref{co1v59}, R. Killip, D. Li, M. Visan, and X. Zhang \cite{KLVZ} established under radial symmetry $(d \ge 4)$
 that solutions must be a soliton up to scale and phase rotation parameters, i.e.
\begin{equation*}
	u(t, x) =e^{i\gamma_0}e^{it}\lambda_0^{\frac{d}{2}} Q_0 (\lambda_0 x) \quad \mbox{ for some parameters } \gamma_0\in [0,2\pi] \mbox{ and } \lambda_0\in \mathbb{R}_+.
\end{equation*}
A vital step in their proof is to prove the compactness property  energy level $\dot{H}_x^1$. To prove such nontrivial property, some technical tools are needed, such as ``incoming/outgoing" decomposition for radial functions in $L_x^2$. Their approach has been successfully applied to other dispersive models, we refer to \cite{LZ28,Murphy}.
In \cite{LZ1}, D. Li and X. Zhang further gave the classification of the minimal-mass blowup solutions to \eqref{eq1.6v59} when $d \ge 2$ for radial initial data, where they relax the radial assumption to splitting-spherically symmetric assumption on the initial data  when $d \ge 4$. The approach in \cite{KLVZ,LZ1} heavily relies on radial/splitting-spherically symmetry assumption, so it seems to be quite difficult to adapt their strategy to the non-radial case.

We remind the readers that all of the above results need $H^1_x$ initial data. To prove the conjecture for $L_x^2$ initial data, D. Li and X. Zhang \cite{LZ0} proved the two-way non-scattering solutions in fact lie in $H_x^{1+}$ in the radial case for $d \ge 4$, and thus showed the two-way non-scattering minimal-mass blowup solutions must be solitons up to symmetries, we also refer to \cite{LZ,LZ2}. For related results on other dispersive models, we refer to \cite{LZeng}. Recently, B. Dodson \cite{D2,D1} resolved the conjecture completely for non-radial initial data when $1\leq d \leq 15$. 
To resolve the conjecture in the non-radial case in $L_x^2$, he developed the long time Strichartz estimate integrated with modulation theory. Specially, a long time Strichartz estimate adapting to the solutions around the orbit of ground state was developed to address the difficulty caused by the Galilean transformations and spatial translations. A pivotal contribution in \cite{D2,D1} lies in the design of an iterative framework. By combining this framework with modulation analysis, he established the almost monotonicity of the scale parameter—a result that cannot be achieved through compactness arguments alone.

\subsection{On the mass critical coupled nonlinear Schr\"odinger system}

\subsubsection{Finitely coupled NLSS}

    {Returning to the coupled nonlinear Schr\"odinger system \eqref{1.1}, it admits significantly more 
    soliton waves when $N<\infty$ compared to the nonlinear Schr\"odinger equation \eqref{eq1.6v59}. 
    Specifically, the NLSS \eqref{1.1} allows for standings waves of the form $(e^{i\lambda_1 t}Q_1,\cdots, e^{i\lambda_Nt}Q_N)$, where the parameters $\lambda_j$ are not all equal and $Q_j$ are strictly positive. Obviously,  $(Q_1,\cdots,Q_N)$ satisfies
\begin{equation}\label{multisw}
	\Delta_{\R^2}Q_j-\lambda_j Q_j=-2 \sum\limits_{k \in \mathbb{Z}_N } Q_k^2   Q_j - Q_j^3,\quad \forall \ j\in\Z_N.
\end{equation}
The existence of such solutions $(Q_1, \cdots, Q_N)$ for certain ranges of $(\lambda_1, \cdots, \lambda_N)$ was established by 
B. Sirakov \cite{Si}. 
This richer structure of solitons solutions introduces greater complexity in the dynamical behavior of solutions to \eqref{1.1} compared to the single-component case. 
}

{By analogy with the NLS \eqref{eq1.6v59}, we now consider 
the minimal mass that ensures the global well-posedness of the NLSS \eqref{1.1}.} 
N. V. Nguyen, R. Tian, B. Deconinck, and N. Sheils \cite{NTDS} studied the best constant of the corresponding vector-valued Gagliardo-Nirenberg inequality  of \eqref{1.1}. This is further studied by the first author together with Z. Guo, G. Hwang, and H. Yoon \cite{CGHY}. Now let us briefly discuss the best constant $C_{GN}$ of the Gagliardo-Nirenberg inequality related to the system \eqref{1.1}:
\begin{equation*}
	\sum_{j\in\Z_N}\int_{\R^2}  F_j(\mathbf{f})\cdot \bar{f}_j \, \mathrm{d} x\leq C_{GN}\|\mathbf{f}\|^2_{L^2_x l^2}\|\nabla\mathbf{f}\|^2_{L^2_x l^2} .
\end{equation*}
We define Weinstein functional of \eqref{1.1} as follows:
\begin{equation}\label{qzmz}
	J(\mathbf{f})=\frac{\sum\limits_{j\in\Z_N}\int_{\R^2}  F_j(\mathbf{f})\cdot \bar{f}_j \, \mathrm{d} x}{\|\mathbf{f}\|^2_{L^2_x l^2}\|\nabla\mathbf{f}\|^2_{L^2_x l^2}} ,
\end{equation}
which is invariant under homogeneity and scaling symmetry, that is, $J(\mathbf{u})=J( \mathbf{u}^{\lambda,\mu})$ where $\mathbf{u}^{\lambda,\mu}(x)=\mu \mathbf{u}(\lambda x)$ for any $\mu, \lambda>0$. Then by standard variational argument,
a maximizer $\mathbf{Q}$ of the Weinstein functional $J$ weakly solves the system of Euler-Lagrange
equations
\begin{equation}\label{EL}
	\Delta_{\R^2}\mathbf{Q}-\mathbf{Q}=-\mathbf{F}(\mathbf{Q}).
\end{equation}
 Using the fact that $\big\|\nabla |f|\big\|_{L_x^2}\leq \|\nabla f\|_{L_x^2}$ and then using Schwartz rearrangement, we can assume that each component of $\mathbf{Q}$ is non-negative, radial and decreasing function. When $N=2$, obviously, $(0, Q_0)$, $(Q_0,0)$ and $ \left(\sqrt{\frac{1}{3}}Q_0,\sqrt{\frac{1}{3}}Q_0 \right)$ are the non-negative radial decreasing solutions of \eqref{EL}. In fact, J. Wei and W. Yao \cite{WY} proved that they are the only non-negative  decreasing solutions of the system \eqref{EL} up to scaling and space translation. Since $J \left( \left(\sqrt{\frac{1}{3}}Q_0,\sqrt{\frac{1}{3}}Q_0 \right) \right)> J((0,Q_0))=J((Q_0,0))$, we see that $\mathbf{Q}= \left(\sqrt{\frac{1}{3}}Q_0,\sqrt{\frac{1}{3}}Q_0 \right)$. Since J. Wei and W. Yao \cite{WY}'s result can be  extended to the $3\leq N< \infty$ case, then one can also prove that $\mathbf{Q}= \left(\sqrt{\frac{1}{2N-1}}Q_0,\cdots,\sqrt{\frac{1}{2N-1}}Q_0 \right)$, and we have
 \begin{equation*}
 	C_{GN}=J \left( \left(\sqrt{\frac{1}{2N-1}} Q_0 ,\cdots,\sqrt{\frac{1}{2N-1}} Q_0  \right) \right)=\frac{2N-1}{N}\frac{\|Q_0\|^4_{L_x^4}}{\|Q_0\|^2_{L_x^2}\|\nabla Q_0\|^2_{L_x^2}}=\frac{2(2N-1)}{N}\|Q_0\|^{- 2}_{L_x^2}.
 \end{equation*}
Thus, the following sharp  Gagliardo-Nirenberg inequality holds:
\begin{lemma}[Sharp Gagliardo-Nirenberg inequality]\label{le1.1v19}
For any $\mathbf{u}\in H^{1}_{x} l^2(\R^2\times\Z_N)$, we have
\begin{equation}\label{SGN}
	\sum_{j\in\Z_N}\int_{\R^2}  F_j(\mathbf{u})\cdot \bar{u}_j \, \mathrm{d} x\leq \frac{2(2N-1)}{N}\|Q_0\|^{- 2}_{L_x^2}\|\mathbf{u}\|^2_{L^2_x l^2}\|\nabla\mathbf{u}\|^2_{L^2_x l^2},
\end{equation}
the equality holds if and only if $\mathbf{u}= \left(e^{i\gamma_1}\lambda \sqrt{\frac{1}{2N-1}}Q_0 (\lambda x+\tilde{x}),\cdots,e^{i\gamma_N}\lambda\sqrt{\frac{1}{2N-1}} Q_0 (\lambda x+\tilde{x}) \right)$ for some $(\gamma_1,\cdots,\gamma_N)\in [0,2\pi]^N,\lambda>0$ and $\tilde{x}\in \mathbb{R}^2$.
\end{lemma}
From now on, we call $\mathbf{Q}= \left(\sqrt{\frac{1}{2N-1}}Q_0,\cdots,\sqrt{\frac{1}{2N-1}}Q_0 \right)$ the ground state of  system  \eqref{1.1}.  By \eqref{SGN} and conservation laws, for any initial data $\mathbf{u}_0\in H^{1}_{x} l^2$ with  $\|\mathbf{u}_0\|_{L^2_x l^2}< \sqrt{\frac{N}{2N-1}}\|Q_0 \|_{L_x^2}$, then $\|\mathbf{u}(t)\|_{H^{1}_{x} l^2}\lesssim 1$ uniformly with $t$, which yields the solution of system \eqref{1.1} is global well-posed. One may also expect that for merely $L^2_x l^2$ initial data with $\|\mathbf{u}_0\|_{L^2_x l^2}< \sqrt{\frac{N}{2N-1}}\|Q_0 \|_{L_x^2}$, then the solution of system \eqref{1.1} is global well-posed and scatters. In fact, in \cite{CGHY}, the first author together with Z. Guo, G. Hwang, and H. Yoon have proved global well-posedness and scattering of \eqref{1.1} when the mass of the initial data is strictly less than the mass of $\mathbf{Q}$. Thus, we can see that the soliton $e^{it} \mathbf{Q}$ and the pseudo-conformal transformation of the soliton $e^{it} \mathbf{Q}$ are the minimal-mass blowup solutions among all non-scattering solutions of \eqref{1.1}. It is natural to characterize the minimal-mass blowup solutions of \eqref{1.1}. We will address this problem later in this paper. 

\subsubsection{Infinitely coupled NLSS}

For the infinitely coupled nonlinear Schr\"odinger system \eqref{1.1}, 
by using the approximation of the best constant of the sharp Gagliardo-Nirenberg inequality related to the finitely coupled NLSS in Lemma \ref{le1.1v19}, the first author together with Z. Guo, G. Hwang, and H. Yoon \cite{CGHY}  proved the following sharp Gagliardo-Nirenberg inequality: 

\begin{lemma}[Sharp Gagliardo-Nirenberg inequality for the infinite vector function]
For $\mathbf{u} \in H^1_x l^2(\RR^2\times\ZZ)$, we have
\begin{align}\label{GN-infity}
\sum_{j\in\Z}\int_{\R^2}  F_j(\mathbf{u})\cdot \bar{u}_j \,\mathrm{d}x  \leq\frac4{\|Q_0\|_{L_x^2(\RR^2)}^2} \|\nabla \mathbf{u}\, \|_{L_x^2 l^2(\RR^2\times \ZZ)}^2 \|\mathbf{u}\,\|_{L_x^2 l^2(\RR^2\times \ZZ)}^2.
\end{align}

\end{lemma}
Then by the concentration-compactness/rigidity theorem method, they gave global well-posedness and scattering of \eqref{1.1} when the mass of the initial data is strictly less than $\frac12 \|Q_0 \|_{L_x^2}^2$.

In summary, in \cite{CGHY}, they proved scattering in the region $\mathcal{S}:= \{ \mathbf{u}_0:  M(\mathbf{u}_0) < M_N \}$, where $M_N = {\frac{N}{2N-1}} \|Q_0\|^2_{L_x^2}$ when $N< \infty$ and $M_\infty = {\frac12} \|Q_0 \|^2_{L_x^2}$ when $N= \infty$. 

\subsection{Main results.}
In this paper, we will  give the dynamics of the solution to \eqref{1.1} when the mass is equal to $M_N$ for $N< \infty$, and when it is equal to or even slightly greater than $M_\infty$ for $N= \infty$. 
\subsubsection{Finitely coupled NLSS}
First, we give a classification of solutions to \eqref{1.1} for $N< \infty$ when the mass of the initial data is equal to the mass of $\mathbf{Q}$, where
\begin{align}\label{1.12}
\mathbf{Q} = (Q, \cdots, Q) = \Big( \sqrt{\tfrac1{2N-1}} Q_0, \cdots, \sqrt{\tfrac1{2N-1}} Q_0   \Big) .
\end{align}
We prove that any non-scattering solution to NLSS must coincide with
\begin{align*}
e^{it}\mathbf{Q} \text{ or }
\frac1{t}  \mathbf{Q} \left( \frac{x}t \right)  e^{i \frac{|x|^2-4}{4t}}
\end{align*}
up to the symmetries \eqref{sym1}-\eqref{sym5} in the critical space $L_x^2l^2$. In particular, 
we exclude 
the existence of standing waves of the form $$(e^{i\lambda_1t}Q_1,\cdots, e^{i\lambda_N t}Q_N),\quad \text{ where } \lambda_j\mbox{ are not all equal, }$$ 
at the threshold mass $M(\mathbf{Q})$.
\begin{theorem}[Classification of minimal-mass blowup solutions to NLSS for $N< \infty$] \label{main1}

Let $\mathbf{u}$ be a solution to \eqref{1.1} with mass $M(\mathbf{u})=M(\mathbf{Q})$, and let $I$ be its  maximal lifespan. 
Assume $\mathbf{u}$ blows up in the sense that $\|\mathbf{u} \|_{L_{t,x}^4 l^2 ( I \times \mathbb{R}^2 \times \mathbb{Z}_N)} = \infty$. Then only the following two scenarios occur:
\begin{enumerate}
 \item $I = \mathbb{R}$ and $\mathbf{u}$ is a soliton solution, i.e. there exist $\lambda>0, (\gamma_1,\cdots,\gamma_N)\in [0,2\pi]^N$ and $ (\tilde{x},\xi)\in \mathbb{R}^2\times\mathbb{R}^2$ such that
\begin{align}\label{soliton}
	u_j(t,x) =e^{i\gamma_j-it |\xi|^2}e^{i\lambda^2t}e^{ix\cdot\xi}\lambda  \sqrt{ \frac1{2N-1}}  Q_0
    \left(\lambda(x-2t\xi)-\tilde{x} \right),\quad \forall \, j\in\mathbb{Z}_N.
\end{align}
\item $I = (- \infty, T)$ or $(T, \infty)$, and $\mathbf{u}$ is a pseudo-conformal transformation of the soliton solution, i.e. there exist $\lambda>0, \left(\gamma_1,\cdots,\gamma_N \right)\in [0,2\pi]^N$ and $ \left(\tilde{x},\xi \right)\in \mathbb{R}^2\times\mathbb{R}^2$ such that
\begin{align}\label{pseudosoliton}
	u_j(t,x) = \frac{\lambda}{| T-t | }e^{i\gamma_j}e^{\frac{i
    |x-\xi|^2}{4(t-T )}}e^{ i \frac{\lambda^2}{t- T }} \sqrt{ \frac1{2N-1}} Q_0\left(\frac{\lambda(x-\xi)-(T-t)\tilde{x}}{ T -t} \right),\quad\forall \,  j\in\mathbb{Z}_N.
\end{align}
\end{enumerate}

\end{theorem}

\begin{remark} For the following system of equations with general coupling coefficients
\begin{equation*}
		\begin{cases}
			i\partial_t u_1 + \Delta_{\mathbb{R}^2} u_1 = -  \alpha_{11}|u_1|^2 u_1 -   \alpha_{12}|u_2|^2 u_1-\cdots-\alpha_{1N}|u_N|^2 u_1,\\
			i\partial_t u_2 + \Delta_{\mathbb{R}^2} u_2 = -  \alpha_{21}|u_1|^2 u_2 -   \alpha_{22}|u_2|^2 u_2-\cdots-\alpha_{2N}|u_N|^2 u_2,\\
            \qquad\qquad\qquad\qquad\qquad\qquad\qquad\vdots\\
            i\partial_t u_N + \Delta_{\mathbb{R}^2} u_N = -  \alpha_{N1}|u_1|^2 u_N -   \alpha_{N2}|u_2|^2 u_N-\cdots-\alpha_{NN}|u_N|^2 u_N,
		\end{cases}
	\end{equation*}
  where $\alpha_{ij}>0$ for $i,j=1,\cdots,N$, the analysis of the existence/nonexistence of ground states, as well as the investigation of their uniqueness and nondegeneracy properties when such ground states exist, constitutes a matter of profound significance in this field. For a comprehensive exploration of these topics, we refer the reader to \cite{CZou, DW, JLLLW, LW, NTDS, PWW, SWX,Si, WW,WY} and the references therein. Notably, under the assumptions that the ground state is uniquely determined and  nondegenerate, the rigidity properties of the corresponding Schr\"odinger system can be rigorously established by adapting the methodology used in our proof. 
  
\end{remark}

\subsubsection{Infinitely coupled NLSS}

Next, we turn to the infinitely coupled nonlinear Schr\"odinger system, that is \eqref{1.1} for $N = \infty$. Let \begin{equation*}
		C(M)=\begin{cases}
			2\sum\limits_{j=1}^{[\frac{M}{2}]}j^2  ,  \quad\text{if}\; M \mbox { is odd}; \\
            \\
		2\sum\limits_{j=1}^{\frac{M}{2}}j^2-\left(\frac{M}{2} \right)^2,	\quad\text{if}\; M \mbox { is even}.
		\end{cases}
	\end{equation*}
Now we present the global well-posedness and scattering of the cubic focusing infinitely coupled
nonlinear Schr\"odinger system in $L_x^2 h^1(\mathbb{R}^2 \times \mathbb{Z})$:

\begin{theorem}[Global well-posedness and scattering for the cubic focusing 
infinitely coupled NLSS]\label{th1.2}
Given any positive integer $M\geq 2$. Suppose that $\mathbf{u}_0\in L_x^2h^1(\R^2\times\Z)$ satisfies 
\begin{equation*}\|\mathbf{u}_0\|^2_{L_x^2\dot{h}^1(\R^2\times\Z)}\leq C(M)\cdot \frac{2M-1}{2M}\|\mathbf{u}_0\|^2_{L_x^2l^2(\R^2\times\Z)}\footnotemark\quad \mbox{ and }\quad \|\mathbf{u}_0\|^2_{L_x^2l^2(\R^2\times\Z)}< \frac{M}{2M-1}\|Q_0\|^2_{L_x^2(\R^2)}. 
\end{equation*}
\footnotetext{Here $L_x^2 \dot{h}^1(\R^2\times\Z)$ is defined in in Subsection \ref{subse1.6v33}.}
Then there exists a unique global solution $\mathbf{u}(t)$ to \eqref{1.1} satisfying
\begin{equation*}
\left\|\mathbf{u}(t)\right\|_{L_{t,x}^4 l^2(\mathbb{R}\times \mathbb{R}^2  \times \mathbb{Z})}  \lesssim_{\|\mathbf{u}_0\|_{L_x^2l^2}} 1.
\end{equation*}
Furthermore, the solution scatters in $L_x^2 h^1$ in the sense that there exists $\mathbf{u}^{\,\pm} \in  L^2_x h^1(\R^2\times\Z)$
such that
	\[
		\left\|\mathbf{u}-e^{it\Delta_{\RR^2}}\mathbf{u}^{\,\pm} \right\|_{L^2_x h^1(\RR^2\times\ZZ)}
		:=\bigg\| \bigg( \sum\limits_{j\in \mathbb{Z}} \langle j\rangle^2 \big|  u_j(t) - e^{it\Delta_{\mathbb{R}^2}} u_j^{\pm  } \big|^2 \bigg)^\frac12  \bigg\|_{L_x^2(\mathbb{R}^2 )} \to 0, \text{ as } t\to \pm \infty.
	\]
\end{theorem}

\begin{remark}
Given any  $\mathbf{u}_0\in L_x^2h^1$ that satisfies $M(\mathbf{u}_0)=\frac{1}{2}\|Q_0\|^2_{L_x^2}$, then there must exist $M_0\geq2$ such that $\|\mathbf{u}_0\|^2_{L_x^2\dot{h}^1}\leq C(M_0)\cdot \frac{2M_0-1}{2M_0}\|\mathbf{u}_0\|^2_{L_x^2l^2}$. Therefore, Theorem \ref{th1.2} directly implies that  there does not exist minimal-mass blowup solution for \eqref{1.1} in $L_x^2h^1(\R^2\times\Z)$. 

\end{remark}

As a consequence of Theorem \ref{th1.2}, by the argument as in the proof of \cite[Theorem 3.9]{CGYZ}, we can obtain the global well-posedness and scattering of the large-scale solution of the focusing cubic NLS on $\mathbb{R}^2 \times \mathbb{T}$ in $L_x^2 H_y^1$, where $\mathbb{T} = \mathbb{R}/2\pi \mathbb{Z}$.

\begin{theorem}\label{le3.11v63old}
Suppose that $\phi \in L^2_x H_y^1(\mathbb{R}^2\times \mathbb{T})$ satisfies   $\| \phi \|^2_{L_{x}^2\dot{H}_y^1(\mathbb{R}^2\times \mathbb{T}) } \leq C(M)\cdot \frac{2M-1}{2M}  \|\phi \|^2_{L^2_{x,y}(\mathbb{R}^2\times \mathbb{T})}$ and $\| \phi \|^2_{L_{x,y}^2(\mathbb{R}^2\times \mathbb{T}) } < \frac{M}{2M-1}\|Q_0\|^2_{L_x^2}$ for some positive integer $M\geq 2$.
Then there is a sufficiently large constant $\lambda_0=  \lambda_0 (\phi)$  such that  for $\lambda \ge \lambda_0$, we have a unique global solution
$U_\lambda \in C_t^0 L_x^2 H_y^1(\mathbb{R} \times \mathbb{R}^2 \times \mathbb{T})$ of
\begin{equation*}
\begin{cases}
i \partial_t U_\lambda + \Delta_{\mathbb{R}^2 \times \mathbb{T}} U_\lambda = - \left|U_\lambda\right|^2 U_\lambda,
\\
U_\lambda(0,x,y) = \frac1\lambda \phi\left( \frac{x}\lambda, y \right).
\end{cases}
\end{equation*}
Moreover, for $\lambda \ge \lambda_0$, we have
\begin{align*}
\left\|U_\lambda \right\|_{L_t^\infty L_x^2 H_y^1 \cap L_{t,x}^4 H_y^1(\mathbb{R} \times \mathbb{R}^2 \times \mathbb{T})} \lesssim 1.
\end{align*}
As a consequence, $U_\lambda$ scatters in $L_x^2 H_y^1$ in the sense that there exist $ \left\{ U_\lambda^\pm \right\} \in L_x^2 H_y^1 $ such that
\begin{align*}
\big\| U_\lambda (t) - e^{it \Delta_{\mathbb{R}^2 \times \mathbb{T}} } U_\lambda^\pm \big\|_{L_x^2 H_y^1 } \to 0, \text{ as } t \to \pm \infty.
\end{align*}
\end{theorem}

Let
\begin{align*}
\tilde{U}_\lambda (t,x,y) =  \lambda  U_\lambda ( \lambda^2 t , \lambda x, y),
\end{align*}
we have
\begin{align*}
\begin{cases}
i \partial_t \tilde{U}_\lambda + \Delta_{\mathbb{R}^2 \times { \mathbb{T}_{ \lambda^{-1} } } } \tilde{U}_\lambda = - \big|\tilde{U}_\lambda \big|^2 \tilde{U}_\lambda, \\[0.5em]
\tilde{U}_\lambda(0,x,y)  = \phi(x,\lambda y).
\end{cases}
\end{align*}
Thus, we get the global well-posedness and scattering in $L_x^2 H_y^1(\mathbb{R}^2 \times \mathbb{T}_{\lambda^{-1} } )$ of the focusing cubic NLS on the small cylinder $\mathbb{R}^2 \times  \mathbb{T}_{ \lambda^{-1} } $, when $\lambda$ is sufficiently large.
\begin{theorem}\label{le3.11v63}
Suppose that $\phi \in L^2_x H_y^1(\mathbb{R}^2\times \mathbb{T})$ satisfies   $\| \phi \|^2_{L_{x}^2\dot{H}_y^1(\mathbb{R}^2\times \mathbb{T}) } \leq C(M)\cdot \frac{2M-1}{2M}  \|\phi \|^2_{L^2_{x,y}(\mathbb{R}^2\times \mathbb{T})}$ and $\| \phi \|^2_{L_{x,y}^2(\mathbb{R}^2\times \mathbb{T}) } < \frac{M}{2M-1}\|Q_0\|^2_{L_x^2}$ for some positive integer $M\geq 2$. Then there is $\lambda_0=  \lambda_0 (\phi)$ sufficiently large such that for $\lambda \ge \lambda_0$, we have a unique global solution
$U_\lambda \in C_t^0 L_x^2 H_y^1(\mathbb{R} \times \mathbb{R}^2 \times  \mathbb{T}_{ \lambda^{-1}})$ of
\begin{equation*}
i \partial_t U_\lambda + \Delta_{\mathbb{R}^2 \times  \mathbb{T}_{ \lambda^{-1} }} U_\lambda = -  \left|U_\lambda \right|^2 U_\lambda,
\end{equation*}
with $U_\lambda(0,x,y) = \phi(x,\lambda y)$.
Moreover, for $\lambda \ge \lambda_0$,
\begin{align*}
\|U_\lambda \|_{L_t^\infty L_x^2 H_y^1 \cap L_{t,x}^4 H_y^1(\mathbb{R} \times \mathbb{R}^2 \times  \mathbb{T}_{  \lambda^{-1} })} \lesssim 1.
\end{align*}
As a consequence, $U_\lambda$ scatters to the solution of the linear equation $ i \partial_t V_\lambda + \Delta_{\mathbb{R}^2 \times  \mathbb{T}_{ \lambda^{-1} } } V_\lambda = 0$ in $L_x^2 H_y^1(\mathbb{R}^2 \times  \mathbb{T}_{ \lambda^{-1}}) $ when $\lambda$ is sufficiently large.

\end{theorem}

\subsection{Summary of the proof. }
In this subsection, we will give a sketch of the proof of Theorem \ref{main1} and Theorem \ref{th1.2}.

\textbf{ Part I. Proof of Theorem \ref{main1}.} 
The basic idea of this part is similar to that of \cite{D2,D1}. To prove Theorem \ref{main1}, it suffices to consider the solutions that blows up forward in time. The proof of Theorem \ref{main1} reduces to the resolution of the following theorem:
\begin{theorem}\label{t2.3}
Let $0 < \eta_{\ast} \ll 1$ be a small, fixed constant to be defined later. Suppose  $\mathbf{u}$ is a  solution to $(\ref{1.1})$ on the maximal interval of existence $I \subset \mathbb{R}$, with $ M(\mathbf{u}) =  M(\mathbf{Q})$. If $\mathbf{u}$ blows up forward in time, and satisfies 
\begin{equation}\label{f2.32}
\sup_{t \in [0, \sup(I))}\inf_{\substack{ \left(\lambda,\tilde{x},\xi \right)\in \mathbb{R}_{+}\times\mathbb{R}^2\times\mathbb{R}^2,
\\(\gamma_1,\cdots, \gamma_N) \in [0, 2 \pi]^N}} \sum_{j\in\mathbb{Z}_N}
\left\|  e^{i \gamma_j}e^{ix\cdot\xi}\lambda u_{j} \left(t, \lambda  \left(x+\tilde{x} \right) \right) - Q  (x) \right\|^2_{L_x^{2}} \leq \eta_{\ast},
\end{equation}
then $\mathbf{u}$ is either a soliton solution of the form \eqref{soliton} or a pseudoconformal transformation of a soliton of the form \eqref{pseudosoliton}.

\end{theorem}
\begin{remark}
The constant $0 < \eta_{\ast} \ll 1$ will be chosen as a small fixed quantity  determined by some other theorems in the rest of the  paper.
\end{remark}

To prove Theorem \ref{main1} from Theorem \ref{t2.3}, we show that $\mathbf{u}(t)$ converges to $\mathbf{Q}$ along a subsequence up to translation, Galilean boost, phase rotation, and scaling. The result is presented in the following theorem: 
\begin{theorem}\label{claim2.1}
	Let $\mathbf{u}$ be a blowup forward in time solution to \eqref{1.1} when $N< \infty$ satisfying $ M(\mathbf{u}) =  M(\mathbf{Q})$. Let $I$ be the maximal lifespan of the solution $\mathbf{u}$. Then there exists a sequence $t_{n} \nearrow \sup I $ as $n \to \infty$ and a family of parameters $\lambda_{n} > 0$, $ \left(\tilde{x}_n,\xi_n \right)\in \mathbb{R}^2\times\mathbb{R}^2$, and $(\gamma_{1,n},\cdots,\gamma_{N,n}) \in [0,2\pi]^N$ such that
\begin{equation}\label{f2.31}
		\lim_{n\to\infty}\sum_{j\in\mathbb{Z}_N}
\left\|   e^{i\gamma_{j,n}}e^{ix\cdot\xi_n} \lambda_{n} u_j \left(t_{n}, \lambda_{n}  \left(x+\tilde{x}_n \right) \right) -Q (x) \right\|^2_{L_x^2}=0.
	\end{equation}
\end{theorem}

The proof of this theorem is done by following the argument in \cite{D4,D3} with most of the concrete computations have been rigorously established in \cite{CGHY}. For brevity, we relegate the detailed proof to the appendix. 
{Once this convergence is established, Theorem \ref{main1} follows. }

We will now give an informal description of the proof of Theorem \ref{t2.3}. 
{The proof relies on the following key  estimates}.

\subsubsection{Modulation analysis}
Using the modulation analysis,  there exist almost everywhere differentiable parameter functions
\begin{equation*}
	( \gamma_1(t), \cdots, \gamma_N(t)) : I\rightarrow [0,2\pi]^N,\quad \xi(t): I\rightarrow \mathbb{R}^2 ,\quad x(t): I\rightarrow \mathbb{R}^2 ,\quad \lambda(t): I\rightarrow \mathbb{R}_+ ,
\end{equation*}
such that for each $j \in \mathbb{Z}_N$,
\begin{equation*}
	e^{i\gamma_j (t)} e^{ix\cdot\xi(t)}\lambda(t)u_j  \big(t,\lambda(t)x+x(t) \big) = Q(x)+\epsilon_j (t,x) 
\end{equation*}
with $\boldsymbol{\epsilon}(t) = (\epsilon_1(t), \cdots, \epsilon_N(t)) $ satisfies the following orthogonality conditions: 
\begin{equation}\label{ql}
\aligned
	& \langle \boldsymbol{\epsilon}(t), \boldsymbol{\chi}_0\rangle_{L^2_x l^2 }=\langle \boldsymbol{\epsilon}(t), i\boldsymbol{\chi}_{0,j} \rangle_{L^2_x l^2 }
= \langle \boldsymbol{\epsilon}(t),{ \mathbf{Q}}_{x_1}\rangle_{L^2_x l^2 } = \langle \boldsymbol{\epsilon}(t),  \mathbf{Q}_{x_2}\rangle_{L^2_x l^2 }
\\
& =\langle \boldsymbol{\epsilon}(t), i\mathbf{Q}_{x_1}\rangle_{L^2_x l^2 }= \langle \boldsymbol{\epsilon}(t), i\mathbf{Q}_{x_2 }\rangle_{L^2_x l^2 }=0.
\endaligned
\end{equation}
Here,     $ i\mathbf{Q}_{x_j}=(iQ_{x_j},\cdots,iQ_{x_j}), \mathbf{Q}_{x_j}=(Q_{x_j},\cdots,Q_{x_j})$ for $  j\in\{1,2\}$, $\boldsymbol{\chi}_0=(\chi_0,\cdots,\chi_0)$ and $i\boldsymbol{\chi}_{0,j}=(0,\cdots, 0,\underset{j\text{-th }}{i\chi_0},0,\cdots,0)$ for $ j\in \mathbb{Z}_N$, where $\chi_0$ is the  eigenvector corresponding to the unique negative eigenvalue of the linearized operator $L_{0,+}:=-\Delta+1-3Q_0^2$.  We remind readers that the coupled terms allow for additional phase rotation symmetry, which is reflected in the different parameters $\gamma_j$. By  local theory, if $\sup I=\infty$, then Theorem \ref{t2.3} holds as long as $\epsilon(t_0)=0$ at some time $t_0\in I$.
Let $s(t)=\int_0^t 
{\lambda(\tau)^{-2}} \, \mathrm{d}\tau$, then we derive the following controls of parameters:
\begin{equation}\label{qm1}
	\int_{a}^{a + 1} \left|\sum_{j\in\mathbb{Z}_N}\left((\gamma_{j})_s + 1 - \frac{x_{s}}{\lambda} \cdot\xi(s) - |\xi(s)|^{2} \right) \right| \, \mathrm{d} s
\lesssim \int_{a}^{a + 1}  \left\| \boldsymbol{\epsilon}(s)  \right\|_{L^2_{x} l^2}^{2} \, \mathrm{d} s,
\end{equation}
\begin{equation}\label{qm2}
	\int_{a}^{a + 1}  \left|(\gamma_{j})_s + 1 - \frac{x_{s}}{\lambda} \cdot\xi(s) - |\xi(s)|^{2}  \right|  \, \mathrm{d} s
 \lesssim \int_{a}^{a + 1}  \left\| \boldsymbol{\epsilon}(s)  \right\|_{L^2_{x} l^2} \, \mathrm{d} s,\ \forall \ j\in \Z_{N},
\end{equation}
\begin{equation}\label{qm3}
	\int_{a}^{a + 1}  \left|\xi_{s} - \frac{\lambda_{s}}{\lambda} \xi(s) \right| \, \mathrm{d} s
\lesssim \int_{a}^{a + 1}  \left\| \boldsymbol{\epsilon}(s)  \right\|_{L^2_{x} l^2}^{2} \, \mathrm{d} s,
\end{equation}
\begin{equation}\label{qm4}
	\int_{a}^{a + 1}  \left|\frac{\lambda_{s}}{\lambda} \right| \, \mathrm{d} s
\lesssim \int_{a}^{a + 1}  \left\| \boldsymbol{\epsilon}(s)  \right\|_{L^2_{x} l^2} \,  \mathrm{d} s, 
\end{equation}
and
\begin{equation}\label{qm5}
	\sup_{s \in [a, a + 1]} \| \boldsymbol{\epsilon}(s) \|_{L_{x}^{2}l^2} \sim \inf_{s \in [a, a + 1]} \| \boldsymbol{\epsilon}(s) \|_{L_{x}^{2}l^2}.
\end{equation}
These estimates will be crucial for the bootstrap argument in the proof.

\subsubsection{A priori estimate of the reminder $\|\boldsymbol{\epsilon}(t)\|_{L_{x}^2l^2}$}

Thanks to the spectral properties of the linearized operators $L_{+}$ and $L_{-}$, the energy of $\mathbf{Q}+\boldsymbol{\epsilon}$ satisfies the following positivity inequality in the inhomogeneous Sobolev space $H_x^1l^2$: 
\begin{equation}\label{qc0}
		E \left(\mathbf{Q} + \boldsymbol{\epsilon} \right) \gtrsim  \left\| \boldsymbol{\epsilon}(x )  \right\|_{H^{1}_x  l^2 \left(\mathbb{R}^2 \times\mathbb{Z}_N \right)}^{2} .
	\end{equation}
Since we are actually working in the critical space $L_x^2l^2$, it is natural to consider the low-frequency truncation of the solution $\mathbf{u}$. We aim to recover positive definiteness estimates similar to the one above. However, the energy of the low-frequency truncation of the solution is no longer conserved, necessitating an analysis of the energy change over an interval. To achieve this, we establish estimates for the high-frequency part of the solution, specifically the long time Strichartz estimates:
    \begin{align}\label{qc}
& \left\| P_{\geq k_0} \mathbf{u}  \right\|_{U_{\Delta}^{2} \left(J,L_x^2l^2 \right)} +   \left\|  \left|P_{\geq k_0} \mathbf{u} \right|\cdot
\left|P_{\leq k_0 - 3} \mathbf{u} \right|  \right\|_{L_{t,x}^{2}l^2 \left(J \times \mathbb{R}^{2}\times\mathbb{Z}_N \right)} \notag \\
& \lesssim  \left(\frac{1}{T} \int_{J}   \left\| \boldsymbol{\epsilon}(t)  \right\|_{L^2_{x} l^2}^{2} \lambda(t)^{-2} \, \mathrm{d} t \right)^{1/2} + 
{T^{-10}},
\end{align}
where $T=\int_J 
{\lambda(t)^{-2} } \, \mathrm{d} t$ and $2^{3k_0}=T$. 
This type of estimate was introduced by B. Dodson. For the system \eqref{1.1}, the presence of coupling terms prevents the direct application of bilinear estimates, as in \cite{D2}, to each component. Instead, we employ interaction 
Morawetz estimates to derive a bilinear estimate for the system \eqref{1.1}, which then allows us to establish the long time Strichartz estimate \eqref{qc}. It is important to note that \eqref{qc} is \emph{not} an  priori  estimate for arbitrary solutions of $\eqref{1.1}$; rather, it captures additional information arising from the fact that $\mathbf{u}$ is close to the orbit of $\mathbf{Q}$. Notice that if $\mathbf{u}$ is a soliton, then \eqref{qc} can be obtained by the fact that $Q$ is smooth with rapidly decreasing derivatives. Now with the help of \eqref{qc}, we will demonstrate that the energy increment over the interval $J$ can be controlled by the integral average of the remainder $\|\boldsymbol{\epsilon}(t)\|_{L_{x}^2l^2}$ with  the rapidly decaying term, that is,
\begin{equation*}
\sup_{t \in J} E \left(P_{\leq k _0+ 9} \mathbf{u}(t) \right)
 \lesssim \frac{2^{2k_0}}{T} \int_{J}  \left\| \boldsymbol{\epsilon}(t)  \right\|_{L^2_{x} l^2}^{2} \lambda(t)^{-2}  \, \mathrm{d} t
 + \sup_{t \in J} \frac{|\xi(t)|^{2}}{\lambda(t)^{2}} + 2^{2k_0} T^{-10}.
\end{equation*}
Finally, using the above estimates, we recover the initial inequality \eqref{qc0} for the low-frequency truncation of the solution in the following sense:
\begin{align*}
& \sup_{t \in J}  \left\|\nabla P_{\leq k _0+ 9} \left(e^{-i \gamma_j(t)} e^{-ix \cdot \frac{\xi(t)}{\lambda(t)}} \frac{1}{\lambda(t)} \epsilon_j \left(t, \frac{x - x(t)}{\lambda(t)} \right) \right)  \right\|_{L_x^2}^{2}
\notag \\
& \lesssim \frac{2^{2k_0}}{T} \int_{J} \left\| \boldsymbol{\epsilon}(t)  \right\|_{L^2_{x} l^2}^{2} \lambda(t)^{-2} \, \mathrm{d} t
+ \sup_{t \in J} \frac{|\xi(t)|^{2}}{\lambda(t)^{2}} + 2^{2k_0} T^{-10}
\end{align*}
and
\begin{equation}\label{qccc2}
\sup_{t \in J} \left\| \boldsymbol{\epsilon}(t)  \right\|_{L^2_{x} l^2}^{2}
 \lesssim \frac{2^{2k_0} T^{1/50}}{\eta_{1}^{2} T} \int_{J} \left\| \boldsymbol{\epsilon}(t) \right\|_{L^2_{x} l^2}^{2} \lambda(t)^{-2} \, \mathrm{d} t
 + \frac{T^{1/50}}{\eta_{1}^{2}} \sup_{t \in J} \frac{|\xi(t)|^{2}}{\lambda(t)^{2}} + 2^{2k_0} \frac{T^{1/50}}{\eta_{1}^{2}} T^{-10}, 
\end{equation}
where $\eta_1$ is a small fixed constant.

\subsubsection{Virial  inequality}

Let $\chi(r) \in C^{\infty}([0, \infty))$ be a smooth, radial function such that $\chi(r) = 1$ for $0 \leq r \leq 1$,
and $\chi(r)$ is supported on $r \leq 2$. Define 
\begin{equation*}
\phi(r) = \int_{0}^{r} \psi \left(\frac{\eta_1 s}{2R} \right)  \, \mathrm{d} s, \quad\text{ where } \psi(x):=\chi^2(|x|).
\end{equation*}
Consider the time derivative of the Morawetz potential,
\begin{equation*}
\frac{d}{d t}M(t)= \frac{d}{d t}\sum_{j\in\mathbb{Z}_N}\int \phi(|x|) \frac{x}{|x|} \cdot \Im \left[\overline{P_{\leq k _0+ 9} u_j} \nabla P_{\leq k _0+ 9} u_j \right](t,x) \,   \mathrm{d} x. 
\end{equation*}
It includes an energy structure, then under some assumptions on the modulation parameters, we can utilize the estimates established in  previous subsections to derive a virial inequality of the following form:
\begin{align}\label{qn}
 & \int_a^b \left\| \boldsymbol{\epsilon}(t) \right\|_{L^2_{x} l^2}^{2} \lambda(t)^{-2} \, \mathrm{d} t \notag
\\
& \leq 3 \sum_{j\in\mathbb{Z}_N} \langle w_j(a), Q + x \cdot \nabla Q \rangle_{L_x^{2}} - 3\sum_{j\in\mathbb{Z}_N} \langle w_j(b), Q + x \cdot \nabla Q \rangle_{L_x^{2}}
+ \frac{T^{1/15}}{\eta_{1}^{2}} \sup_{t \in J} \frac{|\xi(t)|^{2}}{\lambda(t)^{2}} + O \left(T^{-8} \right),
\end{align}
where $w_j(t)=\Im \epsilon_j(t), \ \forall \ j\in \Z_{N}$. A direct consequence of the above inequality is that if $\left|\frac{\xi(t)}{\lambda(t)} \right|$ is bounded, then the quantity $$\int_a^b \left\| \boldsymbol{\epsilon}(t) \right\|_{L^2_{x} l^2}^{2} \lambda(t)^{-2} \, \mathrm{d} t$$
is bounded. Moreover, if the modulation parameters satisfy the required conditions on the maximal lifespan of the solution, additional regularity information about the remainder term $\boldsymbol{\epsilon}$ can be obtained. However, we have not yet established any precise pointwise control over the modulation parameters. 
It is noteworthy that the first two terms on the right-hand side of \eqref{qn} can be replaced by $O \left(\|\boldsymbol{\epsilon}(a)\|_{L_x^2l^2}+\|\boldsymbol{\epsilon}(b)\|_{L_x^2l^2} \right)$, as 
a direct consequence of H\"older's inequality, which means that the new inequality, after substitution, is invariant under $L_x^2$-invariance scaling transformation. Combining this  with the parameter control relations \eqref{qm1}-\eqref{qm5} from modulation analysis, and under the conditions of Theorem \ref{t2.3}, the inequality \eqref{qn} holds at least over intervals of length scale $\eta_{\ast}^{-1}$. Then, by employing a bootstrap framework, we rigorously prove that over the entire lifespan of $\mathbf{u}$,  the $L_{x}^2l^2$ norm of the remainder $\boldsymbol{\epsilon}(s)$, $\|\boldsymbol{\epsilon}(s)\|_{L_{x}^2l^2}$, not only 
belongs to $L_s^2$ but also to $L_s^p$ for any $p>1$. This additional regularity of $\boldsymbol{\epsilon}$, along with 
the virial inequality \eqref{qn}, 
will lead to the almost monotonic 
decreasing property of the parameter $\lambda(s)$, i.e., 
$\lambda(s)\sim\inf\limits_{s'\in[0,s]}\lambda \left(s' \right)$ for all $s\geq 0$.

After establishing the above key estimates, we can utilize the almost monotonicity of the frequency parameter property  discussed before and the virial inequality \eqref{qn} to provide a rigorous proof of Theorem \ref{t2.3}. We first consider the case where the blowup time satisfies  $\sup I=\infty$. Suppose, for contradiction, that Theorem \ref{t2.3} fails. Then, there exists a sequence $s_n\nearrow\infty$ as $n \to \infty$ such that the norm $\|\boldsymbol{\epsilon}(s_n)\|_{L_x^2l^2}$ 
decays sufficiently fast relative to the growth rate of $s_n$. Over the interval $[0,s^{-1}(s_n)]$, we can  repeat the proof of the long time Strichartz estimate \eqref{qc} to derive an improved decay estimate for the high frequency part of $\mathbf{u}$. This estimate together with estimate \eqref{qccc2}, in turn, allows us to demonstrate that the uniform bound for  $\|\boldsymbol{\epsilon}(t)\|_{L_x^2l^2}$ on the interval $[0,s^{-1}(s_n)]$  tends to zero as $n\to \infty$. This implies $\|\boldsymbol{\epsilon}(0)\|_{L_x^2l^2}=0$, which contradicts with the assumption that Theorem \ref{t2.3} fails. For the blowup in finite time case, the proof can be reduced to the aforementioned case by using the pseudo-conformal transformation \eqref{pseudotransform}. This completes the proof of Theorem \ref{t2.3}.

\textbf{ Part II. Proof of Theorem \ref{th1.2}.} {Let us briefly outline the strategy for proving Theorem \ref{th1.2}. First, observe that while the natural embedding $H_{rad}^1l^2\hookrightarrow L_x^4l^2$ fails to be compact, but the restricted 
embedding $H_{rad}^1l^2\cap L_x^2h^1\hookrightarrow L_x^4l^2$ becomes compact. This distinction directly implies the following property:} 
\begin{align}
    \sup_{ \substack{\mathbf{u}\in H_x^1l^2 \cap L_x^2h^1, \\
    \|\mathbf{u}\|^2_{L_x^2\dot{h}^1}\leq C(M) \|\mathbf{u}\|^2_{L_x^2l^2}
    }}J(\mathbf{u})=J_M= \frac{2(2M-1)}{M}\|Q_0\|^{-2}_{L_x^2}, 
\end{align}
where $J(\mathbf{u})$ is defined as in \eqref{qzmz}. Consequently, this yields 
the following inequality: 
\begin{align}\label{qqr}
	\sum_{j\in\Z_N}\int_{\R^2}  F_j(\mathbf{u})\cdot \bar{u}_j \, \mathrm{d} x\leq \frac{2(2M-1)}{M}\|Q_0\|^{- 2}_{L_x^2}\|\mathbf{u}\|^2_{L^2_x l^2}\|\nabla\mathbf{u}\|^2_{L^2_x l^2}, 
\end{align}
for any $\mathbf{u}\in H_x^1l^2\cap L_x^2h^1$ with $\|\mathbf{u}\|^2_{L_x^2\dot{h}^1} 
\leq C(M) \|\mathbf{u}\|^2_{L_x^2l^2}$. Now we use the same argument in \cite{CGHY}, the proof of Theorem \ref{th1.2} can be reduced to excluding the almost periodic solution $\mathbf{v}(t)$. We find that \eqref{qqr} is enough for us to exclude the existence of  $\mathbf{v}(t)$ in the scenario  $\int_0^{\infty}N(t)^3 \, \mathrm{d} t<\infty$\footnote{Here $N(t)$ is the frequency scale parameter of $\mathbf{v}$, see \eqref{eq2.10v20} for details.}.  In fact, in this scenario, 
we can show that $\mathbf{v}\in H_x^1l^2$ and $E\left(e^{ix\cdot\xi_{\infty}}\mathbf{v} \right)=0$ for some $\xi_{\infty}\in\R^2$. This contradicts with \eqref{qqr}, since 
\begin{align*}
0< \left\|e^{ix\cdot\xi_{\infty}}\mathbf{v} \right\|^2_{L_x^2l^2}=\|\mathbf{v}\|^2_{L_x^2l^2}<\frac{M}{2M-1}\|Q_0\|^2_{L_x^2} .
\end{align*}
For the 
{remaining case $\int_0^{\infty}N(t)^3 \, \mathrm{d}t=\infty$}, we adapt the methodology of \cite{CGHY} to exclude $\mathbf{v}$. 
Crucially, in \cite{CGHY}, the assumption that the mass is strictly less than $\frac{1}{2}\|Q_0\|^2_{L_x^2l^2}$ is only utilized to establish the following  positivity estimate
\begin{align}\label{mmn}
  E\left(\chi(x)e^{ix\cdot\xi}P_{\leq k}\mathbf{v}(t,x) \right)\gtrsim \left\|\nabla \left(\chi(x)e^{ix\cdot\xi}P_{\leq k}\mathbf{v}(t,x) \right) \right\|_{L_x^2},
\end{align}
where $ \chi\in C^{\infty}_0(\R^2)$ with $|\chi(x)|\leq 1$, and the implicit constant is independent  of $\chi$, $k\in \Z$ and $\xi\in \R^2$. This estimate follows directly from the sharp Gagliardo-Nirenberg inequality \eqref{GN-infity} and the estimate $$\left\|\chi(x)e^{ix\cdot\xi}P_{\leq k}\mathbf{v}(t,x) \right\|^2_{L_x^2l^2}\leq \|\mathbf{v}(t,x)\|^2_{L_x^2l^2}< \frac{1}{2}\|Q_0\|^2_{L_x^2}.$$
In our setting,  $$\left\|\chi(x)e^{ix\cdot\xi}P_{\leq k}\mathbf{v}(t,x) \right\|^2_{L_x^2\dot{h}^1}\leq C(M) \left\|\chi(x)e^{ix\cdot\xi}P_{\leq k}\mathbf{v}(t,x) \right\|^2_{L_x^2l^2}$$
cannot generally be deduced from 
$$\|\mathbf{v}(t,x)\|^2_{L_x^2\dot{h}^1}<\frac{2M-1}{2M} \cdot C(M)\|\mathbf{v}(t,x)\|^2_{L_x^2l^2},$$ 
thus we need to divide into the following two cases: 
\begin{enumerate}
    \item when $\left\|\chi(x)e^{ix\cdot\xi}P_{\leq k}\mathbf{v}(t,x) \right\|^2_{L_x^2l^2}< \frac{2M-1}{2M}\|\mathbf{v}(t,x)\|^2_{L_x^2l^2}<\frac{1}{2}\|Q_0\|^2_{L_x^2}$, we directly use \eqref{GN-infity} to recover estimate \eqref{mmn}.
\item  when $\left\|\chi(x)e^{ix\cdot\xi}P_{\leq k}\mathbf{v}(t,x)\right\|^2_{L_x^2l^2}\geq \frac{2M-1}{2M}\|\mathbf{v}(t,x)\|^2_{L_x^2l^2}$, we have 
 \begin{align*}
  \left\|\chi(x)e^{ix\cdot\xi}P_{\leq k}\mathbf{v}(t,x)\right\|^2_{L_x^2\dot{h}^1}\leq \frac{2M-1}{2M}\cdot C(M) \|\mathbf{v}(t,x)\|^2_{L_x^2l^2}\leq C(M)\left\|\chi(x)e^{ix\cdot\xi}P_{\leq k}\mathbf{v}(t,x)\right\|^2_{L_x^2l^2}
 \end{align*}
 and 
  \begin{align*}
  \left\|\chi(x)e^{ix\cdot\xi}P_{\leq k}\mathbf{v}(t,x)\right\|^2_{L_x^2l^2}\leq \|\mathbf{v}(t,x)\|^2_{L_x^2l^2}<\frac{M}{2M-1}\|Q_0\|^2_{L_x^2} . 
 \end{align*}
 Therefore, we can directly restore  estimate \eqref{mmn} by  using \eqref{qqr}. 

\end{enumerate}

{Now that  we have restored the positivity estimate \eqref{mmn}, we can directly transplant the proof of Theorem 5.3 in \cite{CGHY} to our setting and exclude the existence of the solution  $\mathbf{v}$, thereby completing the proof of Theorem \ref{th1.2}.}  

\textbf{Organization of the paper.} The rest of the paper is organized as follows: In Section \ref{reduction},  we reduce Theorem \ref{main1} to  Theorem \ref{t2.3}.
In Section \ref{se3v314},
 we first establish the spectral properties for the linearized operator
around the ground state $\mathbf{Q}$, and based on that, we develop the modulation analysis. We are  devoted to establishing a universal control  of the reminder $\boldsymbol{\epsilon}$'s $L^2_x l^2$ norm  
in Section \ref{sec:aprieps}. 
In  Section \ref{Sec:viria}, we derive a  virial-type inequality, and then based on that and Theorem \ref{claim2.1}, we complete the proof of Theorem \ref{t2.3} in Section \ref{Sec:Thm2.3}. And 
in Section \ref{Sec:Thm1.2}, we prove Theorem \ref{th1.2}. Finally, we provide the proof of Theorem \ref{claim2.1}  in the appendix.

\subsection{Notation and Preliminaries}\label{subse1.6v33}
We use the notation $X\lesssim Y$ to indicate that there exists some constant $C>0$ such  that $X \le C Y$. Similarly, we write $X \sim Y$ if 
$X\lesssim Y \lesssim X$.

For a vector-valued function $\mathbf{f}(t,x)  = \left\{f_j(t,x)  \right\}_{j\in \mathbb{Z}_N  }$ and $s\geq0$, we define 
\begin{align*}
\| \mathbf{f} \|_{L_t^p L_x^q h^s } := \bigg \|\bigg(\sum_{j \in \mathbb{Z}_N }  \langle j \rangle^{2s} \left| f_j(t,x) \right|^2 \bigg)^\frac12 \bigg\|_{L_t^p L_x^q},
\end{align*}
with $\langle j\rangle=\sqrt{1+j^2}.$
When $s = 0$, we write $L_t^p L_x^q h^0$ as $L_t^p L_x^q l^2$.

We also define the discrete nonisotropic Sobolev space. For a sequence of real-valued functions $\boldsymbol{\phi} =  \left\{\phi_k \right\}_{k\in \mathbb{Z}_N }$ and $s_1,s_2 \ge 0$, we define
\begin{align*}
  \|\boldsymbol{\phi} \|_{H_x^{s_1} h^{s_2}}  
 =   \bigg \| \bigg( \sum\limits_{k\in \mathbb{Z}_N } \langle k\rangle^{2s_2} \left|\langle \nabla_x\rangle^{ s_1} \phi_k (x)  \right|^2  \bigg)^\frac12 \bigg\|_{L_x^{2}}.
\end{align*}
and
\begin{align*}
  \|\boldsymbol{\phi} \|_{\dot{H}_x^{s_1} h^{s_2}}  
 =   \bigg \| \bigg( \sum\limits_{k\in \mathbb{Z}_N } \langle k\rangle^{2s_2} \left| |\nabla_x|^{ s_1} \phi_k (x)  \right|^2  \bigg)^\frac12 \bigg\|_{L_x^{2}}.
\end{align*}
In particular, when $s_1 = 0$, we denote the space $H_x^{s_1} h^{s_2}$ as $L_x^2 h^{s_2}$. Furthermore, when $s_2 = 0$, we denote $L_x^2 h^{s_2}$ as  $L_x^2 l^2$. Similarly, we define $L_x^2\dot{h}^{s_2}$ as follows:
\begin{align*}
  \|\boldsymbol{\phi} \|_{L_x^{2} \dot{h}^{s_2}}  
 =   \bigg \| \bigg( \sum\limits_{k\in \mathbb{Z}_N }  k^{2s_2} \left |\phi_k (x)  \right|^2  \bigg)^\frac12 \bigg\|_{L_x^{2}}.
\end{align*}
For $L_x^{2} l^{2}$, the related inner product $\langle \cdot \rangle_{L_x^{2} l^{2}}$ is given by 
$\langle f, g \rangle_{ L_x^{2} l^{2} } : = \Re \int \sum\limits_{j\in \mathbb{Z}_N} f_j (x) \overline{g_j (x)} \,\mathrm{d}x. $ For other Hilbert space $H$ defined in this paper, the related inner product are real inner product. 

For a function $f\in L_{loc}^1(\R^2)$, we use $\hat{f}$ or $\mathcal{F} (f)$ to denote the spatial Fourier transform of $f$:
$$\mathcal{F}(f) (\xi) = \hat{f}(\xi) =(2 \pi)^{- 1 } \int_{\R^2}e^{- i x\cdot \xi} f(x) \,\mathrm{d}x.$$
Let $\varphi\in C_0^\infty(\R)$ be a real-valued, non-negative, even, and radially-decreasing function such that
\begin{align*}
\varphi(\xi) =
\begin{cases}
1, \ |\xi| \leq 1, \\
0, \  |\xi| \ge 2,
\end{cases}
\end{align*}
and define $\chi(\xi)=\varphi(\xi)-\varphi(2\xi)$. For a  number $N\in \mathbb{N}_+$, we define the Littlewood-Paley projectors:
\begin{align*}
\widehat{P_Nf}(\xi):=\chi\big(\tfrac{\xi}{2^N} \big)\widehat{f}(\xi),  \quad \widehat{P_{\leq N}f}(\xi):=\varphi\big(\tfrac{\xi}{2^N} \big)\widehat{f}(\xi),\quad \widehat{P_{\geq N}f}(\xi):=\Big(1-\varphi\big(\tfrac{\xi}{2^N} \big)\Big)\widehat{f}(\xi).
\end{align*}
We also set that 
\begin{align*}
\widehat{P_0f}(\xi)=\widehat{P_{\leq0}f}:=\varphi\big({\xi} \big)\widehat{f}(\xi),  \quad {P_Nf}={P_{\leq N}f}=0,\quad\forall \ N\in \Z, \ N\leq -1  .
\end{align*}
In the following, we provide the Strichartz estimate for the (NLSS), which is an easy extension of the Strichartz estimate for the (NLS). 

\begin{definition}[Strichartz admissible pair]
We call a pair $(q,r)$ is Strichartz admissible if $2 < q \le \infty$, $2\le r < \infty$, and $\frac1q + \frac1r = \frac12$.
\end{definition}
\begin{lemma}[Strichartz estimate, \cite{CGYZ}]\label{th2.1v31}
For any Strichartz admissible pairs $(q,r)$ and $(\tilde{q}, \tilde{r})$, we have
	\begin{enumerate}
		\item $\big\|e^{it \Delta} \mathbf{f}\,\big\|_{L_t^q L_x^r l^2 (\RR\times\RR^2\times\ZZ_N )}   \lesssim \|\mathbf{f}\,\|_{L_x^2 l^2 (\RR^2\times\ZZ_N )}$,
 		\item $\displaystyle\left\|\int_0^t e^{i(t-s) \Delta} \mathbf{F}(s,x) \,\mathrm{d}s \right\|_{L_t^q L_x^r l^2 (\RR\times\RR^2\times\ZZ_N )} \lesssim \|\mathbf{F} \|_{L_t^{\tilde{q}'} L_x^{\tilde{r}'} l^2 (\RR\times\RR^2\times\ZZ_N )}$. 
\end{enumerate}
\end{lemma}
The linear profile decomposition in $L_x^2(\mathbb{R}^2)$ for the mass-critical nonlinear Schr\"odinger equation 
is established by F. Merle and L. Vega \cite{MV}. After the work of R. Carles and S. Keraani \cite{CK} on the one-dimensional case,
P. B\'egout and A. Vargas \cite{BV} establish the linear profile decomposition of the mass-critical nonlinear Schr\"odinger equation 
for arbitrary dimensions by the refined Strichartz inequality \cite{Bo1} and bilinear restriction estimate \cite{T1}. This can be easily extended to give the linear profile decomposition in $L_x^2l ^2(\mathbb{R}^2 \times \mathbb{Z}_N)$ with $N< \infty$ for the mass-critical nonlinear Schr\"odinger system. 

\begin{proposition}[Linear profile decomposition in $L_x^2 l^2( \mathbb{R}^2 \times \mathbb{Z}_N)$] \label{pro3.9v23}Suppose that $N<\infty$. Let $\{\mathbf{u}_n\}_{n\ge 1}$ 
be a bounded sequence in $L_x^2 l^2(\mathbb{R}^2 \times \mathbb{Z}_N )$. Then after passing to a subsequence if necessary,
there exist $J^* \in \{0,1, \cdots\} \cup \{\infty\}$,
functions $\boldsymbol{\phi}^{j}$ in $L_x^2 l^2(\mathbb{R}^2 \times \mathbb{Z}_N )$ and mutually orthogonal frames $(\lambda_n^j, t_n^j, x_n^j,\xi_n^j )_{n\ge 1} \subseteq (0,\infty) \times \mathbb{R} \times \mathbb{R}^2 \times \mathbb{R}^2$,
which means
\begin{align*}
\frac{\lambda_n^j}{\lambda_n^k} + \frac{\lambda_n^k}{\lambda_n^j} + \lambda_n^j \lambda_n^{k} |\xi_n^j - \xi_n^{k}|^2  + \frac{|x_n^j- x_n^k |^2 }{\lambda_n^j \lambda^k_n}
+ \frac{|(\lambda_n^j)^2 t_n^j - (\lambda_n^k)^2 t_n^k |}{\lambda_n^j \lambda_n^k} \to \infty, \text{ as } n \to \infty, \text{ for } j\ne k,
\end{align*}
and for every $J \le J^*$, a sequence $\mathbf{r}_n^J \in L_x^2 l^2 (\mathbb{R}^2 \times \mathbb{Z}_N)$, such that
\begin{align*}
\mathbf{u}_n(x)  = \sum\limits_{j=1}^J \frac1{\lambda_n^j} e^{ix\cdot\xi_n^j} \left(e^{it_n^j\Delta_{\mathbb{R}^2   }}  \boldsymbol{\phi}^j\right)\left(\frac{x-x_n^j}{\lambda_n^j} \right)  + \mathbf{r}_n^J(x),
\end{align*}
furthermore, 
\begin{align*}
& \lim\limits_{n\to \infty} \left(\| \mathbf{u}_n\|_{L_x^2 l^2   }^2 - \sum\limits_{j=1}^J \left\|   \frac1{\lambda_n^j} e^{ix\cdot\xi_n^j} \left(e^{it_n^j\Delta_{\mathbb{R}^2   }}  \boldsymbol{\phi}^j\right)\left(\frac{x-x_n^j}{\lambda_n^j}\right)
\right\|_{ L_x^2 l^2  }^2 - \| \mathbf{r}_n^J\|^2_{L_x^2 l^2}  \right)  = 0, \notag \\
& \lambda_n^j   e^{-it_n^j \Delta_{\mathbb{R}^2  }}\left(e^{-i(\lambda_n^j x + x_n^j)\cdot \xi_n^j}  \mathbf{r}_n^J\left(\lambda_n^j x+ x_n^j \right)\right)  \rightharpoonup 0 \text{ in } L_x^2 l^2 ,  \text{ as  } n\to \infty,\text{ for each }  j\le J,\notag\\
&  \limsup\limits_{n\to \infty} \|e^{it \Delta_{\mathbb{R}^2    }} \mathbf{r}_n^{J} \|_{L_{t,x}^4  l^2 (\mathbb{R}\times \mathbb{R}^2 \times \mathbb{Z}_N)}   \to 0, \text{ as } J\to J^*.
\end{align*}
\end{proposition}

\section{Theorem \ref{t2.3}  implying Theorem \ref{main1}}\label{reduction}
In this section, we show that  Theorem \ref{main1} follows by  Theorem \ref{t2.3}.
Let $\mathbf{u}(t)$ be a  blowup solution to \eqref{1.1} with maximal lifespan $I$ and
$M(\mathbf{u}) =  M(\mathbf{Q})$. For any time $t_0\in I$, we define the distance to the $(N+5)$-dimensional manifold of symmetries acting on  \eqref{1.12} by
\begin{equation*}
\inf_{\substack{ \left(\lambda,\tilde{x},\xi \right)\in \mathbb{R}_{+}\times\mathbb{R}^2\times\mathbb{R}^2,
		\\(\gamma_1,\cdots, \gamma_N) \in [0, 2 \pi]^N}} \sum_{j\in\mathbb{Z}_N}
\left\|   e^{i \gamma_j}e^{ix\cdot\xi} \lambda u_{j} \left(t_0,\lambda \left(x+\tilde{x} \right) \right) - {Q} (x) \right\|^2_{L_x^{2}}.
\end{equation*}
First, we prove that there exist $\lambda_{0} > 0$, $\left(\tilde{x}_0,\xi_0 \right)\in \mathbb{R}^2\times \mathbb{R}^2$, and $(\gamma_{0,1},\cdots,\gamma_{0, N}) \in [0,2\pi]^N$ such that  this infimum is attained.

\begin{lemma}\label{l2.1}
There exist $\lambda_{0} > 0$, $ \left(\tilde{x}_0,\xi_0 \right)\in \mathbb{R}^2\times \mathbb{R}^2$, and $(\gamma_{0,1},\cdots,\gamma_{0,N}) \in [0,2\pi]^N$ such that
\begin{align*} 
& \sum_{j\in\mathbb{Z}_N} \left\| u_{j}(t_0,x) - e^{-i \gamma_{0,j}} \lambda_{0}^{-1} e^{-ix\cdot\xi_0}  {Q}  \left(\lambda_{0}^{-1} x+\tilde{x}_0 \right)  \right\|^2_{L_x^{2}} \\
 =& \inf_{\substack{ \left(\lambda,\tilde{x},\xi \right)\in \mathbb{R}_{+}\times\mathbb{R}^2\times\mathbb{R}^2\\(\gamma_1,\cdots, \gamma_N) \in [0, 2 \pi]^N}}\sum_{j\in\mathbb{Z}_N}
\left\| u_{j}(t_0,x) - e^{-i \gamma_j}e^{-ix\cdot\xi} \lambda^{-1}  Q  \left(\lambda^{-1} x+\tilde{x} \right)  \right\|^2_{L_x^{2}}.
\end{align*}
\end{lemma}
\begin{proof}
Obviously,
\begin{equation*} 
\sum_{j\in\mathbb{Z}_N} \left\| u_{j}(t_0,x) - e^{-i \gamma_j}e^{-ix\cdot\xi} \lambda^{-1}   Q  \left(\lambda^{-1} x+\tilde{x} \right) \right\|^2_{L_x^{2}}
\end{equation*}
is continuous as a function of $\lambda,\tilde{x},\xi$, and $\gamma_1,\cdots,\gamma_N$. We divide into two cases:

\textbf{Case I}. For all $\lambda > 0$ and $\left(\tilde{x},\xi \right)\in \mathbb{R}^2\times\mathbb{R}^2$,
\begin{equation*} 
\inf_{(\gamma_1,\cdots,\gamma_N) \in [0, 2\pi]^N} \sum_{j\in\mathbb{Z}_N}  \left\| u_{j}(t_0,x) - \lambda^{-1} e^{ - i\gamma_j }e^{- ix \cdot\xi }  Q  \left(\lambda^{-1} x+\tilde{x} \right) \right\|_{L_x^{2}}^{2}=2 \left\|\mathbf{Q} \right\|_{L^2_x l^2}^2.
\end{equation*}
In this case, by direct calculation, for any $j\in\mathbb{Z}_N$,
\begin{align}\label{f2.5}
	& \left\| u_{j}(t_0,x) - \lambda^{-1} e^{i\gamma_j }e^{ix \cdot \xi }Q  \left(\lambda^{-1} x+\tilde{x} \right)
\right\|_{L_x^{2}}^{2}
+ \left\| u_{j}(t_0,x) +  \lambda^{-1} e^{i\gamma_j }e^{ix \cdot \xi } Q  \left(\lambda^{-1} x+\tilde{x} \right)  \right\|_{L_x^{2}}^{2} \notag \\
 =& 2 \| Q  \|_{L_x^{2}}^{2}+2 \|u_{j}\|_{L_x^{2}}^{2},
\end{align}
Changing  variable, \eqref{f2.5} implies
\begin{equation}\label{f2.6}
	\frac{1}{2 \pi} \int_{0}^{2 \pi} \left\| u_{j}(t_0,x) - \lambda^{-1} e^{i\gamma_j }e^{ix \cdot \xi }Q
\left(\lambda^{-1} x+\tilde{x} \right)  \right\|_{L_x^{2}}^{2} \,\mathrm{d} \gamma_j
= \| Q  \|_{L_x^{2}}^{2}+ \|u_{j}\|_{L_x^{2}}^{2}. 
\end{equation}
Integrating \eqref{f2.6} over variables
$\gamma_k$ for $k\neq j$ in  $[0, 2\pi]^{N-1}$ and summing over $j$, we obtain 
\begin{align*} 
	& \frac{1}{(2 \pi)^N} \int_{0}^{2\pi}\cdots\int_{0}^{2 \pi} \sum_{j\in\mathbb{Z}_N}
\left\| u_{j}(t_0,x) - \lambda^{-1} e^{i\gamma_j }e^{ix \cdot \xi }Q  \left(\lambda^{-1} x+\tilde{x} \right)  \right\|_{L_x^{2}}^{2} \,
\mathrm{d} \gamma_1\cdots  \mathrm{d} \gamma_{N} \\
=&N \| Q  \|_{L_x^{2}}^{2}+\sum_{j\in\mathbb{Z}_N} \|u_{j}\|_{L_x^{2}}^{2}  =2 \left\|\mathbf{Q} \right\|^2_{L_x^2l^2},
\end{align*}
then \eqref{f2.6} implies for any $ \lambda > 0, \  \left(\tilde{x},\xi \right)\in \mathbb{R}^2\times\mathbb{R}^2, \  (\gamma_1,\cdots,\gamma_N) \in [0, 2 \pi]^N$
\begin{equation*} 
\sum_{j\in\mathbb{Z}_N} \left\| u_{j}(t_0,x) - \lambda^{-1} e^{i\gamma_j }e^{ix \cdot \xi  }Q  \left(\lambda^{-1} x+\tilde{x} \right)  \right\|_{L_x^{2}}^{2}
= 2 \left\| \mathbf{Q}  \right\|_{L^2_{x} l^2}^{2}.
\end{equation*}
Thus, in this case, we can simply take $\lambda_{0} = 1,  \tilde{x}_0=0, \xi_0=0$ and  $\gamma_{0,j} = 0$ for all 
$j\in\mathbb{Z}_N$. 

\textbf{Case II}.  There exist  $\lambda_0 > 0$ and $\left(\tilde{x}_0,\xi_0 \right)\in \mathbb{R}^2\times\mathbb{R}^2$ such that 
\begin{equation*} 
\inf_{(\gamma_1,\cdots,\gamma_N) \in [0, 2\pi]^N} \sum_{j\in\mathbb{Z}_N}  \left\| u_{j}(t_0,x) - \lambda_0^{-1} e^{ - i\gamma_j }e^{- ix \cdot\xi_0 }  Q  \left(\lambda_0^{-1} x+\tilde{x}_0 \right) \right\|_{L_x^{2}}^{2}=2 \left\|\mathbf{Q} \right\|_{L^2_x l^2}^2 
\end{equation*}
and
\begin{equation*} 
\inf_{\substack{ \left(\lambda,\tilde{x},\xi \right)\in \mathbb{R}_{+}\times\mathbb{R}^2\times\mathbb{R}^2, \\(\gamma_1,\cdots, \gamma_N) \in [0, 2 \pi]^N}}\sum_{j\in\mathbb{Z}_N} \left\|u_j(t_0,x)-\lambda^{-1}e^{-i\gamma_j}e^{-ix\cdot\xi} Q  \left(\lambda^{-1} x+\tilde{x} \right) \right\|_{L_x^{2}}^{2}< 2 \left\|\mathbf{Q} \right\|^2_{L_x^2l^2}.
\end{equation*}
In this case, we claim that
\begin{align}\label{f2.10}
	\lim_{\lambda \nearrow \infty}  \sum_{j\in\mathbb{Z}_N}
\left\| u_{j}(t_0,x) - e^{-i \gamma_j}e^{-ix\cdot\xi} \lambda^{-1} Q  \left(\lambda^{-1} x+\tilde{x} \right)  \right\|^2_{L_x^{2}}
= 2  \left\| \mathbf{Q} \right\|_{L_x^{2}l^2}^{2}
\end{align}
and 
\begin{align}\label{f2.11}
	\lim_{\lambda \searrow 0}  \sum_{j\in\mathbb{Z}_N} \left\| u_{j}(t_0,x) - e^{-i \gamma_j}e^{-ix\cdot\xi} \lambda^{-1} Q  \left(\lambda^{-1} x+\tilde{x} \right) \right\|^2_{L_x^{2}} = 2  \left\| \mathbf{Q}  \right\|_{L^2_{x} l^2}^{2},
\end{align}
uniformly in $ \left(\tilde{x},\xi \right)\in\mathbb{R}^2\times\mathbb{R}^2$ and $(\gamma_1,\cdots,\gamma_N)\in[0,2\pi]^N$.
Indeed,
\begin{align*}
	& \sum_{j\in\mathbb{Z}_N} \left\| u_{j}(t_0,x) - e^{-i \gamma_j}e^{-ix\cdot\xi} \lambda^{-1} Q \left(\lambda^{-1} x+\tilde{x} \right) \right\|^2_{L_x^{2}} \\
=&\left\|\mathbf{Q} \right\|^2_{L_x^2l^2}
+ \left\|\mathbf{u}_0 \right\|^2_{L_x^2l^2}+2\sum_{j\in\mathbb{Z}_N} \left\langle u_{j}(x) , e^{-i\gamma_j}e^{-ix\cdot\xi}\lambda^{-1} Q \left(\lambda^{-1} x+\tilde{x} \right) \right\rangle_{L_x^2}\\
	=&2 \left\|\mathbf{Q} \right\|^2_{L_x^2l^2}+2\sum_{j\in\mathbb{Z}_N}  \left\langle u_{j}(x) , e^{-i\gamma_j}e^{-ix\cdot\xi}\lambda^{-1} Q \left(\lambda^{-1} x+\tilde{x} \right) \right\rangle_{L_x^2}.
\end{align*}
For any $\delta>0$, we can choose $N+1$ smooth functions $\phi_{1,\delta},\cdots,\phi_{N,\delta},\psi_{\delta}$ supported in a compact region $K_{\delta}$ such that
\begin{equation}
	 \|Q-\psi_{\delta}\|_{L_x^2}<\delta,\quad \sum_{j\in\mathbb{Z}_N}\|u_{j}-\phi_{j,\delta}\|_{L_x^2}<\delta.\notag
\end{equation}
Then, we have
\begin{align}
	&\ \  \sum_{j\in\mathbb{Z}_N} \left\langle u_{j}(x) , e^{-i\gamma_j}e^{-ix\cdot\xi}\lambda^{-1} Q \left(\lambda^{-1} x+\tilde{x} \right)  \right\rangle_{L_x^2}
\notag \\
& =
\sum_{j\in\mathbb{Z}_N} \left\langle \phi_{j,\delta}, e^{-i\gamma_j}e^{-ix\cdot\xi}\lambda^{-1} \psi_\delta  \left(\lambda^{-1} x+\tilde{x} \right)  \right\rangle_{L_x^2}\notag\\
	&\  \ +\sum_{j\in\mathbb{Z}_N} \left\langle u_{j}-\phi_{j,\delta}, e^{-i\gamma_j}e^{-ix\cdot\xi}\lambda^{-1} Q \left(\lambda^{-1} x+\tilde{x} \right)  \right\rangle_{L_x^2}
+\sum_{j\in\mathbb{Z}_N} \left\langle u_{j}, e^{-i\gamma_j}e^{-ix\cdot\xi}\lambda^{-1} \left[Q-\psi_{\delta} \right] \left(\lambda^{-1} x+\tilde{x} \right)  \right\rangle_{L_x^2}\notag\\
	& =\sum_{j\in\mathbb{Z}_N} \left\langle \phi_{j,\delta}, e^{-i\gamma_j}e^{-ix\cdot\xi}\lambda^{-1} \psi_\delta  \left(\lambda^{-1} x+\tilde{x} \right)  \right\rangle_{L_x^2}
+O (\delta)\label{plm1}\\
	& =\sum_{j\in\mathbb{Z}_N} \left\langle e^{i\gamma_j }e^{i( x - \tilde{x}) \cdot\xi}\lambda\phi_{j,\delta} \left(\lambda \left(x-\tilde{x} \right) \right),  \psi_\delta (x) \right\rangle_{L_x^2}+O(\delta).\label{plm2}
\end{align}
By \eqref{plm1} and \eqref{plm2}, it is easy to see that the claims \eqref{f2.10} and \eqref{f2.11} hold.
Thus, \eqref{f2.10} and \eqref{f2.11} imply that there exist $0 < \lambda_{1} < \lambda_{2} < \infty$ such that
\begin{align*} 
& \inf_{\substack{ \left(\lambda,\tilde{x},\xi \right)\in \mathbb{R}_{+}\times\mathbb{R}^2\times\mathbb{R}^2, \\ (\gamma_1,\cdots, \gamma_N) \in [0, 2 \pi]^N}} \sum_{j\in\mathbb{Z}_N} \left\|u_j(t_0,x)-\lambda^{-1}e^{-i\gamma_j}e^{-ix\cdot\xi} Q \left(\lambda^{-1} x+\tilde{x} \right)  \right\|_{L_x^{2}}^{2} \\
& = \inf_{\substack{\lambda \in [\lambda_{1}, \lambda_{2}], \\ \left(\tilde{x},\xi \right)\in\mathbb{R}^2\times\mathbb{R}^2, \\(\gamma_1,\cdots, \gamma_N) \in [0, 2 \pi]^N}} \sum_{j\in\mathbb{Z}_N} \left\|u_j(t_0,x)-\lambda^{-1}e^{-i\gamma_j}e^{-ix\cdot\xi} Q \left(\lambda^{-1} x+\tilde{x} \right)  \right\|_{L_x^{2}}^{2}.
\end{align*}
Also, one can deduce that
\begin{align}\label{f2.15}
	\lim_{ \left|\tilde{x} \right| \to  \infty}  \sum_{j\in\mathbb{Z}_N} \left\| u_{j}(t_0,x) - e^{-i \gamma_j}e^{-ix\cdot\xi} \lambda^{-1} Q \left(\lambda^{-1} x+\tilde{x} \right) \right\|^2_{L_x^{2}} = 2 \left\| \mathbf{Q} \right\|_{L^2_{x} l^2}^{2}
\end{align}
and
\begin{align}\label{f2.16}
	\lim_{|\xi| \to  \infty}  \sum_{j\in\mathbb{Z}_N}
\left\| u_{j}(t_0,x) - e^{-i \gamma_j}e^{-ix\cdot\xi} \lambda^{-1} Q \left(\lambda^{-1} x+\tilde{x} \right) \right\|^2_{L_x^{2}}
= 2  \left\| \mathbf{Q} \right\|_{L^2_{x} l^2}^{2}
\end{align}
uniformly in  $\lambda\in[\lambda_1,\lambda_2], \xi\in\mathbb{R}^2,  (\gamma_1,\cdots,\gamma_N)\in[0,2\pi]^N$ and $\lambda\in [\lambda_1,\lambda_2],\tilde{x}\in\mathbb{R}^2, (\gamma_1,\cdots,\gamma_N)\in[0,2\pi]^N$ respectively. Thus \eqref{f2.15} and \eqref{f2.16} also imply that there exist $\left(\tilde{x}_1,\tilde{x}_2,\xi_1,\xi_2 \right)\in\mathbb{R}^2\times\mathbb{R}^2\times\mathbb{R}^2\times\mathbb{R}^2$, $ \left|\tilde{x}_1 \right|< \left|\tilde{x}_2 \right|$, $|\xi_1|<|\xi_2|$ so that
\begin{align*} 
& 	\inf_{\lambda \in [\lambda_{1}, \lambda_{2}]} \inf_{\substack{ \left(\tilde{x},\xi \right)\in\mathbb{R}^2\times\mathbb{R}^2,
\\(\gamma_1,\cdots, \gamma_N) \in [0, 2 \pi]^N}} \sum_{j\in\mathbb{Z}_N}
\left\|u_j(t_0,x)-\lambda^{-1}e^{-i\gamma_j}e^{-ix\cdot\xi} Q \left(\lambda^{-1} x+\tilde{x} \right) \right\|_{L_x^{2}}^{2}\\
 =&\inf_{\substack{\lambda\in [\lambda_1,\lambda_2], \\ \left|\tilde{x} \right|\in \left[ \left|\tilde{x}_1 \right|, \left|\tilde{x}_2 \right| \right] ,  \\|\xi|\in[|\xi_1|,|\xi_2|],
\\(\gamma_1,\cdots, \gamma_N) \in [0, 2 \pi]^N}} \sum_{j\in\mathbb{Z}_N} \left\|u_j(t_0,x)-\lambda^{-1}e^{-i\gamma_j}e^{-ix\cdot\xi} Q \left(\lambda^{-1} x+\tilde{x} \right) \right\|_{L_x^{2}}^{2}.
\end{align*}
Since $[\lambda_{1}, \lambda_{2}] \times \left[ \left|\tilde{x}_1 \right|, \left|\tilde{x}_2 \right| \right]\times [|\xi_1|,|\xi_2|] \times [0, 2 \pi]^N$ is a compact set, by continuity, there exist $\lambda_{0} > 0$, $ \left(\tilde{x}_0,\xi_0 \right)\in\mathbb{R}^2\times\mathbb{R}^2$, and $(\gamma_{0,1},\cdots,\gamma_{0,N}) \in [0, 2\pi]^N$ such that
\begin{align*}\label{f2.18}	
& \sum_{j\in\mathbb{Z}_N} \left\| u_{j}(t_0,x) - e^{-i \gamma_{j}} \lambda_{0}^{-1} e^{-ix\cdot\xi_0}Q \left(\lambda_{0}^{-1} x+\tilde{x}_0 \right)  \right\|^2_{L_x^{2}} \\
 =& \inf_{\substack{ \left(\lambda,\tilde{x},\xi \right)\in \mathbb{R}_{+}\times\mathbb{R}^2\times\mathbb{R}^2,
\\(\gamma_1,\cdots, \gamma_N) \in [0, 2 \pi]^N}}\sum_{j\in\mathbb{Z}_N} \left\| u_{j}(t_0,x) - e^{-i \gamma_j}e^{-ix\cdot\xi} \lambda^{-1}
Q \left(\lambda^{-1} x+\tilde{x} \right) \right\|^2_{L_x^{2}}.
\end{align*}
This completes the proof of the lemma. 
\end{proof}
Next, we show that the distance function is upper semicontinuous in time $t$ on the lifespan $I$ and continuous when it takes small values. 
\begin{lemma}[Upper semicontinuity of the distance to the ground state]\label{l2.2}
	The  function
	\begin{equation}\label{f2.19}
		\alpha(t)=\inf_{\substack{ \left(\lambda,\tilde{x},\xi \right)\in \mathbb{R}_{+}\times\mathbb{R}^2\times\mathbb{R}^2,
\\(\gamma_1,\cdots, \gamma_N) \in [0, 2 \pi]^N}} \sum_{j\in\mathbb{Z}_N} \left\|  e^{i \gamma_j}e^{ix\cdot\xi} \lambda u_{j} \left(t, \lambda  \left(x+\tilde{x} \right) \right) - Q (x) \right\|^2_{L_x^{2}}
	\end{equation}
	is upper semicontinuous  for any $t \in I$. Moreover, the function \eqref{f2.19} is continuous  when it is small.  
\end{lemma}
\begin{proof}
	We only need to consider the case that $\alpha(t)\neq 0 \mbox{ for all } t\in I$. First we prove that the upper semicontinuous property. Fix  $t_{0} \in I$. By Lemma \ref{l2.1} and the symmetries \eqref{sym1}-\eqref{sym5}, we can set
	\begin{equation*} 
		\left\|  \mathbf{u}(t_{0},  x) - \mathbf{Q}(x) \right\|^2_{L^2_{x} l^2} = \inf_{\substack{ \left(\lambda,\tilde{x},\xi \right)\in \mathbb{R}_{+}\times\mathbb{R}^2\times\mathbb{R}^2, \\(\gamma_1,\cdots, \gamma_N) \in [0, 2 \pi]^N}} \sum_{j\in\mathbb{Z}_N}
\left\|  e^{i \gamma_j}e^{ix\cdot\xi} \lambda u_{j} \left(t_0, \lambda  \left(x+\tilde{x} \right) \right) - Q(x )  \right\|^2_{L_x^{2}}.
	\end{equation*}
	For $t$ close to $t_{0}$, let
	\begin{equation*} 
		\boldsymbol{\epsilon}(t, x) = \mathbf{u}(t, x) - e^{i(t - t_{0})} \mathbf{Q}(x).
	\end{equation*}
	Since $e^{i(t - t_{0})} \mathbf{Q}$ and $\mathbf{u}$ both solve $\eqref{1.1}$, we have 	\begin{equation}\label{f2.22}
		i  \partial_t  {\epsilon}_j  + \Delta  {\epsilon}_{j}  = F_j(e^{i(t-t_0)}\mathbf{Q})-{ {F}}_j ( \mathbf{u} )
  = O \left(Q^2 \left|\boldsymbol{\epsilon} \right| \right)+O \left( \left|\boldsymbol{\epsilon} \right|^3 \right)
  =O \left( \left|\mathbf{u} \right|^2 \left|\boldsymbol{\epsilon} \right| \right)+O \left( \left|\boldsymbol{\epsilon} \right|^3 \right),\  \forall \, j\in \mathbb{Z}_N.
	\end{equation}
Now we choose a  small open neighborhood $J'$ of $t_0$ such  that $\min\limits_{t\in J'} \left\|\boldsymbol{\epsilon}(t) \right\|_{L^2_x l^2}\geq \frac{1}{2} \left\|\boldsymbol{\epsilon}(t_0) \right\|_{L^2_x l^2}$.  Then \eqref{f2.22} implies
	\begin{align}\label{f2.28}
\aligned
		\frac{ \mathrm{d} }{ \mathrm{d} t}  \left( \left\| \boldsymbol{\epsilon}(t)  \right\|_{L^2_{x} l^2}^{2} \right)
&\lesssim { \|Q\|_{L_x^{\infty}}^{2}  \left\| \boldsymbol{\epsilon}(t)  \right\|_{L^2_x l^2}^{2}
+ \|Q\|_{L_x^6}^{3}  \left\| \boldsymbol{\epsilon}(t)  \right\|_{L^2_x l^2}
+ \left\| \mathbf{u}(t)  \right\|_{L^6_x l^2}^{3}  \left\| \boldsymbol{\epsilon}(t)  \right\|_{L^2_x l^2}} \\
		&\lesssim_{ \left\| \boldsymbol{\epsilon}(t_0) \right\|_{L^2_x l^2}}
		\left( \| Q \|_{L_x^{\infty}}^{2} +\| Q \|_{L_x^6}^3 + \left\| \mathbf{u}(t)  \right\|_{L_x^6 l^2}^{3} \right)
\left\| \boldsymbol{\epsilon}(t)  \right\|_{L^2_x  l^2}^{2}, \quad\forall \,  t\in J'.
\endaligned
	\end{align}
By Gronwall's inequality, \eqref{f2.28} implies
	\begin{align}\label{qac}
	\left\| \boldsymbol{\epsilon}(t)  \right\|_{L^2_{x} l^2}^{2}\leq \exp\left( {\int_{t_0}^t\left( \| Q \|_{L_x^{\infty}}^{2} +\| Q \|_{L_x^{6}}^{3} +
\left\| \mathbf{u} \right\|_{L^{6}_x l^2}^{3}\right) \,  \mathrm{d} s}\right)
\left\| \boldsymbol{\epsilon}(t_0 )  \right\|_{L^2_{x} l^2}^{2}, \quad \forall \, t\in J'.
	\end{align}
Finally, using the fact that $\mathbf{u} \in L_{t, loc}^{3} L_{x}^{6}l^2$, \eqref{qac} implies
	\begin{align*} 
& \limsup_{t \rightarrow t_{0}} \inf_{\substack{ \left(\lambda,\tilde{x},\xi \right)\in \mathbb{R}_{+}\times\mathbb{R}^2\times\mathbb{R}^2 , \\(\gamma_1,\cdots, \gamma_N) \in [0, 2 \pi]^N}} \sum_{j\in\mathbb{Z}_N}
\left\|  e^{i \gamma_j}e^{ix\cdot\xi} \lambda u_{j} \left(t, \lambda  \left(x+\tilde{x} \right) \right) - Q  (x) \right\|^2_{L_x^{2}} \\
 \leq&   \inf_{\substack{ \left(\lambda,\tilde{x},\xi \right)\in \mathbb{R}_{+}\times\mathbb{R}^2\times\mathbb{R}^2,
\\(\gamma_1,\cdots, \gamma_N) \in [0, 2 \pi]^N}} \sum_{j\in\mathbb{Z}_N}
\left\|  e^{i \gamma_j}e^{ix\cdot\xi} \lambda u_{j} \left(t_0, \lambda  \left(x+\tilde{x} \right) \right) - Q  (x) \right\|^2_{L_x^{2}}, \notag
	\end{align*}
which establishes upper semicontinuity.

Now we turn to prove the continuity of \eqref{f2.19} when $\|\boldsymbol{\epsilon}(t)\|_{L_x^2l^2}$ is sufficiently small. Let $J$ be an open neighborhood of $t_{0}$ such that  $ \left\| \mathbf{u} \right\|_{L_{t,x }^{4}  l^2 \left(J \times \mathbb{R}^2\times\mathbb{Z}_N \right)} = 1$(guaranteed by local theory),
then using  Strichartz estimate on $J$, we get
	\begin{equation*} 
	\left\| \boldsymbol{\epsilon} \right \|_{L_{t}^{\infty} L_{x}^{2}l^2 \cap L_{t,x}^{4} l^2 \left(J \times \mathbb{R}^2\times\mathbb{Z}_N  \right)}
\lesssim  \left\| \boldsymbol{\epsilon}(t_{0})  \right\|_{L_x^{2}l^2} +  \left\| \boldsymbol{\epsilon}  \right\|_{L_{t,x}^{4}l^2 \left(J \times \mathbb{R}^2\times\mathbb{Z}_N \right)}
\left\| \mathbf{u} \right\|_{L_{t,x}^{4} l^2 \left(J \times \mathbb{R}^2\times\mathbb{Z}_N \right)}^{2} +  \left\| \boldsymbol{\epsilon}  \right\|_{L_{t,x}^{4}l^2 \left(J \times \mathbb{R}^2\times\mathbb{Z}_N \right)}^3.
	\end{equation*}
	 Therefore, for $ \left\| \boldsymbol{\epsilon}(t_{0})  \right\|_{L^2_{x} l^2}$ small, we can split the interval $J$ into finitely many pieces and then use the bootstrap argument to obtain 
\begin{equation}\label{f2.24}
	\left\| \boldsymbol{\epsilon}(t)  \right\|_{L_t^{\infty}L_x^2l^2\left(J\times\mathbb{R}^2\times \mathbb{Z}_N  \right) }  
+ \left\| \boldsymbol{\epsilon}(t)  \right\|_{L_t^{p}L_x^ql^2 \left(J\times\mathbb{R}^2\times \mathbb{Z}_N  \right)} \lesssim  \left\| \boldsymbol{\epsilon}(t_{0})  \right\|_{L^2_{x} l^2}, \mbox{ for all admissible pair } (p,q)
	\end{equation}
	and
	\begin{equation}\label{f2.25}
		\lim_{t \rightarrow t_{0}} \left\| \boldsymbol{\epsilon}(t)  \right\|_{L_x^{2} l^2}
=  \left\| \boldsymbol{\epsilon}(t_{0})  \right\|_{L_x^{2}l^2}.
	\end{equation}
Now suppose that
	\begin{align*} 
		& \liminf_{t \rightarrow t_{0}} \inf_{\substack{ \left(\lambda,\tilde{x},\xi \right)\in \mathbb{R}_{+}\times\mathbb{R}^2\times\mathbb{R}^2, \\(\gamma_1,\cdots, \gamma_N) \in [0, 2 \pi]^N}}
\sum_{j\in\mathbb{Z}_N} \left\|  e^{i \gamma_j}e^{ix\cdot\xi} \lambda u_{j} \left(t, \lambda  \left(x+\tilde{x} \right) \right) - Q (x)   \right\|^2_{L_x^{2}} \\
 <&   \inf_{\substack{ \left(\lambda,\tilde{x},\xi \right)\in \mathbb{R}_{+}\times\mathbb{R}^2\times\mathbb{R}^2,
\\(\gamma_1,\cdots, \gamma_N) \in [0, 2 \pi]^N}} \sum_{j\in\mathbb{Z}_N} \left\|  e^{i \gamma_j}e^{ix\cdot\xi} \lambda u_{j} \left(t_0, \lambda  \left(x+\tilde{x} \right) \right) - Q (x)  \right\|^2_{L_x^{2}},
	\end{align*}
then by Lemma \ref{l2.1}, there exist a sequence $t_{n}' \rightarrow t_{0}$ as $n \to \infty$, $\lambda_{n}' > 0$, $ \left(\tilde{x}',\xi' \right)\in\mathbb{R}^2\times\mathbb{R}^2$, and $\left(\gamma_{1,n}',\cdots,\gamma_{N,n}' \right) \in [0,2\pi]^N$ such that
	\begin{equation*}
		\sum_{j\in\mathbb{Z}_N} \left\|  e^{i \gamma_j'}e^{ix\cdot\xi'} \lambda u_{j} \left(t_n', \lambda'  \left(x+ \tilde{x}' \right) \right) - Q (x)  \right\|^2_{L_x^{2}}
<   \inf_{\substack{ \left(\lambda,\tilde{x},\xi \right)\in \mathbb{R}_{+}\times\mathbb{R}^2\times\mathbb{R}^2,
\\(\gamma_1,\cdots, \gamma_N) \in [0, 2 \pi]^N}} \sum_{j\in\mathbb{Z}_N} \left\|  e^{i \gamma_j}e^{ix\cdot\xi} \lambda u_{j} \left(t_0, \lambda  \left(x+\tilde{x} \right) \right) - Q  (x) \right\|^2_{L_x^{2}}.
	\end{equation*}
For $t_{n}'$ sufficiently close to $t_{0}$, we can repeating the arguments leading to \eqref{f2.24} and \eqref{f2.25} with $t_{n}'$ as the initial data yields a contradiction. Therefore, the semicontinuity of \eqref{f2.19} will be upgraded to continuity when $\|\alpha(t_0)\|_{L_x^2l^2}=\|\boldsymbol{\epsilon}(t_0)\|_{L_x^2l^2}$ is small.
\end{proof}

After the above preparation, we will show how Theorem \ref{main1} follows from Theorem \ref{t2.3}.

\begin{proof}[Theorem \ref{t2.3} $\Rightarrow $ Theorem \ref{main1}]
Let $\mathbf{u}(t)$ be the solution to $(\ref{1.1})$ with initial data $\mathbf{u}_{0}$ satisfying $M(\mathbf{u}_0) =  M(\mathbf{Q})$. We divide into two cases:

\textbf{Case I}. There exists some $t_{0} > 0$ such that
\begin{equation*}
\sup_{t \in [t_{0}, \sup(I))} \inf_{\substack{ \left(\lambda,\tilde{x},\xi \right)\in \mathbb{R}_{+}\times\mathbb{R}^2\times\mathbb{R}^2,
\\(\gamma_1,\cdots, \gamma_N) \in [0, 2 \pi]^N}} \sum_{j\in\mathbb{Z}_N}
\left\|  e^{i \gamma_j}e^{ix\cdot\xi} \lambda  u_{j} \left(t, \lambda  \left(x+\tilde{x} \right) \right) - Q (x)  \right\|^2_{L_x^{2}} \leq \eta_{\ast},
\end{equation*}
where $\eta_{\ast}$ is the small parameter in Theorem \ref{t2.3}. In this case, after translating in time so that $t_{0} = 0$, \eqref{f2.32} holds, thus we can directly  use Theorem \ref{t2.3} to prove Theorem \ref{main1}.

\textbf{Case II}.  There exists a sequence $t_{n}^{-} \nearrow \sup(I)$ such that
\begin{equation}\label{f2.35}
\inf_{\substack{ \left(\lambda,\tilde{x},\xi \right)\in \mathbb{R}_{+}\times\mathbb{R}^2\times\mathbb{R}^2,
\\(\gamma_1,\cdots, \gamma_N) \in [0, 2 \pi]^N}}
\sum_{j\in\mathbb{Z}_N} \left\|  e^{i \gamma_j}e^{ix\cdot\xi} \lambda  u_{j} \left(t_n^-, \lambda  \left(x+\tilde{x} \right) \right) - Q (x)  \right\|^2_{L_x^{2}}  > \eta_{\ast},
\end{equation}
for every $n$. In this case, after passing to a subsequence, we can suppose that for every $n$, $t_{n}^{-} < t_{n} < t_{n + 1}^{-}$, where $t_{n}$ is the sequence in \eqref{f2.31} and $t_{n}^{-}$ is the sequence in \eqref{f2.35}. The fact that
\begin{equation}\label{f2.36}
\inf_{\substack{ \left(\lambda,\tilde{x},\xi \right)\in \mathbb{R}_{+}\times\mathbb{R}^2\times\mathbb{R}^2,
\\(\gamma_1,\cdots, \gamma_N) \in [0, 2 \pi]^N}}
\sum_{j\in\mathbb{Z}_N} \left\|  e^{i \gamma_j}e^{ix\cdot\xi} \lambda u_{j} \left(t, \lambda  \left(x+\tilde{x} \right) \right) - Q (x)  \right\|^2_{L_x^{2}}
\end{equation}
is upper semicontinuous as a function of $t$, and  continuous when \eqref{f2.36} is small, guarantees that there exists a small, fixed $0 < \eta_{\ast} \ll 1$ such that the sequence $t_{n}^{+}$, defined by
\begin{equation*} 
t_{n}^{+} = \inf \bigg\{ t \in I : \sup_{\tau \in [t, t_{n}]} \inf_{\substack{ \left(\lambda,\tilde{x},\xi \right)\in \mathbb{R}_{+}\times\mathbb{R}^2\times\mathbb{R}^2,
\\(\gamma_1,\cdots, \gamma_N) \in [0, 2 \pi]^N}} \sum_{j\in\mathbb{Z}_N}
\left\|  e^{i \gamma_j}e^{ix\cdot\xi}  \lambda  u_{j} \left(\tau, \lambda  \left(x+\tilde{x} \right) \right) - Q(x)   \right\|^2_{L_x^{2}}
< \eta_{\ast}  \bigg\},
\end{equation*}
satisfies $t_{n}^{+} \nearrow \sup I$ as $n\to \infty $ and
\begin{equation}\label{f2.38}
\inf_{\substack{\left(\lambda,\tilde{x},\xi \right)\in \mathbb{R}_{+}\times\mathbb{R}^2\times\mathbb{R}^2,
\\(\gamma_1,\cdots, \gamma_N) \in [0, 2 \pi]^N}}
\sum_{j\in\mathbb{Z}_N} \left\|  e^{i \gamma_j}e^{ix\cdot\xi} \lambda u_{j} \left(t_n^+, \lambda  \left(x+\tilde{x} \right) \right) - Q(x)   \right\|^2_{L_x^{2}}
= \eta_{\ast}.
\end{equation}
Indeed, the fact that \eqref{f2.36} is upper semicontinuous as a function of $t$ implies that
\begin{equation*}
\bigg\{ 0 \leq t < t_{n} : \inf_{\substack{ \left(\lambda,\tilde{x},\xi \right)\in \mathbb{R}_{+}\times\mathbb{R}^2\times\mathbb{R}^2,
\\(\gamma_1,\cdots, \gamma_N) \in [0, 2 \pi]^N}} \sum_{j\in\mathbb{Z}_N}
\left\|  e^{i \gamma_j}e^{ix\cdot\xi} \lambda u_{j} \left(t, \lambda  \left(x+\tilde{x} \right) \right) - Q( x)   \right\|^2_{L_x^{2}}  \geq \eta_{\ast}  \bigg\}
\end{equation*}
is a closed set. Since this set is also bounded, it has a maximal element $t_{n}^{+}$, and $t_{n}^{+} \geq t_{n}^{-}$. The fact that \eqref{f2.36} is upper semicontinuous in time also implies that
\begin{equation}\label{edc0}
\inf_{\substack{ \left(\lambda,\tilde{x},\xi \right)\in \mathbb{R}_{+}\times\mathbb{R}^2\times\mathbb{R}^2,
\\(\gamma_1,\cdots, \gamma_N) \in [0, 2 \pi]^N}}
\sum_{j\in\mathbb{Z}_N} \left\|  e^{i \gamma_j}e^{ix\cdot\xi} \lambda u_{j} \left(t_n^+, \lambda  \left(x+\tilde{x} \right) \right) - Q (x)  \right\|^2_{L_x^{2}}  \geq \eta_{\ast}.
\end{equation}
On the other hand, since
\begin{equation}\label{edc}
\inf_{\substack{ \left(\lambda,\tilde{x},\xi \right)\in \mathbb{R}_{+}\times\mathbb{R}^2\times\mathbb{R}^2,
\\(\gamma_1,\cdots, \gamma_N) \in [0, 2 \pi]^N}}
\sum_{j\in\mathbb{Z}_N} \left\|  e^{i \gamma_j}e^{ix\cdot\xi} \lambda u_{j} \left(t, \lambda  \left(x+\tilde{x} \right) \right) - Q (x)  \right\|^2_{L_x^{2}}
 < \eta_{\ast} \quad \text{for all} \quad t_{n}^{+} < t < t_{n},
\end{equation}
as in the proof of \eqref{f2.24} and \eqref{f2.25}, we can choose sufficiently small $\eta_{\ast}$ 
 so that \eqref{f2.36} is continuous on interval $ \left(t_n^+ -\delta,t_n^+ +\delta \right)$ for some $0 < \delta \ll 1$. Thus,  \eqref{edc} also implies
\begin{equation}\label{f2.41}
\inf_{\substack{ \left(\lambda,\tilde{x},\xi \right)\in \mathbb{R}_{+}\times\mathbb{R}^2\times\mathbb{R}^2,
\\(\gamma_1,\cdots, \gamma_N) \in [0, 2 \pi]^N}}
\sum_{j\in\mathbb{Z}_N} \left\|  e^{i \gamma_j}e^{ix\cdot\xi} \lambda u_{j} \left(t_n^+, \lambda  \left(x+\tilde{x} \right) \right) - Q (x)  \right\|^2_{L_x^{2}}  \leq \eta_{\ast}.
\end{equation}
By \eqref{edc0} and \eqref{f2.41}, we have
\begin{equation}
	\inf_{\substack{ \left(\lambda,\tilde{x},\xi \right)\in \mathbb{R}_{+}\times\mathbb{R}^2\times\mathbb{R}^2,
\\(\gamma_1,\cdots, \gamma_N) \in [0, 2 \pi]^N}} \sum_{j\in\mathbb{Z}_N}
\left\|  e^{i \gamma_j}e^{ix\cdot\xi} \lambda u_{j} \left(t_n^+, \lambda \left(x+\tilde{x} \right) \right) - Q (x)  \right\|^2_{L_x^{2}}
 = \eta_{\ast}.\notag
\end{equation}
Next using  \eqref{f2.38}, the definition of $t_n$ and $ t_n^+$, and then following the proof of \eqref{f2.24}, we derive
\begin{equation}\label{f2.42}
	\lim_{n \rightarrow \infty} \left\| \mathbf{u} \right\|_{L_{t,x}^{4}l^2( \left[t_{n}^{+}, t_{n} \right] \times \mathbb{R}^2\times\mathbb{Z}_N )} = \infty.
\end{equation}
Applying Proposition \ref{pro3.9v23} to ${\mathbf{u} \left(t_{n}^{+}, x \right)}$, the fact that $\mathbf{u}$ is a minimal-mass blowup solution that blows up forward in time and the fact that  $t_{n}^{+} \nearrow \sup I$ as $n\to \infty$ imply that there exist $\mathbf{w}_0\in L_x^2 l^2$ and parameters $\widehat{\lambda}_n > 0$, $ \left(\widehat{x}_n,\widehat{\xi}_n \right)\in\mathbb{R}^2\times\mathbb{R}^2$, 
and  $\left(\widehat{\gamma}_{n,1},\cdots,\widehat{\gamma}_{n,N} \right) \in [0,2\pi]^N$ such that
\begin{equation}\label{f2.43}
\sum_{j\in\mathbb{Z}_N}
\left\|e^{i\widehat{\gamma}_{n ,j}}e^{ix\cdot\widehat{ \xi}_n}\widehat{\lambda}_n  u_j \left(t_{n}^{+}, \widehat{\lambda}_n
\left(x+ \widehat{x}_n \right) \right)-w_{0,j} (x) \right\|^2_{L^2_x} \rightarrow 0, \text{ as } n \to \infty .
\end{equation}
Moreover, $t_{n}^{+} \nearrow \sup I$ as $n\to \infty$ implies $ \left\| \mathbf{w}_{0}  \right\|_{L_x^{2}l^2} = \left\| \mathbf{Q}  \right\|_{L_x^{2}l^2}$, and by \eqref{f2.38},
\begin{equation}\label{f2.44}
\inf_{\substack{ \left(\lambda,\tilde{x},\xi \right)\in \mathbb{R}_{+}\times\mathbb{R}^2\times\mathbb{R}^2,
\\(\gamma_1,\cdots, \gamma_N) \in [0, 2 \pi]^N}}
\sum_{j\in\mathbb{Z}_N} \left\|  e^{i \gamma_j}e^{ix\cdot\xi} \lambda w_{0,j} \left( \lambda \left(x+\tilde{x} \right) \right) - Q (x)  \right\|^2_{L_x^{2}}
 = \eta_{\ast}.
\end{equation}
Let $\mathbf{w}$ be the solution to \eqref{1.1} with initial data $\mathbf{w}_{0}$. Then $t_n^+\nearrow \sup I$ as $n \to \infty$ implies that $\mathbf{w}$ blows up both forward and backward in time. Denote the maximal interval of existence of $\mathbf{w}$ as $\tilde{I}$. We claim that
\begin{equation}\label{f2.45}
\inf_{\substack{ \left(\lambda,\tilde{x},\xi \right)\in \mathbb{R}_{+}\times\mathbb{R}^2\times\mathbb{R}^2,
\\(\gamma_1,\cdots, \gamma_N) \in [0, 2 \pi]^N}}
\sum_{j\in\mathbb{Z}_N} \left\|  e^{i \gamma_j}e^{ix\cdot\xi} \lambda w_{j} \left( t,\lambda \left(x+\tilde{x} \right) \right) - Q (x)   \right\|^2_{L_x^{2}} \leq \eta_{\ast},\quad \forall \,  t\in [0, \sup\tilde{I}).
\end{equation}
Indeed, for any compact interval $[0, M]\subseteq  \tilde{I}$, \eqref{f2.43} and the stability lemma guarantee that as $n \to \infty$, 
\begin{align}\label{ijz}
\left\|\mathbf{u}'_n(t)-\mathbf{w}(t) \right\|_{L_t^{\infty} L_x^2l^2\cap L_{t, x}^4  l^2([0,M]\times\mathbb{R}^2\times\mathbb{Z}_N)}\to 0,
\end{align}
where
$$\mathbf{u}'_n(t)= \left(e^{i\widehat{\gamma}_{n,1}}e^{ix\cdot\widehat{\xi}_n}\widehat{\lambda}_n u_{1} \left(t_n^+ +\widehat{\lambda}_n^2 t,\widehat{\lambda}_n x+ \widehat{x}_n \right),\cdots,e^{i\widehat{\gamma}_{n,N}}e^{ix\cdot\widehat{\xi}_n}\widehat{\lambda}_n u_{N} \left(t_n^++\widehat{\lambda}_n^2 t,\widehat{\lambda}_n x+\widehat{x}_n \right)\right).$$
Clearly, \eqref{f2.42}, \eqref{f2.44}, \eqref{ijz}, and the arbitrariness in the choice of $M$ 
imply that \eqref{f2.45} holds.

However, Theorem \ref{t2.3}, \eqref{f2.44}, and \eqref{f2.45} imply that $\mathbf{w}$ must be of the form \eqref{pseudosoliton}. Such a solution scatters backward in time, which contradicts the fact that $\mathbf{w}$ blows up both forward and backward in time. Thus \eqref{f2.35} cannot happen, which precludes possibility of Case II. So we finish the proof.
\end{proof}

\section{Modulation analysis}\label{se3v314}
In this section,  we first establish the spectral properties for the linearized operator
around the ground state $\mathbf{Q}$, and based on that, we develop the modulation analysis. 

Let $\mathbf{u}(t)=e^{i(t-t_0)}[\mathbf{Q}+\boldsymbol{\epsilon}(t)]$, 
where $\boldsymbol{\epsilon}(t)=\boldsymbol{\epsilon}_1(t)+i\boldsymbol{\epsilon}_2(t)$ for some $t_0\in I$. Then the $2N$-dimensional vector $\boldsymbol{\varepsilon}:=(\boldsymbol{\epsilon}_1, \boldsymbol{\epsilon}_2)$ 
satisfies the  following evolution equation:
\begin{equation*}
	\partial_{t}\boldsymbol{\varepsilon}+\mathcal{L}\boldsymbol{\varepsilon}=\mathcal{R}\boldsymbol{\varepsilon},
\end{equation*}
where $\mathcal{R}\boldsymbol{\varepsilon}$ is the nonlinear term with at least two components of $\boldsymbol{\epsilon}$, and
\begin{align*}
	\mathcal{L} &= \begin{pmatrix}O & -L_{+} \\
		L_{-} & O \\
	\end{pmatrix}.
\end{align*}
The operators $L_+$ and $L_-$ are defined as follows: 
\begin{align}
	L_{+} &=\begin{pmatrix}
1-\Delta- (2N+1)  Q^2& -4 Q^2 &\cdots &-4 Q^2		\\
		-4 Q^2
		& 1-\Delta- (2N+ 1) Q^2 &\cdots &-4 Q^2\\
		\vdots &\vdots &\ddots &\vdots\\
		-4 Q^2 &-4 Q^2 &\cdots & 1 -\Delta-(2N+1) Q^2
	\end{pmatrix},
\label{L+}\\
\intertext{ and }
	L_{-}&=\begin{pmatrix}1-\Delta- (2N-1)  Q^2& 0 &\cdots &0	\\
		0
		& 1-\Delta- (2N - 1) Q^2 &\cdots &0\\
		\vdots &\vdots &\ddots &\vdots\\
		0 &0 &\cdots & 1 -\Delta-(2N - 1)  Q^2
	\end{pmatrix}.\label{L-}
\end{align}
It is easy to check that $L_{+}$ and $ L_{-}$ are self-adjoint operators in $H_x^1l^2(\mathbb{R}^2\times \mathbb{Z}_N)$, and $\mathcal{L}$ is self-adjoint in $H_x^1l^2(\mathbb{R}^2\times \mathbb{Z}_{2N})$. 
The spectral structure of $L_{+}$ plays a key role in the paper, so we begin by studying 
the spectral properties of these linearized operators. 

First, we  invoke the following classical results:
\begin{theorem}[\cite{FHRY}]\label{L01spectral}
	For the operator $L_{0,+}:=-\Delta+1-3Q_0^2$ defined in  real-valued Sobolev space $H^1_x(\mathbb{R}^2)$.
    The following four properties hold:
	
	\begin{enumerate}
		\item ${L}_{0,+}$ is a self-adjoint operator and $\sigma_{ess}({L}_{0,+}) = [1, \infty)$.
		
		\item $Ker({L}_{0,+}) = span \left\{ (Q_0)_{x_{1}}, (Q_0)_{x_{2}} \right\}$.
		
		\item $L_{0,+}$ has a unique negative eigenvalue $\lambda_{0}$ associated with a positive, radially symmetric eigenfunction $\chi_{0}$. Moreover, there exists $\delta > 0$ such that $|\chi_{0}(x)| \lesssim e^{-\delta |x|}$ for all $x \in \mathbb{R}^{2}$.
		
		\item If $u\in H^1_x(\mathbb{R}^2) $ satisfies $\langle u,\chi_0\rangle_{L^2_x}=0$, then $\langle {L}_{0,+} u,u \rangle_{L_x^2}\geq 0$.
	\end{enumerate}
\end{theorem}

\begin{theorem}[\cite{FHRY}]\label{L02spectral}
	The operator $L_{0,-}:=-\Delta+1-Q_0^2$ is positive in real-valued Sobolev space $H^1_x(\mathbb{R}^2)$, 
    i.e., 
	\begin{align}\label{PL_-}
	\left\langle L_{0,-} u, u \right\rangle_{L^2_{x}}=\int_{\mathbb{R}^2} \left( |\nabla u|^2+u^2-Q_0^2u^2 \right) \,  \mathrm{d} x\geq 0.
	\end{align}
\end{theorem}
Define  $\boldsymbol{\chi}_0:=(\chi_{0},\cdots,\chi_{0})$. 
By direct calculation, $L_{+}\boldsymbol{\chi}_0=\lambda_{0}\boldsymbol{\chi}_0$, 
so $\lambda_{0}$ is also a negative eigenvalue for $L_{+}$. In fact, the following lemma shows that  $\lambda_{0}$ must be the unique 
negative eigenvalue.
\begin{lemma}\label{negative1}
	For any real-valued $\mathbf{u}\in H_x^1l^2(\mathbb{R}^2\times\mathbb{Z}_{N})$ with $\mathbf{u}\perp \mathbf{Q}$, we have
	\begin{align*}
	\left\langle L_{+} \mathbf{u}, \mathbf{u}  \right\rangle_{L_x^2l^2}\geq 0.
	\end{align*}
\end{lemma}
\begin{proof}
	Recall that $\mathbf{Q}$ is obtained by the maximization problem
    \begin{align*}
     \sup\limits_{ \substack{ \mathbf{u} \in H_x^1 l^2,\\ \mathbf{u} \ne 0} }
     J( \mathbf{u}),
    \end{align*}
    where
    \begin{align*}
        J(\mathbf{f})=\frac{\sum_{j\in\Z_N}\int_{\R^2}  F_j(\mathbf{f})\cdot \bar{f}_j \, \mathrm{d} x}{\|\mathbf{f}\|^2_{L^2_x l^2}\|\nabla\mathbf{f}\|^2_{L^2_x l^2}}. 
    \end{align*}
    Thus, 
    \begin{align*}
     \frac{ \mathrm{d} }{ \mathrm{d}  {\varepsilon}} \bigg|_{\epsilon=0}J \left(\mathbf{Q}+\varepsilon \mathbf{u} \right)=0,\quad\forall\ \mathbf{u}\in H^1_xl^2{\backslash \{0\}}
    \end{align*}
    and
	\begin{align}\label{mini}
		\frac{ \mathrm{d}^2 }{ \mathrm{d}  {\varepsilon}^2} \bigg|_{\epsilon=0}J \left(\mathbf{Q}+\varepsilon \mathbf{u} \right)\leq 0.
	\end{align}
If $\mathbf{u}\perp \mathbf{Q}$, then by direct calculation, \eqref{mini} is equivalent to
\begin{align*}
\left\langle L_{+} \mathbf{u}, \mathbf{u}  \right\rangle_{L_x^2l^2}
&\gtrsim_{ \left\|\mathbf{Q} \right\|_{L_x^2l^2}} \left(\int_{\mathbb{R}^2}\sum_{j\in\mathbb{Z}_N}\Delta Q u_j \,\mathrm{d}x \right)^2
-\left(\int_{\mathbb{R}^2}\sum_{j\in\mathbb{Z}_N} (2N-1)Q^2 u_j \,\mathrm{d}x \right)^2\notag\\
	&\geq  \left(\int_{\mathbb{R}^2}\sum_{j\in\mathbb{Z}_N} Qu_j \,\mathrm{d}x  \right)\left(\int_{\mathbb{R}^2}\sum_{j\in\mathbb{Z}_N} \left(\Delta Q+(2N-1)Q\right)u_j \,\mathrm{d}x \right)=0.
\end{align*}
This completes the proof.\end{proof}
Next, we determine the null space of $L_{+}$. In \cite{DW,WY}, the authors have proved that it is $2$-dimensional when $N=2$. Specifically, they proved the following lemma:
\begin{lemma}[\cite{DW,WY}]\label{null}
	The  null space of the operator $L_{+}$ in real-valued $H_x^1 l^2(\mathbb{R}^2\times\mathbb{Z}_{2})$ is spanned by $$\big\{({Q}_{x_1}, {Q}_{x_1}), ({Q}_{x_2}, {Q}_{x_2})\big\}$$
    when $N=2$.
\end{lemma}

In fact, they have proven the above conclusion for the more general coupled Schr\"odinger system. However, it seems hard to extend this conclusion to the case of more than two components via using their argument. Nevertheless, for the system \eqref{1.1}, we can straightforwardly extend this  conclusion to the multiply coupled case:
\begin{lemma}\label{assum}
The null space of the operator $L_{+}$ in real-vauled $H_x^1 l^2(\mathbb{R}^2\times\mathbb{Z}_{N})$ is spanned by
\begin{align*}
\big\{({Q}_{x_1},\cdots,{Q}_{x_1}), ({Q}_{x_2},\cdots,{Q}_{x_2})\big\}
\end{align*}
when $N\geq 2$.

\begin{proof}
Suppose that $\mathbf{r}\in H_x^1 l^2(\mathbb{R}^2\times\mathbb{Z}_{N})$ is a  vector-valued function such that $L_{+}\mathbf{r}=0$. Then by \eqref{L+}, $\mathbf{r}$ satisfies the following coupled system:
\begin{equation}\label{C1.1}
		\begin{cases}
			\Delta r_1 -r_1+(2N+1)Q^2 r_1 +\sum\limits_{\substack{j\in\mathbb{Z}_N\\j\neq 1}} 4Q^2 r_j= 0,\\
			\quad\qquad\qquad\qquad\qquad\vdots\\
             \Delta r_N-r_N +(2N+1)Q^2 r_N +\sum\limits_{\substack{j\in\mathbb{Z}_N\\j\neq N}} 4Q^2 r_j= 0.\\
		\end{cases}
	\end{equation}
   We claim that $r_j=r_k, \; \forall\ j,k\in \mathbb{Z}_N$. To prove this, let $h=r_1-r_2$, then by \eqref{C1.1}, $h$ satisfies the following linear equation
   \begin{equation*}
   \Delta h -h +(2N-3)Q^2 h = 0.
   \end{equation*}
   Multiplying by $h$ in the above equation and recalling that $Q=\sqrt{\frac{1}{2N-1}}Q_0$, we obtain 
   \begin{equation*}
       \int_{\mathbb{R}^2} \left( |\nabla h|^2+h^2-\frac{2N-3}{2N-1}Q_0^2h^2 \right) \,  \mathrm{d}x =0.
   \end{equation*}
   By \eqref{PL_-}, $\|Q_0h\|_{L_x^2}=0$, which implies $r_1-r_2=h=0$. Similar argument can be applied to the other components, so $r_1=\cdots=r_N$. Setting $r_1=\cdots=r_N=r$, the system \eqref{C1.1} reduces to the single equation
    $$\Delta r-r+ (6N-3) Q^2r=L_{0,+} r=0. $$
By Theorem \ref{L01spectral}, $h\in span\big\{({Q_0})_{x_1}, ({Q_0})_{x_2}\big\}=span\big\{{Q}_{x_1}, {Q}_{x_2}\big\} $.
\end{proof}
\end{lemma}

Now, Lemma \ref{negative1} and Lemma \ref{assum} yield the following positivity estimate for the operator $L_{+}$:
\begin{theorem}\label{positive}
	 There exists a positive constant $c_0$ such that
	\begin{align*}
	\left\langle L_{+} \mathbf{u}, \mathbf{u}  \right\rangle_{L_x^2l^2} \geq c_0 \left\langle \mathbf{u}, \mathbf{u} \right\rangle_{L_x^2l^2},
	\end{align*}
	for any real-valued $\mathbf{u}\in H_x^1l^2 \left(\mathbb{R}^2\times\mathbb{Z}_{N} \right)$ and $\mathbf{u}\perp  \left\{\boldsymbol{\chi}_0, \mathbf{Q}_{x_1}, \mathbf{Q}_{x_2} \right\}$.
\end{theorem}
\begin{proof}
	Let $\sigma=\inf  \left\{ \left\langle L_{+} \mathbf{u},   \mathbf{u}  \right\rangle_{L_x^2l^2}:
    \left\|\mathbf{u} \right\|_{L_x^2l^2}=1,   \mathbf{u}\perp
\left\{\boldsymbol{\chi}_0, \mathbf{Q}_{x_1}, \mathbf{Q}_{x_2} \right\} \right\}$. 
By Lemma \ref{negative1}, $\sigma\geq 0$. We only need to show that $\sigma>0$. Suppose $\sigma=0$, then arguing as in \cite{MF1,weistein}, 
there exists a vector-valued function $\boldsymbol{\epsilon}_*\in H_x^1l^2 \left(\mathbb{R}^2\times\mathbb{Z}_N \right)$ such that the minimum is attained at
	$\boldsymbol{\epsilon}_*$ and satisfies
	\begin{align}
		&L_{+}\boldsymbol{\epsilon}_*=\alpha \boldsymbol{\epsilon}_*+ \beta{\boldsymbol{\chi}_0}+\gamma \mathbf{Q}_{x_1}+ \eta \mathbf{Q}_{x_2}\label{langr},\\
		& \left\langle L_{+}\boldsymbol{\epsilon}_*, \boldsymbol{\epsilon}_* \right\rangle_{L_x^2l^2}=0, \quad  \left\|\boldsymbol{\epsilon}_* \right\|_{L_x^2l^2}=1,\quad \boldsymbol{\epsilon}_*\perp \left\{\boldsymbol{\chi}_0, \mathbf{Q}_{x_1}, \mathbf{Q}_{x_2} \right\}.\notag
	\end{align}
    Taking  the inner product of \eqref{langr} with $\boldsymbol{\epsilon}_*$, we have $ \left\langle L_{+}\boldsymbol{\epsilon}_*, \boldsymbol{\epsilon}_* \right\rangle_{L_x^2l^2}=\alpha$, so $\alpha=0$. Then by taking inner products  with $\boldsymbol{\chi}_0, \mathbf{Q}_{x_1}$, and $\mathbf{Q}_{x_2}$ respectively, we obtain
	$\beta= \gamma = \eta=0$. Thus,  $L_{+}\boldsymbol{\epsilon}_*=0$. But by Lemma \ref{null} and  $\boldsymbol{\epsilon}_*\perp  \left\{ \mathbf{Q}_{x_1}, \mathbf{Q}_{x_2} \right\}$, we conclude  $\boldsymbol{\epsilon}_*=0$, which is a contradiction. Therefore, $\sigma>0$.
\end{proof}

Our next goal is  to present a decomposition of the solution $\mathbf{u}(t)$ and its initial properties. This decomposition is pivotal in deriving the uniform control of the reminder in the next section. Recalling Lemma \ref{l2.1}, if $\mathbf{u}(t)$ is the blowup solution in Theorem \ref{t2.3}, it admits the decomposition:
\begin{align*}
	& u_j(t)=e^{i\gamma_{0, j} (t)}e^{ix\cdot\xi_0(t)}\lambda_{0}(t)\left[Q(\lambda_{0}(t)( x + x_0(t)) )+\epsilon_{0,j}(t,\lambda_{0}( x + x_0(t)) )\right], 
    j\in\mathbb{Z}_N  \\
    & \mbox{ with }  \|\boldsymbol{\epsilon}_0(t)\|_{L^2_x l^2}=\alpha(t)\leq \eta_{\ast}.
\end{align*}
Now, we will refine this decomposition via  using the implicit function theorem according
to the following theorem:
\begin{theorem}\label{modulation}
 Let $\mathbf{u}\in L_x^2l^2 \left(\mathbb{R}^2\times\mathbb{Z}_N \right)$. There exists $\alpha > 0$ sufficiently small such that if there exist $\lambda_{0} > 0$, $(\gamma_{0,1}\cdots \gamma_{0,N})\in [0,2\pi]^N$, $x_{0} \in \mathbb{R}^{2}$, and $\xi_{0} \in \mathbb{R}^{2}$ satisfying 
\begin{equation*} 
\bigg(\sum\limits_{j\in\mathbb{Z}_N}  \bigg\| e^{i \gamma_{0,j}} e^{ix \cdot \xi_{0}} \lambda_{0} u_j (\lambda_{0} x + x_{0}) -
Q(x) \bigg\|_{L_x^{2}}^2\bigg)^{1/2} \leq \alpha,
\end{equation*}
then there exist unique $\lambda > 0$, $(\gamma_1,\cdots,\gamma_N) \in [0,2\pi]^N$, $\tilde{x} \in \mathbb{R}^{2}$ and $\xi \in \mathbb{R}^{2}$ such that if $\boldsymbol{\epsilon}(x)=(\epsilon_1,\cdots, \epsilon_N)$ with
\begin{equation*} 
\epsilon_j(x) = e^{i \gamma_j} e^{ix \cdot \xi} \lambda u_j \left(\lambda x + \tilde{x} \right) - Q(x),\quad j\in\mathbb{Z}_N,
\end{equation*}
then the following orthogonality conditions hold: 
\begin{align}\label{f4.3}
\aligned
\left\langle \boldsymbol{\epsilon}, \boldsymbol{\chi}_{0}  \right\rangle_{L^2_x l^2}
&= \left\langle\boldsymbol{\epsilon}, \mathbf{Q}_{x_{1}} \right\rangle_{L^2_x l^2}
= \left\langle\boldsymbol{\epsilon}, \mathbf{Q}_{x_{2}} \right\rangle_{L^2_x l^2}
=  \left\langle\boldsymbol{\epsilon}, i\mathbf{ Q}_{x_{1}} \right\rangle_{L^2_x l^2}=  \left\langle\boldsymbol{\epsilon},i\mathbf{ Q}_{x_{2}} \right\rangle_{L^2_x l^2} = 0, \\
& \left\langle \boldsymbol{\epsilon}, i\boldsymbol{ \chi}_{0,j}  \right\rangle_{L^2_x l^2  }=\left\langle \epsilon_j, i \chi_{0}  \right\rangle_{L_x^2  }=0,\quad \forall\ j\in \mathbb{Z}_N,
\endaligned
\end{align}
and the following estimate holds: 
\begin{align}\label{f4.4}
\aligned
& \left\| \boldsymbol{\epsilon}  \right\|_{L^2_x l^2} +   \left|\frac{\lambda}{\lambda_{0}} - 1 \right|
+ \sum\limits_{j\in\mathbb{Z}_N} \left|\gamma_j - \gamma_{0,j} - \xi_{0} \cdot  \left(\tilde{x} - x_{0} \right) \right|
 +  \left|\xi - \frac{\lambda}{\lambda_{0}} \xi_{0} \right| +  \left|\frac{\tilde{x} - x_{0}}{\lambda_{0}} \right|  \\
&\lesssim \bigg(\sum\limits_{j\in\mathbb{Z}_N} \bigg\| e^{i \gamma_{0,j}} e^{ix \cdot \xi_{0}} \lambda_{0} u_j \left(\lambda_{0} x + x_{0} \right) - Q(x) \bigg\|_{L_x^{2}}^2\bigg)^{1/2},
\endaligned
\end{align}
here   $ i\mathbf{Q}_{x_j}=(iQ_{x_j},\cdots,iQ_{x_j}), \mathbf{Q}_{x_j}=(Q_{x_j},\cdots,Q_{x_j})$ for $  j\in\{1,2\}$, and $i\boldsymbol{\chi}_{0,j}=(0,\cdots, 0,\underset{\text{j-th }}{i\chi_0},0,\cdots,0)$ for $ j\in \mathbb{Z}_N$.
\end{theorem}

\begin{proof}
Let
	\begin{equation*} 
		 {\epsilon}_{0,j}(x) = e^{i \gamma_{0,j}} e^{ix \cdot \xi_0} \lambda u_j \left(\lambda_0 x + \tilde{x}_0 \right) - Q(x),\         j\in\mathbb{Z}_N, 
	\end{equation*}
and let the vector-valued function $\mathbf{f}$ be an element of the set
\begin{equation*} 
\mathbf{f} \in  \left\{ i\boldsymbol{\chi}_{0,1},\cdots, i\boldsymbol{\chi}_{0,N}, \mathbf{ \chi}_{0}, \mathbf{Q}_{x_{1}},  \mathbf{Q}_{x_{2}}, i\mathbf{ Q}_{x_{1}}, i\mathbf{Q}_{x_{2}}  \right\}.
\end{equation*}
By H{\"o}lder's inequality,
\begin{equation*}
 \left|\langle \boldsymbol{\epsilon}_0, \mathbf{ f}  \rangle_{L_x^2l^2} \right| \lesssim  \left\| \boldsymbol{\epsilon}_0  \right\|_{L^2_x l^2}.
\end{equation*}
Next we define the map $ \mathbf{F}: L^2_x l^2(\mathbb{R}^2\times\mathbb{Z}_N) \times\mathbb{R}^{N}\times\mathbb{R}^5 \to \mathbb{R}^{N}\times\mathbb{R}^5 $ by
\begin{align}\label{inverseimp}
	&  \mathbf{F}  \left(\mathbf{u},\boldsymbol{\gamma},\lambda,\xi_1,\xi_2,\tilde{x}_1,\tilde{x}_2 \right) \\\nonumber
 =&\Big( \left\langle \boldsymbol{\epsilon}, \boldsymbol{\chi}_{0,1}  \right\rangle_{L^2_x l^2}, \cdots, \left\langle \boldsymbol{\epsilon}, \boldsymbol{\chi}_{0,N}  \right\rangle_{L^2_x l^2},
	\left\langle \boldsymbol{\epsilon}, i \boldsymbol{\chi}_{0} \right\rangle_{L^2_x l^2}, \left\langle\boldsymbol{\epsilon}, \mathbf{Q}_{x_{1}} \right\rangle_{L^2_x l^2} ,
	\left\langle\boldsymbol{\epsilon}, \mathbf{Q}_{x_{2}} \right\rangle_{L^2_x l^2},
	\left\langle\boldsymbol{\epsilon}, i\mathbf{ Q}_{x_{1}} \right\rangle_{L^2_x l^2},  \left\langle\boldsymbol{\epsilon},i\mathbf{ Q}_{x_{2}} \right\rangle_{L^2_x l^2} \Big).
\end{align}
It's easy to check that the inner products in \eqref{inverseimp} are $C^{1}$ functions of $\gamma$, $\lambda$, $\tilde{x}$, and $\xi$.
 Indeed, $\forall  \, j \in \mathbb{Z}_N$,
\begin{equation}\label{f4.9}
\left|\frac{\partial}{\partial \gamma_j} \left\langle \boldsymbol{\epsilon}, \mathbf{f}  \right\rangle_{L^2_x l^2}\right|
= \bigg|\big\langle i e^{i \gamma_j} e^{ix \cdot \xi} \lambda u_j \left(\lambda x+\tilde{x} \right)  , f_j(x ) \big\rangle_{L_x^2}\bigg|
\lesssim  \left\| \mathbf{u} \right\|_{L^2_x l^2} \left\| \mathbf{f}  \right\|_{L^2_x l^2}.
\end{equation}
Next, since every component of $\mathbf{f}$ is smooth with rapidly  decreasing derivatives, $\forall  \, l  = 1, 2, $
\begin{align}
&\left|\frac{\partial}{\partial \xi_l}  \left\langle \boldsymbol{\epsilon}, \mathbf{f}   \right\rangle_{L^2_x l^2}\right|
= \bigg|\sum\limits_{j\in\mathbb{Z}_N}\big\langle i e^{i \gamma_j} e^{ix_l  \cdot \xi} \lambda u_j \left(\lambda x+\tilde{x} \right) , f_j (x ) \big\rangle_{L_x^2}\bigg|
\lesssim  \left\| \mathbf{u}  \right\|_{L^2_x l^2}  \left\| x_l\mathbf{f}  \right\|_{L^2_x l^2}.\label{f4.10}
\end{align}
Integrating by parts, $\forall  \, l = 1, 2$,
\begin{align}
&\left|\frac{\partial}{\partial \tilde{x}_l} \left\langle \boldsymbol{\epsilon}, \mathbf{f}  \right\rangle_{L^2_x l^2}\right|
=\bigg|\sum\limits_{j\in\mathbb{Z}_N}  \left\langle e^{i \gamma_j} e^{ix \cdot \xi} \lambda  \partial_{x_l } u_j \left(\lambda x + \tilde{x} \right),
 f_j(x)   \right\rangle_{L_x^{2}}\bigg| \label{f4.12} \\
 = &\bigg|-\frac{1}{\lambda} \sum\limits_{j\in\mathbb{Z}_N}  \left\langle e^{i \gamma_j} e^{ix \cdot \xi} \lambda  u_j \left(\lambda x + \tilde{x} \right), \partial_{x_l} f_j (x) \right\rangle_{L_x^{2}}
 - \frac{\xi_l}{\lambda} \sum\limits_{j\in\mathbb{Z}_N}  \left\langle i e^{i \gamma_j } e^{ix \cdot \xi} \lambda u_j \left(\lambda x + \tilde{x} \right), f_j (x) \right\rangle_{L_x^{2}}\bigg|\notag\\
\lesssim& \frac{1}{\lambda}  \left\| \mathbf{u}  \right\|_{L^2_x l^2}  \left\| \partial_{x_l} \mathbf{f}  \right\|_{L^2_x l^2}
+ \frac{|\xi_l|}{\lambda} \left\| \mathbf{u}  \right\|_{L^2_x l^2}  \left\| \mathbf{f}  \right\|_{L^2_x l^2}.\notag
\end{align}
Finally,
\begin{align}\label{f4.14}
\aligned
&\left|\frac{\partial}{\partial \lambda}  \left\langle \boldsymbol{\epsilon}, \mathbf{f}  \right\rangle_{L^2_x l^2}\right|
= \bigg|\sum\limits_{j\in\mathbb{Z}_N} \left\langle   e^{i \gamma_j} e^{ix \cdot \xi} u \left(\lambda x + \tilde{x} \right) + e^{i \gamma_j } e^{ix \cdot \xi} \lambda x \cdot \nabla u \left (\lambda x + \tilde{x} \right), f_j (x) \right\rangle_{L_x^2} \bigg|\\
\lesssim& \frac{1}{\lambda}  \left\| \mathbf{u}  \right\|_{L^2_x l^2}  \left\| \mathbf{f}  \right\|_{L^2_x l^2} + \frac{|\xi|}{\lambda}  \left\| \mathbf{u}  \right\|_{L^2_x l^2}
\left\| |x| \mathbf{f}  \right\|_{L^2_x l^2} + \frac{1}{\lambda}  \left\| \mathbf{u}  \right\|_{L^2_x l^2}  \left\| |x| \nabla \mathbf{f}  \right\|_{L^2_x l^2}.
\endaligned
\end{align}
Therefore, the inner product is a $C^{1}$ function of $\gamma$, $\xi$, $\lambda$, and $\tilde{x}$. Repeating the above calculations would also show that the inner product is $C^{2}$.

Computing \eqref{f4.9}-\eqref{f4.14} at $u_j = e^{-i\gamma_{0,j}}e^{-ix\cdot\xi_0}\frac{1}{\lambda_0}Q \left(\frac{x}{\lambda_{0}}-\tilde{x} \right)$, $\xi=\xi_0, \tilde{x}=\tilde{x}_0,  \lambda= \lambda_0, \gamma_j = \gamma_{0,j},\ \forall \,  {j\in\mathbb{Z}_N}$, we have
\begin{align*}
&\frac{\partial}{\partial \gamma_j}  \left\langle \boldsymbol{\epsilon}, \mathbf{f}  \right\rangle_{L_x^2l^2} \bigg|_{u_1 = e^{-i\gamma_{0,1}}e^{-ix\cdot\xi_0}\frac{1}{\lambda_0}Q \left(\frac{x}{\lambda_{0}}-\tilde{x} \right),\cdots, u_N= e^{-i\gamma_{0,N}}e^{-ix\cdot\xi_0}\frac{1}{\lambda_0}Q \left(\frac{x}{\lambda_{0}}-\tilde{x} \right), \lambda = \lambda_{0}, \gamma_j=\gamma_{0,j}, \tilde{x} =\tilde{x}_0, \xi = \xi_0}\\
  =& \langle i Q, f_j \rangle_{L_x^2} =
\begin{cases}
	0 ,& \text{if} \quad \mathbf{f} \in  \left\{ \boldsymbol{\chi}_{0}, \mathbf{Q}_{x_{1}}, \mathbf{Q}_{x_{2}},i \mathbf{Q}_{x_{1}}, i\mathbf{Q}_{x_{2}} \right\},\\
	0, &\text{if}  \quad \mathbf{f} \in  \left\{{i\boldsymbol{\chi}_{0,k}}, k\in\mathbb{Z}_N, k\neq j \right\},\\
	\langle Q, \chi_0\rangle_{L_x^{2}} > 0,  &\text{if} \quad \mathbf{f} = i\boldsymbol{\chi}_{0,j}.\notag
\end{cases}
\end{align*}
The fact that $\langle Q, \chi_0\rangle_{L_x^{2}} > 0$ follows from the fact that $\chi_{0} > 0$ and $Q > 0$. Next, for $l=1,2$, we have
\begin{align*}
	&\frac{\partial}{\partial \xi_l}  \left\langle \boldsymbol{\epsilon}, \mathbf{f}  \right\rangle_{L^2_x l^2} \bigg|_{u_1 = e^{-i\gamma_{0,1}}e^{-ix\cdot\xi_0}\frac{1}{\lambda_0}Q \left(\frac{x}{\lambda_{0}}-\tilde{x} \right),\cdots, u_N= e^{-i\gamma_{0,N}}e^{-ix\cdot\xi_0}\frac{1}{\lambda_0}Q \left(\frac{x}{\lambda_{0}}-\tilde{x} \right), \lambda = \lambda_{0}, \gamma_j=\gamma_{0,j}, \tilde{x} =\tilde{x}_0, \xi = \xi_0}\\
	 =& \sum\limits_{j\in\mathbb{Z}_N}\langle i x_lQ, f_j \rangle_{L_x^2} =
	\begin{cases}
		0 ,& \text{if} \quad \mathbf{f} \in  \left\{ \boldsymbol{\chi}_{0}, \mathbf{Q}_{x_{1}}, \mathbf{Q}_{x_{2}} \right\},\\
		0, &\text{if}  \quad \mathbf{f} \in  \left\{{i\boldsymbol{\chi}_{0,k}}, \  k\in\mathbb{Z}_N \right\},\\
		-\frac{\delta_{k l}N}{2}\|Q\|^2_{L_x^2 \left(\mathbb{R}^2 \right)},  &\text{if} \quad \mathbf{f} = i\mathbf{Q}_{x_k },\quad k\in\{1,2\},
	\end{cases}\notag
\end{align*}
and
\begin{align*}
	&\frac{\partial}{\partial \tilde{x}_l} \left\langle \boldsymbol{\epsilon}, \mathbf{f}  \right\rangle_{L^2_x l^2} \bigg|_{u_1 = e^{-i\gamma_{0,1}}e^{-ix\cdot\xi_0}\frac{1}{\lambda_0}Q \left(\frac{x}{\lambda_{0}}-\tilde{x} \right),\cdots, u_N= e^{-i\gamma_{0,N}}e^{-ix\cdot\xi_0}\frac{1}{\lambda_0}Q \left(\frac{x}{\lambda_{0}}-\tilde{x} \right), \lambda = \lambda_{0}, \gamma_j=\gamma_{0,j}, \tilde{x} =\tilde{x}_0, \xi = \xi_0}\\
	  = &\sum\limits_{j\in\mathbb{Z}_N}\langle Q_{x_l}, f_j \rangle_{L_x^2} =
	\begin{cases}
		0 ,& \text{if} \quad \mathbf{f} \in  \left\{ \boldsymbol{\chi}_{0}, i\mathbf{Q}_{x_{1}}, i\mathbf{Q}_{x_{2}}  \right\},\\
		0, &\text{if}  \quad \mathbf{f} \in  \left\{i\boldsymbol{\chi}_{0,k},\ k\in\mathbb{Z}_N \right\},\\
		\frac{\delta_{k l}N}{2}\|\nabla Q\|^2_{L^2_x\left(\mathbb{R}^2 \right)},  &\text{if} \quad \mathbf{f} = \mathbf{Q}_{x_k },\quad k\in\{1,2\},
	\end{cases}\notag
\end{align*}
where $\delta_{kl}$ is standard Kronecker symbol. Finally,
\begin{align*}
	&\frac{\partial}{\partial \lambda}  \left\langle \boldsymbol{\epsilon}, \mathbf{f}  \right\rangle_{L^2_x l^2} \bigg|_{u_1 = e^{-i\gamma_{0,1}}e^{-ix\cdot\xi_0}\frac{1}{\lambda_0}Q \left(\frac{x}{\lambda_{0}}-\tilde{x} \right),\cdots, u_N= e^{-i\gamma_{0,N}}e^{-ix\cdot\xi_0}\frac{1}{\lambda_0}Q \left(\frac{x}{\lambda_{0}}-\tilde{x} \right), \lambda = \lambda_{0}, \gamma_j=\gamma_{0,j}, \tilde{x} =\tilde{x}_0, \xi = \xi_0}\\
	  =& \sum\limits_{j\in\mathbb{Z}_N}\langle Q+x\cdot\nabla Q, f_j \rangle_{L_x^2} =
	\begin{cases}
		0 ,& \text{if} \quad \mathbf{f} \in  \left\{ \mathbf{Q}_{x_{1}}, \mathbf{Q}_{x_{2}},i\mathbf{Q}_{x_{1}}, i\mathbf{Q}_{x_{2}} \right\},\\
		0, &\text{if}  \quad \mathbf{f} \in  \left\{{i\boldsymbol{\chi}_{0,k}},\   k\in\mathbb{Z}_N \right\},\\
		\frac{2N}{\lambda_0}\langle Q, \chi_0\rangle_{L_x^2}>0,  &\text{if} \quad \mathbf{f} = \boldsymbol{\chi}_{0},
	\end{cases}\notag
\end{align*}
where we used the fact that $L_{+}$ is self-adjoint and $L_{+}\boldsymbol{\chi}_0=-\lambda_0 \boldsymbol{\chi}_0$.

Therefore, by the implicit function theorem associated to $F$, there exist  unique $\lambda > 0$, $\tilde{x} \in \mathbb{R}^{2}$, $\xi \in \mathbb{R}^{2}$,  and $(\gamma_{1}, \cdots , \gamma_{N})\in [0,2\pi]^N $ in a neighborhood of $\lambda_{0}$, $\xi_{0}$, $x_{0}$, and $(\gamma_{0,1}, \cdots , \gamma_{0,N})$ such that \eqref{f4.3} and \eqref{f4.4} hold.
 The proof of uniqueness  is identical to the proof of uniqueness in \cite{D2}, so we omit it.
\end{proof}

By Theorem \ref{modulation}, if $\mathbf{u}(t)$ is the solution as in Theorem \ref{t2.3}, then there exist functions $\lambda(t)$, $x(t)$, $\xi(t)$, and  $\gamma_1(t),\cdots, \gamma_{N}(t)$ on $[0,\sup I)$ such that if
\begin{equation}\label{decomp}
\epsilon_j(t,x) = e^{i \gamma_j(t)} e^{ix \cdot \xi(t)} \lambda u_j(\lambda(t) x + x(t)) - Q(x),\  j\in \mathbb{Z}_N,
\end{equation}
then the following orthogonality conditions hold: 
\begin{align}\label{orthod}
\aligned
\left\langle \boldsymbol{\epsilon}(t), \boldsymbol{\chi}_{0}  \right\rangle_{L^2_x l^2}
 &=  \left\langle\boldsymbol{\epsilon}(t), \mathbf{Q}_{x_{1}}  \right\rangle_{L^2_x l^2} = \left\langle\boldsymbol{\epsilon}(t), \mathbf{Q}_{x_{2}} \right\rangle_{L^2_x l^2}
=  \left\langle\boldsymbol{\epsilon}(t), i\mathbf{ Q}_{x_{1}} \right\rangle_{L^2_x l^2}=  \left\langle\boldsymbol{\epsilon}(t),i\mathbf{ Q}_{x_{2}} \right\rangle_{L^2_x l^2} = 0, \\
&\left\langle \boldsymbol{\epsilon}(t), i \chi_{0,j} \right\rangle_{L^2_x l^2}=\left\langle \epsilon_j(t), i \chi_{0} \right\rangle_{L_x^2}=0,\qquad \forall \,  j\in \mathbb{Z}_N.
\endaligned
\end{align}
Also as in \cite{D2} and \cite{D1}, it can be shown that $\lambda(t)$, $x(t)$, $\xi(t)$ and  $\gamma_1(t),\cdots, \gamma_{N}(t)$ are continuous functions on $[0, \sup I)$, and are differentiable almost everywhere on $[0, \sup I)$. Moreover, we have the following lemma:
\begin{lemma}\label{lemma23}
	Let $\mathbf{u}$ be the solution in Theorem \ref{t2.3}. If $\eta_{\ast}$ is sufficiently small, then
\begin{align}\label{qaze}
	\left\|\mathbf{u}(t) \right\|^4_{L_{t,x}^4l^2 \left(J\times\mathbb{R}^2\times\mathbb{Z}_N \right)} \lesssim 1+\int_{J}
    {\lambda(t)^{-2}}  \, \mathrm{d} t.
	\end{align}
\end{lemma}
\begin{proof}
	Let $0<\eta\ll 1$ to be chosen temporarily and partition $J$ into some tiny subintervals $I_j$ so that
	\begin{align*}
		\frac{\eta}{1+\int_{J} \lambda(t)^{-2}dt}\int_J \lambda(t)^{-2} \, \mathrm{d} t \leq \int_{I_j} \lambda(t)^{-2} \,  \mathrm{d} t  \leq \eta.
	\end{align*}
	This requires at most $\eta^{-1}\left(1+\int_{J} \lambda(t)^{-2}\, \mathrm{d}t \right)$ intervals. For each $j$, we can choose $t_j\in I_j$ so that
	\begin{align*}
		\lambda(t_j)^2|I_j|\leq 2\eta.
	\end{align*}
	Recalling that $$e^{i\gamma_j(t)}e^{ix\cdot\xi(t)}\lambda(t)u_j \left(t,\lambda(t)x+x(t) \right)=Q(x)+ \epsilon_j(t), \ \forall \,  j\in\mathbb{Z}_N$$
    and $\sup\limits_{t\in[0,\infty)} \left\|\boldsymbol{\epsilon}(t) \right\|_{L^2_x l^2}\leq \eta_{\ast}$, by the Strichartz estimate, for any interval $J_{\alpha}\subset I_j$ with $t_j\in J_{\alpha}$, we have
	\begin{align*}		
& \left\|\mathbf{u}(t) \right\|_{L_{t,x}^4l^2 \left(J_{\alpha}\times\mathbb{R}^2\times\mathbb{Z}_N \right)} \\
=& \left\|e^{ix\cdot\xi(t_j)}\mathbf{u}(t,x-2(t-t_j)\xi(t_j)) \right\|_{L_{t,x}^4l^2 \left(J_{\alpha}\times\mathbb{R}^2\times\mathbb{Z}_N \right)}\notag\\
		\lesssim&   \left\|e^{i(t-t_j)\Delta} \left(e^{ix\cdot\xi(t_j)}\mathbf{u}(t_j,x) \right) \right\|_{L_{t,x}^4l^2 \left(J_{\alpha}\times\mathbb{R}^2\times\mathbb{Z}_N \right)} + \left\|e^{ix\cdot\xi(t_j)}\mathbf{u}(t,x-2(t-t_j)\xi(t_j)) \right\|^3_{L_{t,x}^4l^2 \left(J_{\alpha}\times\mathbb{R}^2\times\mathbb{Z}_N \right)}\notag\\
		\lesssim& \eta_{\ast}+ \left\|e^{i(t-t_j)\Delta}\lambda(t_j)^{-1}Q \left(\lambda(t_j)^{-1}(x-x(t_j)) \right) \right\|_{L_{t,x}^4 \left(J_{\alpha}\times\mathbb{R}^2 \right)} + \left\|\mathbf{u}(t) \right\|^3_{L_{t,x}^4l^2 \left(J_{\alpha}\times\mathbb{R}^2\times\mathbb{Z}_N \right)} .
	\end{align*}
Choosing $C_0\gg 1$ such that $\|P_{\geq C_0}Q\|_{L_x^2}\leq \eta_0$, we continue:
\begin{align}\nonumber
	&\hspace{10ex}\lesssim \eta_{\ast}+ 
\left\|e^{i(t-t_j)\Delta} P_{\leq C_0\lambda(t_j)}\lambda(t_j)^{-1}Q \left(\lambda(t_j)^{-1}(x-x(t_j)) \right) \right\|_{L_{t,x}^4 \left(J_{\alpha}\times\mathbb{R}^2 \right)} \\\nonumber
	&\hspace{13ex}+ \left\|e^{i(t-t_j)\Delta} P_{\ge C_0\lambda(t_j)}\lambda(t_j)^{-1}Q \left(\lambda(t_j)^{-1}(x-x(t_j)) \right) \right\|_{L_{t,x}^4 \left(J_{\alpha}\times\mathbb{R}^2 \right)}
+ \left\|\mathbf{u}(t) \right\|^3_{L_{t,x}^4l^2 \left(J_{\alpha}\times\mathbb{R}^2\times\mathbb{Z}_N \right)} \\\nonumber
	&\hspace{10ex}\lesssim \eta_{\ast}+\eta_0+|J_{\alpha}|^{1/2}C_0\lambda(t_j)
 + \left\|\mathbf{u}(t) \right\|^3_{L_{t,x}^4l^2 \left(J_{\alpha}\times\mathbb{R}^2\times\mathbb{Z}_N \right)} \\\label{qd1}
	&\hspace{10ex}\lesssim \eta_{\ast}+\eta_0
+C_0 \eta^{1/2}+ \left\|\mathbf{u}(t) \right\|^3_{L_{t,x}^4l^2 \left(J_{\alpha}\times\mathbb{R}^2\times\mathbb{Z}_N \right)}.
\end{align}
Choosing sufficiently small $\eta_{\ast}$, $\eta_0$ and  then choosing sufficiently small $\eta$ so that $\eta_{\ast}+\eta_0+C_0\eta^{1/2}\ll 1$, then by a bootstrap argument, \eqref{qd1} implies
\begin{align*}
	\left\|\mathbf{u}(t) \right\|_{L_{t,x}^4l^2 \left(I_j\times\mathbb{R}^2\times\mathbb{Z}_N \right)}\leq 1.
\end{align*}
Using the bound on the number of the intervals $I_j$, we can obtain \eqref{qaze}.
\end{proof}
Define the monotone function $s : [0, \sup I ) \rightarrow \mathbb{R}$,
\begin{equation}\label{f4.23}
	s(t) = \int_{0}^{t} \lambda(\tau)^{-2} \, \mathrm{ d} \tau.
\end{equation}
This function is invertible and by Lemma \ref{lemma23}, $s([0,\infty))=\infty$. Let $t(s) : [0, \infty) \rightarrow [0, \sup I)$ be its inverse function. Define 
\begin{equation}\label{f4.24}
	\aligned
	\lambda(s) = \lambda(t(s)), \qquad \gamma_j(s) = \gamma_j(t(s)), \qquad \forall \, j\in \mathbb{Z}_N, \\
	x(s) = x(t(s))=(x_1(t(s)), x_2(t(s))),  \qquad \xi(s) = \xi(t(s))=(\xi_1(t(s)), \xi_2(t(s))),\qquad
	\endaligned
\end{equation}
and 
\begin{equation}\label{f4.25}
	\epsilon_j (s, x) = e^{i \gamma_j(s)} e^{ix \cdot\xi(s)} \lambda(s) u_j (t(s), \lambda(x) x + x(s)) - Q(x)=\tilde{u}_j(s)-Q(x), \quad \forall \,  j\in\mathbb{Z}_N.
\end{equation}
In the final part of this section, we are devoted to establish the explicit control for the  geometric parameters \eqref{f4.24}.
First, assume $\mathbf{u}(t)\in H_x^1l^2$. Then for any $j\in \mathbb{Z}_N$, we can directly compute 
\begin{equation}\label{f4.26}
	\aligned
	\partial_s \epsilon_j  = & i (\gamma_j)_s (Q + \epsilon_j) + i \xi_{s}\cdot x (Q + \epsilon_j) + \frac{\lambda_{s}}{\lambda}
\left(\frac{1}{2} (Q + \epsilon_j) + x\cdot \nabla(Q + \epsilon_j) \right)
- i \frac{\lambda_{s}}{\lambda} \xi(s) \cdot x (Q + \epsilon_j) \\
	& + \frac{x_{s}}{\lambda} \cdot\nabla (Q + \epsilon_j) - i \frac{x_{s}}{\lambda} \cdot\xi(s) (Q + \epsilon_j)
 + 2 \xi(s)\cdot\nabla (Q + \epsilon_j)- i |\xi(s)|^{2} (Q + \epsilon_j) \\
	& + i \Delta(Q + \epsilon_j) + i |Q + \epsilon_j|^{2} (Q + \epsilon_j)
+2i\sum_{\substack{k\in\mathbb{Z}_N\\k\neq j}}|Q + \epsilon_k|^{2}(Q + \epsilon_j).
	\endaligned
\end{equation}
Taking $\mathbf{f} \in  \left\{ \boldsymbol{\chi}_0, i \boldsymbol{\chi}_{0,1},\cdots, i \boldsymbol{\chi}_{0,N},  \mathbf{Q}_{x_1}, \mathbf{Q}_{x_2}, i\mathbf{Q}_{x_1}, i\mathbf{Q}_{x_2}  \right\}$, by \eqref{orthod}, we have 
\begin{equation*} 
	\frac{ \mathrm{d} }{ \mathrm{d} s} \left\langle \boldsymbol{\epsilon}, \mathbf{f} \right\rangle_{L_x^{2}l^2}
 = \left\langle \boldsymbol{\epsilon}_{s}, \mathbf{f} \right\rangle_{L_x^{2}l^2} = 0.
\end{equation*}
By direct calculation, we derive that if $ \mathbf{f} \in \left\{ \boldsymbol{\chi}_0, \mathbf{Q}_{x_1}, \mathbf{Q}_{x_2}, i \mathbf{Q}_{x_1}, i\mathbf{Q}_{x_2} \right\}$, then
\begin{align}\nonumber
	&  \sum_{j\in\mathbb{Z}_N} \left\langle i (\gamma_{j})_s (Q + \epsilon_j), f_j \right\rangle_{L_x^{2}}
 =  \sum_{j\in\mathbb{Z}_N} \left\langle i (\gamma_{j})_s \epsilon_j, f_j  \right\rangle_{L_x^{2}}, \\\label{f4.28}
	&   \left\langle i (\gamma_{j})_s (Q + \epsilon_j), i\chi_0 \right\rangle_{L_x^{2}}
= (\gamma_{j})_s \left\langle Q, \chi_0  \right\rangle_{L_x^2}, \quad \forall \, j \in \mathbb{Z}_N.
\end{align}
When $\mathbf{f} \in  \left\{ \mathbf{Q}_{x_1}, \mathbf{Q}_{x_2}, \boldsymbol{\chi}_0, i\boldsymbol{\chi}_{0,1},\cdots,i\boldsymbol{\chi}_{0,N} \right\}$, we get
\begin{align}\label{f4.29}
	& \left\langle i \left(\xi_{s} - \frac{\lambda_{s}}{\lambda} \xi(s) \right)\cdot x  \left(\mathbf{Q} + \boldsymbol{\epsilon} \right), \mathbf{f} \right\rangle_{L^2_{x} l^2}
= O \left( \left|\xi_{s}(s) - \frac{\lambda_{s}}{\lambda} \xi(s) \right|  \left\| \boldsymbol{\epsilon}  \right\|_{L^2_{x} l^2} \right),\\\nonumber
	&  \left\langle i \left(\xi_{s} - \frac{\lambda_{s}}{\lambda} \xi(s) \right)\cdot x  \left(\mathbf{Q} + \boldsymbol{\epsilon} \right), i\mathbf{Q}_{x_l } \right\rangle_{L^2_{x} l^2}
= - \frac{N}{2}  \left( (\xi_{l })_s(s) - \frac{\lambda_{s}}{\lambda} \xi_l (s) \right) \| Q \|_{L_x^{2}}^{2}
 + O \left( \left|\xi_{s}(s) - \frac{\lambda_{s}}{\lambda} \xi(s) \right|  \left\| \boldsymbol{\epsilon}  \right\|_{L^2_{x} l^2} \right), 
\end{align}
with $l = 1, 2.$ If $\mathbf{f} \in  \left\{\boldsymbol{\chi}_0, \mathbf{Q}_{x_1}, \mathbf{Q}_{x_2}, i\mathbf{Q}_{x_1}, i\mathbf{Q}_{x_2}  \right\}$, then we obtain
\begin{align}\label{f4.30}
	& \left\langle - i \frac{x_{s}}{\lambda} \cdot\xi(s)
\left(\mathbf{Q} + \boldsymbol{\epsilon} \right), \mathbf{f} \right\rangle_{L^2_{x} l^2}
 = \sum_{j\in\mathbb{Z}_N} \left\langle -i \frac{x_{s}}{\lambda}\cdot \xi(s) \epsilon_j, f_j \right\rangle_{L_x^{2}},\\\nonumber
	&  \left\langle -i \frac{x_{s}}{\lambda} \cdot\xi(s)  \left(\mathbf{Q}+\boldsymbol{\epsilon} \right), i \boldsymbol{\chi}_{0,j} \right\rangle_{L^2_{x} l^2}
= -\frac{x_{s}}{\lambda} \cdot\xi(s)\langle Q, \chi_0 \rangle_{L_x^2}, \quad \forall  \, j\in \mathbb{Z}_N.
\end{align}
When $\mathbf{f} \in  \left\{\boldsymbol{\chi}_0, \mathbf{Q}_{x_1}, \mathbf{Q}_{x_2}, i\mathbf{Q}_{x_1}, i\mathbf{Q}_{x_2}  \right\}$, one has
\begin{align}\label{f4.31}
	\aligned
	& \left\langle i |\xi(s)|^{2}  \left(\mathbf{Q} + \boldsymbol{\epsilon} \right), \mathbf{f} \right\rangle_{L^2_{x} l^2}
 = \sum_{j\in\mathbb{Z}_N}  \left\langle i |\xi(s)|^{2} \epsilon_j, f_j \right\rangle_{L^2_{x}},\\
	&  \left\langle i |\xi(s)|^2 \mathbf{Q}, i \boldsymbol{\chi}_{0,j} \right\rangle_{L^2_{x} l^2} =
	|\xi(s)|^2 \left\langle Q, \chi_0  \right\rangle_{L_x^2}, \quad \forall \,  j\in \mathbb{Z}_N.
	\endaligned
\end{align}
If  $\mathbf{f}  \in
\left\{\mathbf{Q}_{x_1}, \mathbf{Q}_{x_2}, i\mathbf{Q}_{x_1}, i\mathbf{Q}_{x_2}, i\boldsymbol{\chi}_{0,1},\cdots, i\boldsymbol{\chi}_{0,N}  \right\}$, we estimate
\begin{align}\label{f4.32}
	\aligned
	& \frac{\lambda_{s}}{\lambda} \left\langle \left(\mathbf{Q} + \boldsymbol{\epsilon} \right)
 + x\cdot\nabla \left(\mathbf{Q} + \boldsymbol{\epsilon} \right), \mathbf{f} \right\rangle_{L^2_{x} l^2}
= O \left( \left|\frac{\lambda_{s}}{\lambda} \right|  \|  \boldsymbol{\epsilon}  \|_{L^2_{x} l^2} \right),\\
& 	\frac{\lambda_{s}}{\lambda}  \left\langle \left(\mathbf{Q} + \boldsymbol{\epsilon} \right) + x\cdot\nabla \left(\mathbf{Q} + \boldsymbol{\epsilon} \right), \boldsymbol{\chi}_{0} \right\rangle_{L^2_{x} l^2}
 = \frac{2N\lambda_{s}}{ \lambda_{0}\lambda} \langle Q, \chi_0\rangle_{L_x^2}
 + O \left( \left|\frac{\lambda_{s}}{\lambda} \right|  \left\|\boldsymbol{\epsilon}  \right\|_{L^2_{x} l^2} \right).
	\endaligned
\end{align}
If $\mathbf{f} \in
\left\{\boldsymbol{\chi}_0, i\boldsymbol{\chi}_{0,1},\cdots, i\boldsymbol{\chi}_{0,N}, i\mathbf{Q}_{x_1}, i\mathbf{Q}_{x_2}  \right\}$, we have
\begin{align}\label{f4.33}
 & 	\left\langle  \left(\frac{x_{s}}{\lambda} + 2 \xi(s) \right) \cdot\nabla \left(\mathbf{Q} + \boldsymbol{\epsilon} \right), \mathbf{f} \right\rangle_{L^2_{x} l^2}  =  O \left( \left|\frac{x_{s}}{\lambda} + 2 \xi(s) \right| \|  \boldsymbol{\epsilon}  \|_{L^2_{x} l^2} \right),\\\nonumber
& 	\left\langle \left(\frac{x_{s}}{\lambda} + 2 \xi(s) \right ) \cdot\nabla \left(\mathbf{Q} + \boldsymbol{\epsilon} \right), \mathbf{Q}_{x_l } \right\rangle_{L^2_{x} l^2} 
= N \left(\frac{(x_{l })_s}{\lambda} + 2 (\xi_l )_s(s) \right)  \| Q_{x_l } \|_{L^2_{x}}^{2}
 + O \left(  \left|\frac{x_{s}}{\lambda} + 2 \xi(s) \right|  \left\| \boldsymbol{\epsilon}  \right\|_{L^2_{x} l^2} \right), l = 1,2.
\end{align}
Finally, taking $\boldsymbol{\epsilon} = \boldsymbol{\epsilon}_{1}  + i \boldsymbol{\epsilon}_{2}$,
\begin{equation}\label{f4.34}
	\aligned
	i \Delta \left(\mathbf{Q} + \boldsymbol{\epsilon} \right) + \mathbf{\mathbf{F}} \left(\mathbf{Q}+\boldsymbol{\epsilon} \right)
	=  i \mathbf{Q} + i L_+ \boldsymbol{\epsilon_{1}} -  L_{-} \boldsymbol{\epsilon_{2}}
+ O\left( \left|\boldsymbol{\epsilon} \right|^{2} \left( \left|\mathbf{\tilde{u}}(s) \right| + |Q| \right)\right),
	\endaligned
\end{equation}
where $\mathbf{\tilde{u}}$ is defined in \eqref{f4.25} and $L_+$, $L_-$ are given by \eqref{L+} and \eqref{L-}. Since $L_+$ and $ L_{-}$ are self-adjoint operators,
$ \left\langle \boldsymbol{\epsilon}_{1}, \boldsymbol{\chi}_0 \right\rangle_{L^2_{x}  l^2 } =0$  and $L_+ \mathbf{Q}_{x_1}=L_+ \mathbf{Q}_{x_2} = 0$,
\begin{align}
	& \left\langle i \Delta \left(\mathbf{Q} + \boldsymbol{\epsilon} \right)   + i\mathbf{\mathbf{F}} \left(\mathbf{Q}+\boldsymbol{\epsilon} \right), i\boldsymbol{\chi}_{0,j} \right\rangle_{L^2_{x}  l^2 } \notag 
\\
&  = \left\langle Q,\chi_0 \right\rangle_{L_x^2}
+ O \left( \left\|\boldsymbol{\epsilon} \right\|_{L^2_x l^2} \right)
+ O\left( \left\langle \left|\boldsymbol{\epsilon} \right|^{2} \left( \left|\tilde{\boldsymbol{u}} \right| + |Q| \right), |\chi_0| \right\rangle_{L^2_{x}}\right),
 \quad \forall  \, j\in\mathbb{Z}_N, \notag\\
& 	\left\langle i \Delta \left(\mathbf{Q} + \boldsymbol{\epsilon} \right)   + i \mathbf{\mathbf{F}} \left(\mathbf{Q}+\boldsymbol{\epsilon} \right), i\boldsymbol{\chi}_{0} \right\rangle_{L^2_{x}  l^2 }
=  N \langle Q,\chi_0\rangle_{L_x^2}
+ O\left( \left\langle \left|\boldsymbol{\epsilon} \right|^{2}
\left( \left|\mathbf{\tilde{u}}(s) \right| + |Q| \right), |\chi_0| \right\rangle_{L^2_{x}}\right), \notag \\
& 	\left\langle i \Delta \left(\mathbf{Q} + \boldsymbol{\epsilon} \right)   + i\mathbf{\mathbf{F}} \left(\mathbf{Q}+\boldsymbol{\epsilon} \right), \boldsymbol{\chi}_{0} \right\rangle_{L^2_{x}  l^2 }
= - \left( \boldsymbol{\epsilon}_{2},  L_{-} \boldsymbol{\chi}_{0} \right)_{L^2_{x} l^2  }
+ O\left( \left\langle \left|\boldsymbol{\epsilon} \right|^{2}
\left( \left|\mathbf{\tilde{u}}(s) \right| + |Q| \right), |\chi_{0}| \right\rangle_{L^2_{x} }\right),\notag\\
& 	\left\langle i \Delta \left(\mathbf{Q} + \boldsymbol{\epsilon} \right)   + i\mathbf{\mathbf{F}} \left(\mathbf{Q}+\boldsymbol{\epsilon} \right), i\mathbf{Q}_{x_l } \right\rangle_{L^2_{x}  l^2 }
=  O\left( \left\langle  \left|\boldsymbol{\epsilon} \right|^{2} \left( \left|\mathbf{\tilde{u}}(s) \right| + |Q| \right), |Q_{x_l }| \right\rangle_{L^2_{x} }\right), l = 1, 2,\notag \\
& 	\left\langle i \Delta \left(\mathbf{Q} + \boldsymbol{\epsilon} \right)   + i\mathbf{\mathbf{F}} \left(\mathbf{Q}+\boldsymbol{\epsilon} \right), \mathbf{Q}_{x_l } \right\rangle_{L^2_{x} l^2 }
= - \left( \boldsymbol{\epsilon}_{2},  L_{-} \mathbf{Q}_{x_l } \right)_{L^2_{x} l^2}
+ O\left( \left\langle \left|\boldsymbol{\epsilon} \right|^{2} \left( \left|\mathbf{\tilde{u}}(s) \right| + |Q| \right), |Q_{x_l }| \right\rangle_{L^2_{x}}\right), l = 1,2.\label{f4.35}
\end{align}
Combining \eqref{f4.28}-\eqref{f4.35} and  the orthogonal relations \eqref{orthod}, we have proved
\begin{align}
& \sum_{j\in\mathbb{Z}_N} \left( (\gamma_{j})_s + 1 - \frac{x_{s}}{\lambda} \xi(s) - |\xi(s)|^{2} \right) 
\left\langle Q, \chi_0 \right\rangle_{L_x^2 }
+ O \left( \left|\xi_{s} - \frac{\lambda_{s}}{\lambda} \xi(s) \right|  \left\| \boldsymbol{\epsilon}  \right\|_{L^2_{x} l^2} \right)
+ O \left( \left|\frac{\lambda_{s}}{\lambda} \right|  \left\| \boldsymbol{\epsilon}  \right\|_{L^2_{x} l^2} \right)\notag \\
&  + O \left( \left|\frac{x_{s}}{\lambda} + 2 \xi(s) \right|  \left\| \boldsymbol{\epsilon} \right\|_{L^2_{x} l^2} \right)
 + O \left( \left\|\boldsymbol{\epsilon}(s)  \right\|_{L^2_{x} l^2}^{2} \|Q\|_{L_x^{\infty}}
 + \left\| \boldsymbol{\epsilon}(s) \right \|_{L^2_{x} l^2} \left\| \boldsymbol{\epsilon}(s)  \right\|_{L^{4}_x l^2} \left\| \mathbf{\tilde{u}}(s)  \right\|_{L^{4}_x l^2} \right) = 0, \label{f4.36}
\end{align}
and $\forall \ j\in \Z_N$,
\begin{align}\label{f4.37}
& \left((\gamma_{j})_s + 1 - \frac{x_{s}}{\lambda} \xi(s) - |\xi(s)|^{2} \right)
\left\langle Q, \chi_0 \right\rangle_{L_x^2}
+ O \left( \left|\xi_{s} - \frac{\lambda_{s}}{\lambda} \xi(s) \right|  \left\| \boldsymbol{\epsilon}  \right\|_{L^2_{x} l^2}  \right)
+ O \left(  \left|\frac{\lambda_{s}}{\lambda} \right|  \left\| \boldsymbol{\epsilon}  \right\|_{L^2_{x} l^2} \right) \\\nonumber
& 	+O \left(  \left\|\boldsymbol{\epsilon} \right\|_{L^2_x l^2} \right)
+ O \left( \left|\frac{x_{s}}{\lambda} + 2 \xi(s) \right|  \left\| \boldsymbol{\epsilon}  \right\|_{L^2_{x} l^2} \right)
 + O \left( \left\|\boldsymbol{\epsilon}(s)  \right\|_{L^2_{x} l^2}^{2} \|Q\|_{L_x^{\infty}}
 + \left\| \boldsymbol{\epsilon}(s)  \right\|_{L^2_{x} l^2} \left\| \boldsymbol{\epsilon}(s)  \right\|_{L^{4}_x l^2} \left\| \mathbf{\tilde{u}}(s)  \right\|_{L^{4}_x l^2} \right) = 0,
\end{align}
and
\begin{equation}\label{f4.38}
	\aligned
& 	\frac{2N\lambda_{s}}{ \lambda_{0}\lambda} \langle Q, \chi_0\rangle_{L_x^2}
-  \left\langle \boldsymbol{\epsilon}_{2},  L_{-} \boldsymbol{\chi}_0 \right\rangle_{L^2_{x}}
+ O \left( \left|\xi_{s} - \frac{\lambda_{s}}{\lambda} \xi(s) \right|  \left\| \boldsymbol{\epsilon}  \right\|_{L^2_{x} l^2} \right)
 + O \left(  \left|\frac{\lambda_{s}}{\lambda} \right|  \left\| \boldsymbol{\epsilon}   \right\|_{L^2_{x} l^2} \right) \\
& 	+\sum_{j\in\mathbb{Z}_N} \left\langle i (\gamma_{j})_s    {\epsilon}_j, {\chi}_0 \right\rangle_{L^2_{x}}
+\sum_{j\in\mathbb{Z}_N} \left\langle i {\epsilon}_j ,{\chi}_0 \right\rangle_{L_x^2}
-\sum_{j\in\mathbb{Z}_N} \left\langle i \frac{x_{s}}{\lambda}\cdot \xi(s)  {\epsilon}_j, {\chi}_0 \right\rangle_{L^2_{x}}
-\sum_{j\in\mathbb{Z}_N} \left\langle i |\xi(s)|^{2} \epsilon_j, {\chi}_0 \right\rangle_{L^2_{x}}\\
& 	+ O\left( \left|\frac{x_{s}}{\lambda} + 2 \xi(s)  \right|  \left\| \boldsymbol{\epsilon} \right\|_{L^2_{x} l^2} \right)
 + O\left( \left\| \boldsymbol{\epsilon}  \right\|_{L^2_{x} l^2}^{2}  \left(\| Q \|_{L_x^{\infty}}^{3} +  \left\| \boldsymbol{\epsilon}  \right\|_{L_x^{\infty}}^{3} \right) \right) = 0,
	\endaligned
\end{equation}
and
\begin{align}\label{f4.39}
	& - \frac{N}{2} \left((\xi_{l })_s(s) - \frac{\lambda_{s}}{\lambda} \xi_l (s) \right) \| Q \|_{L^2_{x}}^{2}
+ O \left( \left|\xi_{s} - \frac{\lambda_{s}}{\lambda} \xi(s) \right|  \left\| \boldsymbol{\epsilon}  \right\|_{L^2_{x} l^2} \right)
 + O \left( \left|\frac{\lambda_{s}}{\lambda} \right|  \left\| \boldsymbol{\epsilon}  \right\|_{L^2_{x} l^2} \right) \\\nonumber
& 	+\sum_{j\in\mathbb{Z}_N} \left\langle  (\gamma_{j})_s   \epsilon_j, Q_{x_l } \right\rangle_{L^2_{x}}
+\sum_{j\in\mathbb{Z}_N} \left\langle \epsilon_j ,Q_{x_l } \right\rangle_{L_x^2}
-\sum_{j\in\mathbb{Z}_N} \left\langle  \frac{x_{s}}{\lambda}\cdot \xi(s) \epsilon_j, Q_{x_l } \right\rangle_{L^2_{x}}
-\sum_{j\in\mathbb{Z}_N}  \left\langle  |\xi(s)|^{2} \epsilon_j, Q_{x_l } \right\rangle_{L^2_{x}}\\\nonumber
& 	+ O \left( \left|\frac{x_{s}}{\lambda} + 2 \xi(s) \right|  \| \boldsymbol{\epsilon} \|_{L^2_{x} l^2} \right)
+ O \left( \left\|\boldsymbol{\epsilon}(s)  \right\|_{L^2_{x} l^2}^{2}\|Q\|_{L_x^{\infty}}
+ \left\| \boldsymbol{\epsilon}(s)  \right\|_{L^2_{x} l^2} \left\| \boldsymbol{\epsilon}(s)  \right\|_{L^{4}_x l^2} \left\| \mathbf{\tilde{u}}(s)  \right\|_{L^{4}_x l^2} \right) = 0,  \ l = 1, 2,
\end{align}
and
\begin{equation}\label{f4.42}
	\aligned
	& N \left(\frac{(x_l )_s}{\lambda} + 2 (\xi_l )_s(s) \right) \| Q_{x_l } \|_{L^2_{x}}^{2}
-  \left\langle \boldsymbol{\epsilon}_{2},  L_{-} \mathbf{Q}_{x_l } \right\rangle_{L^2_{x} l^2}
+ O \left( \left|\xi_{s} - \frac{\lambda_{s}}{\lambda} \xi(s) \right|  \left\| \boldsymbol{\epsilon}  \right\|_{L^2_{x} l^2} \right)
 + O \left(  \left|\frac{\lambda_{s}}{\lambda} \right|  \left\| \boldsymbol{\epsilon}  \right\|_{L^2_{x} l^2} \right) \\
& 	+\sum_{j\in\mathbb{Z}_N} \left\langle i (\gamma_{j})_s   \epsilon_j, Q_{x_l } \right\rangle_{L^2_{x}}
+\sum_{j\in\mathbb{Z}_N} \left\langle i \epsilon_j ,Q_{x_l } \right\rangle_{L_x^2}
-\sum_{j\in\mathbb{Z}_N} \left\langle i \frac{x_{s}}{\lambda}\cdot \xi(s) \epsilon_j, Q_{x_l } \right\rangle_{L^2_{x}}
-\sum_{j\in\mathbb{Z}_N}  \left\langle i |\xi(s)|^{2} \epsilon_j, Q_{x_l } \right\rangle_{L^2_{x}}\\
& 	+ O \left( \left|\frac{x_{s}}{\lambda} + 2 \xi(s)  \right|  \left\| \boldsymbol{\epsilon}  \right\|_{L^2_{x} l^2} \right)
+ O \left( \left\|\boldsymbol{\epsilon}(s)  \right\|_{L^2_x l^2}^{2} \|Q\|_{L_x^{\infty}}
+ \left\| \boldsymbol{\epsilon}(s)  \right\|_{L^2_x l^2}  \left\| \boldsymbol{\epsilon}(s)  \right\|_{L^4_x l^2} \left\| \mathbf{\tilde{u}}(s)  \right\|_{L^4_x l^2} \right) = 0, \ l = 1,2 .
	\endaligned
\end{equation}
On the other hand, by changing variables and applying Lemma \ref{lemma23}, for any $a\geq 0$, we have 
\begin{align}\label{wsz1}
\int_{a}^{a + 1} \left\| \mathbf{\tilde{u}}(s)  \right\|_{ L_x^4l^2}^4  \, \mathrm{d} s
=\int_{s^{-1}(a)}^{s^{-1}(a+1)} \left\|\mathbf{u}(t) \right\|_{ L_x^4l^2}^4 \, \mathrm{ d} t\lesssim 1. 
\end{align}
From \eqref{f4.25}, we further have
\begin{align}\label{wsz2}
	\int_{a}^{a + 1}  \left\| \boldsymbol{\epsilon}(s)  \right\|_{ L_x^4l^2}^4 \,\mathrm{d}s \lesssim 1.
\end{align}
By some algebraic manipulation, 
$\eqref{wsz1}$, $\eqref{wsz2}$ and \eqref{f4.36}-\eqref{f4.42} imply that
\begin{equation}\label{f4.45}
	\int_{a}^{a + 1}  \left|\frac{\lambda_{s}}{\lambda} \right| \, \mathrm{d} s
\lesssim \int_{a}^{a + 1}  \left\| \boldsymbol{\epsilon}(s)  \right\|_{L^2_{x} l^2} \,  \mathrm{d} s.
\end{equation}
Obviously, \eqref{f4.45} implies that
\begin{align}\label{eqilam}
	\min_{s\in [a,a+1]}\lambda(s) \sim \max_{s\in [a,a+1]}\lambda(s).
\end{align}
Next we choose $t_0\in  \left[s^{-1}(a),s^{-1}(a+1) \right]$ so that
$$ \left\|\boldsymbol{\epsilon}(t_0) \right\|_{L^2_x l^2}
=\min_{t\in  \left[s^{-1}(a),s^{-1}(a+1) \right]} \left\|\boldsymbol{\epsilon}(t) \right\|_{L^2_x l^2}
=\min_{s\in [a,a+1]} \left\|\boldsymbol{\epsilon}(s) \right\|_{L^2_x l^2}$$
and define
\begin{align}\label{begin1}
\tilde{\epsilon}_j(t)=u_j(t)-e^{i(t-t_0)\Delta} \left(e^{-i\gamma_j(t_0)}e^{-ix\cdot\xi(t_0)}\frac{1}{\lambda(t_0)}Q \left(\frac{x-x(t_0)}{\lambda(t_0)} \right) \right)
=u_j(t)-\tilde{Q}(t),\ \forall \,  j\in\mathbb{Z}_N. 
\end{align}
By the Strichartz estimate, $\forall \,  J\subset  \left[s^{-1}(a),s^{-1}(a+1) \right]$,
\begin{align}\label{boot1}
\aligned
& \left\|\boldsymbol{\tilde{\epsilon}}(t) \right\|_{L_t^{\infty}L_x^2l^2\cap L_{t,x}^4l^2 \left(J\times\mathbb{R}^2\times\mathbb{Z}_N \right)}
\\
 \lesssim&  \left\|\boldsymbol{\epsilon}(t_0) \right\|_{L^2_x l^2}
+ \left\| \boldsymbol{\tilde{\epsilon}}(t)  \right\|_{L_{t,x}^{4}l^2 \left(J \times \mathbb{R}^2\times\mathbb{Z}_N \right)}
 \left\| \frac{1}{\lambda(t_0)}Q \left(\frac{x-x(t_0)}{\lambda(t_0)} \right)  \right\|_{L_{t,x}^{4}  \left(J \times \mathbb{R}^2 \right)}^{2}
+  \left\| \boldsymbol{\epsilon}  \right\|_{L_{t,x}^{4}l^2 \left(J \times \mathbb{R}^2\times\mathbb{Z}_N \right)}^3.
\endaligned
\end{align}
On the other hand, \eqref{eqilam} implies that $$ \left\| \frac{1}{\lambda(t_0)}Q \left(\frac{x-x(t_0)}{\lambda(t_0)} \right)  \right\|_{L_{t,x}^{4} \left(J \times \mathbb{R}^2 \right)} \lesssim 1. $$
Using a standard argument, if $\eta_{\ast}$ is sufficient small, then
\begin{align}\label{conc}
\left\|\boldsymbol{\tilde{\epsilon}}(t) \right\|_{L_t^{\infty}L_x^2l^2\cap L_{t,x}^4l^2 \left( \left[s^{-1}(a),s^{-1}(a+1) \right]\times\mathbb{R}^2\times\mathbb{Z}_N \right)}
\lesssim \left\|\boldsymbol{\epsilon}(t_0) \right\|_{L^2_x l^2}.
\end{align}
By the definition of $\boldsymbol{\epsilon}(t)$, \eqref{conc} implies
\begin{align*}
\max_{s\in [a,a+1]}  \left\|\boldsymbol{\epsilon}(s) \right\|_{L^2_x l^2}
= \max_{t\in  \left[s^{-1}(a),s^{-1}(a+1) \right]} \left\|\boldsymbol{\epsilon}(t) \right\|_{L^2_x l^2}
\sim\min_{t\in  \left[s^{-1}(a),s^{-1}(a+1) \right]} \left\|\boldsymbol{\epsilon}(t) \right\|_{L^2_x l^2}
=\min_{s\in [a,a+1]} \left\|\boldsymbol{\epsilon}(s) \right\|_{L^2_x l^2}.
\end{align*}
Moreover, recalling \eqref{f4.25} and  setting
$$\check{\epsilon}_j(s)=e^{i \gamma_j(s)} e^{ix \cdot\xi(s)} \lambda(s) \tilde{\epsilon}_j(t(s), \lambda(s ) x + x(s)),\  \check{Q}(s)=e^{i \gamma_j(s)} e^{ix \cdot\xi(s)} \lambda(s) \tilde{Q}(t(s), \lambda(s ) x + x(s)),\ 
\forall \, j\in\mathbb{Z}_N,$$
 we have
\begin{align}
\aligned
&  	\left\langle  \left|\boldsymbol{\epsilon} \right|^{2}  \left(  \left|\mathbf{\tilde{u}}(s) \right| + |Q| \right), |f| \right\rangle_{L^2_{x}}    \\
\lesssim& 	\left\langle \left|\boldsymbol{\epsilon} \right|^{2} \left(  \left|\boldsymbol{\check{\epsilon}}(s) \right| + \left|\boldsymbol{\check{\epsilon}}(s) \right|+ |Q| \right), |f| \right\rangle_{L^2_{x}} \\
	\lesssim& \left\|\boldsymbol{\epsilon}(s) \right\|^2_{L^2_x l^2}
+\left\|\boldsymbol{\epsilon}(s) \right\|^2_{L^2_x l^2}
\left\|\mathbf{\check{Q}}(s) \right\|_{L^{\infty}_x l^2}
\\
&   +\left\|\boldsymbol{\epsilon}(s) \right\|_{L^2_x l^2}
\left\|\boldsymbol{\epsilon}(s) \right\|_{L^4_x l^2} \left\|\boldsymbol{\check{\epsilon}}(s) \right\|_{L^{4}_x l^2},
\quad \forall\,  f\in\{\chi_0, Q, Q_{x_1}, Q_{x_2}\}.
\endaligned
\end{align}
Thus, the terms $O \left( \left\|\boldsymbol{\epsilon}(s)  \right\|_{L^2_{x} l^2}^{2} \|Q\|_{L^{\infty}_x }
+ \left\| \boldsymbol{\epsilon}(s)  \right\|_{L^2_{x} l^2} \left\| \boldsymbol{\epsilon}(s) \right\|_{L^{4}_x l^2} \left\| \mathbf{\tilde{u}}(s) \right\|_{L^{4}_x l^2} \right)$ in \eqref{f4.36}-\eqref{f4.42} can be replaced by
$$O\left( \left\|\boldsymbol{\epsilon}(s) \right\|^2_{L^2_x l^2}
+ \left\|\boldsymbol{\epsilon}(s) \right\|^2_{L^2_x l^2} \left\|\mathbf{\check{Q}}(s) \right\|_{L^{\infty}_x l^2} + \left\|\boldsymbol{\epsilon}(s) \right\|_{L^2_x l^2}
\left\|\boldsymbol{\epsilon}(s) \right\|_{L^4_x l^2} \left\|\boldsymbol{\check{\epsilon}}(s) \right\|_{L^{4}_x l^2}\right). $$
Changing variable, we have
\begin{align*}
\left\|\mathbf{\check{Q}}(s) \right\|_{L_s^2L_x^{\infty}l^2 \left( [a,a+1]\times\mathbb{R}^2\times\mathbb{Z}_N \right)}
\lesssim \int_{ \left[s^{-1}(a),s^{-1}(a+1) \right]}\frac{1}{\lambda(t)^2}  \,  \mathrm{d} t\lesssim 1,
\end{align*}
and 
\begin{align}\label{end1}
\left\|\boldsymbol{\check{\epsilon}}(s) \right\|_{L_s^4L_x^{4}l^2 \left([a,a+1]\times\mathbb{R}^2\times\mathbb{Z}_N \right)}
\lesssim \left\|\boldsymbol{\tilde{\epsilon}}(t) \right\|_{L_t^4L_x^{4}l^2 \left( \left[s^{-1}(a),s^{-1}(a+1) \right]\times\mathbb{R}^2\times\mathbb{Z}_N \right)}
\lesssim \min_{s\in[a,a+1]} \left\|\boldsymbol{\epsilon}(s) \right\|_{L^2_x l^2}.
\end{align}
Doing some algebra with \eqref{f4.36}-\eqref{f4.42} again, we can prove that for any $a\geq 0$,
\begin{equation}\label{f4.52}
	\int_{a}^{a + 1} \left|\sum_{j\in\mathbb{Z}_N}\left((\gamma_{j})_s + 1 - \frac{x_{s}}{\lambda} \cdot\xi(s) - |\xi(s)|^{2} \right) \right| \, \mathrm{d} s
\lesssim \int_{a}^{a + 1}  \left\| \boldsymbol{\epsilon}(s)  \right\|_{L^2_{x} l^2}^{2} \, \mathrm{d} s,
\end{equation}
\begin{equation}\label{f4.53}
	\int_{a}^{a + 1}  \left|(\gamma_{j})_s + 1 - \frac{x_{s}}{\lambda} \cdot\xi(s) - |\xi(s)|^{2}  \right|  \, \mathrm{d} s
 \lesssim \int_{a}^{a + 1}  \left\| \boldsymbol{\epsilon}(s)  \right\|_{L^2_{x} l^2} \, \mathrm{d} s, \quad\forall \ j\in\Z_N,
\end{equation}
\begin{equation}\label{f4.54}
	\int_{a}^{a + 1}  \left|\xi_{s} - \frac{\lambda_{s}}{\lambda} \xi(s) \right| \, \mathrm{d} s
\lesssim \int_{a}^{a + 1}  \left\| \boldsymbol{\epsilon}(s)  \right\|_{L^2_{x} l^2}^{2} \, \mathrm{d} s,
\end{equation}
and
\begin{equation}\label{f4.55}
	\int_{a}^{a + 1}  \left|\frac{x_{s}}{\lambda} + 2 \xi (s) \right| \, \mathrm{d} s
 \lesssim \int_{a}^{a + 1}  \left\| \boldsymbol{\epsilon}(s)  \right\|_{L^2_{x} l^2} \, \mathrm{d} s.
\end{equation}
If $\mathbf{u}(t)$ is not in $H^1_xl^2$, then \eqref{f4.45} and \eqref{f4.52}-\eqref{f4.55} can be proven using a standard approximation process; we refer to \cite{D2,D1} for details.
\begin{remark}
We also proved \begin{equation}\label{f4.56}
		\sup_{s\in [a,a+1]} \left\|\boldsymbol{\epsilon}(s) \right\|_{L^2_x l^2}
\sim\inf_{s\in [a,a+1]} \left\|\boldsymbol{\epsilon}(s) \right\|_{L^2_x l^2}.
	\end{equation}
The implicit constants in \eqref{f4.45} and \eqref{f4.52}-\eqref{f4.56} are independent of $\eta_{\ast}$ and $a$.
\end{remark}

\section{A priori estimate for $\|\boldsymbol{\epsilon}\|_{L^2_x l^2}$}\label{sec:aprieps}
This section is devoted to establishing a universal control  of the reminder $\boldsymbol{\epsilon}$'s $L^2_x l^2$ norm. In the first subsection, under the assumption that $\mathbf{u}_0\in H^{1}_{x} l^2$, we 
derive a pointwise estimate for $\|\boldsymbol{\epsilon}\|_{H^{1}_{x} l^2}$  via using the spectral properties derived in the previous section. Since we are actually working in the critical space $L^2_x l^2$, it is necessary to establish corresponding estimates for the low-frequency truncated reminder $P_{\leq L}\boldsymbol{\epsilon}$ for some sufficiently large $L$. However, $P_{\leq L}\boldsymbol{\epsilon}$ is not a solution to the nonlinear system $\eqref{1.1}$ and its energy $E(P_{\leq L}\mathbf{u})$ is no longer conserved, so in the second subsection we derive a non-priori estimate -- long time Strichartz estimate and then  in the third subsection, we establish an almost conservation law for $E(P_{\leq L} \mathbf{u})$. In the final subsection, we give an upper bound for the mass of the reminder term $\|\boldsymbol{\epsilon}\|_{L^2_x l^2}$ and the  truncated kinetic energy $\|\nabla P_{\leq L}\boldsymbol{\epsilon}\|_{L^2_x l^2}$ up to some $L^2_x l^2$ invariant symmetries.

\subsection{A uniform estimate for $\boldsymbol{\epsilon}$ in the inhomogeneous space $H^{1}_{x} l^2$}

Following  \cite{D2}, \cite{D1} and \cite{MR}, the spectral properties in Theorem \ref{positive} and the decomposition in Theorem \ref{modulation}  yield a priori bound in the inhomogeneous space $H^{1}_{x} l^2$:
\begin{theorem}\label{bouenergy}
	Suppose that  $\boldsymbol{\epsilon}(x)\in H^{1}_{x} l^2(\mathbb{R}^2\times\mathbb{Z}_N)$,  $\boldsymbol{\epsilon} \perp \left\{\boldsymbol{\chi}_0, i\boldsymbol{\chi}_{0,1},\cdots,i\boldsymbol{\chi}_{0,N},\mathbf{Q}_{x_1}, \mathbf{Q}_{x_2}, i\mathbf{Q}_{x_1}, i\mathbf{Q}_{x_2} \right\}$,
	$ \left\| \boldsymbol{\epsilon}(x)  \right\|_{L^2_{x} l^2} \ll 1$, and $ \left\| \mathbf{Q} + \boldsymbol{\epsilon}  \right\|_{L^2_{x} l^2} =  \left\| \mathbf{Q}  \right\|_{L^2_{x} l^2}$. Then
	\begin{equation*} 
		E \left(\mathbf{Q} + \boldsymbol{\epsilon} \right) \gtrsim  \left\| \boldsymbol{\epsilon}(x )  \right\|_{H^{1}_x  l^2 \left(\mathbb{R}^2 \times\mathbb{Z}_N \right)}^{2} .
	\end{equation*}
\end{theorem}
\begin{proof}
	Let $\boldsymbol{\epsilon}=\mathbf{v}+i\mathbf{w}$.  Decomposing the energy and integrating by parts, 
    we obtain 
	\begin{align}\label{f4.58}
		\aligned
		E \left(\mathbf{Q} + \boldsymbol{\epsilon} \right)
		= & \frac{N}{2} \int  \left(Q_{x_1}^{2}+Q_{x_2}^2 \right)  \mathrm{d} x+ \sum_{j\in\mathbb{Z}_N}  \int \nabla v_j(x)\cdot\nabla Q(x) \,  \mathrm{d} x
		+ \frac{1}{2} \sum_{j\in\mathbb{Z}_N} \| \nabla v_{j} \|_{L^2_{x}}^{2}\\
		& + \frac{1}{2}\sum_{j\in\mathbb{Z}_N}\| \nabla w_{j} \|_{L^2_{x}}^{2}- \frac{(2N-1)N}{4} \int Q(x)^{4} \,   \mathrm{d} x
		- \sum_{j\in\mathbb{Z}_N}\int (2N-1)Q(x)^3 v_j(x) \, \mathrm{d} x
		\\
		& - \frac{N+3}{2} \sum_{j\in\mathbb{Z}_N}\int Q(x)^{2}(v_j(x) )^2  \, \mathrm{d} x - 2  \int \sum_{j\in\mathbb{Z}_N}\sum_{\substack{k\in\mathbb{Z}_N\\k\neq j}}Q(x)^2v_kv_j  \, \mathrm{d} x \\
		& - \frac{N+1}{2} \sum_{j\in\mathbb{Z}_N}\int Q(x)^{2}(w_j)^2 \,  \mathrm{d} x
		- \int O \left( \left|\boldsymbol{\epsilon}(x) \right|^{3} Q(x)  +  \left|\boldsymbol{\epsilon}(x) \right|^{4} \right) \,  \mathrm{d} x.
		\endaligned
	\end{align}
	Since $E \left(\mathbf{Q} \right) = 0$,
	\begin{equation}\label{f4.59}
		\frac{N}{2} \int |\nabla Q|^2  \, \mathrm{d} x - \frac{(2N-1)N}{4} \int Q^{4} \,   \mathrm{d} x = 0.
	\end{equation}
	Next, integrating by parts,
	\begin{align}\label{f4.60}
    \aligned 
		& \sum_{j\in\mathbb{Z}_N}\int \nabla Q(x)\cdot \nabla v_j(x) \,  \mathrm{d} x
		- (2N-1)\sum_{j\in\mathbb{Z}_N}\int Q(x)^{3}v_j(x) \,  \mathrm{d} x \\
		=& - \sum_{j\in\mathbb{Z}_N}\int  \left(\Delta Q(x) + (2N-1)Q(x)^{3} \right) v_j(x)  \, \mathrm{d} x  = \sum_{j\in\mathbb{Z}_N}\int Q(x) v_j(x)  \, \mathrm{d} x.
		\endaligned 
	\end{align}
	Using the fact that $\left\| \mathbf{Q} + \boldsymbol{\epsilon}  \right\|_{L^2_{x} l^2} =  \left\| \mathbf{Q}  \right\|_{L^2_{x} l^2}$, we have
	\begin{align}\label{f4.61}
    \aligned 
		& \frac{1}{2}  \left\| \mathbf{Q}  \right\|_{L^2_{x} l^2}^{2} - \frac{1}{2}  \left\| \mathbf{Q} + \boldsymbol{\epsilon}  \right\|_{L^2_{x} l^2}^{2}
		+ \frac{1}{2}  \left\| \boldsymbol{\epsilon}  \right\|_{L^2_{x} l^2}^{2} \\
		=& -\langle \mathbf{Q}, \boldsymbol{\epsilon}\rangle_{L^2_{x} l^2} =  -\sum_{j\in\mathbb{Z}_N} \int Q(x)v_j(x) \,  \mathrm{d} x
		= \frac{1}{2}  \left\| \boldsymbol{\epsilon}  \right\|_{L^2_{x} l^2}^{2}
		=\frac{1}{2}\sum_{j\in\mathbb{Z}_N}\|v_j\|^2_{L_x^2}+\frac{1}{2}\sum_{j\in\mathbb{Z}_N}\|w_j\|^2_{L_x^2}. 
        \endaligned 
	\end{align}
	Therefore,
	\begin{equation}\label{energybou}
		\aligned
		E \left(\mathbf{Q} + \boldsymbol{\epsilon} \right) = &   \frac{1}{2} \sum_{j\in\mathbb{Z}_N}\| \nabla v_{j} \|_{L^2_{x}}^{2}
		+ \frac{1}{2}\sum_{j\in\mathbb{Z}_N}\| \nabla w_{j} \|_{L^2_{x}}^{2} \\
        & + \frac{1}{2}\sum_{j\in\mathbb{Z}_N}\|v_j\|_{L_x^2}^2
		+\frac{1}{2}\sum_{j\in\mathbb{Z}_N}\|w_j\|_{L_x^2}^2
- \frac{N+3}{2} \sum_{j\in\mathbb{Z}_N}\int Q(x)^{2}(v_j)^2  \,  \mathrm{d} x \\
		& - 2  \int \sum_{j\in\mathbb{Z}_N}\sum_{\substack{k\in\mathbb{Z}_N\\k\neq j}}Q(x)^2v_kv_j \,  \mathrm{d} x
		- \frac{N+1}{2} \sum_{j\in\mathbb{Z}_N}\int Q(x)^{2}(w_j)^2 \, \mathrm{ d} x
		\\
        & - \int O \left(  \left|\boldsymbol{\epsilon}(x) \right|^{3} Q +  \left|\boldsymbol{\epsilon}(x) \right|^{4} \right) \,  \mathrm{d} x.
		\endaligned
	\end{equation}
	Recalling \eqref{L+},
	\begin{align*} 
		& \frac{1}{2} \sum_{j\in\mathbb{Z}_N}\| \nabla v_{j} \|_{L^2_{x}}^{2} + \frac{1}{2}\sum_{j\in\mathbb{Z}_N}\|v_j\|^2_{L_x^2}
		- \frac{N+3}{2} \sum_{j\in\mathbb{Z}_N}\int Q(x)^{2}(v_j)^2  \,  \mathrm{d} x
		- 2 \int \sum_{j\in\mathbb{Z}_N}\sum_{\substack{k\in\mathbb{Z}_N\\ k\neq j}}Q(x)^2v_kv_j \, \mathrm{d} x
		\\
         =& \left\langle L_{+}\mathbf{v},\mathbf{v} \right\rangle_{L^2_x l^2}.
	\end{align*}
	On one hand, since $\mathbf{v}\perp \left\{\boldsymbol{\chi}_0, \mathbf{Q}_{x_1}, \mathbf{Q}_{x_2} \right\} $, by Theorem \ref{positive},
	\begin{align}\label{bouqw1}
    \aligned 
		& \frac{1}{2} \sum_{j\in\mathbb{Z}_N}\| \nabla v_{j} \|_{L^2_{x}}^{2}
		+ \frac{1}{2}\sum_{j\in\mathbb{Z}_N}\|v_j\|^2_{L_x^2}- \frac{N+3}{2} \sum_{j\in\mathbb{Z}_N}\int Q(x)^{2} (v_j)^2 \,   \mathrm{ d} x
		- 2 \int \sum_{j\in\mathbb{Z}_N}\sum_{\substack{k\in\mathbb{Z}_N\\k\neq j}}Q(x)^2v_kv_j \,   \mathrm{d} x \\
		=&\frac{1}{2} \left\langle L_{+}\mathbf{v},\mathbf{v} \right\rangle_{L^2_x l^2}
		\geq \frac{C_{N}}{2}  \left\|\mathbf{v} \right\|^2_{L^2_x l^2}.
        \endaligned 
	\end{align}
	On the other hand,
	\begin{align}\label{bouqw2}
    \aligned 
		\frac{1}{2} \sum_{j\in\mathbb{Z}_N}\| \nabla v_{j} \|_{L^2_{x}}^{2}
		& =\frac{1}{2} \left\langle L_{+}\mathbf{v},\mathbf{v} \right\rangle_{L^2_x l^2}
		-\frac{1}{2}\sum_{j\in\mathbb{Z}_N}\|v_j\|^2_{L_x^2}+ \frac{N+3}{2} \sum_{j\in\mathbb{Z}_N}\int Q(x)^{2}(v_j)^2 \,   \mathrm{d} x  
		\\
        & \leq C_{N,\|Q\|_{L^{\infty}}}  \left\langle L_{+}\mathbf{v},\mathbf{v} \right\rangle_{L^2_x l^2}.
        \endaligned 
	\end{align}
	Thus, by \eqref{bouqw1} and \eqref{bouqw2},
	\begin{align}\label{bouqw}
		& \frac{1}{2} \sum_{j\in\mathbb{Z}_N}\| \nabla v_{j} \|_{L^2_{x}}^{2}
		+ \frac{1}{2}\sum_{j\in\mathbb{Z}_N}\|v_j\|^2_{L_x^2}- \frac{N+3}{2} \sum_{j\in\mathbb{Z}_N}\int Q(x)^{2} (v_j)^2  \,  \mathrm{d} x
		- 2 \int \sum_{j\in\mathbb{Z}_N}\sum_{\substack{k\in\mathbb{Z}_N\\k\neq j}}Q(x)^2v_kv_j  \, \mathrm{d} x \notag\\
		=&\frac{1}{4}  \left\langle L_{+}\mathbf{v},\mathbf{v} \right\rangle_{L^2_x l^2}
		+\frac{1}{4} \left\langle L_{+}\mathbf{v},\mathbf{v} \right\rangle_{L^2_x l^2}\geq  \frac{C_N}{4} \left\|\mathbf{v} \right\|^2_{L^2_x l^2}
		+\frac{1}{8C_{N,\|Q\|_{L^{\infty}}}} \left\|\nabla\mathbf{v} \right\|^2_{L^2_x l^2} \gtrsim  \left\|\mathbf{v} \right\|^2_{H^{1}_{x} l^2}.
	\end{align}
	Similarly,
	\begin{align}\label{bouqwe1}
    \aligned
		&\hspace{2ex}\frac{1}{2} \sum_{j\in\mathbb{Z}_N}\| \nabla w_{j} \|_{L^2_{x}}^{2}
		+ \frac{1}{2}\sum_{j\in\mathbb{Z}_N}\|w_j\|^2_{L_x^2}- \frac{N+1}{2} \sum_{j\in\mathbb{Z}_N}\int Q(x)^{2}(w_j)^2  \,   \mathrm{d} x \\
		&=  \frac{1}{2} \sum_{j\in\mathbb{Z}_N}\| \nabla w_{j} \|_{L^2_{x}}^{2} + \frac{1}{2}\sum_{j\in\mathbb{Z}_N}\|w_j\|^2_{L_x^2}
		\\
        & \quad - \frac{N+3}{2} \sum_{j\in\mathbb{Z}_N}\int Q(x)^{2}(w_j)^2  \, \mathrm{d} x
		- 2 \int \sum_{j\in\mathbb{Z}_N}\sum_{\substack{k\in\mathbb{Z}_N\\k\neq j}}Q(x)^2w_kw_j \,  \mathrm{d} x  \\
		&\hspace{3ex} +\sum_{j\in\mathbb{Z}_N}\int Q(x)^{2}(w_j)^2  \, \mathrm{ d} x +
        2\int \sum_{j\in\mathbb{Z}_N}\sum_{\substack{k\in\mathbb{Z}_N\\k\neq j}}Q(x)^2w_kw_j  \, \mathrm{d} x  \\
		&\geq  \frac{C_N}{4} \left\|\mathbf{w} \right\|^2_{L^2_x l^2}
		+\frac{1}{8C_{N,\|Q\|_{L_x^{\infty}}}} \left\|\nabla\mathbf{w} \right\|^2_{L^2_x l^2}- (2N-3)\sum_{j\in\mathbb{Z}_N}\int Q(x)^{2}(w_j)^2  \,   \mathrm{d} x .
        \endaligned
	\end{align}	
	Since $ \left\langle w_j,\chi_0 \right\rangle_{L_x^2}=0,\  \forall \, j\in\mathbb{Z}_N$, using Theorem \ref{L01spectral},
	\begin{align}\label{bouqwe2}
		\aligned
		&\hspace{2ex}\frac{1}{2} \sum_{j\in\mathbb{Z}_N} \| \nabla w_{j} \|_{L^2_{x}}^{2}
		+ \frac{1}{2}\sum_{j\in\mathbb{Z}_N} \|w_j\|^2_{L_x^2}- \frac{N+1}{2} \sum_{j\in\mathbb{Z}_N}\int Q(x)^{2} (w_j)^2  \,  \mathrm{d} x \\
		&=  \frac{1}{2} \sum_{j\in\mathbb{Z}_N} \| \nabla w_{j} \|_{L^2_{x}}^{2} + \frac{1}{2}\sum_{j\in\mathbb{Z}_N} \|w_j\|^2_{L_x^2}
		\\
        &\quad  - \frac{6N-3}{2} \sum_{j\in\mathbb{Z}_N}\int Q(x)^{2}(w_j)^2   \, \mathrm{d} x +\frac{5N-4}{2}\sum_{j\in\mathbb{Z}_N} \int Q(x)^{2}(w_j)^2  \,  \mathrm{d} x \\
	&\gtrsim\sum_{j\in\mathbb{Z}_N} \left\langle L_{0,+}w_j, w_j \right\rangle_{L_x^2} +\sum_{j\in\mathbb{Z}_N}\int Q(x)^{2}(w_j)^2 \,\mathrm{d}x \gtrsim \sum_{j\in\mathbb{Z}_N}\int Q(x)^{2}(w_j)^2 \,\mathrm{d}x.
		\endaligned
	\end{align}
	Combining \eqref{bouqwe1} and \eqref{bouqwe2}, we  have
	\begin{align}\label{bouqwe}
		\hspace{2ex}\frac{1}{2} \sum_{j\in\mathbb{Z}_N}
		\| \nabla w_{j} \|_{L^2_{x}}^{2} + \frac{1}{2}\sum_{j\in\mathbb{Z}_N}\|w_j\|^2_{L_x^2}
		- \frac{N+1}{2} \sum_{j\in\mathbb{Z}_N}\int Q(x)^{2}(w_j)^2   \, \mathrm{d} x \gtrsim \left\|\mathbf{w} \right\|^2_{H^1_xl^2} .
	\end{align}
	Finally, by the interpolation and $ \left\| \boldsymbol{\epsilon}  \right\|_{L^2_{x} l^2} \ll 1$, we have 
	\begin{equation}\label{bouqwer1}
		\int \left|\boldsymbol{\epsilon}(t,x) \right|^{4}  \, \mathrm{d} x
		\lesssim \left\| \nabla\boldsymbol{\epsilon} \right\|_{L^2_x l^2}^{2} \left\| \boldsymbol{\epsilon}  \right\|_{L^2_{x} l^2}^{2}
		\ll \left\| \nabla\boldsymbol{\epsilon} \right\|_{L^2_x l^2}^{2},
	\end{equation}
	and
	\begin{equation}\label{bouqwer2}
		\int Q(x) \left|\boldsymbol{\epsilon}(t,x) \right|^{3}  \, \mathrm{d} x
		\lesssim  \left\| \boldsymbol{\epsilon}  \right\|_{L^2_{x} l^2}^{2}  \left\| \nabla\boldsymbol{\epsilon}  \right\|_{L^2_x l^2}
		\ll  \left\| \boldsymbol{\epsilon}  \right\|_{L^2_x l^2}^{2}+ \left\| \nabla\boldsymbol{\epsilon} \right\|_{L^2_x l^2}^{2} .
	\end{equation}
	Collecting \eqref{bouqw}, \eqref{bouqwe}, \eqref{bouqwer1} and \eqref{bouqwer2}, we  complete the proof 
	of Theorem \ref{bouenergy}.
\end{proof}
\begin{remark}\label{remark0}
	From the proof above, we see that if
	$$\boldsymbol{\epsilon}\perp \left\{i\boldsymbol{\chi}_{0,1},\cdots,i\boldsymbol{\chi}_{0,N},\boldsymbol{\chi}_0,i\mathbf{Q}_{x_1},i\mathbf{Q}_{x_2},\mathbf{Q}_{x_1},\mathbf{Q}_{x_2} \right\},
	\quad\boldsymbol{\epsilon}=\mathbf{v}+i\mathbf{w},$$
	then there exists a constant $\lambda_1>0$ such that
\begin{align}\label{key0}
		\aligned
		& 	\frac{1}{2} \sum_{j\in\mathbb{Z}_N} \left\| \nabla v_{j}  \right\|_{L^2_{x}}^{2}
		+ \frac{1}{2}\sum_{j\in\mathbb{Z}_N} \| \nabla w_{j} \|_{L^2_{x}}^{2} + \frac{1}{2}\sum_{j\in\mathbb{Z}_N} \|v_j\|_{L_x^2}^2 +\frac{1}{2}\sum_{j\in\mathbb{Z}_N}\|w_j\|_{L_x^2}^2  \\
		&	- \frac{N+3}{2} \sum_{j\in\mathbb{Z}_N}\int Q(x)^{2}(v_j)^2  \,  \mathrm{ d} x
		- 2 \int \sum_{j\in\mathbb{Z}_N}\sum_{\substack{k\in\mathbb{Z}_N\\k\neq j}}Q(x)^2v_kv_j  \,  \mathrm{d} x
		- \frac{N+1}{2} \sum_{j\in\mathbb{Z}_N}\int Q(x)^{2}(w_j)^2  \,  \mathrm{d} x\\ 
		\geq& \lambda_1 \left\|\boldsymbol{\epsilon} \right\|^2_{H^{1}_{x} l^2}.
		\endaligned
	\end{align}
\end{remark}

\subsection{Long time Strichartz estimate}
In this subsection, we derive a non-priori estimate for the blowup solution $\mathbf{u}$. Specifically, we will prove that on some interval $J$ with $K\sim \int_{J}\frac{1}{\lambda(t)^2} \mathrm{d} t$, the Strichartz norms of  $P_{\geq K^{1/3}}\mathbf{u}$ on $J$ are actually dominated by its linear part (in Duhamel's formula) and, consequently, can be bounded by the initial data on the interval $J$.

\subsubsection{Two atomic spaces.}

Similar to  \cite{D5} and \cite{D2}, since  the endpoint Strichartz estimate fails in dimension two, we need to use the functional spaces $U_\Delta^p$ and $V_\Delta^p$, introduced in \cite{KTa,HHK}. In our context, we define the variant versions of the standard $U_\Delta^p$ and $V_\Delta^p$ spaces, namely $U_\Delta^p(I,L_x^2l^2 \left(\R^2\times\Z_N  \right))$ and $V_\Delta^p(I,L_x^2l^2 \left(\R^2\times\Z_N  \right))$:
\begin{definition}[$U_\Delta^p(I,L_x^2l^2 \left(\R^2\times\Z_N  \right))$ space] Let $1\leqslant p<\infty$. We define $U_\Delta^p(\R, L_x^2l^2 \left(\R^2\times\Z_N  \right) )$ as an atomic space with atoms  $\mathbf{v}^\lambda(t,x)$ given by 
	\begin{align*}
		\mathbf{v}^\lambda(t,x)=\sum_{k=0}^{M}\chi_{[t_k,t_{k+1}]}(t)e^{it\Delta} \mathbf{v}_k^\lambda(x),
\quad\sum_{k=0}^M\left\| \mathbf{v}_k^\lambda(x) \right\|_{L_x^2l^2 \left(\R^2\times\Z_N  \right)}^p=1.
	\end{align*}
	In the summation above, $M$ can be finite or infinite. If $M$ is finite, we further assume $t_0=-\infty$ and $t_{M+1}=+\infty$. The norm of $U_\Delta^p (\R,L_x^2l^2 \left(\R^2\times\Z_N  \right))$ is defined as 
	\begin{align*}
\| \mathbf{v} \|_{U_\Delta^p(\R,L_x^2l^2 \left(\R^2\times\Z_N  \right))}
		=\inf\left\{\sum_{\lambda\in\Z} |c_\lambda|:  \mathbf{v} =\sum_{\lambda\in\Z}c_\lambda  \mathbf{v}^\lambda, \hspace{1ex} \mathbf{v}^\lambda\mbox{ is an atom}\right\},
	\end{align*}
	where the infimum is taken over all atom decompositions. 
	For any  interval $I\subseteq\R$, we can define the local version as
	\begin{align*}
		\| \mathbf{v}\|_{U_\Delta^p \left(I,L_x^2l^2 \left(\R^2\times\Z_N  \right) \right)}=\| \mathbf{v} \chi_I(t)\|_{U_\Delta^p \left(\R,L_x^2l^2 \left(\R^2\times\Z_N  \right) \right)}.
	\end{align*}
\end{definition}
Next, we define the space $V_\Delta^p \left(I,L_x^2l^2 \right)$, which is also useful for our analysis.
\begin{definition}[$V_\Delta^p \left(I,L_x^2l^2 \right)$ space]
	Let $1\leqslant p<\infty$. We define $V_\Delta^p(\R, L_x^2 l^2 )$ as the space of  right-continuous functions $ \mathbf{v} \in L_t^{\infty} L_x^2l^2 \left(\R\times\mathbb{R}^2\times\Z_N \right)$ endowed with the norm
	\begin{align*}
		\| \mathbf{v} \|_{V_\Delta^p \left(\R,L_x^2l^2 \right)}^p
=\| \mathbf{v} \|_{L_t^\infty L_x^2l^2}^p+\sup_{\{t_k\}\nearrow\infty}
\sum_{k} \big\|e^{-it_{k+1}\Delta} \mathbf{v} (t_{k+1})-e^{- it_k\Delta} \mathbf{v}(t_k)\big\|_{L_x^2l^2}^p<\infty.
	\end{align*}
For any  interval $I\subseteq\R$, we can define the local version  as
		\begin{align*}
			\| \mathbf{v} \|_{V_\Delta^p \left(I,L_x^2 l^2 \right)}=\| \mathbf{v} \chi_I(t)\|_{V_\Delta^p \left(\R,L_x^2l^2 \right)}.
		\end{align*}
\end{definition}
Next, we present some useful properties of  the $U_\Delta^p \left(I,L^2_xl^2 \right)$ and $V_\Delta^p \left(I,L^2_xl^2 \right)$ spaces.
\begin{lemma}\label{property}
Suppose $1<p<q<\infty$ and $t_0\leqslant t_1\leqslant t_2$. For $I\subset \R$,  we have
	\begin{align}
		& U_\Delta^p \left(I,L_x^2l^2 \right)\subseteq V_\Delta^p \left(I,L_x^2l^2 \right)\subseteq U_\Delta^q \left(I,L_x^2l^2 \right),\label{embeddingvp}
        \\
		&\| \mathbf{v} \|_{U_\Delta^p \left([t_0,t_1],L_x^2l^2 \right)}\leqslant\| \mathbf{v} \|_{U_\Delta^p \left([t_0,t_2],L_x^2l^2 \right)},  \notag \\
		&\| \mathbf{v} \|_{U_\Delta^2 \left([t_0,t_2],L_x^2l^2 \right)}^2\leqslant\|  \mathbf{v} \|_{U_\Delta^2 \left([t_0,t_1],L_x^2l^2 \right)}^2
+\| \mathbf{v} \|_{U_\Delta^2 \left([t_1,t_2],L_x^2l^2 \right)}^2,  \notag\\
		&\| \mathbf{v} \|_{L_t^pL_x^ql^2 \left(I\times\R^2\times\Z_N  \right)}\lesssim_p \| \mathbf{v} \|_{U_\Delta^p \left(I,L_x^2l^2 \right)}, \text{ where }   p\geq2 \mbox{ and } (p,q) \mbox{ is an admissible pair, }\label{u-pstrichartz}\\
&\| \mathbf{v} \|_{U_\Delta^2 \left(I,L_x^2l^2  \right)}\lesssim_p \min_{\tilde{t}\in I} \left\| \mathbf{v}  \left(\tilde{t} \right) \right\|_{L_x^2l^2}
+\|i\partial_t \mathbf{v} +\Delta  \mathbf{v} \|_{L_t^{p'}L_x^{q'}l^2 \left(I\times\R^2\times\Z_N  \right)},  
\label{dupstrichartz}\\
& \text{ where }   p\geq2 \mbox{ and } (p,q) \mbox{ is an admissible pair. } \notag
	\end{align}
\end{lemma}
Furthermore, we have the following estimates:
	\begin{lemma}\label{up-dual}
		For any $p>2$ and let $I$ be an interval with $t_0\in I$, we have
		\begin{align*}
			\left\|\int_{t_0}^t e^{i(t-s)\Delta} \mathbf{F} (\tau) \,  \mathrm{d} \tau\right\|_{U_\Delta^2 \left(I,L_x^2l^2 \right)}
\lesssim_p\sup_{\| \mathbf{G} \|_{V_\Delta^p \left(I,L_x^2l^2 \right)}=1}\big\|| \mathbf{G} |\cdot |  \mathbf{F} |\big\|_{L_{t,x}^1}.	
		\end{align*}
	\end{lemma}
The proof of  Lemma  \ref{property} and Lemma \ref{up-dual} are similar to that in \cite{D5}, so we omit them.
\subsubsection{Long time Strichartz estimate - Type I}
Having introduced the atomic spaces, we now present the following long time Strichartz estimate:
\begin{theorem}[Long time Strichartz estimate-I] \label{t5.9}
Let $\mathbf{u}$ be the  solution in Theorem \ref{t2.3}. Suppose $0< \eta_1\ll \eta_0\ll1$ are two sufficiently small constants to be determined later,  and $J = [a,b]$ is an interval for which
\begin{align}\label{f5.7}
\aligned
& \lambda(t) \geq \frac{1}{\eta_{1}},\quad \frac{|\xi(t)|}{\lambda(t)} \leq \eta_{0} \quad  \text{for all} \quad t \in J,
\quad \int_{J} \lambda(t)^{-2} \, \mathrm{d} t = T=2^{3k_0}, \\
\text{ and } & 
\sup_{t \in J}
    \left\|\boldsymbol{\epsilon}(t) \right\|_{L^2_x l^2}\leq\eta_{\ast}\leq \eta_0,
     \quad \left(\int_{|\xi|\geq\eta^{-1/2}_1}  \left|\hat{Q}(\xi) \right|^2 \,  \mathrm{d} \xi \right)^{1/2} \leq \eta_0.
\endaligned
\end{align}
Splitting $J$ into $2^{3k_0}$ small subintervals $J_i=[a_i,b_i]$ such that
$$b_i=a_{i+1}, \quad\forall\, 1\leq i\leq 2^{3k_0}-1,\quad \int_{J_i}
{\lambda(t)^{-2}} \, \mathrm{d} t=1, $$
and define the intervals $G^j_{k_\ast}$ by induction:
\begin{align*}
	G^0_{k_{\ast}}=J_{k_{\ast}}, \quad\forall  \, 1\leq k_{\ast}\leq 2^{3k_0}; \quad
\forall \, 1\leq j\leq 3k_0,\quad	G_{k_{\ast}}^{j+1}=G_{2k_{\ast}-1}^{j}\cup G_{2k_{\ast}}^{j}, \quad \forall \, 1\leq k_{\ast}\leq 2^{3k_0-j-1}.
\end{align*}
Finally, define the  long time Strichartz norm:
\begin{equation*}
 \left\| \mathbf{u} \right\|_{X \left(J \times \mathbb{R}^{2}\times\mathbb{Z}_N \right)}^{2}
 =  \sup_{0 \leq i \leq k_0} \sup_{1\leq k_{\ast}\leq 2^{3k_0-3i}}\left( \left\| P_{\geq i} \mathbf{u}  \right\|_{U_{\Delta}^{2} \left(G_{k_{\ast}}^{3i}, L_x^2l^2 \right)}^{2}  + \eta_{0}^{-\frac1{10}}
 \left\|  \left|P_{\geq i} \mathbf{u} \right|\cdot |P_{\leq i - 3}  \mathbf{u} | \right\|_{L_{t,x}^{2} l^2 \left(G_{k_{\ast}}^{3i} \times \mathbb{R}^{2}\times\mathbb{Z}_N \right)}^{2}\right).
\end{equation*}
Then the long time Strichartz estimate,
\begin{equation}\label{f5.10}
\left\| \mathbf{u} \right\|_{X \left(J \times \mathbb{R}^{2}\times\mathbb{Z}_N\right)} \lesssim 1,
\end{equation}
holds with implicit constant independent of $T$.
\end{theorem}
\begin{proof}
Let
\begin{equation*}
\left\| \mathbf{u}  \right\|_{X_j \left(J \times \mathbb{R}^{2}\times\mathbb{Z}_N \right)}^{2}
=  \sup_{0 \leq i \leq j} \sup_{1\leq k_{\ast}\leq 2^{3k_0-3i}}\left( \left\| P_{\geq i} \mathbf{u} \right\|_{U_{\Delta}^{2} \left(G_{k_{\ast}}^{3i},L_x^2l^2 \right)}^{2}
+ \eta_{0}^{-\frac1{10}}
 \left\|  \left|P_{\geq i} \mathbf{u} \right|\cdot |P_{\leq i - 3}  \mathbf{u} |  \right\|_{L_{t,x}^{2}l^2 \left(G_{k}^{3i} \times \mathbb{R}^{2}\times\mathbb{Z}_N \right)}^{2}\right) .
\end{equation*}
Since $\lambda(t)\geq \frac{1}{\eta_{1}}$, arguing as in \eqref{begin1}-\eqref{end1}, for $0\leq i\leq 9$,
if $\eta_0$ is sufficiently small and $\eta_1\ll \eta_0$, then there exists  constant $C_i$ such that
\begin{align*}
	\left\|\mathbf{u} \right\|_{L_{t,x}^4l^2 \left(G^{3i}_{k_{\ast}}\times\mathbb{R}^2\times\mathbb{Z}_N \right)}\leq C_{i}\eta_0,\quad \forall \, 1\leq k_{\ast}\leq 2^{3k_0-3i} .
\end{align*}
Thus, by \eqref{dupstrichartz}, we have
\begin{equation*} 
 \left\| \mathbf{u}  \right\|_{U_{\Delta}^{2} \left(G^{3i}_{k_{\ast}}, L_x^2l^2 \right)} \lesssim_i 1, \quad \forall \,  1\leq k_{\ast}\leq 2^{3k_0-3i}, 
\end{equation*}
and by H\"older's inequality,
\begin{align*}
	\frac{1}{\eta_{0}^{1/10}}   \left\| \left|P_{\geq i} \mathbf{u} \right|\cdot |P_{\leq i - 3}  \mathbf{u} |  \right\|_{L_{t,x}^{2}l^2 \left(G_{k_{\ast}}^{3i} \times \mathbb{R}^{2}\times\mathbb{Z}_N \right)}^{2}
\begin{cases}
		=0,&i=0,1,2;\\
		\lesssim C^2_i\eta^{19/10}_0,
&i=3,4,5,6,7,8,9.
	\end{cases}
\end{align*}
Therefore,
\begin{align}\label{longtime0}
	\left\| \mathbf{u} \right\|_{X_0 \left(J \times \mathbb{R}^{2} \right)}
\leq\cdots\leq \left\| \mathbf{u} \right\|_{X_9 \left(J \times \mathbb{R}^{2} \right)}\lesssim 1.
\end{align}
 This is the base case.

Next, we prove the inductive step.
For any $m\in\mathbb{Z}_N$, we decompose the nonlinear terms as
\begin{align*} 
& P_{\geq i}\mathbf{F}_m \left(\mathbf{u} \right) \\
= &  \sum_{j\in\mathbb{Z}_N} \Big[P_{\geq i} O  \left(P_{\geq i - 3} u_j\overline{P_{\geq i - 3} u_j}P_{\geq i - 3}u_m \right)
+ P_{\geq i} O \left(P_{\geq i - 3} u_j\overline{P_{\geq i - 3} u_j} P_{\leq i - 3} u_m \right) \\
&\qquad\qquad+ P_{\geq i} O \left(P_{\geq i - 3} u_j\overline{P_{\leq i - 3} u_j}P_{\leq i - 3} u_m \right) \Big].
\end{align*}
Then, by Lemma \ref{up-dual} (set $p=100/49$) and  \eqref{f5.7}, for a fixed interval $G_{k_{\ast}}^{3i}$,
\begin{align}\nonumber
	\left\|P_{\geq i}\mathbf{u} \right\|_{U_{\Delta}^2 \left(G^{3i}_{k_{\ast}},L_x^2l^2 \right) }
&\lesssim \min_{t_0\in G_{k_{\ast}}^{3i}} \left\|e^{i(t-t_0)\Delta}P_{\geq i}\mathbf{u}(t_0) \right\|_{U_{\Delta}^2 \left(G^{3i}_{k_{\ast}},L_x^2l^2 \right)}
+ \sup_{ \left\|\mathbf{v} \right\|_{V_{\Delta}^{100/51} \left(G_{k}^{3i},L_x^2l^2 \right)}=1}
\left\| \left|\mathbf{v} \right| \left|\mathbf{\mathbf{F}} \left(\mathbf{u} \right) \right| \right\|_{L_{t,x}^1 \left(G_{k}^{3i}\times\mathbb{R}^2 \right)} \\\label{estkeyy}
	&\lesssim \eta_0+A_1+B_1+C_1,\quad 1\leq k_{\ast}\leq 2^{3k_0-3i},
\end{align}
where
\begin{align*}
	&A_1=\sup_{ \left\|\mathbf{v} \right\|_{V_{\Delta}^{100/51} \left(G_{k}^{3i},L_x^2l^2 \right)}=1}
\left\|  \left|\mathbf{v} \right| \left|P_{\geq i-3}\mathbf{u} \right|^{3} \right\|_{L_{t,x}^{1} \left(G_{k}^{3i}\times\mathbb{R}^2 \right)} +\sup_{ \left\|\mathbf{v} \right\|_{V_{\Delta}^{100/51} \left(G_{k}^{3i},L_x^2l^2 \right)}=1}
\left\|  \left|\mathbf{v} \right| \left|P_{\geq i-3}\mathbf{u} \right|^{2} \left|P_{\leq i-3}\mathbf{u} \right|
\right\|_{L_{t,x}^{1} \left(G_{k}^{3i}\times\mathbb{R}^2 \right)},\notag\\
&B_1=\sup_{ \left\|\mathbf{v} \right\|_{V_{\Delta}^{100/51} \left(G_{k}^{3i},L_x^2l^2 \right)}=1}
\left\|  \left|\mathbf{v} \right| \left|P_{\leq i-3}\mathbf{u} \right|
\left|P_{i-6\leq\cdot \leq i-3}\mathbf{u} \right| \left|P_{\geq i-3}\mathbf{u} \right| \right\|_{L_{t,x}^{1} \left(G_{k}^{3i}\times\mathbb{R}^2 \right)},\notag\\
	& C_1= \sup_{ \left\|\mathbf{v} \right\|_{V_{\Delta}^{100/51} \left(G_{k}^{3i},L_x^2l^2 \right)}=1}
\left\| \left|\mathbf{v} \right| \left|P_{\leq i-3}\mathbf{u} \right| \left|P_{\leq i-6}\mathbf{u} \right|
\left|P_{\geq i-3}\mathbf{u} \right| \right\|_{L_{t,x}^{1} \left(G_{k}^{3i}\times\mathbb{R}^2 \right)}.
\end{align*}
By H\"older's inequality,
\begin{equation}\label{f5.16}
\aligned
A_1+B_1
\lesssim
  \left\| \mathbf{v} \right\|_{L_{t}^{3} L_{x}^{6}l^2}
\left\| P_{\geq i-3}\mathbf{u} \right\|_{L_{t}^{3} L_{x}^{6}l^2}^{2} \left\| P_{\leq i-3}\mathbf{u} \right\|_{L_{t}^{\infty} L_{x}^{2}l^2}.
\endaligned
\end{equation}
Using \eqref{embeddingvp} and \eqref{u-pstrichartz}, we obtain
\begin{equation*}
\left\| \mathbf{v} \right\|_{L_{t}^{\infty} L_{x}^{2}l^2 \left(G_{k_{\ast}}^{3i}\times\mathbb{R}^2 \right)}
+ \left\| \mathbf{v} \right\|_{L_{t}^{3} L_{x}^{6}l^2 \left(G_{k_{\ast}}^{3i}\times\mathbb{R}^2 \right)}
\lesssim   \left\| \mathbf{v} \right\|_{V_{\Delta}^{100/51} \left(G_{k_{\ast}}^{3i},L_x^2l^2 \right)} = 1.
\end{equation*}
Therefore, when $i > 9$, by \eqref{f5.7}, H\"older's inequality  and $L^p$ interpolation,
\begin{align}\label{f5.18}
\aligned
A_1+B_1
& \lesssim \left\| P_{\geq i-6}\mathbf{u} \right\|_{L_{t}^{3} L_{x}^{6} l^2 \left(G_{k_{\ast}}^{3i}\times\mathbb{R}^2\times\mathbb{Z}_N \right)}^{2}
\\
& \lesssim \left\| P_{\geq i-6}\mathbf{u} \right\|_{L_{t}^{5/2} L_{x}^{10}l^2  \left(G_{k_{\ast}}^{3i}\times\mathbb{R}^2\times\mathbb{Z}_N \right)}^{5/3}
\left\| P_{\geq i-6}\mathbf{u} \right\|_{L_{t}^{\infty} L_{x}^{2} l^2  \left(G_{k_{\ast}}^{3i}\times\mathbb{R}^2\times\mathbb{Z}_N \right)}^{1/3}\\
& \lesssim \eta_{0}^{1/3}  \left\| \mathbf{u} \right\|_{X_{i - 3} \left(J \times \mathbb{R}^2\times\mathbb{Z}_N \right)}^{5/3}.
\endaligned
\end{align}
For the term $C_1$, obviously
\begin{align*}
	C_1
&\lesssim \left\| \left|P_{\geq i - 3} \mathbf{u} \right|\cdot \left|P_{\leq i - 6} \mathbf{u} \right| \right\|_{L_{t,x}^{2} \left(G_{k_{\ast}}^{3i}\times\mathbb{R}^2 \right)} \cdot \sup_{ \left\|\mathbf{v} \right\|_{V_{\Delta}^{100/51} \left(G_{k}^{3i},L_x^2l^2 \right)}=1}
\left\|\mathbf{v}\cdot P_{\leq i-3}\mathbf{u} \right\|_{L_{t,x}^2 \left(G_{k}^{3i}\times\mathbb{R}^2 \right)}\notag\\
&\lesssim \eta_{0}^{1/20} \left\| \mathbf{u} \right\|_{X_{i - 3} \left(J \times \mathbb{R}^2\times\mathbb{Z}_N \right)} 
\cdot \sup_{ \left\|\mathbf{v} \right\|_{V_{\Delta}^{100/51} \left(G_{k}^{3i}, L_x^2l^2 \right)}=1}
\left\|\mathbf{v}\cdot P_{\leq i-3}\mathbf{u} \right\|_{L_{t,x}^2 \left(G_{k}^{3i}\times\mathbb{R}^2 \right)}.
\end{align*}
Next, we suppose that it is true that for any $\left\|\mathbf{v}_{0} \right\|_{L^2_{x} l^2}=1$, where $\hat{v}_{0,j}$ is supported on $|\xi| \geq 2^{i}$ for any $j\in\mathbb{Z}_N$,
\begin{equation}\label{f5.19}
\sup_{ \left\|\mathbf{v}_{0} \right\|_{L_x^2l^2}=1}
 \left\| e^{it \Delta} \mathbf{v}_{0}\cdot P_{\leq i-3} \mathbf{u}  \right\|_{L_{t,x}^{2} \left(G_{k}^{3i}\times\mathbb{R}^2 \right)}
  \lesssim 1 + \eta_{0} \left\| \mathbf{u} \right\|_{X_{i - 3} \left(J \times \mathbb{R}^2\times\mathbb{Z}_N \right)}^{3}.
\end{equation}
Now, suppose $\mathbf{v}$ admits an atom decomposition $  \mathbf{v} =\sum_{\lambda\in\Z}c_\lambda  \mathbf{v}^\lambda$, then \eqref{f5.19} implies
\begin{align}\label{f5.20}
\left\|\mathbf{v}\cdot P_{\leq i-3}\mathbf{u} \right\|_{L_{t,x}^2 \left(G_{k}^{3i}\times\mathbb{R}^2 \right)}
&\lesssim \sup_{\lambda\in\mathbb{Z}} \left\|\mathbf{v}^{\lambda}\cdot P_{\leq i-3}\mathbf{u} \right\|_{L_{t,x}^2 \left(G_{k}^{3i}\times\mathbb{R}^2 \right)}\cdot \sum_{\lambda\in\mathbb{Z}}|c_{\lambda}|\notag\\
&\lesssim \sup_{\lambda\in\mathbb{Z}} \left(\sum_{k=0}^{M_{\lambda}}\left\|\mathbf{v}_k^{\lambda}\cdot P_{\leq i-3}\mathbf{u} \right\|^2_{L_{t,x}^2 \left(G_{k}^{3i}\times\mathbb{R}^2 \right)}\right)^{1/2}\cdot \sum_{\lambda\in\mathbb{Z}}|c_{\lambda}|\notag\\
&\lesssim (1 + \eta_{0} \left\| \mathbf{u} \right\|_{X_{i - 3} \left(J \times \mathbb{R}^2\times\mathbb{Z}_N \right)}^{3})\cdot\sum_{\lambda\in\mathbb{Z}}|c_{\lambda}|, 
\end{align}
where $\mathbf{v}^{\lambda}=\sum\limits_{k=0}^{M_{\lambda}}\chi_{[t_k,t_{k+1}]}(t)e^{it\Delta} \mathbf{v}_k^\lambda(x),
\ 
\sum\limits_{k=0}^{M_{\lambda}}\left\| \mathbf{v}_k^\lambda(x) \right\|_{L_x^2l^2 \left(\R^2\times\Z_N  \right)}^2=1.$ Using \eqref{embeddingvp}  and  the definition of $U^2_{\Delta}\left(G_k^{3i}, L_x^2l^2 \right)$, we have 
\begin{equation*}
\sup_{\left\|\mathbf{v} \right\|_{V_{\Delta}^{100/51}\left(G_{k}^{3i}, L_x^2l^2 \right)}=1}
\left\|\mathbf{v}\cdot P_{\leq i-3}\mathbf{u} \right\|_{L_{t,x}^2 \left(G_{k}^{3i}\times\mathbb{R}^2 \right)}
\lesssim 1 + \eta_{0} \left\| \mathbf{u} \right\|_{X_{i - 3} \left(J \times \mathbb{R}^2\times\mathbb{Z}_N \right)}^{3},
\end{equation*}
and therefore,
\begin{equation}\label{f5.21}
C_1\lesssim  \eta_{0}^{1/20} \left\| \mathbf{u} \right\|_{X_{i - 3} \left(J \times \mathbb{R}^2\times\mathbb{Z}_N \right)}
 \left(1 + \eta_{0} \left\| \mathbf{u} \right\|_{X_{i - 3} \left(J \times \mathbb{R}^2\times\mathbb{Z}_N \right)}^{3} \right).
\end{equation}
Inserting \eqref{f5.16} and \eqref{f5.21} into \eqref{estkeyy}, we obtain  
\begin{align}\label{longtime1}
\aligned
	& \left\|P_{\geq i}\mathbf{u} \right\|_{U_{\Delta}^2 \left(G^{3i}_{k_{\ast}}, L_x^2l^2 \right)} \\
 \lesssim &  \eta_0+\eta_{0}^{1/3} \left\| \mathbf{u} \right\|_{X_{i - 3} \left(J \times \mathbb{R}^2\times\mathbb{Z}_N \right)}^{5/3}
\\
& +\eta_{0}^{1/20} \left\| \mathbf{u} \right\|_{X_{i - 3} \left(J \times \mathbb{R}^2\times\mathbb{Z}_N \right)} \left(1 + \eta_{0} \left\| \mathbf{u} \right\|_{X_{i - 3} \left(J \times \mathbb{R}^2\times\mathbb{Z}_N \right)}^{3} \right),
\quad \forall \, 1\leq k_{\ast}\leq 2^{3k_0-3i}.
\endaligned
\end{align}
Next, we prove \eqref{f5.19}. Following the strategy in \cite{D2}, we will prove it via the interaction Morawetz estimate. Let
\begin{equation*} 
\mathbf{v}(t,x) = e^{it \Delta} \mathbf{v}_{0},
\end{equation*}
where $\left\| \mathbf{v}_{0} \right\|_{L^2_{x}l^2} = 1$ and $\forall \, k\in\mathbb{Z}_N, \ \hat{v}_{0,k}$ is supported on $|\xi|\sim 2^m$ for some $m \geq i$. For a fixed direction $\omega\in \mathbb{S}^1$, we define the interaction Morawetz potential
\begin{align*} 
M_{\omega}(t) = & \sum_{j\in\mathbb{Z}_N} \sum_{k\in\mathbb{Z}_N} \int |v_k(t, y)|^{2} \frac{(x - y)_{\omega}}{|(x - y)_{\omega}|}
\Im \left[\overline{P_{\leq i+3}u_{j}} \partial_{\omega} P_{\leq i+3}u_j \right](t,x) \, \mathrm{d} x  \mathrm{d} y
\\
& + \sum_{j\in\mathbb{Z}_N}\sum_{k\in\mathbb{Z}_N} \int |P_{\leq i+3}u_j(t,y) |^{2} \frac{(x - y)_{\omega}}{|(x - y)_{\omega}|} \Im \left[\overline{v_k} \partial_{\omega} v_k \right](t,x) \,  \mathrm{d}x  \mathrm{d} y, \notag
\end{align*}
where $(x-y)_{\omega}=(x-y)\cdot \omega$ and $\partial_{\omega}=\nabla \cdot \omega$.
Noticing that $P_{\leq i-3}\mathbf{u}$ solves the equation:
\begin{equation}\label{f5.25}
i \partial_{t} P_{\leq i - 3}{u}_j + \Delta P_{\leq i - 3}{u}_j +  { {F}}_j \left( P_{\leq i - 3}\mathbf{u} \right)
 =  { {F}}_j \left(P_{\leq i - 3}\mathbf{u} \right)
 - P_{\leq i - 3} { {F}}_j \left(\mathbf{u} \right)
  = -\mathcal{N}^{(j)}_{i - 3},\quad \forall \, j\in\mathbb{Z}_N.
\end{equation}
Making a direct computation,
\begin{equation}\label{f5.26}
\aligned
\frac{ \mathrm{d}}{ \mathrm{d} t} M_{\omega}(t)
 =  & 2 \sum_{j\in\mathbb{Z}_N}\sum_{k\in\mathbb{Z}_N}\int \int_{x_{\omega} = y_{\omega}}
\left|\partial_{\omega} \left(\overline{v_k(t,y)} P_{\leq i-3}u_{j}(t,x) \right) \right|^{2} \, \mathrm{d} x  \mathrm{d} y
\\
& - \sum_{j\in\mathbb{Z}_N}\sum_{k\in\mathbb{Z}_N}\int \int_{x_{\omega} = y_{\omega}} |v_k(t,y)|^{2}
\left( \overline{P_{\leq i-3}u_{j}}
 { {F}}_j \left(P_{\leq i - 3}\mathbf{u}  \right) \right)(t,x)  \, \mathrm{d} x  \mathrm{d} y \\
& + \sum_{j\in\mathbb{Z}_N}\sum_{k\in\mathbb{Z}_N}\int\int |v_k(t,y)|^{2} \frac{(x - y)_{\omega}}{|(x - y)_{\omega}|}
\cdot \Re \left[\overline{P_{\leq i-3}u_{j}} \partial_{\omega} \mathcal {N}^{(j)}_{i - 3} \right](t,x) \, \mathrm{d} x  \mathrm{d} y \\
& -  \sum_{j\in\mathbb{Z}_N}\sum_{k\in\mathbb{Z}_N} \int\int |v_k(t,y)|^{2} \frac{(x - y)_{\omega}}{|(x - y)_{\omega}|}
\cdot \Re \left[\overline{\mathcal {N}^{(j)}_{i - 3}} \partial_{\omega} P_{\leq i-3}u_{j} \right](t,x) \, \mathrm{d} x  \mathrm{d} y \\
& + 2 \sum_{j\in\mathbb{Z}_N}\sum_{k\in\mathbb{Z}_N}
\int\int \Im \left[\overline{P_{\leq i-3}u_{j}} \mathcal {N}^{(j)}_{i - 3} \right](t, y) \frac{(x - y)_{\omega}}{|(x - y)_{\omega}|}
\cdot \Im \left[\overline{v}_k \partial_{\omega} v_k \right](t,x) \, \mathrm{d} x  \mathrm{d} y.
\endaligned
\end{equation}
On the other hand, as calculated in \cite{D2}, for any $f,g\in L_x^2(\mathbb{R}^2)$, and $j_1,j_2\in\mathbb{N}$, we have
\begin{align}\label{f5.27}
	\aligned
	& \int\int  \left|\partial_{\omega} \left(P_{\leq j_1} f ( x_{1}, x_{2}) \overline{P_{\geq j_2} g( x_{1}, x_{2})} \right) \right|^{2} \,
\mathrm{d} x_{1} \mathrm{d} x_{2}
\\
& \lesssim 2^{j_1}  \int\int\int  \left|\partial_{\omega} \left(P_{\leq j_1}f( x_{1}, x_{2}) \overline{P_{\geq j_2}g ( x_{1}, y_{2})} \right) \right|^{2}
\, \mathrm{d} x_{1}  \mathrm{d} x_{2} \mathrm{d} y_{2}.
	\endaligned
\end{align}
Let
\begin{align*} 
	M(t)=\int_{ \mathbb{S}^1} M_{\omega}(t)  \, \mathrm{d} \omega .
\end{align*}
By \eqref{f5.26}, \eqref{f5.27}, Bernstein's inequality and the fact that
\begin{align*}
	\int_{ \mathbb{S}^1} \frac{x_{\omega}}{|x_{\omega}|} \partial_{\omega} \,  \mathrm{d} \omega
=C\frac{x}{|x|}\cdot\nabla, \quad \int_{ \mathbb{S}^1} \delta (\cos(\theta)r) \,  \mathrm{d} \theta
=\frac{C'}{r} \quad\mbox{ for some $C, C' > 0$},
\end{align*}
we obtain 
 \begin{equation}\label{f5.29}
 	\aligned
 	& 2^{2m} \sum_{j\in\mathbb{Z}_N}\sum_{k\in\mathbb{Z}_N}
 \left\| \overline{v}_k P_{\leq i-3}u_{j} \right\|_{L_{t,x}^{2} \left(G^{3i}_{k_{\ast}} \times \mathbb{R}^{2} \right)}^{2}  \\
  \lesssim & 2^{i}\sup_{t\in G^{3i}_{k_{\ast}}}  | M(t) | 
 +2^iC'\sum_{j\in\mathbb{Z}_N}\sum_{k\in\mathbb{Z}_N}\int\int |v_k(t,y)|^{2} \frac{1}{|x-y|}
 \left(
 \overline{P_{\leq i-3}u_{j}} { {F}}_j \left(P_{\leq i - 3}\mathbf{u} \right) \right) (t,x)\, \mathrm{d} x  \mathrm{d} y \\
 & 	- 2^iC\sum_{j\in\mathbb{Z}_N}\sum_{k\in\mathbb{Z}_N}\int_{G^{3i}_{k_{\ast}}}\int\int |v_k(t,y)|^{2} \frac{(x - y)}{|(x - y)|}
 \cdot \Re \left[\overline{P_{\leq i-3}u_{j}}\cdot \nabla \mathcal {N}^{(j)}_{i - 3} \right](t,x)  \, \mathrm{d} x  \mathrm{d} y  \mathrm{d} t \\
 &  + 2^iC\sum_{j\in\mathbb{Z}_N}\sum_{k\in\mathbb{Z}_N}\int_{G^{3i}_{k_{\ast}}}\int\int |v_k(t,y)|^{2} \frac{(x - y)}{|(x - y)|}
  \cdot \Re \left[\overline{\mathcal {N}^{(j)}_{i - 3}} \nabla P_{\leq i-3}u_{j} \right](t,x) \, \mathrm{d}x  \mathrm{d}y  \mathrm{d} t \\
&  	- 2^{i+1}C \sum_{j\in\mathbb{Z}_N}\sum_{k\in\mathbb{Z}_N}\int_{G^{3i}_{k_{\ast}}}\int\int \Im \left[\overline{P_{\leq i-3}u_{j}} \mathcal {N}^{(j)}_{i - 3} \right](t, y)
 \frac{(x - y)}{|x - y |} \cdot \Im \left[\overline{v}_k \nabla v_k \right](t,x) \, \mathrm{d}x  \mathrm{d} y  \mathrm{d} t.
 	\endaligned
 \end{equation}
Also, by Bernstein's inequality,
 \begin{align}\label{estttt1}
 	\sup_{t\in J} |M(t) | \lesssim 2^m.
 \end{align}
Next we estimate the remaining terms in \eqref{f5.29}. We begin by proving some useful estimates.
Noticing that by \eqref{f5.7} and Bernstein's inequality,
\begin{align}\label{keya}
\aligned
	\left\|P_{\leq i-3}\mathbf{u}(t) \right\|_{L_{t,x}^{\infty}l^2 \left(J\times\mathbb{R}^2\times\mathbb{Z}_N \right)}
&\lesssim \left\|P_{\leq \eta^{1/2}_{1}}\mathbf{u}(t) \right\|_{L_{t,x}^{\infty}l^2 \left(J\times\mathbb{R}^2\times\mathbb{Z}_N \right)}
+ \left\|P_{\eta_1^{1/2}\leq\cdot\leq i-3}\mathbf{u}(t) \right\|_{L_{t,x}^{\infty}l^2 \left(J\times\mathbb{R}^2\times\mathbb{Z}_N \right)}\\
	&\lesssim \eta_1^{1/4}+\eta_02^i\lesssim \eta_02^i.
\endaligned
\end{align}
	Similarly, for any $p>2$, we have
	\begin{align}\label{keyz}
	\left\|P_{\leq i-3}\mathbf{u}(t) \right\|_{L_{t}^{\infty}L_x^pl^2 \left(J\times\mathbb{R}^2\times\mathbb{Z}_N \right)}
&\lesssim  \eta_02^{i(1-\frac{2}{p})},
	\end{align}
if $\eta_1$ is sufficiently small (depending on $p$ and $\eta_0$).

Recalling the definition of the $X_{i-3}$ norm, for all $0\leq j \leq i-3$, we have
\begin{align}\nonumber
	\left\| P_{j}\mathbf{u} \right\|^{20}_{L_t^{20} L_x^{20/9}l^2 \left(G^{3i}_{k_{\ast}}\times\mathbb{R}^2\times\mathbb{Z}_N \right)}
	\lesssim
    \sum_{G^{3j}_l\subset G^{3i}_{k_{\ast}}} \left\|P_{j}\mathbf{u} \right\|^{20}_{L_t^{20}L_x^{20/9}l^2 \left(G^{3j}_{l}\times\mathbb{R}^2\times\mathbb{Z}_N \right)}
\lesssim  2^{3i-3m} \left\|\mathbf{u} \right\|^{20}_{X_{i-3} \left(J\times\mathbb{R}^2\times\mathbb{Z}_N \right)}.
\end{align}
Thus, 
\begin{align}\label{xnorm}
 \left\| P_{\leq i-3}\mathbf{u} \right\|_{L_t^{20}L_x^{40/13}l^2 \left(G^{3i}_{k_{\ast}}\times\mathbb{R}^2\times\mathbb{Z}_N \right)}
\lesssim & \sum_{m=0}^{i-3}2^{\frac{3i-3m}{20}}2^{m/4} \left\|\mathbf{u} \right\|_{X_{i-3} \left(J\times\mathbb{R}^2\times\mathbb{Z}_N \right)} 
	\lesssim 2^{i/4} \left\|\mathbf{u} \right\|_{X_{i-3} \left(J\times\mathbb{R}^2\times\mathbb{Z}_N \right)}.
\end{align}
Similarly, we can also prove that
	\begin{align}\label{xnorm1}
		\left\|\nabla P_{\leq i-3}\mathbf{u} \right\|_{L_t^{3}L_x^{6}l^2 \left(G^{3i}_{k_{\ast}}\times\mathbb{R}^2\times\mathbb{Z}_N \right)}
&\lesssim 2^i \left\|\mathbf{u} \right\|_{X_{i-3} \left(J\times\mathbb{R}^2\times\mathbb{Z}_N \right)}.
	\end{align}
The estimates \eqref{keya}-\eqref{xnorm1} are sufficient for us to estimate the remaining terms in \eqref{f5.29}.
First, by \eqref{embeddingvp}, \eqref{u-pstrichartz}, \eqref{keya}, \eqref{xnorm1} and the Hardy-Littlewood-Sobolev inequality,
\begin{align}\nonumber
	& 2^iC\sum_{j\in\mathbb{Z}_N}\sum_{k\in\mathbb{Z}_N} \int\int |v_k(t,y)|^{2} \frac{1}{|x-y|}
\left( \overline{P_{\leq i-3}u_{j} { {F}}_j  \left(P_{\leq i - 3}\mathbf{u} \right)} \right) (t,x) \, \mathrm{d} x  \mathrm{d} y\\\nonumber
	 \lesssim& 2^i \left\| \mathbf{v} \right\|_{L_t^{40/3}L_x^{40/17}l^2 \left(G_{k_{\ast}}^{3i}
     \times
     \mathbb{R}^2
     \times\mathbb{Z}_N \right)}^2
\left\| P_{\leq i-3}\mathbf{u} \right\|_{L_t^{20}L_x^{40/13}l^2 \left(G_{k_{\ast}}^{3i}\times\mathbb{R}^2\times\mathbb{Z}_N \right)}^3
\left\|P_{\leq i-3}\mathbf{u}(t) \right\|_{L_{t}^{\infty}L_x^{8/3}l^2 \left(J\times\mathbb{R}^2\times\mathbb{Z}_N \right)}
\\\label{estttt2}
 \lesssim& \eta_02^{2i} \left\|\mathbf{u} \right\|^{2}_{X_{i-3} \left(J\times\mathbb{R}^2\times\mathbb{Z}_N \right)}.
\end{align}
Now consider the terms
\begin{equation}\label{f5.37}
 { {F}}_j \left(P_{\leq i - 3}\mathbf{u} \right) - P_{\leq i - 3} { {F}}_j \left(\mathbf{u} \right)
= -\mathcal {N}^{(j)}_{i - 3},\quad \forall \, j\in\mathbb{Z}_N.
\end{equation}
where
\begin{align}\nonumber
	\mathcal{N}_{i-3}^{(j)}
=&P_{\leq i-3} {F}_j \left(\mathbf{u} \right)- {F}_j \left(P_{\leq i-3}\mathbf{u} \right)
=P_{\leq i-3} {F}_j \left(\mathbf{u} \right) -P_{\leq i-3} {F}_j \left(P_{\leq i-6}\mathbf{u} \right)
+ {F}_j \left(P_{\leq i-6}\mathbf{u} \right)- {F}_j \left(P_{\leq i-3}\mathbf{u} \right)\\\label{f5.38}
	=&\mathcal{N}_{i-3}^{(j,1)}+\mathcal{N}_{i-3}^{(j,2)},\\\nonumber
\mathcal{N}_{i-3}^{(j,1)}=& \sum_{k\in\mathbb{Z}_N} O \left(P_{\leq i-6}u_k P_{\leq i-6}u_kP_{i-6\leq\cdot\leq i-3}u_j \right)
+ \sum_{k\in\mathbb{Z}_N} P_{\leq i-3}O \left(P_{\leq i-6}u_k P_{\leq i-6}u_kP_{ \geq i-6}u_j \right), \\\nonumber
	\mathcal{N}_{i-3}^{(j,2)}=&\sum_{k\in\mathbb{Z}_N} P_{\leq i-3}O \left(P_{\leq i-6}u_k P_{\geq i-6}u_kP_{\geq i-6}u_j \right)
	+\sum_{k\in\mathbb{Z}_N} P_{\leq i-3}O\left(P_{\geq i-6}u_k P_{\geq i-6}u_kP_{\geq i-6}u_j \right)\\\nonumber
	&   +\sum_{k\in\mathbb{Z}_N}O \left(P_{\leq i-6}u_k P_{i-6 \leq\cdot\leq i-3}u_kP_{i-6 \leq\cdot\leq i-3}u_j \right)\\\nonumber
	&    +\sum_{k\in\mathbb{Z}_N}O \left(P_{i-6 \leq\cdot\leq i-3}u_k P_{i-6 \leq\cdot\leq i-3}u_kP_{i-6 \leq\cdot\leq i-3}u_j \right).
\end{align}
Using \eqref{f5.18}, \eqref{xnorm1} and Bernstein's inequality,
\begin{align}\label{estqq1}
\aligned
 & 2^iC\sum_{j\in\mathbb{Z}_N}\sum_{k\in\mathbb{Z}_N}\int_{G^{3i}_{k_{\ast}}} \int \int |v_k(t,y)|^{2} \frac{ x - y }{|x - y |}
 \cdot \Re \left[\overline{\mathcal N^{(j,2)}_{i - 3}} \nabla P_{\leq i-3}u_{j} \right](t,x)  \,\mathrm{d} x  \mathrm{d} y  \mathrm{d} t \\
  \lesssim& 2^i \left\|P_{\geq i-6}\mathbf{u} \right\|^2_{L_t^{3}L_x^{6}l^2 \left(G_{k_{\ast}}^{3i}\times\mathbb{R}^2\times\mathbb{Z}_N \right)}
 \left\| \mathbf{u} \right\|_{L_{t}^{\infty}L_x^{2}l^2 \left(G_{k_{\ast}}^{3i}\times\mathbb{R}^2\times\mathbb{Z}_N \right)}
 \left\|\nabla P_{\leq i-3}\mathbf{u} \right\|_{L_{t}^{3}L_x^{6}l^2 \left(G_{k_{\ast}}^{3i}\times\mathbb{R}^2\times\mathbb{Z}_N \right)} \\
   \lesssim& 2^{2i}\eta_{0}^{1/3} \left\| \mathbf{u} \right\|_{X_{i - 3} \left(J \times \mathbb{R}^2\times\mathbb{Z}_N \right)}^{8/3}.
 \endaligned
\end{align}
Noticing that by the Coifman-Meyer multiplier theorem,
$$\left\|P_{\leq i-3}f P_{\leq i-2}g-P_{\leq i-3} \left(fP_{\leq i-2}g \right) \right\|_{L_x^1}
\lesssim 2^{-i}\|f\|_{L_x^p} \|\nabla g\|_{L_x^p}, \quad p^{-1}+q^{-1}=1.$$
So once again, using \eqref{f5.18}, \eqref{xnorm1} and Bernstein's inequality, we get
\begin{align}\label{f5.40}
& \left\|  \left|\mathcal{N}_{i-3}^{(j,2)} \right|  |P_{\leq i-3}u_j| \right\|_{L_{t,x}^{1} \left(G_{k_{\ast}}^{3i}\times\mathbb{R}^2 \right)}\\
 \lesssim &  \left\|\mathbf{u} \right\|_{L_t^{\infty}L_x^2l^2 \left(G_{k_{\ast}}^{3i}\times\mathbb{R}^2\times\mathbb{Z}_N \right)}
 \left\|P_{\geq i-3}\mathbf{u} \right\|^3_{L_t^3L_x^6l^2 \left(G_{k_{\ast}}^{3i}\times\mathbb{R}^2\times\mathbb{Z}_N \right)}\notag\\
&+\sum_{k\in\mathbb{Z}_N} \left\|O \left(P_{\leq i-6}u_k P_{i-6 \leq\cdot\leq i-3}u_kP_{i-6 \leq\cdot\leq i-3}u_j \right)
 P_{\leq i-3}u_j \right\|_{L_{t,x}^{1} \left(G_{k_{\ast}}^{3i}\times\mathbb{R}^2 \right)}\notag\\
&+\sum_{k\in\mathbb{Z}_N} \left\|O \left(P_{\leq i-6}u_k P_{\geq i-6}u_kP_{\geq i-6}u_j \right)
P_{\leq i-3}u_j \right\|_{L_{t,x}^{1} \left(G_{k_{\ast}}^{3i}\times\mathbb{R}^2 \right)}\notag\\
&+2^{-i} \left\|P_{\geq i-6}\mathbf{u} \right\|^2_{L_t^{3}L_x^{6}l^2 \left(G_{k_{\ast}}^{3i}\times\mathbb{R}^2\times\mathbb{Z}_N \right)}
\left\| \mathbf{u} \right\|_{L_{t}^{\infty}L_x^{2}l^2 \left(G_{k_{\ast}}^{3i}\times\mathbb{R}^2\times\mathbb{Z}_N \right)}
\left\|\nabla P_{\leq i-3}\mathbf{u} \right\|_{L_{t}^{3}L_x^{6}l^2 \left(G_{k_{\ast}}^{3i}\times\mathbb{R}^2\times\mathbb{Z}_N \right)}\notag\\
\lesssim & \eta_{0}^{1/3} \left\| \mathbf{u} \right\|_{X_{i - 3} \left(J \times \mathbb{R}^2\times\mathbb{Z}_N \right)}^{8/3}
+ \eta_{0}^{1/2} \left\| \mathbf{u} \right\|_{X_{i - 3} \left(J \times \mathbb{R}^2\times\mathbb{Z}_N \right)}^{5/2}
+ \eta_0^{1/10} \left\| \mathbf{u} \right \|^2_{X_{i - 3} \left(J \times \mathbb{R}^2\times\mathbb{Z}_N \right)}.
\end{align}
Therefore, we have
\begin{align}\label{estqq2}
\aligned
& 2^{i+1} \sum_{j\in\mathbb{Z}_N}\sum_{k\in\mathbb{Z}_N}\int_{G^{3i}_{k_{\ast}}}\int\int \Im \left[ \overline{P_{\leq i-3}u_{j}} \mathcal {N}^{(j,2)}_{i - 3} \right](t, y)
\frac{x - y }{|x - y|}
\cdot \Im \left[\overline{v}_k \nabla v_k \right](t,x)  \, \mathrm{d} x \mathrm{d} y  \mathrm{d} t\\
 \lesssim&  2^{i+m} \left(\eta_{0}^{1/3} \left\| \mathbf{u} \right\|_{X_{i - 3} \left(J \times \mathbb{R}^2\times\mathbb{Z}_N \right)}^{8/3}
+ \eta_{0}^{1/2} \left\| \mathbf{u} \right\|_{X_{i - 3} \left(J \times \mathbb{R}^2\times\mathbb{Z}_N \right)}^{5/2}
+\eta_0^{1/10} \left\| \mathbf{u} \right\|^2_{X_{i - 3} \left(J \times \mathbb{R}^2\times\mathbb{Z}_N \right)} \right),
\endaligned
\end{align}
and
\begin{align}\label{estqq3}
\aligned
	& 2^iC\sum_{j\in\mathbb{Z}_N}\sum_{k\in\mathbb{Z}_N}
\int_{G^{3i}_{k_{\ast}}}\int\int |v_k(t,y)|^{2} \frac{ x - y }{| x - y |}
\cdot \Re \left[\overline{P_{\leq i-3}u_{j}} \nabla \mathcal{N}^{(j,2)}_{i - 3} \right](t,x) \, \mathrm{d} x  \mathrm{d} y  \mathrm{d} t\\
	 \lesssim & 2^iC  \left|\sum_{j\in\mathbb{Z}_N}\sum_{k\in\mathbb{Z}_N}\int_{G^{3i}_{k_{\ast}}} \int\int |v_k(t,y)|^{2} \frac{ x - y }{| x - y |}
\cdot \Re \left[\nabla \left(\overline{P_{\leq i-3}u_{j}}  \mathcal{N}^{(j,2)}_{i - 3} \right) \right](t,x)  \,\mathrm{d} x  \mathrm{d}y \mathrm{d}t \right| \\
& 	 +2^iC \left|\sum_{j\in\mathbb{Z}_N}\sum_{k\in\mathbb{Z}_N}\int_{G^{3i}_{k_{\ast}}}\int\int |v_k(t,y)|^{2} \frac{ x - y }{| x - y |}
\cdot \Re \left[\nabla\overline{P_{\leq i-3}u_{j}}  \mathcal{N}^{(j,2)}_{i - 3} \right](t,x) \, \mathrm{d} x \mathrm{d} y  \mathrm{d} t \right| \\
	 \lesssim& 2^{2i} \left(\eta_{0}^{1/3}  \left\| \mathbf{u}  \right\|_{X_{i - 3} \left(J \times \mathbb{R}^2\times\mathbb{Z}_N \right)}^{8/3}
+ \eta_{0}^{1/2} \left\| \mathbf{u}  \right\|_{X_{i - 3} \left(J \times \mathbb{R}^2\times\mathbb{Z}_N \right)}^{5/2}
+\eta_0^{1/10} \left\| \mathbf{u} \right\|^2_{X_{i - 3} \left(J \times \mathbb{R}^2\times\mathbb{Z}_N \right)} \right).
\endaligned
\end{align}
Following \eqref{f5.40}-\eqref{estqq3}, we also have
\begin{align}\label{estqq41}
\aligned
&  2^iC \left|\sum_{j\in\mathbb{Z}_N}\sum_{k\in\mathbb{Z}_N}\int_{G^{3i}_{k_{\ast}}}\int\int |v_k(t,y)|^{2} \frac{ x - y }{| x - y |}
 \cdot \Re \left[\overline{P_{i-6\leq\cdot\leq i-3}u_{j}} \nabla \mathcal N^{(j,1)}_{i - 3} \right](t,x) \,\mathrm{d}x \mathrm{d} y \mathrm{d}t \right| \\
&   + 2^iC \left|\sum_{j\in\mathbb{Z}_N}\sum_{k\in\mathbb{Z}_N}\int_{G^{3i}_{k_{\ast}}}\int\int |v_k(t,y)|^{2} \frac{ x - y }{| x - y |}
  \cdot \Re \left[\overline{\mathcal N^{(j,1)}_{i - 3}} \nabla P_{i-6\leq\cdot\leq i-3}u_{j} \right](t,x) \, \mathrm{d}x \mathrm{d} y \mathrm{d} t \right| \\
& +2^{i+1}C \left|\sum_{j\in\mathbb{Z}_N}\sum_{k\in\mathbb{Z}_N}\int_{G^{3i}_{k_{\ast}}} 
\int\int \Im \left[\overline{P_{i-6\leq\cdot\leq i-3}u_{j}} \mathcal N^{(j,1)}_{i - 3} \right](t, y)
\frac{ x - y }{| x - y |} \cdot \Im \left[\overline{v}_k \nabla v_k \right](t,x)\,  \mathrm{d}x \mathrm{d}y \mathrm{d} t \right| \\
 \lesssim& 2^{i+m} \left(\eta_{0}^{1/3}  \left\| \mathbf{u} \right\|_{X_{i - 3} \left(J \times \mathbb{R}^2\times\mathbb{Z}_N \right)}^{8/3}
+ \eta_{0}^{1/2} \left\| \mathbf{u} \right\|_{X_{i - 3} \left(J \times \mathbb{R}^2\times\mathbb{Z}_N \right)}^{5/2}
+\eta_0^{1/10} \left\| \mathbf{u} \right\|^2_{X_{i - 3} \left(J \times \mathbb{R}^2\times\mathbb{Z}_N \right)} \right).
\endaligned
\end{align}
Finally, observe that for each $j,k\in\mathbb{Z}_N$, the Fourier support of
$\mathcal N_{i-3}^{(j,1)}P_{\leq i-6}u_k$
lies on frequencies $|\xi| \geq 2^{i - 9}$. Therefore, integrating by parts and using the Hardy-Littlewood-Sobolev inequality, we obtain 
\begin{equation}\label{estqq42}
\aligned
& 2^{i+1}C \left|\sum_{j\in\mathbb{Z}_N}\sum_{k\in\mathbb{Z}_N}\int \int \int \Im \left[\overline{P_{\leq i-6}u_k}  \mathcal N_{i-3}^{(j,1)} \right](t, y)
\frac{ x - y }{| x - y |} \Im \left[\overline{v_k} \partial_{x_{1}} v_k \right](t,x) \, \mathrm{d} x  \mathrm{d} y \mathrm{d}t \right| \\
 =&2^{i+1}C \Bigg|\sum_{l_1=1}^2\sum_{l_2=1}^2\sum_{j\in\mathbb{Z}_N}\sum_{k\in\mathbb{Z}_N} \int \int \int \partial_{l_1}{ (-\Delta_y)}^{-1}
\left(\Im \left[\overline{P_{\leq i-6}u_k}  \mathcal N_{i-3}^{(j,1)} \right] \right)(t,y)
 \\
 & \qquad \cdot \left(\frac{\delta_{l_1,l_2}}{|x-y|}-\frac{(x-y)_{l_1}(x-y)_{l_2}}{|x-y|^2} \right)
 \Im \left[\overline{v_k} \partial_{l_2} v_k \right](t,x)  \, \mathrm{d} x \mathrm{d}y  \mathrm{d} t  \Bigg|\\
 \lesssim&  \left\| \mathbf{v} \right\|_{L_t^{5/2}L_x^{10}l^2 \left(G_{k_{\ast}}^{3i}\times\mathbb{R}^2 \times\mathbb{Z}_N \right)}
\left\| \nabla\mathbf{v} \right\|_{L_t^{5/2}L_x^{10}l^2 \left(G_{k_{\ast}}^{3i}\times\mathbb{R}^2 \times\mathbb{Z}_N \right)}
\left\| P_{\leq i-3}\mathbf{u} \right\|_{L_t^{20}L_x^{40/11}l^2 \left(G_{k_{\ast}}^{3i}\times\mathbb{R}^2\times\mathbb{Z}_N \right)}^4
\\
  \lesssim& \eta_0^{2}2^{i+m} \left\| \mathbf{u}  \right\|^2_{X_{i - 3} \left(J \times \mathbb{R}^2\times\mathbb{Z}_N \right)}.
\endaligned
\end{equation}
By the triangle inequality, \eqref{estqq41} and \eqref{estqq42} immediately imply that
\begin{align}\label{estqq4}
\aligned
& 2^{i+1} \left|\sum_{j\in\mathbb{Z}_N}\sum_{k\in\mathbb{Z}_N}\int \int \int \Im \left[\overline{P_{\leq i-6}u_k}  \mathcal N_{i-3}^{(j,1)} \right](t, y)
\frac{(x - y)_{1}}{|(x - y)_{1}|} \Im \left[\overline{v_k} \partial_{x_{1}} v_k \right](t,x) \,\mathrm{d}x \mathrm{d} y  \mathrm{d} t \right| \\
 \lesssim& 2^{i+m} \left(\eta_{0}^{1/3}  \left\| \mathbf{u}  \right\|_{X_{i - 3} \left(J \times \mathbb{R}^2\times\mathbb{Z}_N \right)}^{8/3}
 + \eta_{0}^{1/2} \left\| \mathbf{u}  \right\|_{X_{i - 3} \left(J \times \mathbb{R}^2\times\mathbb{Z}_N \right)}^{5/2}
 +\eta_0^{1/10} \left\| \mathbf{u} \right\|^2_{X_{i - 3} \left(J \times \mathbb{R}^2\times\mathbb{Z}_N \right)} \right).
 \endaligned
\end{align}
Using almost the same strategy, we can also deduce that
\begin{align}\label{estqq5}
\aligned
& 	2^i \left|\sum_{j\in\mathbb{Z}_N}\sum_{k\in\mathbb{Z}_N}\int_{G^{3i}_{k_{\ast}}}\int\int  |v_k(t,y)|^{2} \frac{ x - y }{| x - y |}
\cdot \Re \left[\overline{P_{\leq i-3}u_{j}} \nabla \mathcal N^{(j,1)}_{i - 3} \right](t,x) \,\mathrm{d}x \mathrm{d}y \mathrm{d} t \right|\footnotemark	 \\
 & + 2^i \left|\sum_{j\in\mathbb{Z}_N}\sum_{k\in\mathbb{Z}_N}\int_{G^{3i}_{k_{\ast}}}\int\int |v_k(t,y)|^{2} \frac{ x - y }{| x - y |}
 \cdot \Re \left[\overline{\mathcal N^{(j,1)}_{i - 3}} \nabla P_{\leq i-3}u_{j} \right](t,x) \, \mathrm{d}x  \mathrm{d} y  \mathrm{d} t \right| \\
 	\lesssim& 2^{i+m} \left(\eta_{0}^{1/3} \left\| \mathbf{u} \right\|_{X_{i - 3} \left(J \times \mathbb{R}^2\times\mathbb{Z}_N \right)}^{8/3}
+ \eta_{0}^{1/2} \left\| \mathbf{u} \right\|_{X_{i - 3} \left(J \times \mathbb{R}^2\times\mathbb{Z}_N \right)}^{5/2}
+\eta_0^{1/10} \left\| \mathbf{u} \right\|^2_{X_{i - 3} \left(J \times \mathbb{R}^2\times\mathbb{Z}_N \right)} \right).
\endaligned
\end{align}
\footnotetext{For this term, we also need to use integration by parts to transfer the derivative in  $\mathcal N_{i-3}^{(j,1)}$ to the other terms. }
Plugging \eqref{estttt1}, \eqref{estttt2}, \eqref{estqq1},  \eqref{estqq2},  \eqref{estqq3},  \eqref{estqq4} and \eqref{estqq5} into \eqref{f5.29}, we get 
\begin{align}\label{f5.47}
& 2^{2m} \left\|  \left|P_{\leq i-3}\mathbf{u} \right|\cdot \left|\mathbf{v} \right| \right\|_{L_{t,x}^{2}}^{2}
 \notag \\
  \lesssim& 2^{i+m} \left(1+\eta_{0}^{1/3}  \left\| \mathbf{u} \right\|_{X_{i - 3} \left(J \times \mathbb{R}^2\times\mathbb{Z}_N \right)}^{8/3}
 + \eta_{0}^{1/2} \left\| \mathbf{u}  \right\|_{X_{i - 3} \left(J \times \mathbb{R}^2\times\mathbb{Z}_N \right)}^{5/2}
 +\eta_0^{1/10} \left\| \mathbf{u} \right\|^2_{X_{i - 3} \left(J \times \mathbb{R}^2\times\mathbb{Z}_N \right)} \right).
\end{align}
Thus, we have proved \eqref{f5.19} for $\mathbf{v}$ with Fourier support on $|\xi|\sim 2^m$.
For $\mathbf{v}$ with Fourier support on $|\xi|\geq 2^m$, one may verify that the implicit constant in \eqref{f5.47} can be chosen linearly with $\left\|\mathbf{v}\right\|^2_{L_x^2l^2}$ and independent of $m$, so \eqref{f5.19} holds.

Finally, we estimate
\begin{equation}\label{f5.48}
\left\|  |P_{\geq i} \mathbf{u}||P_{\leq i - 3} \mathbf{u}| \right\|_{L_{t,x}^{2}} .
\end{equation}
Comparing with \eqref{f5.19}, this is more complicated because $P_{\geq i}\mathbf{u}$ is not a solution to \eqref{1.1}. As we did in the proof of \eqref{f5.19}, for any $\omega\in \mathbb{S}^1$, we define the interaction Morawetz action 
\begin{align*} 
	\tilde{M}_{m,\omega}(t) = & \sum_{j\in\mathbb{Z}_N}\sum_{k\in\mathbb{Z}_N}\int |P_{i}u_k(t, y)|^{2} \frac{(x - y)_{\omega}}{|(x - y)_{\omega}|}
 \Im \left[\overline{P_{\leq i+3}u_{j}} \partial_{\omega} P_{\leq i+3}u_j \right](t, x)  \, \mathrm{ d} x  \mathrm{d} y   \\
&  + \sum_{j\in\mathbb{Z}_N}\sum_{k\in\mathbb{Z}_N}\int |P_{\leq i+3}u_j(t,y)|^{2} \frac{(x - y)_{\omega}}{|(x - y)_{\omega}|}
  \Im \left[\overline{P_{m}v_k} \partial_{\omega} P_mv_k \right] (t,x) \, \mathrm{d} x  \mathrm{d} y
\end{align*}
and let
\begin{align*}
	\tilde{M}_m(t)=\int_{ \mathbb{S}^1} \tilde{M}_{m,\omega}(t) \,  \mathrm{d} \omega.
\end{align*}
After  tedious calculation, we have 
\begin{equation}\label{f5.50}
\aligned
& \frac{ \mathrm{d} }{ \mathrm{d} t} \tilde{M}_{m,\omega}(t) \notag\\
 = & 2 \sum_{j\in\mathbb{Z}_N}\sum_{k\in\mathbb{Z}_N}\int \int_{x_{\omega} = y_{\omega}}
 \left|\partial_{\omega} \left(\overline{P_{m}u_k(t,y)} P_{\leq i-3}u_{j}(t,x) \right) \right|^{2}  \, \mathrm{d} x  \mathrm{d} y
\\
&  - \sum_{j\in\mathbb{Z}_N}\sum_{k\in\mathbb{Z}_N}\int \int_{x_{\omega} = y_{\omega}}
 |P_mu_k(t,y)|^{2} \left[\overline{P_{\leq i-3}u_{j}}  { {F}}_j \left(P_{\leq i - 3}\mathbf{u} \right)\right](t,x) \, \mathrm{d} x  \mathrm{d} y \\
 & + \sum_{j\in\mathbb{Z}_N}\sum_{k\in\mathbb{Z}_N}\int\int |P_mu_k(t,y)|^{2} \frac{(x - y)_{\omega}}{|(x - y)_{\omega}|}
  \cdot \Re \left[\overline{P_{\leq i-3}u_{j}} \partial_{\omega} \mathcal N^{(j)}_{i - 3} \right](t,x) \, \mathrm{d} x  \mathrm{d}  y \\
 & -  \sum_{j\in\mathbb{Z}_N}\sum_{k\in\mathbb{Z}_N}\int\int |P_mu_k(t,y)|^{2} \frac{(x - y)_{\omega}}{|(x - y)_{\omega}|}
  \cdot \Re \left[\overline{\mathcal N^{(j)}_{i - 3}} \partial_{\omega} P_{\leq i-3}u_{j} \right](t,x)  \, \mathrm{d} x  \mathrm{d} y \\
& + 2 \sum_{j\in\mathbb{Z}_N}\sum_{k\in\mathbb{Z}_N}\int\int \Im \left[\overline{P_{\leq i-3}u_{j}} \mathcal N^{(j)}_{i - 3} \right](t, y)
\frac{(x - y)_{\omega}}{|(x - y)_{\omega}|} \cdot \Im \left[\overline{P_mu_k} \partial_{\omega} P_mu_k \right](t,x) \, \mathrm{ d} x  \mathrm{d} y\\
& + \sum_{j\in\mathbb{Z}_N}\sum_{k\in\mathbb{Z}_N}\int \int |P_{\leq i - 3} u_j(t,y)|^{2} \frac{(x - y)_{\omega}}{|(x - y)_{\omega}|} \Re \left[\overline{P_{m} u_k} \partial_{\omega} P_{m} ( {F}_k) \right](t,x) \,  \mathrm{d} x  \mathrm{d} y \\
& - \sum_{j\in\mathbb{Z}_N}\sum_{k\in\mathbb{Z}_N}\int \int |P_{\leq i - 3} u_j(t,y)|^{2} \frac{(x - y)_{\omega}}{|(x - y)_{\omega}|}
\Re \left[\overline{P_{m} {F}_k } \partial_{\omega} P_{m} u_k \right](t,x) \, \mathrm{d}x  \mathrm{d}y \\
& + 2 \sum_{j\in\mathbb{Z}_N}\sum_{k\in\mathbb{Z}_N}\int \int \Im \left[\overline{P_{m} u_k} P_{m}  {F}_k  \right](t,y) \frac{(x - y)_{\omega}}{|(x - y)_{\omega}|}
\cdot\Im \left[\overline{P_{\leq i-3}u_{j}} \partial_{\omega} P_{\leq j-3}u_{j} \right](t,x)  \,\mathrm{d}x  \mathrm{d}y.
\endaligned
\end{equation}
Recalling  \eqref{f5.27}, as in \eqref{f5.29}, we also have
\begin{align}
	& \ 2^{2m} \sum_{j\in\mathbb{Z}_N}\sum_{k\in\mathbb{Z}_N}
\left\| P_m u_{k}P_{\leq i-3}u_{j}  \right\|_{L_{t,x}^{2} \left(G^{3i}_{k_{\ast}} \times \mathbb{R}^{2} \right)}^{2} \notag  \\
  \lesssim &  2^{i}\sup_{t\in G^{3i}_{k_{\ast}}}| \tilde{M}_m(t)|
 +2^iC'\sum_{j\in\mathbb{Z}_N}\sum_{k\in\mathbb{Z}_N}\int_{G^{3i}_{k_{\ast}}}\int\int |P_{m}u_k(t,y)|^{2} \frac{1}{|x-y|} \left[
 \overline{P_{\leq i-3}u_{j}} { {F}}_j
  \left(P_{\leq i - 3}\mathbf{u} \right)\right](t,x) \, \mathrm{ d} x  \mathrm{d} y \mathrm{d} t\label{ert1}\\
& 	- 2^iC\sum_{j\in\mathbb{Z}_N}\sum_{k\in\mathbb{Z}_N}\int_{G^{3i}_{k_{\ast}}}\int\int |P_{m}u_k(t,y)|^{2} \frac{ x - y }{| x - y |}
\cdot \Re \left[\overline{P_{\leq i-3}u_{j}}\cdot \nabla \mathcal N^{(j)}_{i - 3} \right](t,x)  \, \mathrm{d} x  \mathrm{d} y  \mathrm{d} t \label{ert2}\\
 & + 2^iC\sum_{j\in\mathbb{Z}_N}\sum_{k\in\mathbb{Z}_N}\int_{G^{3i}_{k_{\ast}}}\int\int |P_{m}u_k(t,y)|^{2} \frac{ x - y }{| x - y |}
 \cdot \Re \left[\overline{\mathcal N^{(j)}_{i - 3}} \nabla P_{\leq i-3}u_{j} \right](t,x)  \, \mathrm{d} x  \mathrm{d} y  \mathrm{d} t \label{ert3}\\
& 	- 2^{i+1}C \sum_{j\in\mathbb{Z}_N}\sum_{k\in\mathbb{Z}_N}\int_{G^{3i}_{k_{\ast}}}\int\int \Im \left[\overline{P_{\leq i-3}u_{j}} \mathcal N^{(j)}_{i - 3} \right](t, y)
\frac{ x - y }{| x - y |} \cdot \Im \left[\overline{v}_k \nabla P_{m}u_k \right](t,x)  \, \mathrm{d} x  \mathrm{d} y  \mathrm{d} t\label{ert4}\\
& 	- 2^iC\sum_{j\in\mathbb{Z}_N}\sum_{k\in\mathbb{Z}_N}\int_{G^{3i}_{k_{\ast}}}\int \int |P_{\leq i - 3} u_j(t,y)|^{2} \frac{ x - y }{| x - y |}
\cdot \Re \left[\overline{P_{m} u_k} \nabla P_{m}( {F}_k) \right](t,x)  \, \mathrm{d}x  \mathrm{d} y \mathrm{d} t\label{ert5}\\
& 	+ 2^iC\sum_{j\in\mathbb{Z}_N}\sum_{k\in\mathbb{Z}_N}\int_{G^{3i}_{k_{\ast}}}\int \int |P_{\leq i - 3} u_j(t,y)|^{2} \frac{ x - y }{| x - y |}
\cdot \Re \left[\overline{P_{m} ( {F}_k)} \nabla P_{m} u_k \right](t,x) \, \mathrm{d} x  \mathrm{d} y \mathrm{d} t\label{ert6}\\
& 	- 2^{i+1}C \sum_{j\in\mathbb{Z}_N}\sum_{k\in\mathbb{Z}_N}\int_{G^{3i}_{k_{\ast}}}\int \int \Im \left[\overline{P_{m} u_k} P_{m}( {F}_k) \right](t,y)
 \frac{ x - y }{| x - y |} \cdot\Im \left[\overline{P_{\leq i-3}u_{j}} \nabla P_{\leq i-3}u_{j} \right](t,x) \, \mathrm{d}x \mathrm{ d} y \mathrm{d} t.\label{ert7}
\end{align}
Noting that $\left\|P_{\geq i}\mathbf{u}(t) \right\|_{L_t^{\infty}L_x^2l^2 \left(J\times\mathbb{R}^2\times\mathbb{Z}_N \right)}\leq \eta_0$, and 
following \eqref{estttt1}-\eqref{estqq5}, we have\footnote{Here we use
$ \left\| P_{m}\mathbf{u} \right\|_{L_t^{5/2}L_x^{10}l^2 \left(G_{k_{\ast}}^{3i}\times\mathbb{R}^2\times\mathbb{Z}_N \right)}^2
\lesssim  \left\| P_{m}\mathbf{u}  \right\|_{L_t^{25/12}L_x^{50}l^2 \left(G_{k_{\ast}}^{3i}\times\mathbb{R}^2\times\mathbb{Z}_N \right)}^{5/3}
\left\| P_{m}\mathbf{u}  \right\|_{L_t^{\infty}L_x^{2}l^2 \left(G_{k_{\ast}}^{3i}\times\mathbb{R}^2\times\mathbb{Z}_N \right)}^{1/3}
\lesssim \eta_0^{1/3} \left\| \mathbf{u}  \right\|_{X_{i - 3} \left(J \times \mathbb{R}^2\times\mathbb{Z}_N \right)}^{5/3}$. }
\begin{align}\label{estqqq1}
& |\eqref{ert1}|+|\eqref{ert2}|+|\eqref{ert3}|+|\eqref{ert4}| \\\nonumber
& \lesssim  \eta_0^{1/3}2^{i+m} \left(1+ \left\| \mathbf{u}  \right\|_{X_{i - 3} \left(J \times \mathbb{R}^2\times\mathbb{Z}_N \right)}^{11/3}
+\eta_{0}^{1/3}  \left\| \mathbf{u}  \right\|_{X_{i - 3} \left(J \times \mathbb{R}^2\times\mathbb{Z}_N \right)}^{8/3}
+ \eta_{0}^{1/2}  \left\| \mathbf{u}  \right\|_{X_{i - 3} \left(J \times \mathbb{R}^2\times\mathbb{Z}_N \right)}^{5/2}
+\eta_0^{1/10} \left\| \mathbf{u}  \right\|^2_{X_{i - 3} \left(J \times \mathbb{R}^2\times\mathbb{Z}_N \right)} \right).
\end{align}
Thus, we only need to estimate terms \eqref{ert5}-\eqref{ert7}.

For term \eqref{ert7}, observe that $P_{m}\mathbf{u}=P_{m}P_{\geq i-3}\mathbf{u}$ and $P_{m} \left(|P_{\leq i-6}u_j|^2P_{\leq i-6}u_k \right)=0,\ \forall \, j,k\in\mathbb{Z}_N$, then we can follow the proof of \eqref{f5.40} to deduce that
 \begin{align}\nonumber
 	 & \left|- 2^{i+1}C \sum_{j\in\mathbb{Z}_N}\sum_{k\in\mathbb{Z}_N}\int_{G^{3i}_{k_{\ast}}}\int \int \Im \left[\overline{P_{m} u_k} P_{m} ( {F}_k) \right](t,y)
 \frac{ x - y }{| x - y |}
 \cdot\Im \left[\overline{P_{\leq i-3}u_{j}} \nabla P_{\leq i-3}u_{j} \right](t,x) \, \mathrm{d} x  \mathrm{d} y \mathrm{d} t\right| \\\label{estqqq2}
 & 	\lesssim 2^{i+m} \left(\eta_{0}^{1/3}  \left\| \mathbf{u}   \right\|_{X_{i - 3} \left(J \times \mathbb{R}^2\times\mathbb{Z}_N \right)}^{8/3}
 + \eta_{0}^{1/2}  \left\| \mathbf{u}  \right\|_{X_{i - 3}  \left(J \times \mathbb{R}^2\times\mathbb{Z}_N \right)}^{5/2}
 +\eta_0^{1/10} \left\| \mathbf{u}  \right\|^2_{X_{i - 3} \left(J \times \mathbb{R}^2\times\mathbb{Z}_N \right)} \right).
 \end{align}
The contribution of terms \eqref{ert5} and \eqref{ert6} are also quite similar to \eqref{ert2} and \eqref{ert3}. Just decompose $P_m ({F}_k)$ as
\begin{align*}
	P_m ({F}_k)   =  & \sum_{j\in\mathbb{Z}_N}P_m O \left(P_{\geq i-3}u_j P_{\geq i-3}u_jP_{\geq i-3}u_k \right)
+\sum_{j\in\mathbb{Z}_N}P_m O \left(P_{\geq i-3}u_jP_{\geq i-3}u_jP_{\leq i-3}u_k \right)\\
& +\sum_{j\in\mathbb{Z}_N} (\delta_{kj}-2)
\Bigg[ P_m  \left( |P_{\leq i-3}u_j|^2  P_{\geq i-3}u_k \right)
+P_m \left(P_{\leq i-3}u_j  P_{\leq i-3}u_k\overline{P_{\geq i-3}u_j} \right) \\
 & \quad +P_m \left(\overline{P_{\leq i-3}u_j}  P_{\leq i-3}u_k{P_{\geq i-3}u_j} \right)\Bigg] \\
  :=&  \mathcal{M}^{(k)}_1+\mathcal{M}^{(k)}_2.
\end{align*}
Using Bernstein's inequality and H\"older's inequality, we have
\begin{align}\label{estqqq3}
\aligned
& \left|2^iC\sum_{j\in\mathbb{Z}_N}\sum_{k\in\mathbb{Z}_N}\int_{G^{3i}_{k_{\ast}}}\int \int |P_{\leq i - 3} u_j(t,y)|^{2} \frac{ x - y }{| x - y |}
\cdot \Re \left[\overline{P_{m} u_k} \nabla \mathcal{M}^{(k)}_1  \right](t,x) \, \mathrm{d} x  \mathrm{d} y \mathrm{d} t \right|  \\
& +  \left|2^iC\sum_{j\in\mathbb{Z}_N}\sum_{k\in\mathbb{Z}_N}\int_{G^{3i}_{k_{\ast}}}\int \int |P_{\leq i - 3} u_j(t,y)|^{2} \frac{ x - y }{| x - y |}
\cdot \Re \left[\overline{\mathcal{M}^{(k)}_1} \nabla P_{m} u_k \right](t,x) \, \mathrm{d} x  \mathrm{d} y \mathrm{d} t\right| \\
 & 	\lesssim 2^{i+j} \left\|\mathbf{u} \right\|^3_{L_t^{\infty}L_x^2l^2 \left(G_{k_{\ast}}^{3i}\times\mathbb{R}^2 \times\mathbb{Z}_N \right)}
\left\|P_{\geq i-3}\mathbf{u} \right\|^3_{L_t^{3}L_x^6l^2 \left(G_{k_{\ast}}^{3i}\times\mathbb{R}^2 \times\mathbb{Z}_N \right)}
\lesssim \eta_{0}^{1/2}  \left\| \mathbf{u}  \right\|_{X_{i - 3} \left(J \times \mathbb{R}^2\times\mathbb{Z}_N \right)}^{5/2}.
\endaligned
\end{align}
Setting
\begin{align*}
	\mathcal{M}^{(k)}_3=\sum_{j\in\mathbb{Z}_N}\left[ |P_{\leq i-3}u_j|^2  P_{m}u_k
+P_{\leq i-3}u_j  P_{\leq i-3}u_k\overline{P_{m}u_j}+\overline{P_{\leq i-3}u_j}  P_{\leq i-3}u_k{P_{m}u_j}\right],
\end{align*}
by the Coifman-Meyer multiplier theorem, we have
\begin{align*}
	\left\|\mathcal{M}_2^{(k)}-\mathcal{M}_3^{(k)} \right\|_{L_x^{6/5}}
\lesssim  \left\|\mathbf{u} \right\|_{L_x^2l^2} \left\|P_{m}\mathbf{u} \right\|_{L_x^6l^2} \left\|\nabla P_{\leq i-3}\mathbf{u}  \right\|_{L_x^6l^2}.
\end{align*}
Thus, by Bernstein's inequality and \eqref{xnorm1}, we have
\begin{align}\label{estqqq4}
\aligned
	& \left|2^iC\sum_{j\in\mathbb{Z}_N}\sum_{k\in\mathbb{Z}_N}\int_{G^{3i}_{k_{\ast}}}\int \int |P_{\leq i - 3} u_j(t,y)|^{2} \frac{ x - y }{| x - y |}
\cdot \Re \left[\overline{P_{m} u_k} \nabla  \left(\mathcal{M}^{(k)}_2-\mathcal{M}^{(k)}_3 \right)  \right](t,x) \, \mathrm{d} x  \mathrm{d} y \mathrm{d} t \right|  \\
& 	+  \left|2^iC\sum_{j\in\mathbb{Z}_N}\sum_{k\in\mathbb{Z}_N}\int_{G^{3i}_{k_{\ast}}}\int \int |P_{\leq i - 3} u_j(t,y)|^{2} \frac{ x - y }{| x - y |}
\cdot \Re \left[\overline{ \left(\mathcal{M}^{(k)}_2-\mathcal{M}^{(k)}_3 \right)} \nabla P_{m} u_k \right](t,x) \, \mathrm{d} x  \mathrm{d} y \mathrm{d} t\right| \\
 	& 
    \lesssim 2^{i+j} \left\|\mathbf{u} \right\|^3_{L_t^{\infty}L_x^2l^2 \left(G_{k_{\ast}}^{3i}\times\mathbb{R}^2 \times\mathbb{Z}_N \right)}
\left\|P_{\geq i-3}\mathbf{u} \right\|^3_{L_t^{3}L_x^6l^2 \left(G_{k_{\ast}}^{3i}\times\mathbb{R}^2 \times\mathbb{Z}_N \right)}
\lesssim \eta_{0}^{1/2} \left\| \mathbf{u}  \right\|_{X_{i - 3} \left(J \times \mathbb{R}^2\times\mathbb{Z}_N \right)}^{5/2}.
\endaligned
\end{align}
Finally, by the product rule, we get
\begin{align*}
	\left|\sum_{k\in\mathbb{Z}_N}\Re \left[\overline{P_{m} u_k} \nabla \mathcal{M}^{(k)}_3 -\overline{\mathcal{M}^{(k)}_3}\nabla P_{m} u_k \right] \right|
\lesssim  \left|\nabla P_{\leq i-3}\mathbf{u} \right| \cdot  \left|P_{\leq i-3}\mathbf{u} \right| \cdot \left|P_{m}\mathbf{u} \right|.
\end{align*}
Using H\"older's inequality and \eqref{xnorm1}, we have 
\begin{align}\label{estqqq5}
\aligned
	& \left|-2^iC\sum_{j\in\mathbb{Z}_N}\sum_{k\in\mathbb{Z}_N}\int_{G^{3i}_{k_{\ast}}}\int \int |P_{\leq i - 3} u_j(t,y)|^{2} \frac{ x - y }{| x - y |}
\cdot \Re \left[\overline{P_{m} u_k} \nabla \mathcal{M}^{(k)}_3  \right](t,x) \,\mathrm{d} x  \mathrm{d} y \mathrm{d} t\right|
	 \\
& + \left| 2^iC\sum_{j\in\mathbb{Z}_N}\sum_{k\in\mathbb{Z}_N}\int_{G^{3i}_{k_{\ast}}}\int \int |P_{\leq i - 3} u_j(t,y)|^{2} \frac{ x - y }{| x - y |}
\cdot \Re \left[\overline{\mathcal{M}^{(k)}_3} \nabla P_{m} u_k \right](t,x) \, \mathrm{d} x  \mathrm{d} y \mathrm{d} t\right| \\
	 \lesssim& 2^{i+j} \left\|\mathbf{u} \right\|^3_{L_t^{\infty}L_x^2l^2 \left(G_{k_{\ast}}^{3i}\times\mathbb{R}^2 \times\mathbb{Z}_N \right)}
\left\|P_{\geq i-3}\mathbf{u} \right\|^3_{L_t^{3}L_x^6l^2 \left(G_{k_{\ast}}^{3i}\times\mathbb{R}^2 \times\mathbb{Z}_N \right)}
\lesssim \eta_{0}^{1/2}  \left\| \mathbf{u}  \right\|_{X_{i - 3} \left(J \times \mathbb{R}^2\times\mathbb{Z}_N \right)}^{5/2}.
\endaligned
\end{align}
Inserting \eqref{estqqq1}-\eqref{estqqq5} into \eqref{ert1}-\eqref{ert7}, then summing over $m\geq i$, we obtain 
\begin{align}\label{longtime2}
& \sum_{j\in\mathbb{Z}_N}\sum_{k\in\mathbb{Z}_N} \left\| P_m u_{k}P_{\leq i-3}u_{j} \right\|_{L_{t,x}^{2} \left(G^{3i}_{k_{\ast}} \times \mathbb{R}^{2} \right)}^{2} \\\nonumber
& \lesssim\eta_0^{1/3} \left(1+ \left\| \mathbf{u} \right\|_{X_{i - 3} \left(J \times \mathbb{R}^2\times\mathbb{Z}_N \right)}^{11/3}
+\eta_{0}^{1/3} \left\| \mathbf{u} \right\|_{X_{i - 3} \left(J \times \mathbb{R}^2\times\mathbb{Z}_N \right)}^{8/3}
	+ \eta_{0}^{1/2}  \left\| \mathbf{u}  \right\|_{X_{i - 3} \left(J \times \mathbb{R}^2\times\mathbb{Z}_N \right)}^{5/2}
+\eta_0^{1/10} \left\| \mathbf{u}  \right\|^2_{X_{i - 3} \left(J \times \mathbb{R}^2\times\mathbb{Z}_N \right)}\right).
\end{align}
Using a standard bootstrap argument with \eqref{longtime0}, \eqref{longtime1} and \eqref{longtime2}, we conclude that 
\begin{align*} 
	\left\| \mathbf{u} \right\|_{X \left(J \times \mathbb{R}^{2}\times\mathbb{Z}_N \right)} \lesssim 1.
\end{align*}
This completes the proof.\end{proof}
\begin{remark}
Plugging \eqref{f5.10} into \eqref{xnorm1}, we also have
\begin{align}\label{ssx}
\left\|\nabla P_{\leq k_0}\mathbf{u} \right\|_{L_t^{3}L_x^{6}l^2 \left(J\times\mathbb{R}^2\times\mathbb{Z}_N \right)}
&\lesssim 2^{k_0}.
\end{align}
\end{remark}

\subsubsection{Long time Strichartz estimate - Type II}
The first type of long time Strichartz estimate \eqref{f5.10} indicates that the high frequency truncation of the solution $\mathbf{u}$, $P_{\geq {k_0}}\mathbf{u}$,  actually behaves like the free Schr\"odinger system, thus the Strichartz norms (as well as atomic norms) are actually bounded by the $L_x^2l^2$ norm of the initial data thanks to the Strichartz estimates (Lemma \ref{property}). However, since we are working with the solutions that are close to the orbit of the ground state $\mathbf{Q}$, we can further prove that the Strichartz norms (as well as atomic norms)  of $P_{\geq {k_0}}\mathbf{u}$ are actually dominated by the $L_x^2l^2$ norm of the high frequency component of $\mathbf{Q}$ (which  decays rapidly) plus the $L_x^2l^2$ norm of reminder $\boldsymbol{\epsilon}$ in the sense of the integral average.

\begin{theorem}[Long time Strichartz estimate-II]\label{t6.3}
Under the assumption of Theorem \ref{t5.9}, we have
\begin{align}\label{f5.66}
\aligned
& \left\| P_{\geq k_0} \mathbf{u}  \right\|_{U_{\Delta}^{2} \left(J,L_x^2l^2 \right)} +   \left\|  \left|P_{\geq k_0} \mathbf{u} \right|\cdot
\left|P_{\leq k_0 - 3} \mathbf{u} \right|  \right\|_{L_{t,x}^{2}l^2 \left(J \times \mathbb{R}^{2}\times\mathbb{Z}_N \right)}   \\
\lesssim&  \left(\frac{1}{T} \int_{J}   \left\| \boldsymbol{\epsilon}(t)  \right\|_{L^2_{x} l^2}^{2} \lambda(t)^{-2} \, \mathrm{d} t \right)^{1/2} + T^{-10}.
\endaligned
\end{align}
\end{theorem}
\begin{proof}
	Using \eqref{f5.10} and following the proof of Theorem \ref{t5.9}, one can prove that there exists $\delta>0$ so that
	\begin{align}
	& \left\|P_{\geq i+3}\mathbf{u} \right\|_{U^2_{\Delta} \left(G_{k_{\ast}}^{3(i+3)},L_x^2l^2 \right)} \notag \\
\leq& C\inf_{t_0\in G_{k_{\ast}}^{3(i+3)}}  \left\|P_{\geq i+3}\mathbf{u}(t_0) \right\|_{L_x^2l^2}
+C\eta_0^{\delta} \left\|P_{\geq i}\mathbf{u} \right\|_{U^2_{\Delta} \left(G_{k_{\ast}}^{3(i+3)},L_x^2l^2 \right)}\label{bootstrap2}\\
		\leq& C\inf_{t_0\in G_{k_{\ast}}^{3(i+3)}}  \left\|P_{\geq i+3}\mathbf{u}(t_0) \right\|_{L_x^2l^2}
+C\eta_0^{\delta} \left(\sum_{G_l^{3k}\subset G_{k_{\ast}}^{3(i+3)}} \left\|P_{\geq i}\mathbf{u} \right\|^2_{U^2_{\Delta} \left(G_{k_{\ast}}^{3i},L_x^2l^2 \right)} \right)^{1/2}, \ \forall \, G_{k_{\ast}}^{3(i+3)}\subset J.\label{bootstrap3}
	\end{align}
	On the other hand, using the triangle inequality and the integral mean value theorem, for $i\geq 2$,
	\begin{align}\label{bst1}
\aligned
		& \  \inf_{t_0\in G_{k_{\ast}}^{3(i+3)}} \left\|P_{\geq i+3}\mathbf{u}(t_0) \right\|_{L_x^2l^2}\\
&\leq \inf_{t_1\in G_{k_{\ast}}^{3(i+3)}}  \left\|P_{\geq i+3}\boldsymbol{\epsilon}(t_1) \right\|_{L_x^2l^2}
+\sup_{t_2\in G_{k_{\ast}}^{3(i+3)}} \left\|P_{\geq i+3} \left(e^{-ix\cdot\xi(t_2)}\frac{1}{\lambda(t_2)} Q \left(\frac{x-x(t_2)}{\lambda(t_2)} \right) \right) \right\|_{L_x^2l^2} \\
		&\lesssim \left(\frac{2^{3k_0-3(i+3)}}{T} \int_{G_{k_{\ast}}^{3(i+3)}}  \left\| \boldsymbol{\epsilon}(t)  \right\|_{L^2_{x} l^2}^{2} \lambda(t)^{-2}  \, \mathrm{d} t \right)^{1/2}+ 2^{-100i}.
\endaligned
	\end{align}
Let $i=\left[\frac{k_0}{2}\right]$ in \eqref{bst1} and  iterate $\left[\frac{k_0}{6}\right]$ times while using
\eqref{f5.10}, \eqref{bootstrap3},
$\#\left\{G_l^{3i} : G_l^{3i}\subset G_{k_{\ast}}^{3(i+3)} \right\}=2^9$ and the fact that $(C\eta_0)^{\frac{k_0}{6}} \ll T^{-10} \mbox{ if } \eta_0 \mbox{  sufficiently small imply that }$
\begin{align}
\aligned
\left\| P_{\geq k_0} \mathbf{u} \right\|_{U_{\Delta}^{2} \left(J,L_x^2l^2 \right)}
	&\lesssim \left(\frac{1}{T} \int_{J}   \left\| \boldsymbol{\epsilon}(t)  \right\|_{L^2_{x} l^2}^{2} \lambda(t)^{-2} \,  \mathrm{ d}  t \right)^{1/2}
+ T^{-10}+(C\eta_0)^{\frac{k_0}{6}} \\
&\lesssim  \left(\frac{1}{T} \int_{J}   \left\| \boldsymbol{\epsilon}(t)  \right\|_{L^2_{x} l^2}^{2} \lambda(t)^{-2} \,  \mathrm{d} t \right)^{1/2} + {T^{-10}}.\label{bos}
\endaligned
\end{align}
Finally, following the proof of \eqref{f5.48}, \eqref{bos} directly implies
\begin{align*}
	\left\|   \left|P_{\geq k_0} \mathbf{u} \right|\cdot \left|P_{\leq k_0 - 3} \mathbf{u} \right|  \right\|_{L_{t,x}^{2}l^2 \left(J \times \mathbb{R}^{2}\times\mathbb{Z}_N \right)}
\lesssim  \left(\frac{1}{T} \int_{J}   \left\| \boldsymbol{\epsilon}(t)  \right\|_{L^2_{x} l^2}^{2} \lambda(t)^{-2} \, \mathrm{ d} t \right)^{1/2} + {T^{-10}}.
\end{align*}
This together with \eqref{bos} imply \eqref{f5.66}. 
\end{proof}
\begin{remark}
    We can also prove that  $\forall \, k_0\leq m\leq k_0+9$,
\begin{align}\label{ssz}
\left\|  \left|P_{\geq m} \mathbf{u} \right|\cdot \left|P_{\leq m - 3} \mathbf{u} \right|  \right\|_{L_{t,x}^{2}l^2 \left(J \times \mathbb{R}^{2}\times\mathbb{Z}_N \right)}
\lesssim  \left(\frac{1}{T} \int_{J}  \left\| \boldsymbol{\epsilon}(t) \right\|_{L^2_{x} l^2}^{2} \lambda(t)^{-2}  \, \mathrm{d} t \right)^{1/2} + {T^{-10}}.
\end{align}
\end{remark}

\subsection{ The estimate of  truncated energy $E(P_{\leq k_0+9}\mathbf{u})$.} Now we will use the long time Strichartz estimates to derive an increment estimate of $E(P_{\leq k_0+9}\mathbf{u})$.
\begin{theorem}[Increment of $E(P_{\leq k_0+9}\mathbf{u})$]\label{t4.1}
Under the assumptions of Theorem \ref{t5.9}, we have
\begin{equation}\label{f6.1}
\sup_{t \in J} E \left(P_{\leq k _0+ 9} \mathbf{u}(t) \right)
 \lesssim \frac{2^{2k_0}}{T} \int_{J}  \left\| \boldsymbol{\epsilon}(t)  \right\|_{L^2_{x} l^2}^{2} \lambda(t)^{-2}  \, \mathrm{d} t
 + \sup_{t \in J} \frac{|\xi(t)|^{2}}{\lambda(t)^{2}} + 2^{2k_0} T^{-10}.
\end{equation}
\end{theorem}
\begin{proof}
Recalling that Lemma \ref{l2.2} guarantees that $ \left\| \boldsymbol{\epsilon}(t) \right\|_{L^2_{x} l^2}$ is continuous as a function of $t$, so
the integral mean value theorem implies that under the conditions of Theorem \ref{t5.9}, there exists $t_{0}\in [a, b]=J$ such that
\begin{equation*} 
\left\| \boldsymbol{\epsilon}(t_{0}) \right\|_{L^2_{x} l^2}^{2}
= \frac{1}{T} \int_{a}^{b}  \left\| \boldsymbol{\epsilon}(t)  \right\|_{L^2_{x} l^2}^{2} \lambda(t)^{-2}  \, \mathrm{d} t.
\end{equation*}
Using the fact that $Q$ is real-valued, smooth, and has rapidly decaying derivatives, arguing like \eqref{f4.58}-\eqref{energybou}, we have for any $t\in J$
\begin{align}\label{f6.3}
& E \left(P_{\leq k _0+ 9} \mathbf{u}(t) \right)
\\
\nonumber
 & = \frac{N}{2} \frac{|\xi(t)|^{2}}{\lambda(t)^{2}} \| Q \|_{L^2_{x}}^{2} + \frac{1}{2 \lambda(t)^{2}} \left\| \boldsymbol{\epsilon}(t) \right\|_{L^2_{x} l^2}^{2}
  - \frac{|\xi(t)|^{2}}{2 \lambda(t)^{2}} \left\| \boldsymbol{\epsilon}(t)  \right\|_{L^2_{x} l^2}^{2}
  + O \left(2^{2k_0}  \left\| \boldsymbol{\epsilon}(t) \right\|_{L^2_{x} l^2}^{2}  \right)
   + O \left(2^{2k_0} T^{-10} \right).
\end{align}
Thus \eqref{f6.1} holds at $t_{0}$.

Next, we compute the change of energy. By direct calculation,
\begin{align*} 
\frac{ \mathrm{d} }{ \mathrm{d} t} E \left(P_{\leq k _0+ 9} \mathbf{u}(t) \right)
 & = - \left\langle  \frac{ \mathrm{d} }{ \mathrm{d} t}P_{\leq k _0+ 9}\mathbf{u}(t), \Delta P_{\leq k _0+ 9} \mathbf{u}(t) \right\rangle_{L^2_{x} l^2}
 - \left\langle \frac{ \mathrm{d} }{ \mathrm{d} t}P_{\leq k _0+ 9}\mathbf{u}(t), \mathbf{ {F}} \left(P_{\leq k_0+9}\mathbf{u} \right)(t)  \right\rangle_{L^2_{x} l^2} \\
 & = \left\langle i \Delta P_{\leq k _0+ 9} \mathbf{u}(t) + iP_{\leq k_0+9} \mathbf{ {F}} \left(\mathbf{u} \right)(t),
P_{\leq k_0+9} \mathbf{ {F}} \left(\mathbf{u} \right)(t) - \mathbf{ {F}} \left(P_{\leq k_0+9}\mathbf{u} \right)(t) \right\rangle_{L^2_x l^2}.
\end{align*}
First, we  estimate
\begin{equation}\label{f6.5}
\int_{t_{0}}^{t'}  \left\langle i \Delta P_{\leq k _0+ 9} \mathbf{u}(t),
P_{\leq k_0+9} \mathbf{ {F}} \left(\mathbf{u} \right)(t) - \mathbf{ {F}} \left(P_{\leq k_0+9}\mathbf{u} \right)(t) \right\rangle_{L^2_x l^2} \, \mathrm{d} t
\end{equation}
for any $t' \in J$.
Since $P_{\leq k_0+9} \mathbf{ {F}} \left(P_{\leq k_0+6}\mathbf{u} \right)(t) - \mathbf{ {F}} \left(P_{\leq k_0+9}P_{\leq k_0+6}\mathbf{u} \right)(t)=0$, we can further decompose \eqref{f6.5} as
\begin{align}
&\int_{t_{0}}^{t'}  \left\langle i \Delta P_{\leq k _0+ 9} \mathbf{u}(t),
P_{\leq k_0+9} \mathbf{ {F}} \left(\mathbf{u} \right)(t) - \mathbf{ {F}} \left(P_{\leq k_0+9}\mathbf{u} \right)(t) \right\rangle_{L^2_x l^2}  \mathrm{d} t \notag\\
= &\int_{t_{0}}^{t'} \Big\langle i \Delta P_{\leq k_0+ 9} \mathbf{u}(t),
\sum_{k\in\mathbb{Z}_N}O \left[P_{\leq k_0+9} \left(P_{\geq k_0+6}u_j\overline{P_{\leq k_0}u_k} P_{\leq k_0}u_k \right)
-P_{k_0+6\leq\cdot \leq k_0+9}u_j\overline{P_{\leq k_0}u_k} P_{\leq k_0}u_k \right] \Big\rangle_{L^2_x l^2}  \mathrm{d} t\label{estqqqq1}\\
&+\int_{t_{0}}^{t'}
\Big\langle i \Delta P_{\leq k_0} \mathbf{u}(t),
\sum_{k\in\mathbb{Z}_N}O \left[P_{\leq k_0+9} \left(P_{\geq k_0+6}u_j\overline{P_{\geq k_0}u_k} P_{\leq k_0}u_k \right)
-P_{k_0+6\leq\cdot \leq k_0+9}u_j\overline{P_{\geq k_0}u_k} P_{\leq k_0} u_k \right] \Big\rangle_{L^2_x l^2}  \mathrm{d} t\label{estqqqq2}\\
&+\int_{t_{0}}^{t'} \Big\langle i \Delta P_{k_0\leq \cdot \leq k_0+9} \mathbf{u}(t),
\sum_{k\in\mathbb{Z}_N}O \left[P_{\leq k_0+9} \left(P_{\geq k_0+6}u_j\overline{P_{\geq k_0}u_k} P_{\leq k_0}u_k \right)-P_{k_0+6\leq\cdot\leq k_0+9}u_j\overline{P_{\geq k_0}u_k} P_{\leq k_0}u_k \right] \Big\rangle_{L^2_x l^2}  \mathrm{d} t\label{estqqqq3}\\
&+\int_{t_{0}}^{t'}  \Bigg\langle i \Delta P_{\leq k_0+9} \mathbf{u}(t),
\sum_{k\in\mathbb{Z}_N}O \big[P_{\leq k_0+9} \left(P_{\geq k_0+6}u_j\overline{P_{\geq k_0}u_k} P_{\geq k_0}u_k \right)\notag\\
&\hspace{55ex}-P_{k_0+6\leq\cdot\leq k_0+9}u_j\overline{P_{k_0\leq\cdot\leq k_0+9}u_k} P_{k_0\leq\cdot\leq k_0+9}u_k \big] \Bigg\rangle_{L^2_x l^2}  \mathrm{d} t.\label{estqqqq4}
\end{align}
We will estimate terms \eqref{estqqqq1}-\eqref{estqqqq4} separately. 

\underline{\emph{\textbf {Contribution of  \eqref{estqqqq1}}.} } Observe that the Fourier support of
\begin{align*}
	\sum_{k\in\mathbb{Z}_N}O
\left[ P_{\leq k_0+9} \left(P_{\geq k_0+6}u_j\overline{P_{\leq k_0}u_k} P_{\leq k_0}u_k \right)-P_{k_0+6\leq\cdot\leq k_0+9}u_j\overline{P_{\leq k_0}u_k} P_{\leq k_0}u_k \right]
\end{align*}
is supported in $|\xi|\geq 2^{k_0+3}$, so by \eqref{f5.66}, \eqref{ssz}, \eqref{ssx} and the differential identity $\Delta (fg)=\Delta f g+2\nabla f\cdot\nabla g+f\Delta g$, we have
\begin{align}\label{estqz1}
 \eqref{estqqqq1} & =  \int_{t_{0}}^{t'} \Bigg\langle i \Delta P_{k_0+3\leq\cdot\leq k_0+ 9} \mathbf{u}(t),
 \sum_{k\in\mathbb{Z}_N}O \big[P_{\leq k_0+9} \left(P_{\geq k_0+6}u_j\overline{P_{\leq k_0}u_k} P_{\leq k_0}u_k \right)\notag\\
 &\hspace{55ex}-P_{k_0+6\leq\cdot \leq k_0+9} u_j\overline{P_{\leq k_0}u_k} P_{\leq k_0}u_k \big] \Bigg\rangle_{L_x^2l^2} \,  \mathrm{d} t\notag\\
&\lesssim 2^{2k_0} \left\|  \left|P_{k_0+3\leq\cdot\leq k_0+9} \mathbf{u} \right| \left|P_{\leq k_0} \mathbf{u} \right|  \right\|_{L_{t,x}^{2}}
\left( \left\|  \left|P_{k_0+6\leq \cdot\leq k_0+12} \mathbf{u} \right| \left|P_{\leq k_0}\mathbf{u} \right|  \right\|_{L_{t,x}^{2}}
+ \left\|  \left|P_{k_0+6\leq \cdot\leq k_0+9} \mathbf{u} \right| \left|P_{\leq k_0}\mathbf{u} \right|  \right\|_{L_{t,x}^{2}} \right)\notag \\
&\ \  +2^{k_0}  \left\|\nabla P_{\leq k_0}\mathbf{u} \right\|_{L_t^3L_x^6l^2}
\left\| P_{\geq k_0}\mathbf{u} \right\|^2_{L_t^3L_x^6l^2} \left\|\mathbf{u} \right\|_{L_t^{\infty}L_x^2l^2}
+ \left\|\Delta P_{\leq k_0}\mathbf{u} \right\|_{L_t^3L_x^6l^2} \left\| P_{\geq k_0}\mathbf{u} \right\|^2_{L_t^3 L_x^6l^2}
\left\|\mathbf{u} \right\|_{L_t^{\infty}L_x^2l^2}\notag \\
&\lesssim \frac{2^{2k_0}}{T} \int_{J}  \left\| \boldsymbol{\epsilon}(t)  \right\|_{L^2_{x} l^2}^{2} \lambda(t)^{-2}  \, \mathrm{d} t + 2^{2k_0} T^{-10}.
\end{align}

\underline{\emph{\textbf {Contribution of  \eqref{estqqqq2}}.} }
Using \eqref{ssx}, we have 
\begin{align}\label{estqz2}
\eqref{estqqqq2}
&\lesssim
2^{k_0} \left\|\nabla P_{\leq k_0}\mathbf{u} \right\|_{L_t^3L_x^6l^2} \left\| P_{\geq k_0}\mathbf{u} \right\|^2_{L_t^3L_x^6l^2}
\left\|\mathbf{u} \right\|_{L_t^{\infty}L_x^2l^2} \notag
\\
& \lesssim \frac{2^{2k_0}}{T} \int_{J}  \left\| \boldsymbol{\epsilon}(t)  \right\|_{L^2_{x} l^2}^{2} \lambda(t)^{-2} \, \mathrm{d} t + 2^{2k_0} T^{-10}.
\end{align}

\underline{\emph{\textbf {Contribution of \eqref{estqqqq3} and \eqref{estqqqq4}}.}}
For these terms, we notice that all of them have three high-frequency factors, thus
\begin{align}\label{estqz3}
|\eqref{estqqqq3}|+|\eqref{estqqqq4}|
\lesssim 2^{2k_0}  \left\| P_{\geq k_0}\mathbf{u} \right\|^3_{L_t^3L_x^6l^2} \left\|\mathbf{u} \right\|_{L_t^{\infty}L_x^2l^2}
	\lesssim \frac{2^{2k_0}}{T} \int_{J}  \left\| \boldsymbol{\epsilon}(t)  \right\|_{L^2_{x} l^2}^{2} \lambda(t)^{-2} \, \mathrm{d} t + 2^{2k_0} T^{-10}.
\end{align}
Plugging \eqref{estqz1}-\eqref{estqz3} into \eqref{f6.5}, we obtain 
\begin{align}\label{f6.13}
& \left|	\int_{t_{0}}^{t'} \left\langle i \Delta P_{\leq k _0+ 9} \mathbf{u}(t),
P_{\leq k_0+9} \mathbf{ {F}} \left(\mathbf{u} \right)(t) - \mathbf{ {F}} \left(P_{\leq k_0+9}\mathbf{u} \right)(t) \right\rangle_{L^2_x l^2} \, \mathrm{d} t \right|\notag 	
\\
 \lesssim& \frac{2^{2k_0}}{T} \int_{J}  \left\| \boldsymbol{\epsilon}(t)  \right\|_{L^2_{x} l^2}^{2} \lambda(t)^{-2} \, \mathrm{d} t + 2^{2k_0} T^{-10}.
\end{align}
 The estimate of
\begin{align}\label{f6.14}
\begin{aligned}
	& \int_{t_{0}}^{t'}  \left\langle iP_{\leq k_0+9} \mathbf{ {F}} \left(\mathbf{u} \right)(t) , P_{\leq k_0+9} \mathbf{ {F}} \left(\mathbf{u} \right)(t) - \mathbf{ {F}} \left(P_{\leq k_0+9}\mathbf{u} \right)(t) \right\rangle_{L^2_x l^2} \, \mathrm{d} t
\\
= & \int_{t_{0}}^{t'} \langle iP_{\leq k_0+9} \mathbf{ {F}} \left(\mathbf{u} \right)(t), - \mathbf{ {F}} \left(P_{\leq k_0+9}\mathbf{u} \right)(t)\rangle_{L^2_x l^2} \, \mathrm{d} t
\end{aligned}
\end{align}
is similar to \eqref{f6.5}.
Using Littlewood-Paley decomposition and then analyzing the interaction between these terms, we can similarly prove that
\begin{equation}\label{f6.15}
	\eqref{f6.14}\lesssim \frac{2^{2k_0}}{T} \int_{J}  \left\| \boldsymbol{\epsilon}(t)  \right\|_{L^2_{x} l^2}^{2} \lambda(t)^{-2} \, \mathrm{ d} t + 2^{2k_0} T^{-10}.
\end{equation}
Finally, integrating both sides of \eqref{f6.1} and using \eqref{f6.13} as while as \eqref{f6.15}, we complete the proof of Theorem \ref{t4.1}.
\end{proof}
\subsection{An uniform estimate for $\|\boldsymbol{\epsilon}\|_{L^2_x l^2}$ on the interval $J$}
In this final subsection,  we will prove a  uniform estimate for $\|\boldsymbol{\epsilon}(t)\|_{L^2_x l^2}$ on the entire interval $J$,  similar to the one in the first subsection (Theorem \ref{bouenergy}).
\begin{theorem}\label{t6.12}
Under the assumptions of Theorem \ref{t5.9}, for all $j\in\mathbb{Z}_N$, we have
\begin{align}\label{f6.16}
\begin{aligned}
& \sup_{t \in J}  \left\|\nabla P_{\leq k _0+ 9} \left(e^{-i \gamma_j(t)} e^{-ix \cdot \frac{\xi(t)}{\lambda(t)}} \frac{1}{\lambda(t)} \epsilon_j \left(t, \frac{x - x(t)}{\lambda(t)} \right) \right)  \right\|_{L_x^2}^{2}
\\
 \lesssim& \frac{2^{2k_0}}{T} \int_{J} \left\| \boldsymbol{\epsilon}(t)  \right\|_{L^2_{x} l^2}^{2} \lambda(t)^{-2} \, \mathrm{d} t
+ \sup_{t \in J} \frac{|\xi(t)|^{2}}{\lambda(t)^{2}} + 2^{2k_0} T^{-10}
\end{aligned}
\end{align}
and
\begin{equation}\label{f6.17}
\sup_{t \in J} \left\| \boldsymbol{\epsilon}(t)  \right\|_{L^2_{x} l^2}^{2}
 \lesssim \frac{2^{2k_0} T^{1/50}}{\eta_{1}^{2} T} \int_{J} \left\| \boldsymbol{\epsilon}(t) \right\|_{L^2_{x} l^2}^{2} \lambda(t)^{-2} \, \mathrm{d} t
 + \frac{T^{1/50}}{\eta_{1}^{2}} \sup_{t \in J} \frac{|\xi(t)|^{2}}{\lambda(t)^{2}} + 2^{2k_0} \frac{T^{1/50}}{\eta_{1}^{2}} T^{-10}.
\end{equation}
\end{theorem}

\begin{proof}
Let $\tilde{\epsilon}_j=e^{-i \gamma_j(t)} e^{-ix \cdot \frac{\xi(t)}{\lambda(t)}} \frac{1}{\lambda(t)} \epsilon_j \left(t, \frac{x - x(t)}{\lambda(t)} \right)=\tilde{v}_j+\tilde{w}_j, \forall \, j\in\mathbb{Z}_N$. Similar to \eqref{f6.3}, we can expand the energy as:
\begin{align}\label{f6.18}
& E \left(P_{\leq k _0+ 9}  \mathbf{u}(t) \right) \notag \\
  = & \frac{N}{2} \frac{|\xi(t)|^{2}}{\lambda(t)^{2}} \| Q \|_{L^2_{x} l^2}^{2}
 + \frac{1}{2 \lambda(t)^{2}} \left\| \boldsymbol{\epsilon}(t) \right\|_{L^2_{x} l^2}^{2}
 - \frac{|\xi(t)|^{2}}{2 \lambda(t)^{2}} \left\| \boldsymbol{\epsilon}(t)  \right\|_{L^2_{x} l^2}^{2}
  + \mathcal P_{2}(t) + \mathcal P_{3}(t) + \mathcal P_{4}(t) + O \left(T^{-20} \right),
\end{align}
where
\begin{align*}
	\mathcal P_{2}(t)  = & \frac{1}{2} \left\| \nabla \mathbf{\tilde{\epsilon}}  \right\|_{L^2_{x}l^2}^{2}
- \frac{N+3}{2\lambda(t)^{4}}\sum_{j\in\mathbb{Z}_N} \int Q \left(\frac{x - x(t)}{\lambda(t)} \right)^{2}  \left( \left( P_{\leq k _0+ 9} \tilde{v}_j  \right)
\left(t,\frac{x - x(t)}{\lambda(t)} \right) \right)^{2} \, \mathrm{d} x \\
& -\frac{N+1}{2\lambda(t)^{4}}\sum_{j\in\mathbb{Z}_N} \int Q \left(\frac{x - x(t)}{\lambda(t)} \right)^{2}
\left( \left( P_{\leq k _0+ 9} \tilde{w}_j \right) \left(t,\frac{x - x(t)}{\lambda(t)} \right) \right)^{2} \mathrm{d}x\\
	&  - \frac{2}{ \lambda(t)^{4}} \sum_{j\in\mathbb{Z}_N}\sum_{\substack{k\in\mathbb{Z}_N\\k\neq j}} \int Q \left(\frac{x - x(t)}{\lambda(t)} \right)^{2}
\left( P_{\leq k _0+ 9}
 \tilde{v}_k \right)
  \left(t,\frac{x - x(t)}{\lambda(t)} \right) \left( P_{\leq k _0+ 9} \tilde{v}_j \right)   \left(t,\frac{x - x(t)}{\lambda(t)} \right ) \, \mathrm{ d} x,\\
\mathcal{P}_3(t)
= & -\frac{1}{\lambda(t)}\sum_{j\in\mathbb{Z}_N}\int Q \left(\frac{x - x(t)}{\lambda(t)} \right)  \left| \left( P_{\leq k _0+ 9} \tilde{\epsilon}_j \right)
\left(t,\frac{x - x(t)}{\lambda(t)} \right) \right|^2
\left( P_{\leq k _0+ 9} \tilde{v}_j  \right) \left(t,\frac{x - x(t)}{\lambda(t)} \right) \, \mathrm{ d} x \\
& 	-\frac{1}{2\lambda(t)}\sum_{j\in\mathbb{Z}_N}\sum_{\substack{k\in\mathbb{Z}_N\\k\neq j}}\int Q \left(\frac{x - x(t)}{\lambda(t)} \right)
\left| \left( P_{\leq k _0+ 9} \tilde{\epsilon}_j \right)  \left(t,\frac{x - x(t)}{\lambda(t)} \right) \right|^2
\left( P_{\leq k _0+ 9} \tilde{v}_k \right)  \left(t,\frac{x - x(t)}{\lambda(t)} \right) \, \mathrm{d} x \\
& 	-\frac{1}{2\lambda(t)}\sum_{j\in\mathbb{Z}_N}\sum_{\substack{k\in\mathbb{Z}_N\\k\neq j}}\int Q \left(\frac{x - x(t)}{\lambda(t)} \right)
\left| \left( P_{\leq k _0+ 9} \tilde{\epsilon}_k  \right) \left(t,\frac{x - x(t)}{\lambda(t)} \right) \right|^2
\left( P_{\leq k _0+ 9} \tilde{v}_j \right) \left(t,\frac{x - x(t)}{\lambda(t)} \right) \, \mathrm{d} x,\\
\mathcal{P}_4(t)
= & -\frac{1}{4}\sum_{j\in\mathbb{Z}_N}\int   \left| \left( P_{\leq k _0+ 9} \tilde{\epsilon}_j \right)  \left(t,\frac{x - x(t)}{\lambda(t)} \right) \right|^4  \, \mathrm{d} x \\
	& -\frac{1}{4}\sum_{j\in\mathbb{Z}_N}\sum_{\substack{k\in\mathbb{Z}_N\\k\neq j}}\int  \left| \left( P_{\leq k _0+ 9} \tilde{\epsilon}_k \right) \left(t,\frac{x - x(t)}{\lambda(t)} \right) \right|^2
\left| \left( P_{\leq k _0+ 9} \tilde{\epsilon}_j \right)  \left(t,\frac{x - x(t)}{\lambda(t)} \right) \right|^2 \, \mathrm{d} x .
\end{align*}
Split
\begin{align}\label{f6.22}
 	\tilde{\epsilon}_j(t)
 = & P_{\leq k _0+ 9} \Big(e^{-i \gamma_j(t)} e^{-ix \cdot \frac{\xi(t)}{\lambda(t)}}
  \tfrac{1}{\lambda(t)} \epsilon_j \big(t, \tfrac{x - x(t)}{\lambda(t)} \big) \Big)
  \\\nonumber
  =  & e^{-i \gamma_j(t)} e^{-ix \cdot \frac{\xi(t)}{\lambda(t)}} P_{\leq k _0+ 9}
  \Big(\tfrac{1}{\lambda(t)} \epsilon_j \big(t, \tfrac{x - x(t)}{\lambda(t)} \big) \Big) + P_{\leq k _0+ 9} \Big(e^{-i \gamma_j(t)} e^{-ix \cdot \frac{\xi(t)}{\lambda(t)}} \frac{1}{\lambda(t)}
   \epsilon_j \big(t, \tfrac{x - x(t)}{\lambda(t)} \big) \Big) \\\nonumber
  &
   -  e^{-i \gamma_j(t)} e^{-ix \cdot \frac{\xi(t)}{\lambda(t)}} P_{\leq k _0+ 9}
   \Big(\tfrac{1}{\lambda(t)} \epsilon_j \big(t, \tfrac{x - x(t)}{\lambda(t)} \big) \Big).
\end{align}
Noticing that for any $k\in\mathbb{N}, f\in L_x^2$ and $\xi\in\mathbb{R}^2$, we have 
\begin{align*}
\left\|P_{\leq k} \left(e^{-ix\cdot\xi}f \right)-e^{-ix\cdot\xi}P_{\leq k}f \right\|_{L_x^2}
& = \left\|\varphi \left(\frac{x}{2^k} \right)\hat{f}(x-\xi)-\varphi \left(\frac{x-\xi}{2^k} \right)\hat{f}(x-\xi) \right\|_{L_x^2}
\\
& \lesssim 2^{-k}|\xi| \|\hat{f} \|_{L_x^2}\lesssim 2^{-k}|\xi|\|f\|_{L_x^2},
\end{align*}
since $\left|\frac{ \xi(t) }{\lambda(t)} \right| \leq \eta_{0}$, for any $j \in\mathbb{Z}_N$, we obtain 
\begin{align}
\aligned
&\left\|P_{\leq k _0+ 9} \left(e^{-i \gamma_j(t)} e^{-ix \cdot \frac{\xi(t)}{\lambda(t)}} \frac{1}{\lambda(t)} \epsilon_j \left(t, \frac{x - x(t)}{\lambda(t)} \right) \right)-
e^{-i \gamma_j(t)} e^{-ix \cdot \frac{\xi(t)}{\lambda(t)}} P_{\leq k _0+ 9} \left(\frac{1}{\lambda(t)} \epsilon_j \left(t, \frac{x - x(t)}{\lambda(t)} \right) \right) \right\|_{L_x^2} \\
 \lesssim& 2^{-k_0} \frac{|\xi(t)|}{\lambda(t)} \left\| \epsilon_j \right\|_{L^2_{x}}
\lesssim 2^{-k_0} \eta_{0} \| \epsilon_j \|_{L^2_{x}},\label{qr1}
\endaligned
\end{align}
and 
\begin{align}\nonumber
& \left\| \nabla P_{\leq k _0+ 9} \left(e^{-i \gamma_j(t)} e^{-ix \cdot \frac{\xi(t)}{\lambda(t)}} \frac{1}{\lambda(t)} \epsilon_j \left(t, \frac{x - x(t)}{\lambda(t)} \right) \right)
- e^{-i \gamma_j(t)} e^{-ix \cdot \frac{\xi(t)}{\lambda(t)}} \nabla P_{\leq k _0+ 9} \left(\frac{1}{\lambda(t)} \epsilon_j \left(t, \frac{x - x(t)}{\lambda(t)} \right) \right)  \right\|_{L^2_{x}}  \\\nonumber
\lesssim& \frac{|\xi(t)|}{\lambda(t)}
\left\|P_{\leq k _0+ 9} \left(e^{-i \gamma_j(t)} e^{-ix \cdot \frac{\xi(t)}{\lambda(t)}} \frac{1}{\lambda(t)} \epsilon_j \left(t, \frac{x - x(t)}{\lambda(t)} \right) \right) \right\|_{L_x^2}+2^{-k_0}
\frac{|\xi(t)|}{\lambda(t)} \left\| P_{\leq k_0+12} \epsilon_j  \right\|_{L^2_{x}}
 \\\label{qr2}
  \lesssim& \eta_{0} \| \epsilon_j \|_{L^2_{x}}.
\end{align}
Since $Q$ is real-valued, smooth, and has derivatives that are rapid decreasing, let
\begin{align*}
\check{\epsilon}_j(t,x) 
=e^{-i \gamma_j(t)} e^{-ix \cdot \frac{\xi(t)}{\lambda(t)}} \frac{1}{\lambda(t)} \left( P_{\leq k_0+9} \epsilon_j  \right) \left(t, \frac{x - x(t)}{\lambda(t)} \right), \forall \,  j\in\mathbb{Z}_N,
\end{align*}
 by \eqref{qr1} and \eqref{qr2}, $\mathcal{P}_2$ can be further decomposed as
\begin{align}\label{f6.25}
\mathcal P_{2}
= & \frac{1}{2} \left\| \nabla \boldsymbol{\check{\epsilon}} \right\|_{L^2_{x} l^2}^{2}
 - \frac{N+3}{2\lambda(t)^{4}}\sum_{j\in\mathbb{Z}_N} \int Q \left(\frac{x - x(t)}{\lambda(t)} \right)^{2} \left( \left( P_{\leq k_0+9}v_j \right)
  \left(t,\frac{x - x(t)}{\lambda(t)} \right) \right)^{2}  \, \mathrm{d} x \\\nonumber
& -\frac{N+1}{2\lambda(t)^{4}}\sum_{j\in\mathbb{Z}_N} \int Q \left(\frac{x - x(t)}{\lambda(t)} \right)^{2} \left( \left( P_{\leq k_0+9}w_j \right)
 \left(t,\frac{x - x(t)}{\lambda(t)} \right) \right)^{2} \, \mathrm{d} x \\\nonumber
& - \frac{2 }{ \lambda(t)^{4}} \sum_{j\in\mathbb{Z}_N}\sum_{\substack{k\in\mathbb{Z}_N\\k\neq j}} \int Q \left(\frac{x - x(t)}{\lambda(t)} \right)^{2}
 \left( P_{\leq k_0+9}v_k  \right) \left(t,\frac{x - x(t)}{\lambda(t)} \right) \left(  P_{\leq k_0+9}v_j \right)
 \left(t,\frac{x - x(t)}{\lambda(t)} \right) \, \mathrm{d} x  \\\nonumber
  & + O\left(\eta_{0}^{2}  \left\| \boldsymbol{\epsilon} \right\|_{L^2_{x} l^2}^{2} \right)
  + O\left(\eta_{0} \left\| \boldsymbol{\epsilon} \right\|_{L^2_{x} l^2}  \left\| \nabla \mathbf{\check{\boldsymbol{\epsilon}}} \right\|_{L^2_{x} l^2} \right)+ O \left( T^{-20} \right).
\end{align}
By direct calculation,
\begin{align}\label{f6.26}
\aligned
\frac{1}{2} \left\| \nabla \boldsymbol{\check{\epsilon}} \right\|_{L^2_{x} l^2}^{2}
= &  \frac{|\xi(t)|^{2}}{2 \lambda(t)^{4}} \left\| \left( P_{\leq k _0+ 9} \boldsymbol{\epsilon} \right)  \left(t, \frac{x - x(t)}{\lambda(t)} \right) \right\|_{L^2_{x} l^2}^{2}
 + \frac{1}{2\lambda(t)^{2}} \left\|  \left( \nabla P_{\leq k _0+ 9} \boldsymbol{\epsilon}  \right) \left(t, \frac{x - x(t)}{\lambda(t)} \right) \right\|_{L^2_x l^2}^{2}\\
& + \frac{\xi(t)}{\lambda(t)^{3}} \cdot \left\langle \left(  P_{\leq k _0+ 9} \boldsymbol{\epsilon}  \right) \left(t, \frac{x - x(t)}{\lambda(t)} \right),
i \left(  P_{\leq k _0+ 9} \nabla \boldsymbol{\epsilon}  \right) \left(t, \frac{x - x(t)}{\lambda(t)} \right) \right\rangle_{L^2_{x} l^2}.
\endaligned
\end{align}
Rescaling, if $2^{n(t)} = \lambda(t)$\footnote{ $n(t)$ can not  necessarily be an positive integer, but we can similarly define $P_{\leq k_0 + 9+n(t)}u$ as $\mathcal{F}^{-1} \left(\varphi \left(\frac{\xi}{2^{k_0+9+n(t)}} \right)\mathcal{F} u(\xi) \right).$}, then
\begin{align}\nonumber
&  \frac{1}{2 \lambda(t)^{2}} \left\| \left(  \nabla P_{\leq k _0+ 9} \boldsymbol{\epsilon} \right)  \left(t, \frac{x - x(t)}{\lambda(t)} \right) \right\|_{L^2_x l^2}^{2}
- \frac{N+3}{2\lambda(t)^{4}}\sum_{j\in\mathbb{Z}_N} \int Q \left(\frac{x - x(t)}{\lambda(t)} \right)^{2}
\left(  \left( P_{\leq k_0+9}v_j\right) \left(t,\frac{x - x(t)}{\lambda(t)} \right) \right)^{2} \, \mathrm{d} x \\\nonumber
&-\frac{N+1}{2\lambda(t)^{4}}\sum_{j\in\mathbb{Z}_N} \int Q \left(\frac{x - x(t)}{\lambda(t)} \right)^{2}
 \left(  \left( P_{\leq k_0+9}w_j  \right) \left(t,\frac{x - x(t)}{\lambda(t)} \right) \right)^{2} \, \mathrm{d} x \\\nonumber
&- \frac{2 }{  \lambda(t)^{4}} \sum_{j\in\mathbb{Z}_N}\sum_{\substack{k\in\mathbb{Z}_N\\k\neq j}} \int Q \left(\frac{x - x(t)}{\lambda(t)} \right)^{2}
 \left( P_{\leq k_0+9} v_k \right)  \left(t,\frac{x - x(t)}{\lambda(t)} \right) \left(  P_{\leq k_0+9}v_j  \right) \left(t,\frac{x - x(t)}{\lambda(t)} \right) \, \mathrm{d} x \\\label{4.27.0}
 =&  \frac{1}{2 \lambda(t)^{2}} \left\| \nabla P_{\leq k _0+ 9 + n(t)} \boldsymbol{\epsilon}(t, x) \right\|_{L^2_x l^2}^{2}
- \frac{N+3}{2\lambda(t)^{2}}\sum_{j\in\mathbb{Z}_N} \int Q(x)^{2} \left( P_{\leq k_0+9}v_j(t,x) \right)^{2} \, \mathrm{d} x \\\nonumber
&  -\frac{N+1}{2\lambda(t)^{2}}\sum_{j\in\mathbb{Z}_N} \int Q(x)^{2} \left( P_{\leq k_0+9}w_j(t,x) \right)^{2} \, \mathrm{d} x
- \frac{2 }{  \lambda(t)^{2}} \sum_{j\in\mathbb{Z}_N}\sum_{\substack{k\in\mathbb{Z}_N\\k\neq j}} \int Q(x)^{2}  P_{\leq k_0+9}v_k (t,x) P_{\leq k_0+9}v_j (t,x) \, \mathrm{ d} x.
\end{align}
Since $\boldsymbol{\epsilon}(t)$ satisfies the orthogonal conditions \eqref{orthod}, and $Q, \chi_0$ are real-valued, smooth, and have derivatives that are rapid decreasing, we have 
\begin{equation}\label{f6.28}
	p(f):= \left\langle P_{\leq k _0+ 9 + n(t)} \boldsymbol{\epsilon}, f \right\rangle_{L^2_{x} l^2}
 \lesssim T^{-20},
 \quad f \in \left\{i\boldsymbol{\chi}_{0,1},\cdots,i\boldsymbol{\chi}_{0,N}, \boldsymbol{\chi}_0, i\mathbf{Q}_{x_1}, i\mathbf{Q}_{x_2}, \mathbf{Q}_{x_1}, \mathbf{Q}_{x_2} \right\}.
\end{equation}
Let
\begin{align*}
\mathbf{e}(t)
= & P_{\leq k _0+ 9 + n(t)} \boldsymbol{\epsilon}-\sum_{j\in\mathbb{Z}_N}
\frac{p\left(i\boldsymbol{\chi}_{0,j} \right)}{ \left\|i\boldsymbol{\chi}_{0,j} \right\|_{L_x^2l^2}}i\boldsymbol{\chi}_{0,j}
-\frac{p \left(\boldsymbol{\chi}_{0} \right)}{ \left\|\boldsymbol{\chi}_{0} \right\|_{L_x^2l^2}}\boldsymbol{\chi}_{0}
-\frac{p \left(i\mathbf{Q}_{x_1} \right)}{ \left\|i\mathbf{Q}_{x_1} \right\|_{L^2_x l^2}}i\mathbf{Q}_{x_1}
\\
& -\frac{p \left(i\mathbf{Q}_{x_2} \right)}{ \left\|i\mathbf{Q}_{x_2} \right\|_{L^2_x l^2}}i\mathbf{Q}_{x_2} -\frac{p \left(\mathbf{Q}_{x_1} \right)}{ \left\|\mathbf{Q}_{x_1} \right\|_{L^2_x l^2}}\mathbf{Q}_{x_1}
-\frac{p \left(\mathbf{Q}_{x_2} \right)}{ \left\|\mathbf{Q}_{x_2} \right\|_{L^2_x l^2}}\mathbf{Q}_{x_2},
\end{align*}
then $\mathbf{e}(t)$ satisfies the orthogonal conditions \eqref{orthod}. Using  \eqref{key0}, we have 
\begin{align}\label{aux1}
& \frac{1}{2} \left\|\mathbf{e}(t) \right\|_{L^2_x l^2}+ \frac{1}{2} \left\| \nabla \mathbf{e}(t) \right\|_{L^2_{x} l^2}^{2}
- \frac{N+3}{2\lambda(t)^{2}}\sum_{j\in\mathbb{Z}_N} \int Q(x)^2 \left(\Re e_j(t) \right)^2 \, \mathrm{d} x  \\\nonumber
 & -\frac{N+1}{2\lambda(t)^{2}}\sum_{j\in\mathbb{Z}_N} \int Q(x)^2 \left(\Im\varepsilon_j(t) \right)^2 \, \mathrm{d} x
	- \frac{2 }{  \lambda(t)^{2}} \sum_{j\in\mathbb{Z}_N}\sum_{\substack{k\in\mathbb{Z}_N\\k\neq j}} \int Q(x)^2\Re e_k(t)\Re e_j(t) \, \mathrm{d} x
\geq \lambda_1 \left\|\mathbf{e}(t) \right\|^2_{H^{1}_{x} l^2}.
\end{align}
Collecting \eqref{f6.25},  \eqref{4.27.0}, \eqref{f6.28} and  \eqref{aux1}, if $\eta_0$  is sufficiently small,
 then there exists a constant $\lambda_2$ such that
\begin{equation}\label{f6.30}
\frac{1}{2 \lambda(t)^{2}} \left\| \boldsymbol{\epsilon} (t) \right\|_{L^2_{x} l^2}^{2}
+ \mathcal P_{2}(t) \geq \frac{\lambda_{2}}{\lambda(t)^{2}} \left\| \boldsymbol{\epsilon}(t) \right\|_{L^2_{x} l^2}^{2}
+ \lambda_{2} \left\| \nabla\mathbf{\check{\boldsymbol{\epsilon}}}(t) \right\|_{L^2_x l^2}^{2} - O \left(T^{-20} \right).
\end{equation}
Finally, by the interpolation, \eqref{f6.22} and Bernstein's inequality, we have
\begin{equation}\label{f6.31}
|\mathcal{P}_3(t)|
\lesssim \left\|P_{\leq k_0+9}\mathbf{\tilde{\boldsymbol{\epsilon}}}(t) \right\|^2_{L^2_x l^2} \left\|\nabla P_{\leq k_0+9}\mathbf{\tilde{\boldsymbol{\epsilon}}}(t) \right\|^2_{L^2_x l^2}
\lesssim \eta_0^2 \left\|\nabla\mathbf{\check{\boldsymbol{\epsilon}}} \right\|^2_{L^2_x l^2}+2^{-2k_0}\eta_0^4 \left\|\boldsymbol{\epsilon}(t) \right\|^4_{L^2_x l^2},
\end{equation}
and
\begin{align}\label{f6.32}
|\mathcal{P}_4(t)|    \lesssim \frac{1}{\lambda(t)} \left\| P_{\leq k_0+9}\mathbf{\tilde{\boldsymbol{\epsilon}}}(t) \right\|_{L^2_{x} l^2}^{2}
\left\| \nabla P_{\leq k_0+9}\boldsymbol{\tilde{\epsilon}}(t) \right\|^2_{L^2_x l^2}  \lesssim \frac{1}{\lambda(t)} \left\| \boldsymbol{\epsilon} \right\|_{L^2_{x} l^2} \left(\eta_{0} \left\| \mathbf{\check{\boldsymbol{\epsilon}}} \right\|_{\dot{H}^{1}_x l^2}
+ 2^{-k_0} \eta_{0}^2  \left\| \boldsymbol{\epsilon} \right\|_{L^2_{x} l^2}^{2} \right).
\end{align}
If $\eta_{0} \ll 1$ is sufficiently small, then by the Cauchy-Schwarz inequality, \eqref{f6.18}, \eqref{f6.30}, \eqref{f6.31} and \eqref{f6.32}, there exists $\lambda_3>0$ such that
\begin{equation}\label{f6.33}
\aligned
E \left(P_{\leq k _0+ 9} \mathbf{u}(t) \right)
 &\geq \lambda_3 \frac{N}{2}\frac{|\xi(t)|^{2}}{\lambda(t)^{2}} \| Q \|_{L^2_{x}}^{2}
 + \frac{\lambda_{3}}{2 \lambda(t)^{2}} \left\| \boldsymbol{\epsilon} \right\|_{L^2_{x} l^2}^{2}
 + \frac{\lambda_{3}}{2}  \left\|\nabla \boldsymbol{\check{\epsilon}}(t) \right\|_{L^2_x l^2}^{2}  - O \left(T^{-20} \right) \\
&\geq  \frac{\lambda_{3}}{2 \lambda(t)^{2}} \left\| \boldsymbol{\epsilon} \right\|_{L^2_{x} l^2}^{2}
+ \frac{\lambda_{3}}{2} \left\| \tilde{\tilde{\boldsymbol{\epsilon}}} \right\|_{\dot{H}_x^{1}l^2}^{2} - O \left(2^{2k_0} T^{-10} \right).
\endaligned
\end{equation}
Plugging \eqref{f6.33} into \eqref{f6.1} and using \eqref{qr2}, we can finish the proof of Theorem \ref{t6.12}.
\end{proof}



\section{A virial inequality}\label{Sec:viria}
In this section, we derive a  virial-type inequality, which plays a key role in the proof of our main Theorem \ref{t2.3}. Inspired by \cite{D2,D1}, we will prove this estimate by making use of the frequency-localized Morawetz estimate and the long time Strichartz estimates 
 derived in the previous section. 
\begin{theorem}\label{t7.13}
Under the assumptions of Theorem \ref{t5.9}, if $\mathbf{u}(t)$ further satisfies 
\begin{equation}\label{f7.1}
\frac{1}{\eta_{1}} \leq \lambda(t) \leq \frac{1}{\eta_{1}} T^{1/100},
\quad \forall \,  t \in J, \qquad \xi(b) = x(a) = 0,\qquad \sup_{t \in J}|x(t)| \leq R = T^{1/25}.
\end{equation}
Let $\boldsymbol{\epsilon} = \mathbf{v} + i \mathbf{w}$, then for $T=T(\eta_1)$ sufficiently large,
\begin{align}\label{f7.2}
& \int_{a}^{b} \left\| \boldsymbol{\epsilon}(t) \right\|_{L^2_{x} l^2}^{2} \lambda(t)^{-2} \, \mathrm{d} t
\notag \\
 \leq& 3 \sum_{j\in\mathbb{Z}_N} \langle w_j(a), Q + x \cdot \nabla Q \rangle_{L^2_{x}} - 3\sum_{j\in\mathbb{Z}_N} \langle w_j(b), Q + x \cdot \nabla Q \rangle_{L^2_{x}}
+ \frac{T^{1/15}}{\eta_{1}^{2}} \sup_{t \in J} \frac{|\xi(t)|^{2}}{\lambda(t)^{2}} + O \left(T^{-8} \right).
\end{align}
\end{theorem}

\begin{proof}
Let $\chi(r) \in C^{\infty}([0, \infty))$ be a smooth, radial function, satisfying $\chi(r) = 1$ for $0 \leq r \leq 1$,
and supported on $r \leq 2$. Then let
\begin{equation*} 
\phi(r) = \int_{0}^{r} \psi \left(\frac{\eta_1 s}{2R} \right)  \, \mathrm{d} s, \quad\psi(x):=\chi^2(|x|),
\end{equation*}
and we can define the Morawetz action 
\begin{equation}\label{f7.4}
M(t) = \sum_{j\in\mathbb{Z}_N}\int \phi(|x|) \frac{x}{|x|} \cdot \Im \left[\overline{P_{\leq k _0+ 9} u_j} \nabla P_{\leq k _0+ 9} u_j \right](t,x) \,   \mathrm{d} x,
\end{equation}
where $k_0=T^{1/3}$. Following \eqref{f5.25},
\begin{equation*} 
i \partial_{t} P_{\leq k _0+ 9} u_j + \Delta P_{\leq k _0+ 9} u_j +  {F}_j \left(P_{\leq k _0+ 9} \mathbf{u} \right)
 =  {F}_j \left(P_{\leq k _0+ 9} \mathbf{u} \right) - P_{\leq k _0+ 9}  {F}_j \left(\mathbf{u} \right) = -{\mathcal{N}}_j,\quad \forall \, j\in\mathbb{Z}_N.
\end{equation*}
Plugging into \eqref{f7.4} and integrating by parts, we obtain
\begin{align}
\frac{ \mathrm{d} }{ \mathrm{d} t} M(t)
=& \sum_{j\in\mathbb{Z}_N} \Re\int \phi(r) \left[-\Delta \overline{P_{\leq k _0+ 9} u_j} \partial_{r} P_{\leq k _0+ 9} u_j + \overline{P_{\leq k _0+ 9} u_j} \Delta \partial_{r} P_{\leq k _0+ 9} u_j \right] \,\mathrm{d}x  \notag\\
&+ \sum_{j\in\mathbb{Z}_N} \Re\int \phi(r) \left[- {F}_j \left(\overline{P_{\leq k _0+ 9} \mathbf{u}} \right)
\partial_{r} P_{\leq k _0+ 9} u_j + \overline{P_{\leq k _0+ 9} u_j} \partial_{r}  {F}_j \left(P_{\leq k _0+ 9} \mathbf{u} \right) \right] \,\mathrm{d}x  \notag\\
&+\sum_{j\in\mathbb{Z}_N} \Re\int \phi(r) \left[\overline{P_{\leq k _0+ 9} u_j} \partial_{r} \mathcal{N}_j \right](t,x) \,  \mathrm{d} x
- \sum_{j\in\mathbb{Z}_N} \Re\int \phi(r) \left[\bar{\mathcal{ N}_j} \partial_{r} P_{\leq k _0+ 9} u_j \right](t,x) \, \mathrm{ d} x \notag\\
=& 2\sum_{j\in\mathbb{Z}_N} \int \chi^{2} \left(\frac{\eta_1 x}{R} \right) \left|\nabla P_{\leq k _0+ 9} u_j \right|^{2}\, \mathrm{ d} x
 - \sum_{j\in\mathbb{Z}_N}\int \chi^{2} \left(\frac{x}{R} \right)  \left|P_{\leq k _0+ 9} u_j \right|^{4} \, \mathrm{d} x
 \notag \\
 &   -2\sum_{j\in\mathbb{Z}_N}\sum_{\substack{k\in\mathbb{Z}_N\\k\neq j}}\int \chi^{2} \left(\frac{\eta_1 x}{R} \right)  \left|P_{\leq k _0+ 9} u_k \right|^{2}
 \left|P_{\leq k _0+ 9} u_j \right|^{2}  \, \mathrm{d} x\label{estq1}\\
& +  2\sum_{j\in\mathbb{Z}_N}\sum_{m=1}^n\sum_{l=1}^n\Re\int \left[\frac{1}{|x|} \phi(x) - \chi^{2} \left(\frac{\eta_1 x}{R} \right) \right]
\left(\delta_{ml} - \frac{x_{m} x_{l}}{|x|^{2}} \right)
\left(\overline{\partial_{m} P_{\leq k _0+ 9} u_j} \partial_{l} P_{\leq k _0+ 9} u_j \right) \,  \mathrm{d} x \label{estq2}\\
& - \frac{\eta_1^2}{2R^{2}}\sum_{j\in\mathbb{Z}_N} \int \psi'' \left(\frac{\eta_1 x}{R} \right)  \left|P_{\leq k _0+ 9} u_j \right|^{2}  \, \mathrm{d} x
\label{estq3}\\
& + \sum_{j\in\mathbb{Z}_N}\Re\int \phi(r) \left[\overline{P_{\leq k _0+ 9} u_j} \partial_{r} \mathcal{N}_j \right](t,x) \, \mathrm{ d} x
- \sum_{j\in\mathbb{Z}_N}\Re\int \phi(r) \left[\bar{\mathcal{N}_j} \partial_{r} P_{\leq k _0+ 9} u_j \right](t,x) \,  \mathrm{d} x.\label{estq4}
\end{align}

{\bf Step 1: Estimation of $M(b) - M(a)$.}
Since $Q$ is smooth, real-valued, and rapidly decreasing (as are all its derivatives), we get
\begin{align}\nonumber
	&\sum_{j\in\mathbb{Z}_N}\int \phi(|x|) \frac{x}{|x|} 
\Im\bigg[\overline{P_{\leq k _0+ 9} \left(e^{-i \gamma_j(t)} e^{-i(x-x(t)) \cdot \frac{\xi(t)}{\lambda(t)}} \frac{1}{\lambda(t)} Q \left(\frac{x - x(t)}{\lambda(t)} \right) \right)} \\\nonumber
    & \qquad  \cdot \nabla \left(P_{\leq k _0+ 9}  \left(e^{-i \gamma_j(t)} e^{-i(x-x(t)) \cdot \frac{\xi(t)}{\lambda(t)}} \frac{1}{\lambda(t)} Q \left(\frac{x - x(t)}{\lambda(t)} \right) \right) \right) \bigg]  \, \mathrm{d} x\Bigg|_{a}^{b} \\\nonumber
	= &  -\frac{1}{\lambda(t)^{2}} \sum_{j\in\mathbb{Z}_N}\int \phi(|x|)\frac{x}{|x|}\cdot\frac{\xi(t)}{\lambda(t)}
\left|P_{\leq k_0+9} \left(e^{-i(x-x(t))\cdot\frac{\xi(t)}{\lambda(t)}}Q \left(\frac{x - x(t)}{\lambda(t)} \right) \right) \right|^{2}  \, \mathrm{d} x \Bigg|_a^b \\\nonumber
	&+  \frac{1}{\lambda(t)^{3}} \sum_{j\in\mathbb{Z}_N} \int \phi(|x|) \frac{x}{|x|}
 \Im \bigg[\overline{P_{\leq k _0+ 9}  \left(e^{-i \gamma_j(t)} e^{-i(x-x(t)) \cdot \frac{\xi(t)}{\lambda(t)}} \frac{1}{\lambda(t)} Q \left(\frac{x - x(t)}{\lambda(t)} \right) \right)} \\\label{equ:Mabest}
&\qquad \cdot P_{\leq k_0+9} \left[e^{-i\gamma_j(t)}e^{-i(x-x(t))\cdot\frac{\xi(t)}{\lambda(t)}}\nabla Q \left(\frac{x - x(t)}{\lambda(t)} \right) \right](x)   \bigg] \, \mathrm{d} x \Bigg|_{a}^{b}.
\end{align}
Noticing that
\begin{align*}
	\int \phi(|x|) \frac{x}{|x|} \cdot \Im \left[\overline{e^{-i \gamma_j(t)} e^{-i(x-x(t)) \cdot \frac{\xi(t)}{\lambda(t)}} \frac{1}{\lambda(t)} Q \left(\frac{x - x(t)}{\lambda(t)} \right)} e^{-i\gamma_j(t)}e^{-i(x-x(t))\cdot\frac{\xi(t)}{\lambda(t)}}\nabla Q \left(\frac{x - x(t)}{\lambda(t)} \right) \right] \, \mathrm{d} x=0,\ \forall \,  j\in \mathbb{Z}_N,
\end{align*}
we continue:
\begin{align}\label{f7.11}
	\eqref{equ:Mabest}= & \frac{1}{\lambda(t)^{2}} \sum_{j\in\mathbb{Z}_N}\int \phi(|x|)\frac{x}{|x|}
\cdot\frac{\xi(t)}{\lambda(t)}  \left|e^{-ix\cdot\frac{\xi(t)}{\lambda(t)}}Q \left(\frac{x - x(t)}{\lambda(t)} \right) \right|^{2} \, \mathrm{d} x \bigg|_{a}^{b}
\notag \\
& +O \left(\frac{|\xi(a)|}{\lambda(a)}\|Q\|_{L_x^2}
\left\|P_{\geq k_0+9} \left(\frac{1}{\lambda(a)}e^{-ix\cdot\frac{\xi(a)}{\lambda(a)}}Q \left(\frac{x-x(a)}{\lambda(a)} \right) \right) \right\|_{L_x^2} \right)\notag\\
	&+O \left(\frac{|\xi(b)|}{\lambda(b)}\|Q\|_{L_x^2}
\left\|P_{\geq k_0+9} \left(\frac{1}{\lambda(b)}e^{-ix\cdot\frac{\xi(b)}{\lambda(b)}}Q \left(\frac{x-x(b)}{\lambda(b)} \right) \right) \right\|_{L_x^2} \right)
\notag \\
& +O \left(\frac{1}{\lambda(a)}\|\nabla Q\|_{L_x^2}
\left\|P_{\geq k_0+9} \left(\frac{1}{\lambda(a)}e^{-ix\cdot\frac{\xi(a)}{\lambda(a)}}Q \left(\frac{x-x(a)}{\lambda(a)} \right) \right) \right\|^2_{L_x^2} \right)\notag\\
	&+O \left(\frac{1}{\lambda(b)}\|\nabla Q\|_{L_x^2}
\left\|P_{\geq k_0+9} \left(\frac{1}{\lambda(b)}e^{-ix\cdot\frac{\xi(b)}{\lambda(b)}}Q \left(\frac{x-x(b)}{\lambda(b)} \right) \right) \right\|^2_{L_x^2} \right)
\notag \\
& +O \left(\frac{1}{\lambda(a)}\|Q\|_{L_x^2}
\left\|P_{\geq k_0+9} \left(\frac{1}{\lambda(a)}e^{-ix\cdot\frac{\xi(a)}{\lambda(a)}}\nabla Q \left(\frac{x-x(a)}{\lambda(a)} \right) \right) \right\|_{L_x^2} \right)\notag\\
	&+O \left(\frac{1}{\lambda(b)}\|Q\|_{L_x^2}
\left\|P_{\geq k_0+9} \left(\frac{1}{\lambda(b)}e^{-ix\cdot\frac{\xi(b)}{\lambda(b)}}\nabla Q \left(\frac{x-x(b)}{\lambda(b)} \right) \right) \right\|_{L_x^2} \right).
\end{align}
By Bernstein's inequality and \eqref{f7.1}, we have 
\begin{align}
	&\left\|P_{\geq k_0+9} \Big(\tfrac{1}{\lambda(a)}e^{-ix\cdot\frac{\xi(a)}{\lambda(a)}}Q \Big( \tfrac{x-x(a)}{\lambda(a)} \Big) \Big) \right\|_{L_x^2}  +
	\left\|P_{\geq k_0+9} \Big(\tfrac{1}{\lambda(b)}e^{-ix\cdot\frac{\xi(b)}{\lambda(b)}}Q \Big(\tfrac{x-x(b)}{\lambda(b)} \Big) \Big) \right\|_{L_x^2}\lesssim 2^{-100k_0} \label{rfc1}
    \end{align}
    and 
    \begin{align}
	& \left\|P_{\geq k_0+9} \Big(\tfrac{1}{\lambda(a)}e^{-ix\cdot\frac{\xi(a)}{\lambda(a)}}\nabla Q \Big(\tfrac{x-x(a)}{\lambda(a)} \Big) \Big) \right\|_{L_x^2}   + 
	\left\|P_{\geq k_0+9} \Big(\tfrac{1}{\lambda(b)}e^{-ix\cdot\frac{\xi(b)}{\lambda(b)}}\nabla Q \Big(\tfrac{x-x(b)}{\lambda(b)} \Big) \Big) \right\|_{L_x^2}\lesssim 2^{-100 k_0},\label{rfc2}
\end{align}
where the implicit constants in \eqref{rfc1} and \eqref{rfc2} depend only on  $\eta_0$ and $\eta_1$. 
Inserting \eqref{rfc1} and \eqref{rfc2} into \eqref{f7.11}, and using \eqref{f7.1},  we obtain 
\begin{align}\label{f7.14}
	&\sum_{j\in\mathbb{Z}_N}\int \phi(|x|) \frac{x}{|x|}
 \Im \bigg[\overline{P_{\leq k _0+ 9}
\left(e^{-i \gamma_j(t)} e^{-i(x-x(t)) \cdot \frac{\xi(t)}{\lambda(t)}} \frac{1}{\lambda(t)} Q \left(\frac{x - x(t)}{\lambda(t)} \right) \right)}
	\notag \\
& \cdot
\nabla \left(P_{\leq k _0+ 9} \left(e^{-i \gamma_j(t)} e^{-i(x-x(t)) \cdot \frac{\xi(t)}{\lambda(t)}} \frac{1}{\lambda(t)} Q \left(\frac{x - x(t)}{\lambda(t)} \right) \right) \right) \bigg] \, \mathrm{d} x \Bigg|_{a}^{b}\notag\\
	=&\frac{1}{\lambda(t)^{2}} \sum_{j\in\mathbb{Z}_N}\int \phi(|x|)\frac{x}{|x|}\cdot\frac{\xi(t)}{\lambda(t)}
\left|e^{-ix\cdot\frac{\xi(t)}{\lambda(t)}}Q \left(\frac{x - x(t)}{\lambda(t)} \right) \right|^{2} \, \mathrm{ d} x \bigg|_{a}^{b}
+ O \left( T^{-10} \right) \notag\\
	=&-N\frac{1}{\lambda(a)^{2}} \sum_{j\in\mathbb{Z}_N}\int \phi(|x|)\frac{x}{|x|}\cdot\frac{\xi(a)}{\lambda(a)}  \left|Q \left(\frac{x}{\lambda(a)} \right) \right|^{2} \, \mathrm{d} x
+ O \left( T^{-10} \right) = O \left( T^{-10} \right).
\end{align}
Using almost the same argument as above, we can also calculate that
\begin{align}\label{f7.15}
	&\sum_{j\in\mathbb{Z}_N}\int \phi(|x|) \frac{x}{|x|}  \Im \bigg[\overline{P_{\leq k _0+ 9}
\left(e^{-i \gamma_j(t)} e^{-i(x-x(t)) \cdot \frac{\xi(t)}{\lambda(t)}} \frac{1}{\lambda(t)} \epsilon_j \left(\frac{x - x(t)}{\lambda(t)} \right) \right)}
 \notag \\
& \qquad \cdot 
\nabla  \left(P_{\leq k _0+ 9}
 \left(e^{-i \gamma_j(t)} e^{-i(x-x(t)) \cdot \frac{\xi(t)}{\lambda(t)}} \frac{1}{\lambda(t)} Q \left(\frac{x - x(t)}{\lambda(t)} \right) \right) \right) \bigg] \, \mathrm{ d} x \Bigg|_{a}^{b}\notag \\
	=& \sum_{j\in\mathbb{Z}_N}   \int \phi(|x|) \frac{x}{|x|} \cdot\frac{\xi(t)}{\lambda(t)}
\left( \frac{1}{\lambda(t)} v_j \left(t,\frac{x - x(t)}{\lambda(t)} \right)  \frac{1}{\lambda(t)} Q \left(\frac{x - x(t)}{\lambda(t)} \right) \right) \, \mathrm{d}x  \Bigg|_{a}^b\notag\\
	& \quad -\sum_{j\in\mathbb{Z}_N}   \frac{1}{\lambda(t)}\int \phi(|x|) \frac{x}{|x|} \cdot \left[ \frac{1}{\lambda(t)} w_j \left(t,\frac{x - x(t)}{\lambda(t)} \right)  \frac{1}{\lambda(t)} \nabla Q \left(\frac{x - x(t)}{\lambda(t)} \right) \right] \, \mathrm{d}x \Bigg|_{a}^b\notag\\
	& \quad -\sum_{j\in\mathbb{Z}_N}   \frac{1}{\lambda(t)}\int \phi(|x|) \frac{x}{|x|}
\cdot \left[ P_{\geq k_0+9} \left(\frac{1}{\lambda(t)} w_j \left(t,\frac{x - x(t)}{\lambda(t)} \right) \right) P_{\leq k_0+9} \left(\frac{1}{\lambda(t)}
 \nabla Q \left(\frac{x - x(t)}{\lambda(t)} \right) \right) \right] \, \mathrm{d}x \Bigg|_{a}^b \notag \\
 & \quad +O \left(T^{-10} \right).
\end{align}
Noting that $\theta_1(x):=\phi({|x|})\frac{x_1}{|x|}$ and $ \theta_2(x):=\phi({|x|})\frac{x_2}{|x|}$ are smooth functions with bounded partial derivatives of all orders, it is easy to verify that
\begin{align*}
& \left\||\nabla|^{100} \Big(\theta_1(x) P_{\leq k_0+9} \Big(\tfrac{1}{\lambda(t)} Q_{x_1} \Big(\tfrac{x - x(t)}{\lambda(t)} \Big) \Big) \Big) \right\|_{L_x^2}+ \left\||\nabla|^{100} \Big(\theta_2(x) P_{\leq k_0+9} \Big(\tfrac{1}{\lambda(t)} Q_{x_2} \Big(\tfrac{x - x(t)}{\lambda(t)} \Big) \Big) \Big) \right\|_{L_x^2}
\lesssim 1, 
\end{align*}
where the implicit constant depends only on $\eta_1$. Thus
\begin{align}\nonumber
	&\sum_{j\in\mathbb{Z}_N}   \frac{1}{\lambda(t)}\int \phi(|x|) \frac{x}{|x|}
\cdot \left[ P_{\geq k_0+9} \left(\frac{1}{\lambda(t)} w_j
\left(t,\frac{x - x(t)}{\lambda(t)} \right)  \right)
P_{\leq k_0+9} \left(\frac{1}{\lambda(t)} \nabla Q \left(\frac{x - x(t)}{\lambda(t)} \right) \right) \right]  \,\mathrm{d}x \Bigg|_{a}^b   \\\nonumber
	=&\sum_{j\in\mathbb{Z}_N} \sum_{k=1}^2 \frac{1}{\lambda(t)}\Re\int \overline{|\nabla|^{-100}P_{\geq k_0+9}
\left(\frac{1}{\lambda(t)} w_j \left(t,\frac{x - x(t)}{\lambda(t)} \right) \right)} \\\nonumber
& \qquad \cdot |\nabla|^{100}
\left(\theta_k(x) P_{\leq k_0+9} \left(\frac{1}{\lambda(t)} \nabla Q \left(\frac{x - x(t)}{\lambda(t)} \right) \right) \right)  \,\mathrm{d}x \Bigg|_{a}^b \\\label{subqa}
	\lesssim& 2^{-100k_0} \left\|\mathbf{w} \right\|_{L^2_x l^2}\lesssim T^{-10}.
\end{align}
On the other hand, since $\phi(x)=|x|$ for $|x|\leq \eta_1^{-1}R$ and by conditions \eqref{f7.1},
$$ \left\|\chi_{|x|\geq \eta_1^{-1}R}\frac{1}{\lambda(t)}|x|Q \left(\frac{x-x(t)}{\lambda(t)} \right) \right\|_{L_x^2} \lesssim \frac{1}{\eta_1}T^{-100},$$ 
then by H\"older's inequality and $x(a)=\xi(b)=0$,
\begin{align}\label{subqa1}
\aligned
	& \sum_{j\in\mathbb{Z}_N}   \int \phi(|x|) \frac{x}{|x|} \cdot\frac{\xi(t)}{\lambda(t)} \left( \frac{1}{\lambda(t)} v_j \left(t,\frac{x - x(t)}{\lambda(t)} \right)  \frac{1}{\lambda(t)} Q \left(\frac{x - x(t)}{\lambda(t)} \right) \right) \,\mathrm{d}x \Bigg|_{a}^b \\
	=&\sum_{j\in\mathbb{Z}_N}   \int x\cdot\frac{\xi(t)}{\lambda(t)}
\left( \frac{1}{\lambda(t)} v_j \left(t,\frac{x - x(t)}{\lambda(t)} \right)
 \frac{1}{\lambda(t)} Q \left(\frac{x - x(t)}{\lambda(t)} \right) \right) \,\mathrm{d}x  \Bigg|_{a}^b \\
	& +\sum_{j\in\mathbb{Z}_N}   \int_{|x|\geq \eta_1^{-1}R}  \left(\phi(|x|) \frac{x}{|x|}-x \right)\cdot\frac{\xi(t)}{\lambda(t)}
\left( \frac{1}{\lambda(t)} v_j \left(t,\frac{x - x(t)}{\lambda(t)} \right)  \frac{1}{\lambda(t)} Q \left(\frac{x - x(t)}{\lambda(t)} \right) \right) \,\mathrm{d}x  \Bigg|_{a}^b \\
	=&\sum_{j\in\mathbb{Z}_N} \xi(t)\cdot \int v_j(t,x) xQ(x) \, \mathrm{d} x \Big|_a^b+\sum_{j\in\mathbb{Z}_N}  x(t)\cdot\frac{\xi(t)}{\lambda(t)}\int v_j(t,x)Q(x) \, \mathrm{d} x \Big|_a^b+O(T^{-10}) \\
	=&\sum_{j\in\mathbb{Z}_N} \xi(b)\cdot \int v_j(b,x) x Q(x) \, \mathrm{d} x+O \left(T^{-10} \right) .
\endaligned
\end{align}
Recalling that $ \left\langle\boldsymbol{\epsilon},i\mathbf{Q}_{x_1} \right\rangle_{L^2_x l^2}=\left\langle\boldsymbol{\epsilon}, i\mathbf{Q}_{x_2} \right\rangle_{L^2_x l^2} = 0$,
we also have
\begin{align}\label{subqa2}
\aligned
	&\sum_{j\in\mathbb{Z}_N}   \frac{1}{\lambda(t)}\int \phi(|x|) \frac{x}{|x|}
\cdot \left[ \frac{1}{\lambda(t)} w_j \left(t,\frac{x - x(t)}{\lambda(t)} \right)  \frac{1}{\lambda(t)} \nabla Q \left(\frac{x - x(t)}{\lambda(t)} \right) \,\mathrm{d}x \right] \Bigg|_{a}^b \\
	=&\sum_{j\in\mathbb{Z}_N} \int w_j(t,x) x\cdot\nabla Q(x) \, \mathrm{d} x \Big|_a^b+\sum_{j\in\mathbb{Z}_N}  \frac{1}{\lambda(t)}\int w_j(t,x)x(t)\cdot \nabla Q(x) \, \mathrm{d} x \Big|_a^b \\
    & \quad 
+O \left(T^{-10} \right) \\
	=&\sum_{j\in\mathbb{Z}_N} \int w_j(t,x) x\cdot\nabla Q(x) \, \mathrm{d} x \Big|_a^b+O \left(T^{-10} \right).
\endaligned
\end{align}
Substituting \eqref{subqa}, \eqref{subqa1} and \eqref{subqa2} into \eqref{f7.15} and then using H\"older's inequality along with 
 \eqref{f6.17}, we deduce that
\begin{align}\label{f7.19}
\aligned
	&\sum_{j\in\mathbb{Z}_N}\int \phi(|x|) \frac{x}{|x|}    \Im \bigg[\overline{P_{\leq k _0+ 9} \left(e^{-i \gamma_j(t)} e^{-i(x-x(t)) \cdot \frac{\xi(t)}{\lambda(t)}} \frac{1}{\lambda(t)} \epsilon_j
\left(t,\frac{x - x(t)}{\lambda(t)} \right) \right)}  \\
& \qquad \cdot
	 \nabla  \left(P_{\leq k _0+ 9} \left(e^{-i \gamma_j(t)} e^{-i(x-x(t))\cdot \frac{\xi(t)}{\lambda(t)}} \frac{1}{\lambda(t)} Q \left(\frac{x - x(t)}{\lambda(t)} \right) \right) \right) \bigg] \, \mathrm{d} x \bigg|_{a}^{b} \\
	=&\sum_{j\in\mathbb{Z}_N} \xi(b)\cdot \int v_j(b,x) xQ(x)\mathrm{d}x-\sum_{j\in\mathbb{Z}_N} \int w_j(t,x) x\cdot\nabla Q(x) \, \mathrm{d} x \Bigg|_a^b+O \left(T^{-10} \right) \\
	=& -\sum_{j\in\mathbb{Z}_N} \left\langle w_j(t,x),x\cdot\nabla Q  \right\rangle_{L_x^2} \Big|_a^b+O \left( \left\|\boldsymbol{\epsilon} \right\|^2_{L_x^2l^2}
+|\xi(b)|^2 \right)+O(T^{-10}) \\
	=&-\sum_{j\in\mathbb{Z}_N} \left\langle w_j(t,x),x\cdot\nabla Q(x)  \right\rangle_{L_x^2} \Big|_a^b
\\
& \quad +O \left(\frac{2^{2k_0} T^{1/50}}{\eta_{1}^{2} T} \int_{J}  \left\| \boldsymbol{\epsilon}(t)  \right\|_{L^2_{x} l^2}^{2} \lambda(t)^{-2} \, \mathrm{d} t
 + \frac{T^{1/50}}{\eta_{1}^{2}} \sup_{t \in J} \frac{|\xi(t)|^{2}}{\lambda(t)^{2}} + 2^{2k_0} \frac{T^{1/50}}{\eta_{1}^{2}} T^{-10} \right) .
 \endaligned
\end{align}
Similarly, one can verify that
\begin{align}\label{f7.20}
	&\sum_{j\in\mathbb{Z}_N}\int \phi(|x|) \frac{x}{|x|}   \Im \bigg[\overline{P_{\leq k _0+ 9} \left(e^{-i \gamma_j(t)} e^{-i(x-x(t)) \cdot \frac{\xi(t)}{\lambda(t)}} \frac{1}{\lambda(t)}Q \left(\frac{x - x(t)}{\lambda(t)} \right) \right)} \\
& \qquad \cdot
 \nabla  \left(P_{\leq k _0+ 9}
 \left(e^{-i \gamma_j(t)} e^{-i(x-x(t)) \cdot \frac{\xi(t)}{\lambda(t)}} \frac{1}{\lambda(t)} \epsilon_j \left(t,\frac{x - x(t)}{\lambda(t)} \right) \right) \right) \bigg] \, \mathrm{d} x \Bigg|_{a}^{b} \notag \\
	=&-\sum_{j\in\mathbb{Z}_N} \left\langle w_j(t,x),x\cdot\nabla Q(x)+2Q  \right\rangle_{L_x^2} \Big|_a^b \notag \\
    & \quad 
+O \left(\frac{2^{2k_0} T^{1/50}}{\eta_{1}^{2} T} \int_{J}
\left\| \boldsymbol{\epsilon}(t)  \right\|_{L^2_{x} l^2}^{2} \lambda(t)^{-2} \, \mathrm{d} t + \frac{T^{1/50}}{\eta_{1}^{2}} \sup_{t \in J} \frac{|\xi(t)|^{2}}{\lambda(t)^{2}} + 2^{2k_0} \frac{T^{1/50}}{\eta_{1}^{2}} T^{-10} \right). \notag
\end{align}
Finally, by \eqref{f6.16}, \eqref{f6.17} and \eqref{f7.1}, for any $t \in J$,
\begin{align}\label{f7.21}
	&\sum_{j\in\mathbb{Z}_N}\int \phi(|x|) \frac{x}{|x|}  \Im \bigg[\overline{P_{\leq k _0+ 9}  \left(e^{-i \gamma_j(t)} e^{-i(x-x(t))\cdot \frac{\xi(t)}{\lambda(t)}} \frac{1}{\lambda(t)}
 \epsilon_j \left(t,\frac{x - x(t)}{\lambda(t)} \right) \right)}
	\\
& \qquad \cdot \nabla  \left(P_{\leq k _0+ 9}  \left(e^{-i \gamma_j(t)} e^{-i(x-x(t))\cdot \frac{\xi(t)}{\lambda(t)}} \frac{1}{\lambda(t)} \epsilon_j
\left(t,\frac{x - x(t)}{\lambda(t)} \right) \right) \right) \bigg] \, \mathrm{d} x \Bigg|_{a}^{b} \notag \\
	\lesssim& \|\phi\|_{L^{\infty}}\max_{t\in J}
\left\|\boldsymbol{\epsilon} \right\|_{L^2_x l^2} \cdot\max_{t\in J}\left(\sum_{j\in\mathbb{Z}_N}
\left\|\nabla P_{\leq k _0+ 9}
\left( e^{-ix \cdot \frac{\xi(t)}{\lambda(t)}} \frac{1}{\lambda(t)} \epsilon_j \left(t,\frac{x - x(t)}{\lambda(t)} \right) \right) \right\|^2_{L_x^2}\right)^{1/2}\notag\\
	\lesssim& \frac{R}{\eta_1} \left(\frac{2^{2k_0} T^{1/100}}{\eta_{1} T} \int_{J}
\left\| \boldsymbol{\epsilon}(t) \right\|_{L^2_{x} l^2}^{2} \lambda(t)^{-2}  \, \mathrm{d} t + \frac{T^{1/100}}{\eta_{1}} \sup_{t \in J} \frac{|\xi(t)|^{2}}{\lambda(t)^{2}} + 2^{2k_0} \frac{T^{1/100}}{\eta_{1}} T^{-10} \right)\notag\\
	\lesssim& \frac{2^{2k_0} T^{1/20}}{\eta_{1}^{2} T} \int_{J} \left\| \boldsymbol{\epsilon}(t)  \right\|_{L^2_{x} l^2}^{2} \lambda(t)^{-2} \, \mathrm{d} t
+ \frac{T^{1/20}}{\eta_{1}^{2}} \sup_{t \in J} \frac{|\xi(t)|^{2}}{\lambda(t)^{2}} + 2^{2k_0} \frac{T^{1/20}}{\eta_{1}^{2}} T^{-10}. \notag
\end{align}
Since
\begin{align*}
P_{\leq k_0+9}u_j(t)= & P_{\leq k_0+9} \left(e^{-i\gamma_j(t)}e^{-i(x-x(t))\cdot\frac{\xi(t)}{\lambda(t)}}\frac{1}{\lambda(t)}Q \left(\frac{x-x(t)}{\lambda(t)} \right) \right)\\
& +P_{\leq k_0+9} \left(e^{-i\gamma_j(t)}e^{-i(x-x(t))\cdot\frac{\xi(t)}{\lambda(t)}}\frac{1}{\lambda(t)} \epsilon_j \left(t,\frac{x-x(t)}{\lambda(t)} \right) \right), \ \forall \,  j\in\mathbb{Z}_N,
\end{align*}
 then \eqref{f7.14}, \eqref{f7.19}, \eqref{f7.20} and \eqref{f7.21} imply that
\begin{equation}\label{f7.22}
	\aligned
	M(b) - M(a) = &  2 \sum_{j\in\mathbb{Z}_N} \left\langle w_{j}(a), Q + x \cdot \nabla Q \right\rangle_{L^2_{x}}
- 2\sum_{j\in\mathbb{Z}_N} \left\langle w_{j}(b), Q + x \cdot \nabla Q \right\rangle_{L^2_{x}} \\
& + O \left(\frac{2^{2k_0} T^{1/20}}{\eta_{1}^{2} T} \int_{J} \left\| \boldsymbol{\epsilon}(t) \right\|_{L^2_{x} l^2}^{2} \lambda(t)^{-2}  \, \mathrm{d} t
+ \frac{T^{1/20}}{\eta_{1}^{2}} \sup_{t \in J} \frac{|\xi(t)|^{2}}{\lambda(t)^{2}} + 2^{2k_0} \frac{T^{1/20}}{\eta_{1}^{2}} T^{-10} \right).
	\endaligned
\end{equation}

{\bf Step 2: Estimation of  $\frac{ \mathrm{d} }{ \mathrm{d} t}M(t)$.}
 We will estimate the terms \eqref{estq1}-\eqref{estq4} separately. 
 
{\bf Estimate of  \eqref{estq4}.}
Noticing that
\begin{align*}
\mathcal{N}_j
 =&P_{\leq k_0+9} {F}_j \left(\mathbf{u} \right)- {F}_j \left(P_{\leq k_0+9}\mathbf{u} \right)
=P_{\leq k_0+9} {F}_j \left(\mathbf{u} \right)-P_{\leq k_0+9} {F}_j \left(P_{\leq k_0+6}\mathbf{u} \right)
+ {F}_j \left(P_{\leq k_0+6}\mathbf{u} \right)- {F}_j \left(P_{\leq k_0+9}\mathbf{u} \right)\\
=&\mathcal{N}_j^{(1)}+\mathcal{N}_j^{(2)},
\end{align*}
where
\begin{align*}
\mathcal{N}_j^{(1)}
=& \sum_{k\in\mathbb{Z}_N} P_{\leq k _0+ 9}O \left(P_{\leq k_0+6}u_k P_{\leq k_0+6}u_kP_{\geq k_0+6}u_j \right)
 +  \sum_{k\in\mathbb{Z}_N} O \left(P_{\leq k_0+6}u_k P_{\leq k_0+6}u_kP_{ k_0+6\leq\cdot\leq k_0+9}u_j \right), \\
\mathcal{N}_j^{(2)}=&\sum_{k\in\mathbb{Z}_N} P_{\leq k _0+ 9}O \left(P_{\leq k_0+6}u_k P_{\geq k_0+6}u_kP_{\geq k_0+6}u_j \right)
+\sum_{k\in\mathbb{Z}_N} P_{\leq k _0+ 9}O \left(P_{\geq k_0+6}u_k P_{\geq k_0+6}u_kP_{\geq k_0+6}u_j \right)\\
& +\sum_{k\in\mathbb{Z}_N} O \left(P_{\leq k_0+6}u_k P_{k_0+6 \leq\cdot\leq k_0+9}u_kP_{k_0+6 \leq\cdot\leq k_0+9}u_j \right)\\
& +\sum_{k\in\mathbb{Z}_N} O \left(P_{k_0+6 \leq\cdot\leq k_0+9}u_k P_{k_0+6 \leq\cdot\leq k_0+9}u_k P_{k_0+6 \leq\cdot\leq k_0+9}u_j \right).
\end{align*}
Following the calculation in \eqref{f5.40}-\eqref{estqq3} and using Theorem \ref{t6.3}, \eqref{ssz} and \eqref{ssx},
one can prove that there exist $m_1, m_2, m_3\in \{0,\cdots, 9\}$, such that
\begin{align}\nonumber
& 
\sum_{j\in\mathbb{Z}_N}\int_{a}^{b}
 \int \phi(r) \Re \left[\overline{P_{\leq k _0+ 9} u_j} \partial_{r} \mathcal{N}_j^{(2)} \right](t,x) \, \mathrm{d} x  \mathrm{d} t
- \sum_{j\in\mathbb{Z}_N}\int_{a}^{b} \int \phi(r) \Re \left[\overline{\mathcal{N}_j^{(2)}} \partial_{r} P_{\leq k _0+ 9} u_j \right](t,x) \, \mathrm{ d} x  \mathrm{d} t \\\nonumber
& \lesssim  \eta_1^{-1}2^{k_0} R \sum^3_{l=1} \left\|  \left|P_{\geq k_0+m_l+3}\mathbf{u} \right| \left|P_{\leq k_0+m_l} \mathbf{u} \right|  \right\|_{L_{t,x}^{2}}^{2}
+ \eta_1^{-1}R2^{k_0}  \left\|\mathbf{u} \right\|^2_{U_{\Delta}^2(J, L_x^2l^2)} \\ \label{f7.24}
& \lesssim 2^{k_0} R  \left(\frac{1}{T} \int_{a}^{b}  \| \boldsymbol{\epsilon}(t) \|_{L^2_{x} l^2}^{2} \lambda(t)^{-2}  \, \mathrm{d} t + 2^{2k_0} T^{-10} \right).
\end{align}
Next, we further decompose $\mathcal{N}_j^{(1)}$ as  $\mathcal{N}_j^{(1)}=\mathcal{N}_j^{(1,1)}+\mathcal{N}_j^{(1,2)}$, where
\begin{align*}
	\mathcal{N}_j^{(1,1)}
=&  \sum_{k\in\mathbb{Z}_N} P_{\leq k _0+ 9}O\left(P_{k_0+3\leq\cdot \leq k_0+6}u_k P_{k_0+3\leq\cdot \leq k_0+6}u_kP_{\geq k_0+6}u_j \right) \\
& +\sum_{k\in\mathbb{Z}_N} P_{\leq k _0+ 9}O \left(P_{k_0+3\leq\cdot \leq k_0+6}u_k P_{\leq k_0+6}u_kP_{\geq k_0+6}u_j \right)  \\
	&  + \sum_{k\in\mathbb{Z}_N} O \left(P_{k_0+3\leq\cdot \leq k_0+6}u_k P_{k_0+3\leq\cdot \leq k_0+6}u_kP_{ k_0+6\leq\cdot\leq k_0+9}u_j \right)\\
    & +\sum_{k\in\mathbb{Z}_N} O \left(P_{k_0+3\leq\cdot \leq k_0+6}u_k P_{\leq k_0+6}u_kP_{ k_0+6\leq\cdot\leq k_0+9}u_j \right),  \\
	\mathcal{N}_j^{(1,2)}  =  & \sum_{k\in\mathbb{Z}_N} P_{\leq k _0+ 9}O \left(P_{\leq k_0+3}u_k P_{\leq k_0+3}u_kP_{\geq k_0+6}u_j \right)
+\sum_{k\in\mathbb{Z}_N} O \left(P_{\leq k_0+3}u_k P_{\leq k_0+3}u_kP_{ k_0+6\leq\cdot\leq k_0+9}u_j \right). 
\end{align*}
Using similar calculations as in 
 \eqref{f7.24}, we obtain 
\begin{align}\nonumber
&\left|\sum_{j\in\mathbb{Z}_N}\int_{a}^{b} \int \phi(r) \Re \left[\overline{P_{\leq k _0+ 9} u} \partial_{r} \mathcal N^{(1,1)} \right](t,x)  \, \mathrm{d} x  \mathrm{d} t
 - \sum_{j\in\mathbb{Z}_N}\int_{a}^{b} \int \phi(r) \Re \left[\overline{\mathcal N^{(1,1)}} \partial_{r} P_{\leq k _0+ 9} u \right](t,x) \, \mathrm{d}x \mathrm{ d} t \right| \\
 \label{f7.27}
& \lesssim 2^{k_0} R \left(\frac{1}{T} \int_{a}^{b}  \left\| \boldsymbol{\epsilon}(t)  \right\|_{L^2_{x}l^2}^{2} \lambda(t)^{-2}  \, \mathrm{d} t  + 2^{2k_0} T^{-10} \right).
\end{align}
Similarly, 
\begin{align}\nonumber
&\left|\sum_{j\in\mathbb{Z}_N}\int_{a}^{b} \int \phi(r) \Re \left[\overline{P_{k_0+3\leq\cdot\leq k _0+ 9} u} \partial_{r} \mathcal N^{(1,2)} \right](t,x) \, \mathrm{ d} x  \mathrm{d} t
- \sum_{j\in\mathbb{Z}_N}\int_{a}^{b} \int \phi(r) \Re \left[\overline{\mathcal N^{(1,2)}} \partial_{r} P_{k_0+3\leq\cdot\leq k _0+ 9} u \right](t,x) \, \mathrm{ d} x  \mathrm{d} t \right|\\\label{f7.28}
	&\lesssim 2^{k_0} R  \left(\frac{1}{T} \int_{a}^{b}  \left\| \boldsymbol{\epsilon}(t)  \right\|_{L^2_{x}l^2}^{2} \lambda(t)^{-2}  \, \mathrm{d} t  + 2^{2k_0} T^{-10} \right).
\end{align}
Finally, using Bernstein's inequality, the fact that $\phi$ is smooth, rapidly decreasing,
$R = T^{1/25}$, and $ \left\| \mathbf{u}  \right\|_{L_{t,x}^{4}l^2 \left(J \times \mathbb{R}^{2}\times\mathbb{Z}_N \right)} \lesssim T^{1/4}$, we have 
\begin{align}\label{10.10}
 & \left|\sum_{j\in\mathbb{Z}_N}\int_{a}^{b} \int \phi(r) \Re \left[\overline{u_{\leq k_0 + 3}} \partial_{r} \mathcal N^{(1,2)} \right](t,x) \, \mathrm{d} x \mathrm{d} t
 - \sum_{j\in\mathbb{Z}_N}\int_{a}^{b} \int \phi(r) \Re \left[\overline{\mathcal N^{(1,2)}} \partial_{r} u_{\leq k_0 + 3} \right](t,x) \, \mathrm{d} x  \mathrm{d} t\right| \notag\\
 & \lesssim 2^{k_0} \left\| P_{\geq k_0 + 3} \phi(x)  \right\|_{L^{\infty}_x }  \left\|  \left|P_{\geq k_0+6}\mathbf{u} \right| \left|P_{\leq k_0+3}\mathbf{u} \right|  \right\|_{L_{t,x}^{2}l^2}
 \left\| P_{\leq k_0+3}\mathbf{u}  \right\|_{L_{t,x}^{4}l^2}^{2}\notag\\
 & \lesssim  2^{-100k_0} \left\| P_{\geq k_0 + 3} |\nabla|^{99}\phi(x)  \right\|_{L^{\infty}_x }T^{1/2} \left(\eta_0+2^{k_0}T^{-5} \right)
  \lesssim  {T^{-9}}.
\end{align}
Combining \eqref{f7.22}, \eqref{f7.27},\eqref{f7.28}, \eqref{10.10} and using the equality $2^{3k_0}=T$, we obtain 
\begin{align}\label{estq41}
	\int_a^b\eqref{estq4} \, \mathrm{d} t
\lesssim 2^{k_0} R \left(\frac{1}{T} \int_{J}  \left\| \boldsymbol{\epsilon}(t)  \right\|_{L^2_{x} l^2}^{2} \lambda(t)^{-2}  \, \mathrm{d} t+2^{2k_0}T^{-10} \right).
\end{align}

{\bf Estimate of \eqref{estq2}}:
Using the fact that $ \left(\frac{1}{r} \phi(r) - \chi^{2}(r) \right) \left(\delta_{jk} - \frac{x_{j} x_{k}}{|x|^{2}} \right)$ is a positive definite matrix, we have
\begin{align}\label{estq21}
	\eqref{estq2}\geq 0.
\end{align}

{\bf Estimate of \eqref{estq3}}:
noting that $\psi''(r)$ is supported in $[1,2]$ and
$ \left\|\chi_{|x|\geq \eta_1^{-1}R}\frac{1}{\lambda(t)}Q \left(\frac{x-x(t)}{\lambda(t)} \right) \right\|_{L_x^2}\lesssim T^{-100}$, by \eqref{f7.1}, we have 
\begin{align}\label{f7.32}
	& \frac{\eta_1^2}{R^{2}}  \sum_{j\in\mathbb{Z}_N}\int \psi'' \left(\frac{\eta_1 x}{R} \right)
\left|P_{\leq k _0+ 9}  \left(e^{-i \gamma_j(t)} e^{i(x-x(t)) \cdot \frac{\xi(t)}{\lambda(t)}} \frac{1}{\lambda(t)}Q \left(\frac{x - x(t)}{\lambda(t)} \right) \right) \right|^{2} \, \mathrm{d} x \notag \\
 &  \lesssim	\frac{\eta_1^2}{R^{2}}T^{-10}\lesssim \frac{1}{\lambda(t)^2}T^{-10}.
\end{align}
On the other hand, by \eqref{f7.1}, for  sufficiently large $T$, we have
\begin{align}\label{f7.33}
	& \sum_{j\in\mathbb{Z}_N}\frac{\eta_1^2}{R^{2}} \int \psi'' \left(\frac{\eta_1 x}{R} \right)  \left|P_{\leq k _0+ 9} \left(e^{-i \gamma_j(t)} e^{i(x-x(t)) \cdot \frac{\xi(t)}{\lambda(t)}}\frac{1}{\lambda(t)} \epsilon_j \left(t, \frac{x - x(t)}{\lambda(t)} \right) \right)  \right|^{2} \, \mathrm{d} x \notag \\
& \lesssim  \frac{\eta_1^2}{R^2} \left\| \boldsymbol{\epsilon}  \right\|_{L^2_{x} l^2}^{2}
\lesssim \frac{1}{R\lambda(t)^2}  \left\| \boldsymbol{\epsilon}  \right\|_{L^2_{x} l^2}^{2}.
\end{align}
Since
\begin{align*}
P_{\leq k_0+9}u_j(t)= & P_{\leq k_0+9} \left(e^{-i\gamma_j(t)}e^{-i(x-x(t))\cdot\frac{\xi(t)}{\lambda(t)}}\frac{1}{\lambda(t)}Q \left(\frac{x-x(t)}{\lambda(t)} \right) \right)\\
& +P_{\leq k_0+9} \left(e^{-i\gamma_j(t)}e^{-i(x-x(t))\cdot\frac{\xi(t)}{\lambda(t)}}\frac{1}{\lambda(t)} \epsilon_j \left(t,\frac{x-x(t)}{\lambda(t)} \right) \right),\ \forall \,  j\in\mathbb{Z}_N,
\end{align*}
 \eqref{f7.32} and \eqref{f7.33} imply that
\begin{align}\label{estq31}
	\eqref{estq3}\lesssim \frac{1}{\lambda(t)^2}T^{-10}+\frac{1}{R\lambda(t)^2} \left\|\boldsymbol{\epsilon}(t) \right\|^2_{L^2_x l^2} .
\end{align}
Therefore, to complete the proof of Theorem \ref{t7.13}, it remains to obtain a proper lower bound for \eqref{estq41}, i.e.
\begin{align*}
\aligned
& 2\sum_{j\in\mathbb{Z}_N}  \int \chi^{2} \left(\frac{\eta_1 x}{R} \right)  \left|\nabla P_{\leq k _0+ 9} u_j(t,x) \right|^{2} \, \mathrm{d} x 
-\sum_{j\in\mathbb{Z}_N} \int \chi^{2} \left(\frac{\eta_1 x}{R} \right)  \left|P_{\leq k _0+ 9} u_j(t,x) \right|^{4} \, \mathrm{d} x \\
& -2\sum_{j\in\mathbb{Z}_N}\sum_{\substack{k\in\mathbb{Z}_N\\k\neq j}}  \int \chi^{2} \left(\frac{\eta_1 x}{R} \right)  \left|P_{\leq k _0+ 9} u_j(t,x) \right|^{2} \left|P_{\leq k _0+ 9} u_k(t,x)\right|^{2}  \, \mathrm{d} x.
\endaligned
\end{align*}
Since $Q$ is smooth and all its derivatives are rapidly decreasing,  one can easily verify that
\begin{align*}
&\sum_{j\in\mathbb{Z}_N}\frac{1}{2} \int  \left(1 - \chi^{2} \left(\frac{\eta_1 x}{2R} \right) \right)
\left|\nabla P_{\leq k _0+ 9}  \left(e^{-i \gamma_j(t)} e^{-i(x-x(t))
\cdot \frac{\xi(t)}{\lambda(t)}} \frac{1}{\lambda(t)} Q \left(\frac{x - x(t)}{\lambda(t)} \right)  \right) \right|^{2} \, \mathrm{d} x \\
&- \frac{1}{4}\sum_{j\in\mathbb{Z}_N} \int  \left(1 - \chi^{2} \left(\frac{\eta_1 x}{2R} \right) \right)
\left|P_{\leq k _0+ 9}  \left(e^{-i \gamma_j(t)} e^{-i(x-x(t)) \cdot \frac{\xi(t)}{\lambda(t)}} \frac{1}{\lambda(t)} Q \left(\frac{x - x(t)}{\lambda(t)} \right)  \right) \right|^{4} \, \mathrm{d} x \\
&- \frac{1}{2}\sum_{j\in\mathbb{Z}_N}\sum_{\substack{k\in\mathbb{Z}_N\\k\neq j}} \int \left(1 - \chi^{2} \left(\frac{\eta_1 x}{2R} \right) \right)
\left|P_{\leq k _0+ 9} \left(e^{-i \gamma_j(t)} e^{-i(x-x(t)) \cdot \frac{\xi(t)}{\lambda(t)}} \frac{1}{\lambda(t)} Q \left(\frac{x - x(t)}{\lambda(t)} \right) \right) \right|^{2}
  \\
& \qquad
  \cdot \left|P_{\leq k _0+ 9}  \left(e^{-i \gamma_k(t)} e^{-i(x-x(t)) \cdot \frac{\xi(t)}{\lambda(t)}} \frac{1}{\lambda(t)} Q \left(\frac{x - x(t)}{\lambda(t)} \right) \right) \right|^{2} \, \mathrm{d} x \lesssim T^{-10}, \\
&\sum_{j\in\mathbb{Z}_N}\Re\int  \left(1 - \chi^{2} \left(\frac{\eta_1 x}{2R} \right) \right)
\overline{\nabla P_{\leq k _0+ 9}  \left(e^{-i \gamma_j(t)} e^{-i(x-x(t)) \cdot \frac{\xi(t)}{\lambda(t)}} \frac{1}{\lambda(t)} Q \left(\frac{x - x(t)}{\lambda(t)} \right) \right)}  \\
& \qquad  \cdot \nabla P_{\leq k _0+ 9}  \left(e^{-i \gamma_j(t)} e^{-i(x-x(t)) \cdot \frac{\xi(t)}{\lambda(t)}} \frac{1}{\lambda(t)} \epsilon_j \left(t, \frac{x - x(t)}{\lambda(t)} \right) \right) \, \mathrm{d} x \lesssim T^{-10}
\end{align*}
and
\begin{align*}
&-\Re\int  \left(1 - \chi^{2} \left(\frac{\eta_1 x}{2R} \right) \right)
\left|P_{\leq k _0+ 9}  \left(e^{-i \gamma_j(t)} e^{-i(x-x(t))
\cdot \frac{\xi(t)}{\lambda(t)}} \frac{1}{\lambda(t)} Q \left(\frac{x - x(t)}{\lambda(t)} \right) \right) \right|^{2} \\
&\hspace{2ex} \cdot   \overline{ P_{\leq k _0+ 9}
\left(e^{-i \gamma_k(t)} e^{-i(x-x(t)) \cdot \frac{\xi(t)}{\lambda(t)}} \frac{1}{\lambda(t)} Q \left(\frac{x - x(t)}{\lambda(t)} \right) \right) }
 P_{\leq k _0+ 9}  \left(e^{-i \gamma_k(t)} e^{-i(x-x(t)) \cdot \frac{\xi(t)}{\lambda(t)}} \frac{1}{\lambda(t)} \epsilon_k \left(t, \frac{x - x(t)}{\lambda(t)} \right) \right) \, \mathrm{d} x \\
&   \lesssim T^{-10}, \quad\forall \,  j,k\in\mathbb{Z}_N.
\end{align*}
Therefore, from \eqref{f6.18}, we have 
\begin{align}\label{f7.39}
&\frac{1}{2} \sum_{j\in\mathbb{Z}_N}\int \chi^{2} \left(\frac{\eta_1 x}{2R} \right)
\left|\nabla P_{\leq k _0+ 9}
\left(e^{-i \gamma_j(t)} e^{-i(x-x(t))\cdot \frac{\xi(t)}{\lambda(t)}} \frac{1}{\lambda(t)} Q \left(\frac{x - x(t)}{\lambda(t)}  \right) \right) \right|^{2} \, \mathrm{ d} x  \\
 &- \frac{1}{4} \sum_{j\in\mathbb{Z}_N}
 \int  \chi^{2} \left(\frac{\eta_1 x}{2R} \right)
 \left|P_{\leq k _0+ 9}  \left(e^{-i \gamma_j(t)} e^{-i(x-x(t))\cdot \frac{\xi(t)}{\lambda(t)}} \frac{1}{\lambda(t)} Q \left(\frac{x - x(t)}{\lambda(t)}  \right)  \right) \right|^{4}  \, \mathrm{d} x \notag\\
&-\frac{1}{2}\sum_{j\in\mathbb{Z}_N}\sum_{\substack{k\in\mathbb{Z}_N\\k\neq j}}\int  \chi^{2} \left(\frac{\eta_1 x}{2R} \right)
 \left|P_{\leq k _0+ 9}  \left(e^{-i \gamma_k(t)} e^{-i(x-x(t))\cdot \frac{\xi(t)}{\lambda(t)}} \frac{1}{\lambda(t)} Q \left(\frac{x - x(t)}{\lambda(t)} \right) \right) \right|^{2} \notag\\
& \quad
 \cdot \left|P_{\leq k _0+ 9} \left(e^{-i \gamma_j(t)} e^{-i(x-x(t))\cdot \frac{\xi(t)}{\lambda(t)}} \frac{1}{\lambda(t)} Q \left(\frac{x - x(t)}{\lambda(t)} \right)  \right)   \right|^{2} \, \mathrm{ d} x\notag\\
&+ \sum_{j\in\mathbb{Z}_N}\Re\int  \chi^{2} \left(\frac{\eta_1 x}{2R} \right)
\overline{\nabla P_{\leq k _0+ 9}
 \left(e^{-i \gamma_j(t)} e^{-i(x-x(t))  \cdot \frac{\xi(t)}{\lambda(t)}}  \frac{1}{\lambda(t)} Q \left(\frac{x - x(t)}{\lambda(t)} \right) \right)}
\notag \\
& \quad \cdot \nabla P_{\leq k _0+ 9} \left(e^{-i \gamma_j(t)} e^{-i(x-x(t))\cdot \frac{\xi(t)}{\lambda(t)}} \frac{1}{\lambda(t)} \epsilon_j \left(t, \frac{x - x(t)}{\lambda(t)} \right) \right)  \, \mathrm{d} x \notag\\
&-  \sum_{j\in\mathbb{Z}_N}\Re\int \chi^{2} \left(\frac{\eta_1 x}{2R} \right)
\left|P_{\leq k _0+ 9} \left(e^{-i \gamma_j(t)} e^{-i(x-x(t))\cdot \frac{\xi(t)}{\lambda(t)}} \frac{1}{\lambda(t)} Q \left(\frac{x - x(t)}{\lambda(t)} \right)  \right) \right|^{2} \notag\\
& \quad \cdot \overline{ P_{\leq k _0+ 9} \left(e^{-i \gamma_j(t)} e^{-i(x-x(t))\cdot \frac{\xi(t)}{\lambda(t)}} \frac{1}{\lambda(t)} Q \left(\frac{x - x(t)}{\lambda(t)} \right)  \right) }
 \cdot P_{\leq k _0+ 9} \left(e^{-i \gamma_j(t)} e^{-i(x-x(t))\cdot \frac{\xi(t)}{\lambda(t)}} \frac{1}{\lambda(t)} \epsilon_j \left(t, \frac{x - x(t)}{\lambda(t)} \right) \right) \, \mathrm{d} x \notag\\
&-  \frac{1}{2}\sum_{j\in\mathbb{Z}_N}\sum_{\substack{k\in\mathbb{Z}_N\\k\neq j}}\Re\int \chi^{2} \left(\frac{\eta_1 x}{2R} \right)
 \left|P_{\leq k _0+ 9}  \left(e^{-i \gamma_k(t)} e^{-i(x-x(t))\cdot \frac{\xi(t)}{\lambda(t)}} \frac{1}{\lambda(t)} Q \left(\frac{x - x(t)}{\lambda(t)} \right)  \right) \right|^{2} \notag\\
  &\quad \cdot
  \overline{ P_{\leq k _0+ 9} \left(e^{-i \gamma_j(t)} e^{-i(x-x(t))\cdot \frac{\xi(t)}{\lambda(t)}} \frac{1}{\lambda(t)} Q \left(\frac{x - x(t)}{\lambda(t)} \right) \right) }
  \cdot  P_{\leq k _0+ 9} \left(e^{-i \gamma_j(t)} e^{-i(x-x(t))\cdot \frac{\xi(t)}{\lambda(t)}} \frac{1}{\lambda(t)} \epsilon_j \left(t, \frac{x - x(t)}{\lambda(t)} \right) \right) \, \mathrm{d} x \notag\\
&-  \frac{1}{2}\sum_{j\in\mathbb{Z}_N}\sum_{\substack{k\in\mathbb{Z}_N\\k\neq j}}
\Re\int \chi^{2} \left(\frac{\eta_1 x}{2R} \right)
\left|P_{\leq k _0+ 9} \left(e^{-i \gamma_j(t)} e^{-i(x-x(t))\cdot \frac{\xi(t)}{\lambda(t)}} \frac{1}{\lambda(t)} Q \left(\frac{x - x(t)}{\lambda(t)} \right)   \right) \right|^{2} \notag\\
&\quad \cdot
\overline{ P_{\leq k _0+ 9} \left(e^{-i \gamma_j(t)} e^{-i(x-x(t))\cdot \frac{\xi(t)}{\lambda(t)}} \frac{1}{\lambda(t)} Q \left(\frac{x - x(t)}{\lambda(t)} \right)  \right) }
\cdot  P_{\leq k _0+ 9} \left(e^{-i \gamma_j(t)} e^{-i(x-x(t))\cdot \frac{\xi(t)}{\lambda(t)}} \frac{1}{\lambda(t)} \epsilon_j \left(t, \frac{x - x(t)}{\lambda(t)} \right) \right) \, \mathrm{ d} x \notag\\
&= \frac{N}{2} \frac{|\xi(t)|^{2}}{\lambda(t)^{2}} \| Q \|_{L^2_{x}}^{2} + \frac{1}{2 \lambda(t)^{2}}
\left\| \boldsymbol{\epsilon}(t) \right\|_{L^2_{x} l^2}^{2} - \frac{|\xi(t)|^{2}}{2 \lambda(t)^{2}}  \left\| \boldsymbol{\epsilon}(t)  \right\|_{L^2_{x} l^2}^{2} + O \left(2^{2k_0} T^{-10} \right). \notag
\end{align}
Next, we consider the terms involving two $\boldsymbol{\epsilon}$'s. On one side, by the product rule, for any $j\in\mathbb{Z}_N$,
\begin{align*}
& \frac{1}{2} \int \chi^{2} \left(\frac{\eta_1 x}{2R} \right)
\left|\nabla P_{\leq k _0+ 9} \left(e^{-i \gamma_j(t)} e^{-i(x-x(t))\cdot \frac{\xi(t)}{\lambda(t)}} \frac{1}{\lambda(t)} \epsilon_j \left(t, \frac{x - x(t)}{\lambda(t)} \right) \right) \right|^{2} \, \mathrm{d} x \\
 =& \frac{1}{2}
\left\| \chi \left(\frac{\eta_1 x}{2R} \right) P_{\leq k _0+ 9} \left(e^{-i \gamma_j(t)} e^{-i(x-x(t))\cdot \frac{\xi(t)}{\lambda(t)}} \frac{1}{\lambda(t)} \epsilon_j
\left(t, \frac{x - x(t)}{\lambda(t)} \right) \right)  \right\|_{\dot{H}_x^{1}}^{2}
 + O \left(\frac{1}{R \lambda(t)^{2}}  \| \epsilon_j \|_{L^2_{x}}^{2} \right).
\end{align*}
On the other side, \eqref{f7.1} and the fact that $Q_{x_{j}}$ and $\chi_{0}$ are rapidly decreasing imply 
\begin{equation*}
\left\langle \chi \left(\frac{\eta_1  \left(\lambda(t)x + x(t) \right)}{2R} \right) \boldsymbol{\epsilon}, \mathbf{f}
\right\rangle_{L^2_{x}l^2} \lesssim T^{-10},\quad \forall\ f \in \left\{i\boldsymbol{\chi}_{0,1},\cdots,i\boldsymbol{\chi}_{0,N}, \boldsymbol{\chi}_0, i\mathbf{Q}_{x_1}, i\mathbf{Q}_{x_2}, \mathbf{Q}_{x_1}, \mathbf{Q}_{x_2} \right\}.
\end{equation*}
Therefore, following the analysis in \eqref{f6.22}-\eqref{f6.30}, we obtain 
\begin{align}\label{f7.42}
&\frac{1}{2} \sum_{j\in\mathbb{Z}_N}\int \chi^2 \left(\frac{\eta_1 x}{2R} \right)
\left|\nabla P_{\leq k _0+ 9}  \left(e^{-i \gamma_j(t)} e^{-i(x-x(t))\cdot \frac{\xi(t)}{\lambda(t)}} \frac{1}{\lambda(t)} \epsilon_j \left(t, \frac{x - x(t)}{\lambda(t)} \right) \right) \right|^{2}  \, \mathrm{d} x  \\
&- \sum_{j\in\mathbb{Z}_N} \int \chi^2 \left(\frac{\eta_1 x}{2R} \right)
  \left|P_{\leq k _0+ 9} \left(e^{-i \gamma_j(t)} e^{-i(x-x(t)) \cdot \frac{\xi(t)}{\lambda(t)}}\frac{1}{\lambda(t)}Q \left(\frac{x - x(t)}{\lambda(t)} \right)  \right) \right|^{2} \notag \\
  & \quad \cdot
 \left|P_{\leq k _0+ 9} \left(e^{-i \gamma_j(t)} e^{-i(x-x(t))\cdot \frac{\xi(t)}{\lambda(t)}} \frac{1}{\lambda(t)} \epsilon_j \left(t,\frac{x - x(t)}{\lambda(t)} \right) \right) \right|^{2}  \, \mathrm{d} x \notag\\
&- \frac{1}{2}\sum_{j\in\mathbb{Z}_N}
\Re \int \chi^2 \left(\frac{\eta_1 x}{2R} \right) \left(\overline{P_{\leq k _0+ 9} \left(e^{-i \gamma_j(t)} e^{-i(x-x(t)) \cdot \frac{\xi(t)}{\lambda(t)}}\frac{1}{\lambda(t)}Q \left(\frac{x - x(t)}{\lambda(t)} \right) \right)} \right)^{2}
 \notag \\
 & \quad  \cdot  \left(P_{\leq k _0+ 9}  \left(e^{-i \gamma_j(t)} e^{-i(x-x(t))\cdot \frac{\xi(t)}{\lambda(t)}} \frac{1}{\lambda(t)} \epsilon_j \left(t,\frac{x - x(t)}{\lambda(t)} \right) \right) \right)^{2}  \, \mathrm{d} x \notag\\
&-\frac{1}{2}\sum_{j\in\mathbb{Z}_N}\sum_{\substack{k\in\mathbb{Z}_N\\k\neq j}} \int \chi^2 \left(\frac{\eta_1 x}{2R} \right)
\left|P_{\leq k _0+ 9}
\left(e^{-i \gamma_k(t)} e^{-i(x-x(t)) \cdot \frac{\xi(t)}{\lambda(t)}}\frac{1}{\lambda(t)}Q \left(\frac{x - x(t)}{\lambda(t)} \right) \right) \right|^{2}
 \notag \\
 & \quad \cdot  \left|P_{\leq k _0+ 9} \left(e^{-i \gamma_j(t)} e^{-i(x-x(t))\cdot \frac{\xi(t)}{\lambda(t)}} \frac{1}{\lambda(t)} \epsilon_j \left(t,\frac{x - x(t)}{\lambda(t)} \right) \right) \right|^{2} \, \mathrm{ d} x\notag\\
&-\frac{1}{2}\sum_{j\in\mathbb{Z}_N}\sum_{\substack{k\in\mathbb{Z}_N\\k\neq j}} \int \chi^2 \left(\frac{\eta_1 x}{2R} \right)
 \left|P_{\leq k _0+ 9} \left(e^{-i \gamma_j(t)} e^{-i(x-x(t)) \cdot \frac{\xi(t)}{\lambda(t)}}\frac{1}{\lambda(t)}Q \left(\frac{x - x(t)}{\lambda(t)} \right)  \right) \right|^{2} \notag \\
 & \quad \cdot
\left|P_{\leq k _0+ 9}  \left(e^{-i \gamma_k(t)} e^{-i(x-x(t))\cdot \frac{\xi(t)}{\lambda(t)}} \frac{1}{\lambda(t)} \epsilon_k \left(t,\frac{x - x(t)}{\lambda(t)} \right) \right) \right|^{2}  \, \mathrm{d} x\notag\\
&-\sum_{j\in\mathbb{Z}_N}\sum_{\substack{k\in\mathbb{Z}_N\\k\neq j}}
\Re\int \chi^2 \left(\frac{\eta_1 x}{2R} \right)
P_{\leq k _0+ 9} \left(e^{-i \gamma_j(t)} e^{-i(x-x(t)) \cdot \frac{\xi(t)}{\lambda(t)}}\frac{1}{\lambda(t)}Q \left(\frac{x - x(t)}{\lambda(t)} \right) \right)
\notag \\
& \quad \cdot \overline{P_{\leq k _0+ 9} \left(e^{-i \gamma_j(t)} e^{-i(x-x(t)) \cdot \frac{\xi(t)}{\lambda(t)}}\frac{1}{\lambda(t)} \epsilon_j \left(t,\frac{x - x(t)}{\lambda(t)} \right) \right)}\notag\\
&  \quad  \cdot P_{\leq k _0+ 9} \left(e^{-i \gamma_k(t)} e^{-i(x-x(t)) \cdot \frac{\xi(t)}{\lambda(t)}}\frac{1}{\lambda(t)}Q \left(\frac{x - x(t)}{\lambda(t)} \right) \right)
\overline{P_{\leq k _0+ 9}
\left(e^{-i \gamma_k(t)} e^{-i(x-x(t)) \cdot \frac{\xi(t)}{\lambda(t)}}\frac{1}{\lambda(t)} \epsilon_k \left(t,\frac{x - x(t)}{\lambda(t)} \right) \right)}
\, \mathrm{d} x\notag\\
&-\sum_{j\in\mathbb{Z}_N}\sum_{\substack{k\in\mathbb{Z}_N\\k\neq j}} \Re\int \chi^2 \left(\frac{\eta_1 x}{2R} \right)
P_{\leq k _0+ 9}
\left(e^{-i \gamma_j(t)} e^{-i(x-x(t)) \cdot \frac{\xi(t)}{\lambda(t)}}\frac{1}{\lambda(t)}Q \left(\frac{x - x(t)}{\lambda(t)} \right) \right)
\notag \\
& \quad \cdot \overline{P_{\leq k _0+ 9} \left(e^{-i \gamma_j(t)} e^{-i(x-x(t)) \cdot \frac{\xi(t)}{\lambda(t)}}\frac{1}{\lambda(t)} \epsilon_j \left(\frac{x - x(t)}{\lambda(t)} \right) \right)}\notag\\
& \quad  \cdot  P_{\leq k _0+ 9} \left(e^{-i \gamma_k(t)} e^{-i(x-x(t)) \cdot \frac{\xi(t)}{\lambda(t)}}\frac{1}{\lambda(t)} \epsilon_k
\left(\frac{x - x(t)}{\lambda(t)} \right) \right)\overline{P_{\leq k _0+ 9}
\left(e^{-i \gamma_k(t)} e^{-i(x-x(t)) \cdot \frac{\xi(t)}{\lambda(t)}}\frac{1}{\lambda(t)}Q \left(\frac{x - x(t)}{\lambda(t)} \right) \right)} \, \mathrm{d} x\notag\\
&\geq
\left(\frac{\lambda_{1}}{\lambda(t)^{2}} -\frac{1}{2\lambda(t)^{2}} \right)
\left\| \boldsymbol{\epsilon}  \right\|_{L^2_{x}l^2}^{2}
+ \lambda_{1} \sum_{j\in\mathbb{Z}_N}\int \chi \left(\frac{\eta_1 x}{2R} \right)^{2}
\left|\nabla P_{\leq k _0+ 9} \left( e^{-i \gamma_j(t)} e^{-i(x-x(t))\cdot \frac{\xi(t)}{\lambda(t)}} \frac{1}{\lambda(t)} \epsilon_j \left(t, \frac{x - x(t)}{\lambda(t)} \right) \right)  \right|^{2}  \, \mathrm{d} x \notag \\
&\quad  - O \left(2^{2k_0} T^{-10} \right) \notag
\end{align}
with a constant $0<\lambda_1\ll \frac{1}{2}$. Next, we observe that for any $g(x), h(x)\in L_x^2(\mathbb{R}^2)$, by Sobolev embedding and H\"older's inequality, we have 
\begin{align*}
& \int \chi^2 \left(\frac{\eta_1 x}{2R} \right)  |g(x_{1}, x_{2})|^{2} |h(x_1,x_2)|^2  \, \mathrm{d} x_{1}  \, \mathrm{d} x_{2} \notag \\
\leq& \int \int \left|\partial_{x_{2}} \left(\chi \left(\frac{\eta_1 x}{2R} \right) |(gh)(x_{1}, x_{2})|^{2} \right) \right|  \, \mathrm{d} x_{1} \mathrm{ d} x_{2}
\cdot \int \int  \left|\partial_{x_{1}} \left(\chi \left(\frac{\eta_1 x}{2R} \right) |(gh)(x_{1}, x_{2})|^{2} \right) \right|  \, \mathrm{d} x_{1} \mathrm{ d} x_{2} \\
\lesssim&  \left\| \chi \left(\frac{\eta_1 x}{2R} \right) \nabla g  \right\|_{L^2_{x}}^{2} \| h \|_{L^2_{x}}^{2}
+ \left\| \chi \left(\frac{\eta_1 x}{2R} \right) \nabla h  \right\|_{L^2_{x}}^{2} \| g \|_{L^2_{x}}^{2}
+ \frac{\eta_1^2}{R^{2}} \| g \|_{L^2_{x}}^{2}\| h \|_{L^2_{x}}^{2}.
\end{align*}
Thus,
\begin{align}\nonumber
& \int \chi^2 \left(\frac{\eta_1 x}{2R} \right)
\left|P_{\leq k _0+ 9} \left(e^{-i \gamma_k(t)} e^{-i(x-x(t))\cdot \frac{\xi(t)}{\lambda(t)}} \frac{1}{\lambda(t)} \epsilon_k \left(t, \frac{x - x(t)}{\lambda(t)} \right) \right) \right|^{2}
\\\nonumber
& \quad \cdot \left|P_{\leq k _0+ 9} \left(e^{-i \gamma_j(t)} e^{-i(x-x(t))\cdot \frac{\xi(t)}{\lambda(t)}} \frac{1}{\lambda(t)} \epsilon_j \left(t, \frac{x - x(t)}{\lambda(t)} \right) \right) \right|^{2} \, \mathrm{d} x \\\nonumber
 \lesssim& \eta_0^{2} \sum_{j\in\mathbb{Z}_N}
\left\| \chi \left(\frac{\eta_1 x}{2R} \right) \nabla P_{\leq k _0+ 9}  \left(e^{-i \gamma_j(t)} e^{-i(x-x(t))\cdot \frac{\xi(t)}{\lambda(t)}} \frac{1}{\lambda(t)} \epsilon_j \left(t, \frac{x - x(t)}{\lambda(t)} \right)  \right)  \right\|_{L^2_{x}}^{2}
+ \frac{\eta_1^2}{R^{2}}\eta_0^2
\left\| \boldsymbol{\epsilon}(t)  \right\|_{L^2_{x} l^2}^{2}\\\nonumber
 \lesssim& \eta_0^{2} \sum_{j\in\mathbb{Z}_N}
\left\| \chi \left(\frac{\eta_1 x}{2R} \right) \nabla P_{\leq k _0+ 9}
\left(e^{-i \gamma_j(t)} e^{-i(x-x(t))\cdot \frac{\xi(t)}{\lambda(t)}} \frac{1}{\lambda(t)} \epsilon_j \left(t, \frac{x - x(t)}{\lambda(t)} \right)   \right) \right\|_{L^2_{x}}^{2}\\\label{f7.43}
& \qquad + \frac{\eta_0^2}{R\lambda(t)^2}  \left\| \boldsymbol{\epsilon}(t)  \right\|_{L^2_{x} l^2}^{2},\  \forall \,  j,k\in\mathbb{Z}_N.
\end{align}
Finally, by interpolation, for any $j,k\in\mathbb{Z}_N$,
\begin{align}\label{f7.44}
&\int \chi^2 \left(\frac{\eta_1 x}{2R} \right)
\left|P_{\leq k _0+ 9} \left(e^{-i \gamma_k(t)} e^{-i(x-x(t))\cdot \frac{\xi(t)}{\lambda(t)}}
\frac{1}{\lambda(t)} \epsilon_k \left(t, \frac{x - x(t)}{\lambda(t)} \right) \right) \right|^{2}
\\
& \quad \cdot \left|P_{\leq k _0+ 9} \left(e^{-i \gamma_j(t)} e^{-i(x-x(t))\cdot \frac{\xi(t)}{\lambda(t)}} \frac{1}{\lambda(t)} \epsilon_j \left(t, \frac{x - x(t)}{\lambda(t)} \right) \right) \right| \notag\\
&\quad \cdot \left|P_{\leq k _0+ 9}  \left(e^{-i \gamma_j(t)} e^{-ix \cdot \frac{\xi(t)}{\lambda(t)}} \frac{1}{\lambda(t)} Q \left(\frac{x - x(t)}{\lambda(t)} \right) \right) \right| \, \mathrm{d} x \notag\\
&\lesssim \frac{1}{\lambda(t)} \left\|\epsilon_j(t) \right\|_{L_x^2}
\left\|\chi \left(\frac{\eta_1 x}{2R} \right)P_{\leq k _0+ 9} \left(e^{-i \gamma_k(t)} e^{-i(x-x(t))\cdot \frac{\xi(t)}{\lambda(t)}} \frac{1}{\lambda(t)} \epsilon_k
\left(t, \frac{x - x(t)}{\lambda(t)} \right) \right)  \right\|^2_{L_x^4} \notag\\
&\lesssim \frac{\eta_0}{\lambda(t)} \left\| \epsilon_k(t) \right\|_{L_x^2}
\left\| \nabla  \left(\chi \left(\frac{\eta_1 x}{2R}  \right) P_{\leq k _0+ 9}
\left(e^{-i \gamma_k(t)} e^{-i(x-x(t))\cdot \frac{\xi(t)}{\lambda(t)}} \frac{1}{\lambda(t)} \epsilon_k \left(t, \frac{x - x(t)}{\lambda(t)} \right) \right) \right)  \right\|_{L^2_{x}} \notag\\
&\lesssim \frac{\eta_0}{\lambda(t)^2} \left\| \epsilon_k(t) \right\|^2_{L_x^2}
+ \frac{\eta_0\eta_1^2}{R^2} \| \epsilon_k\|^2_{L_x^2}
+\eta_0 \left\|\chi \left(\frac{\eta_1 x}{2R} \right) \nabla P_{\leq k _0+ 9}
\left(e^{-i \gamma_k(t)} e^{-i(x-x(t))\cdot \frac{\xi(t)}{\lambda(t)}} \frac{1}{\lambda(t)} \epsilon_k \left(t, \frac{x - x(t)}{\lambda(t)} \right) \right)  \right\|_{L^2_{x}}^{2}\notag\\
&\lesssim \left(\frac{\eta_0}{R\lambda(t)^2}+\frac{\eta_0}{\lambda(t)^2} \right)
 \left\|\boldsymbol{\epsilon} \right\|^2_{L^2_x l^2}
 +\eta_0\sum_{j\in\mathbb{Z}_N} \left\|\chi \left(\frac{\eta_1 x}{2R} \right) \nabla P_{\leq k _0+ 9}
 \left(e^{-i \gamma_j(t)} e^{-i(x-x(t))\cdot \frac{\xi(t)}{\lambda(t)}} \frac{1}{\lambda(t)} \epsilon_j \left(t, \frac{x - x(t)}{\lambda(t)} \right) \right)  \right\|_{L^2_{x}}^{2}.\notag
\end{align}
By \eqref{f7.39}, \eqref{f7.42}, \eqref{f7.43}, \eqref{f7.44}  and \eqref{f7.1}, for sufficiently small $\eta_0 $ and large $T=T(\eta_1)$, there exists a constant $0<\lambda_1\ll 1$ such that
\begin{align}\label{estq11}
	\eqref{estq1}\gtrsim  \frac{\lambda_1}{\lambda(t)^2} \left\| \boldsymbol{\epsilon}(t)  \right\|_{L^2_{x} l^2}^{2}   - O \left(T^{-9} \right).
\end{align}
Collecting \eqref{estq41}, \eqref{estq21}, \eqref{estq31} and \eqref{estq11}, we have proved that for sufficiently small $\eta_0 $ and large $T=T(\eta_1)$,
\begin{align*}
	\frac{ \mathrm{d} }{ \mathrm{d} t} M(t)\gtrsim \frac{1}{\lambda(t)^2}
\left\| \boldsymbol{\epsilon}(t)  \right\|_{L^2_{x} l^2}^{2}   - O \left(T^{-9} \right) .
\end{align*}
Integrating both sides of the above inequality and using \eqref{f7.22}, along with the fact that $|J|\lesssim_{\eta_1} T^{1/50}$, we obtain 
 \begin{align*}
 	 & \int_{a}^{b}  \left\| \boldsymbol{\epsilon}(t)  \right\|_{L^2_{x} l^2}^{2} \lambda(t)^{-2} \, \mathrm{d} t
 \\
  \leq& 3 \sum_{j\in\mathbb{Z}_N}  \left\langle w_j(a), Q + x \cdot \nabla Q \right\rangle_{L^2_{x}} - 3\sum_{j\in\mathbb{Z}_N}
 \left\langle w_j(b), Q + x \cdot \nabla Q \right\rangle_{L^2_{x}} + \frac{T^{1/15}}{\eta_{1}^{2}} \sup_{t \in J} \frac{|\xi(t)|^{2}}{\lambda(t)^{2}} + O \left( T^{-8} \right).
 \end{align*}
Thus, we can conclude the proof.
\end{proof}
Using H\"older's inequality and \eqref{f7.2}, we get
	\begin{equation*}
		\int_{a}^{b} \left\| \boldsymbol{\epsilon}(t) \right\|_{L^2_{x} l^2}^{2} \lambda(t)^{-2} \, \mathrm{d} t
		\lesssim \|\boldsymbol{\epsilon}(a)\|_{L^2_x l^2}+\|\boldsymbol{\epsilon}(b)\|_{L^2_x l^2}
		+ \frac{T^{1/15}}{\eta_{1}^{2}} \sup_{t \in J} \frac{|\xi(t)|^{2}}{\lambda(t)^{2}} + O \left(T^{-8} \right).
	\end{equation*}
Notice that this inequality is scaling-invariant, using the change of variables \eqref{f4.23}  and then using the modulation theory, we have the following theorem:
\begin{theorem}\label{t7.14}
	Let $\mathbf{u}$ be the solution in Theorem \ref{t2.3}. Suppose $J = [a, b]$ is an interval on which
	\begin{equation*}
		\frac{\sup_{t\in J}\lambda(t)}{\inf_{t\in J}\lambda(t)}\leq T^{1/100} \text{ and }  \int_{J} \lambda(t)^{-2}  \, \mathrm{d} t = T.
	\end{equation*}
	Also, suppose $\boldsymbol{\epsilon} = \mathbf{v} + i \mathbf{w}$, and let $[s_1,s_2]=s([a,b])$. Then there exist sufficiently small $\tilde{\eta}$ and large $\tilde{T}$, 
so that for any $\eta<\tilde{\eta}$ and $T\geq\tilde{T}$ satisfying 
$$\sup_{t\in[a,b]} \left\|\boldsymbol{\epsilon}(t) \right\|_{L^2_x l^2}:=\eta\leq \tilde{\eta},\quad \eta T\ll T^{1/100}, \quad \eta^2 T\ll 1, $$
we have
	\begin{equation*}
		\int_{s_1}^{s_2}  \left\| \boldsymbol{\epsilon}(s)  \right\|_{L^2_{x} l^2}^{2}  \, \mathrm{d} s
 \lesssim  \left\| \boldsymbol{\epsilon}(s_1)  \right\|_{L^2_{x} l^2} +  \left\| \boldsymbol{\epsilon}(s_2)  \right\|_{L^2_{x} l^2}
 + \eta^{4} T^{2+1/15}  + O \left(
 {T^{-8}} \right).
	\end{equation*}
\end{theorem}
\begin{proof}
	Let $\eta_1,\eta_0$ be the same constants  as in Theorem \ref{t5.9} and let $\eta_1^{-1}\min_{t\in [a,b]}\lambda(t):=\tilde{\lambda}$. Notice that $$e^{i\gamma(t)}e^{ix\cdot\xi(t)}{\tilde{\lambda}}^{-1}{\lambda(t)}  \mathbf{u}_{\tilde{\lambda}} \left(t,\tilde{\lambda}^{-1}{\lambda(t)}x+x(t) \right) = \boldsymbol{\epsilon} (t)+Q, $$
$ \mathbf{u}_{\tilde{\lambda}}(t,x):=\tilde{\lambda} \mathbf{u}  \left(t,\tilde{\lambda}x \right)$ as while as $ \mathbf{v}_{\tilde{\lambda}}(t):= \mathbf{u}_{\tilde{\lambda}} \left(\tilde{\lambda}^2 t,x \right)$ is also the solution of \eqref{1.1},
then  there exist parameters
$\gamma_{1,1}(t),\cdots, \gamma_{1,N}, x_1(t), \xi_1(t)$ such that
	\begin{align*}
		&e^{i\gamma_{1,j}(t)}e^{ix\cdot\xi_1(t)}v_{\tilde{\lambda}} \left(t,\tilde{\lambda}^{-1}{\lambda \left(\tilde{\lambda}^2t \right)}x+x_1(t) \right)
 = \tilde{\epsilon}_j(t)+Q,
 \quad \forall \,  t\in \left[\tilde{\lambda}^{-2}a,\tilde{\lambda}^{-2}b \right], \quad\forall \,  j\in\mathbb{Z}_N, \end{align*}
 with $ \xi_1 \left(\tilde{\lambda}^{-2}a \right)=x_1 \left(\tilde{\lambda}^{-2}b \right)=0$, 
 and 
 \begin{align*}
\left\|\mathbf{\tilde{\boldsymbol{\epsilon}}}(t) \right\|_{L^2_x l^2}= \left\|\boldsymbol{\epsilon} \left(\tilde{\lambda}^2t \right) \right\|_{L^2_x l^2},
\quad \forall \,  t\in  \left[\tilde{\lambda}^{-2}a,\tilde{\lambda}^{-2}b \right]  \quad \mbox{and}\quad\mathbf{\tilde{\boldsymbol{\epsilon}}}(t)
 \mbox{ satisfies }\eqref{orthod}.
	\end{align*}
	Denote $\lambda_1(t)=\tilde{\lambda}^{-1}\lambda \left(\tilde{\lambda}^2t \right)$. By \eqref{f4.45}, we have 
	\begin{align*} 
		\frac{1}{\eta_1}\leq\lambda_1(t)\leq \frac{1}{\eta_1}T^{1/100},\quad \forall  \, t\in \left[\tilde{\lambda}^{-2}a,\tilde{\lambda}^{-2}b \right].
	\end{align*}
 Let $\tilde{s}(t):=\int_0^t \frac{\tilde{\lambda}^2}{\lambda \left(\tilde{\lambda}^2\tau \right)^2} \, \mathrm{d} \tau, \
 \left[s_1',s_2' \right]:=\tilde{s} \left( \left[\tilde{\lambda}^{-2}a,\tilde{\lambda}^{-2}b \right] \right)=s([a,b])$. 
 By \eqref{f4.54} and \eqref{f4.55}, we have 
	\begin{equation*}
		\sup_{t \in  \left[\tilde{\lambda}^{-2}a,\tilde{\lambda}^{-2}b \right]}  \left|\frac{\xi_1(t)}{\lambda_1(t)} \right|
=\sup_{\tilde{s} \in \left[\tilde{s}_1, \tilde{s}_2 \right]}
\left|\frac{\xi_1 \left(\tilde{s} \right)}{\lambda_1 \left(\tilde{s} \right)} \right|
\lesssim  \eta_1\eta^2T  \ll \eta_{0},
	\end{equation*}
    and 
	\begin{equation*}
		\sup_{t \in \left[\tilde{\lambda}^{-2}a,\tilde{\lambda}^{-2}b \right]}|x_1(t)|
= \sup_{\tilde{s} \in  \left[\tilde{s}_1, \tilde{s}_2  \right]} \left|x_1 \left(\tilde{s} \right) \right|
\lesssim \frac{1}{\eta_{1}^{2}} T^{\frac{1}{50}} \eta_{1} \eta^2T^2 + \frac{1}{\eta_{1}} T^{\frac{1}{100}} \eta T \ll T^{\frac{1}{25}}.
	\end{equation*}
Therefore, Theorem \ref{t7.13} can be applied  on $ \left[\tilde{\lambda}^{-2}a,\tilde{\lambda}^{-2}b \right]$, yielding 
	\begin{equation}\label{11.12}
		\int_{\tilde{\lambda}^{-2}a}^{\tilde{\lambda}^{-2}b}  \left\| \mathbf{\tilde{\boldsymbol{\epsilon}}}(t)  \right\|_{L^2_{x} l^2}^{2}
        {\lambda_1(t)^{-2}  } \, \mathrm{d} t
\lesssim \left\| \mathbf{\tilde{\boldsymbol{\epsilon}}} \left(\tilde{\lambda}^{-2}a \right)  \right\|_{L^2_{x} l^2}
+  \left\| \mathbf{\tilde{\boldsymbol{\epsilon}}} \left(\tilde{\lambda}^{-2}b \right)  \right\|_{L^2_{x} l^2} +  \eta^{4} T^{2+1/15}+ O \left({T^{-8}} \right).
	\end{equation}
Changing variables, \eqref{11.12} implies 
	\begin{equation*}
		\int_{s_1}^{s_2}  \left\| \boldsymbol{\epsilon}(s)  \right\|_{L^2_{x} l^2}^{2}  \, \mathrm{d} s
\lesssim  \left\| \boldsymbol{\epsilon}(s_1)  \right\|_{L^2_{x} l^2} +  \left\| \boldsymbol{\epsilon}(s_2)  \right\|_{L^2_{x} l^2} + \eta^{4} T^{2+1/15}
 + O \left( {T^{-8}} \right).
	\end{equation*}
This completes the proof of Theorem \ref{t7.14}. 
\end{proof}

\section{The proof of Theorem \ref{t2.3}}\label{Sec:Thm2.3}
In this section, we will complete the proof of Theorem \ref{t2.3}. We observe that the scale parameter $\lambda(t)$ is constant for solitons and monotonically decreasing to zero for pseudo-conformal transformation of solitons. In the first subsection, we will prove that this parameter is \emph{almost stable} around the orbit of the ground state $\mathbf{Q}$, that is, under the assumptions of Theorem \ref{t2.3},  the parameter $\lambda(t)$ is \emph{almost monotonically decreasing}, i.e.  $\lambda(t)\sim \inf_{t'\in[0,t]}\lambda(t')$. In the second subsection, we provide the full proof of Theorem \ref{t2.3}. For the infinite time blowup case, we use the almost monotonicity decreasing property of $\lambda(t)$ and the virial inequality  derived in the previous section to show that there exists a time sequence $t_n$ such that $\|\boldsymbol{\epsilon}(t_n)\|_{L^2_x l^2}$ decays rapidly in some sense. This subtle decay rate sharpens the  uniform bound in \eqref{f6.17}, which leads to the  disappearance of the remainder $\boldsymbol{\epsilon}$. For the finite time blowup case, we apply the  pseudo-conformal transformation to $\mathbf{u}$ to construct an infinite time blowup  solution $\mathbf{v}$ and then prove that such $\mathbf{v}$ must be the soliton.
\subsection{Monotonicity decreasing property of scaling parameter $\lambda(t)$}

First, we use Theorem \ref{t7.13} and  Theorem \ref{t7.14} to prove a regularity property of reminder $\boldsymbol{\epsilon}$. Specifically, we will prove that $ \left\| \boldsymbol{\epsilon}(s) \right\|_{L^2_{x} l^2}$ lies in $L_{s}^{p}$ for any $p > 1$.
\begin{theorem}\label{t8.15}
Let $\mathbf{u}$ be the solution in Theorem \ref{t2.3}, and suppose that
\begin{equation}\label{f8.1}
 \left\| \boldsymbol{\epsilon}(0)  \right\|_{L^2_{x} l^2} = \eta_{\ast},\quad \sup_{t\in I} \left\| \boldsymbol{\epsilon}(t)  \right\|_{L^2_{x} l^2}\leq \eta_{\ast}.
\end{equation}
Then for sufficiently small $\eta_{\ast}$,
\begin{equation}\label{f8.2}
\int_{0}^{\infty}  \left\| \boldsymbol{\epsilon}(s)  \right\|_{L^2_{x} l^2}^{2}  \, \mathrm{d} s \lesssim \eta_{\ast},
\end{equation}
where the implicit constant is independent of $\eta_{\ast}$.

Moreover, for any $j \in \mathbb{N}$, we define
\begin{equation}\label{f8.3}
s_{j} = \inf \left\{ s \in [0, \infty) :  \left\| \boldsymbol{\epsilon}(s)  \right\|_{L^2_{x} l^2} = 2^{-j} \eta_{\ast}  \right\}.
\end{equation}
By definition, $s_{0} = 0$, and by Theorem \ref{claim2.1} and Lemma \ref{l2.2}, $s_{j}$ exists for any $j > 0$. Then for each $j \in \mathbb{N}$, we have
\begin{equation}\label{f8.4}
\int_{s_{j}}^{\infty}  \left\| \boldsymbol{\epsilon}(s)  \right\|_{L^2_{x} l^2}^{2}  \, \mathrm{d} s \lesssim 2^{-j} \eta_{\ast},
\end{equation}
where the  implicit constant is independent of $j$ and $\eta_{\ast}$.
\end{theorem}

\begin{proof}
Set $T_{\ast} = \frac{1}{\eta_{\ast}}$. By \eqref{f4.45} and \eqref{f8.1}, for any $s' \geq 0$,
\begin{equation}\label{f8.5}
 \bigg|\sup_{s \in  \left[s', s' + T_{\ast} \right]} \ln (\lambda(s)) - \inf_{s \in  \left[s', s' + T_{\ast} \right]} \ln(\lambda(s)) \bigg|
  \lesssim \int_{s'}^{s'+T_{\ast}}  \left|\frac{\lambda'(s)}{\lambda(s)} \right| \, \mathrm{d} s \lesssim \eta_{\ast} T_{\ast}\lesssim 1,
\end{equation}
where the implicit constant is independent of $s' \geq 0$. Let $J$ be the largest dyadic integer satisfying 
\begin{equation}\label{f8.6}
J = 2^{j_{\ast}} \leq -(\ln (\eta_{\ast}))^{1/4}.
\end{equation}
By \eqref{f8.5} and the triangle inequality,
\begin{equation}\label{f8.7}
\aligned
\bigg|\sup_{s \in \left[s', s' + J T_{\ast} \right]} \ln(\lambda(s)) - \inf_{s \in  \left[s', s' + J T_{\ast} \right]} \ln(\lambda(s)) \bigg| \lesssim J. 
\endaligned
\end{equation}
Therefore,
\begin{equation}\label{f8.8}
\frac{\sup_{s \in  \left[s', s' + 3 J T_{\ast} \right]} \lambda(s)}{\inf_{s \in  \left[s', s' + 3JT^{\ast} \right]} \lambda(s)}
 \lesssim e^{\ln((T_{\ast})^{1/4})}\lesssim e^{\ln \left(T_{\ast}^{1/1000} \right)}\ll T_{\ast}^{1/100}.
\end{equation}
Since $\eta_{\ast}JT_{\ast}\lesssim J, $ we can choose $T_{\ast}$ sufficiently large   such that Theorem \ref{t7.14} holds. 
By Theorem \ref{t7.14}, for any $1\leq\gamma\leq3$, 
\begin{equation}\label{f8.9}
	\int_{s'}^{s'+\gamma JT_{\ast}}  \left\| \boldsymbol{\epsilon}(s)  \right\|_{L^2_{x} l^2}^{2} \, \mathrm{ d} s
 \lesssim  \left\| \boldsymbol{\epsilon} \left(s' \right)  \right\|_{L^2_{x} l^2} +  \left\| \boldsymbol{\epsilon} \left(s'+\gamma JT_{\ast} \right) \right\|_{L^2_{x} l^2}+\eta_{\ast}^{2-1/25}J^{2+1/25}  + O \left( {J^{-8} T_{\ast}^{- 8}} \right).
\end{equation}
For any $s' > J T_{\ast}$, \eqref{f8.9} implies
\begin{equation*}
\int_{s'}^{s' + J T_{\ast}}  \left\| \boldsymbol{\epsilon}(s)  \right\|_{L^2_{x} l^2}^{2}  \, \mathrm{d} s
\lesssim \inf_{s \in  \left[s' - J T_{\ast}, s' \right]}  \left\| \boldsymbol{\epsilon}(s)  \right\|_{L^2_{x} l^2} + \inf_{s \in  \left[s' + J T_{\ast}, s' + 2J T_{\ast} \right]}
\left\| \boldsymbol{\epsilon}(s)  \right\|_{L^2_{x} l^2} +  \eta_{\ast}^{2-1/25}J^{2+1/25}  + O \left( {J^{-8} T_{\ast}^{-8}} \right).
\end{equation*}
Thus, by the integral mean value theorem, $\forall \, s' \geq 0$,
\begin{align*}
& 	\int_{s'}^{s' + J T_{\ast}}  \left\| \boldsymbol{\epsilon}(s) \right\|_{L^2_{x} l^2}^{2} \,\mathrm{d}s
\\
 \lesssim&  \left\| \boldsymbol{\epsilon}  \left(s' \right) \right \|_{L^2_{x}l^2}
 + \frac{1}{J^{1/2} T_{\ast}^{1/2}} \left(\sup_{a\in\mathbb{N}} \int_{s' + a J T_{\ast}}^{s' + (a + 1) J T_{\ast}} \left\| \boldsymbol{\epsilon}(s) \right\|_{L^2_{x} l^2}^{2}
  \, \mathrm{d} s \right)^{1/2} +  \eta_{\ast}^{2-1/25}J^{2+1/25}  + O \left( {J^{-8} T_{\ast}^{-8}} \right),
\end{align*}
and 
\begin{align*}
& \sup_{a\in\mathbb{N}_+} \int_{s' + a J T_{\ast}}^{s' + (a + 1) J T_{\ast}} \left\| \boldsymbol{\epsilon}(s) \right\|_{L^2_{x} l^2}^{2} \,\mathrm{d}s
 \\
  \lesssim &\frac{1}{J^{1/2} T_{\ast}^{1/2}} \left(\sup_{a \in \mathbb{N}} \int_{s' + a J T_{\ast}}^{s' + (a + 1) J T_{\ast}}
 \left\| \boldsymbol{\epsilon}(s) \right\|_{L^2_{x} l^2}^{2} \, \mathrm{d} s \right)^{1/2} +   \eta_{\ast}^{2-1/25}J^{2+1/25}  + O \left( {J^{-8} T_{\ast}^{-8}} \right).
\end{align*}
Thus,  
\begin{align}\label{qmzy}
	& \sup_{a\in\mathbb{N}} \int_{s' + a J T_{\ast}}^{s' + (a + 1) J T_{\ast}} \left\| \boldsymbol{\epsilon}(s) \right\|_{L^2_{x} l^2}^{2} \,\mathrm{d}s
\\\nonumber
 \lesssim& \left\|\boldsymbol{\epsilon} \left(s' \right) \right\|_{L^2_x l^2}+\frac{1}{J^{1/2} T_{\ast}^{1/2}}
\left(\sup_{a \in \mathbb{N}} \int_{s' + a J T_{\ast}}^{s' + (a + 1) J T_{\ast}}  \left\| \boldsymbol{\epsilon}(s) \right\|_{L^2_{x} l^2}^{2}  \, \mathrm{d} s \right)^{1/2}
+   \eta_{\ast}^{2-1/25}J^{2+1/25}  + O \left( {J^{-8} T_{\ast}^{-8}} \right).
\end{align}
Taking $s' = s_{j_{\ast}}$ and applying the Cauchy-Schwarz inequality, for sufficiently small $\eta_{\ast}$, we have
\begin{equation}\label{f8.14}
\sup_{a \in \mathbb{N}} \int_{s_{j_{\ast}} + a J T_{\ast}}^{s_{j_{\ast}} + (a + 1) J T_{\ast}}  \left\| \boldsymbol{\epsilon}(s) \right\|_{L^2_{x} l^2}^{2} \, \mathrm{d} s
 \lesssim 2^{-j_{\ast}} \eta_{\ast} +  \eta_{\ast}^{2-1/25}J^{2+1/25}  + O \left( {J^{-8} T_{\ast}^{-8}} \right)\lesssim 2^{-j_{\ast}}\eta_{\ast}.
\end{equation}
By the triangle inequality, \eqref{f8.14} implies
\begin{equation*}
\sup_{s' \geq s_{j_{\ast}}} \int_{s'}^{s' + J T_{\ast}}  \left\| \boldsymbol{\epsilon}(s) \right\|_{L^2_{x} l^2}^{2}  \, \mathrm{ d} s \lesssim 2^{-j_{\ast}} \eta_{\ast},
\end{equation*}
and by H{\"o}lder's inequality,
\begin{equation*}
\sup_{s' \geq s_{j_{\ast}}} \int_{s'}^{s' + J T_{\ast}}  \left\| \boldsymbol{\epsilon}(s) \right\|_{L^2_{x} l^2} \, \mathrm{d} s \lesssim 1.
\end{equation*}
Repeating this argument, Theorem $\ref{t8.15}$ can be proved by induction. Indeed, fix a constant $C < \infty$ and suppose that there exists a positive integer $n_{0}$ such that for all integers $0 \leq n \leq n_{0}$,
\begin{equation}\label{f8.17}
\sup_{s' \geq s_{nj_{\ast}}} \int_{s'}^{s' + J^{n} T_{\ast}}
\left\| \boldsymbol{\epsilon}(s) \right\|_{L^2_{x} l^2} \, \mathrm{d} s \leq C, \quad \sup_{s' \geq s_{nj_{\ast}}}
\int_{s'}^{s' + J^{n} T_{\ast}}  \left\| \boldsymbol{\epsilon}(s) \right\|_{L^2_{x} l^2}^{2} \, \mathrm{d} s \leq C J^{-n} \eta_{\ast}.
\end{equation}
Then, by \eqref{f8.5}, for $s' \geq s_{n j_{\ast}}$,
\begin{equation}\label{f8.18}
\frac{\sup_{s \in  \left[s', s' + 3 J^{n + 1} T_{\ast} \right]} \lambda(s)}{\inf_{s \in  \left[s', s' + 3 J^{n + 1} T_{\ast} \right]} \lambda(s)} \ll T_{\ast}^{\frac{1}{100}}.
\end{equation}
\begin{remark}
The $C$ in \eqref{f8.17} will ultimately be given by the implicit constants in Theorems \ref{t7.13} and  Theorem \ref{t7.14}, so for $T_{\ast}$ sufficiently large, \eqref{f8.18} will hold.
\end{remark}
Using \eqref{f8.17} and arguing as in the proof of  Theorem \ref{t7.14}, it is easy to see that for any $1\leq \gamma\leq 3$ and $s'\geq s_{nj_{\ast}}$,
\begin{align}\label{f8.19}
& 	\int_{s'}^{s'+\gamma J^{n+1}T_{\ast}}
\left\| \boldsymbol{\epsilon}(s) \right\|_{L^2_{x} l^2}^{2} \, \mathrm{d} s \\
& \lesssim  \left\| \boldsymbol{\epsilon} \left(s' \right)  \right\|_{L^2_{x} l^2}
+  \left\| \boldsymbol{\epsilon} \left(s'+\gamma J^{n+1}T_{\ast} \right)  \right\|_{L^2_{x} l^2}+J^{n/25}T^{n/25}_{\ast}CJ^{1-n}\eta_{\ast}\int_{s'}^{s'+\gamma J^{n+1}T_{\ast}}  \left\| \boldsymbol{\epsilon}(s)  \right\|_{L^2_{x} l^2}^{2} \, \mathrm{d} s  + O \left({J^{-8(n+1)} T_{\ast}^{-8}} \right). \notag
\end{align}
Replacing \eqref{f8.9} with \eqref{f8.19} and repeating the proof of \eqref{qmzy}, one can prove that
\begin{align*}
	& \sup_{a\in\mathbb{N}} \int_{s' + a J^{n+1} T_{\ast}}^{s' + (a + 1) J^{n+1} T_{\ast}}
 \left\| \boldsymbol{\epsilon}(s) \right\|_{L^2_{x} l^2}^{2}  \,\mathrm{d}s \\
   \lesssim & \left\|\boldsymbol{\epsilon} \left(s' \right) \right\|_{L_x^2l^2}+\frac{1}{J^{\frac{n+1}{2}} T_{\ast}^{1/2}} \left(\sup_{a \in \mathbb{N}} \int_{s' + a J^{n+1} T_{\ast}}^{s' + (a + 1) J^{n+1} T_{\ast}}  \left\| \boldsymbol{\epsilon}(s) \right\|_{L^2_{x} l^2}^{2} \, \mathrm{ d} s \right)^{1/2} \\
&     +   J^{n/25}T^{n/25}_{\ast}CJ^{1-n}\eta_{\ast}\int_{s'}^{s'+\gamma J^{n+1}T_{\ast}}  \left\| \boldsymbol{\epsilon}(s)  \right\|_{L^2_{x} l^2}^{2} \, \mathrm{ d} s + O \left({J^{-8(n+1)} T_{\ast}^{-8}} \right), 
\end{align*}
which have the same  implicit constant with \eqref{f8.9}.
Since $J\gg 1$ and $j_{\ast}\geq 0$, under the conditions required by \eqref{f8.14} ($\eta_{\ast}$ is sufficiently small), we can get
\begin{equation*}
	\sup_{s' \geq s_{(n + 1) j_{\ast}}} \int_{s'}^{s' + J^{n + 1} T_{\ast}} \left\| \boldsymbol{\epsilon}(s) \right\|_{L^2_{x} l^2}^{2} \, \mathrm{d} s \leq C J^{-(n + 1)} T_{\ast}^{-1},
\end{equation*}
and the implicit constant can be considered to be the same as in \eqref{f8.14}.
Then by H{\"o}lder's inequality,
\begin{equation*}
\sup_{s' \geq s_{(n + 1) j_{\ast}}} \int_{s'}^{s' + J^{n + 1} T_{\ast}}  \left\| \boldsymbol{\epsilon}(s) \right\|_{L^2_{x} l^2}  \, \mathrm{d} s \leq C.
\end{equation*}
Therefore, \eqref{f8.17} holds for any integer $n > 0$.

Now, take any $j \in \mathbb{Z}$ and suppose $n j_{\ast} < j \leq (n + 1) j_{\ast}$. Then \begin{equation}\label{f8.23}
	\sup_{s' \geq s_{j}} \int_{s'}^{s' + J^{n+1} T_{\ast}}  \left\| \boldsymbol{\epsilon}(s) \right\|_{L^2_{x} l^2}  \, \mathrm{d} s \leq CJ, \quad \sup_{s' \geq s_{j}}
 \int_{s'}^{s' + J^{n+1} T_{\ast}}  \left\| \boldsymbol{\epsilon}(s) \right\|_{L^2_{x} l^2}^{2}  \, \mathrm{d} s \leq C J^{1-n} \eta_{\ast}.
\end{equation}
Once again, we can prove that \eqref{f8.23} implies a sharper bound:
\begin{equation}\label{f8.24}
\sup_{a \geq 0} \int_{s_{j} + a J^{n + 1} T_{\ast}}^{s_{j} + (a + 1) J^{n + 1} T_{\ast}}
\left\| \boldsymbol{\epsilon}(s) \right\|_{L^2_{x} l^2}^{2} \, \mathrm{d} s \lesssim 2^{-j} \eta_{\ast},
\end{equation}
and by H\"older's inequality,
\begin{equation}\label{f8.25}
\sup_{s' \geq s_{j}} \int_{s'}^{s' + 2^{j} T_{\ast}}  \left\| \boldsymbol{\epsilon}(s) \right\|_{L^2_{x} l^2} \, \mathrm{d} s \lesssim 1,
\end{equation}
with the bound independent of $j$. Finally, \eqref{f8.24} and \eqref{f8.25} also yield a sharper bound:
\begin{equation}\label{f8.26}
\int_{s_{j}}^{s_{j} + 2^{j} J T_{\ast}}  \left\| \boldsymbol{\epsilon}(s) \right\|_{L^2_{x} l^2}^{2}  \, \mathrm{d} s \lesssim 2^{-j} \eta_{\ast},
\end{equation}
and therefore, by the integral mean value theorem,
\begin{equation*}
\inf_{s \in  \left[s_{j}, s_{j} + 2^{j} J T_{\ast} \right]}  \left\| \boldsymbol{\epsilon}(s) \right\|_{L^2_{x} l^2} \lesssim 2^{-j} \eta_{\ast} J^{-1/2},
\end{equation*}
which implies
\begin{equation*} 
s_{j + 1} \in  \left[s_{j}, s_{j} + 2^{j} J T_{\ast} \right].
\end{equation*}
Thus, by \eqref{f8.26} and H{\"o}lder's inequality,
\begin{equation*}
\int_{s_{j}}^{s_{j + 1}}  \left\| \boldsymbol{\epsilon}(s) \right\|_{L^2_{x} l^2}^{2} \, \mathrm{d} s
 \lesssim 2^{-j} \eta_{\ast}, \quad \text{and} \quad \int_{s_{j}}^{s_{j + 1}} \left\| \boldsymbol{\epsilon}(s) \right\|_{L^2_{x} l^2} \, \mathrm{d} s \lesssim 1,
\end{equation*}
with constants independent of $j$. Summing over $j$ yields \eqref{f8.2} and \eqref{f8.4}.
\end{proof}

\begin{corollary}
	Let $\mathbf{u}$ be the solution in Theorem \ref{t2.3}, and suppose that
	\begin{equation*}
	\left\| \boldsymbol{\epsilon}(0)  \right\|_{L^2_{x} l^2} = \eta_{\ast},\quad \sup_{t\in I} \left\| \boldsymbol{\epsilon}(t)  \right\|_{L^2_{x} l^2}\leq \eta_{\ast}.
	\end{equation*}
Then, for any $p>1$,
	\begin{equation*}
		\int_{0}^{\infty}  \left\| \boldsymbol{\epsilon}(s)  \right\|_{L^2_{x} l^2}^{p} \, \mathrm{ d} s \lesssim \eta_{\ast},
	\end{equation*}
where the implicit constant depends only on $p$ when $\eta_{\ast} \ll 1$ is sufficiently small. 
\end{corollary}
\begin{proof}
	By \eqref{f4.45} and \eqref{f8.2}, we have 
	\begin{equation}\label{f8.32}
		\lim_{s \rightarrow \infty} \left\| \boldsymbol{\epsilon}(s) \right\|_{L^2_{x} l^2} = 0.
	\end{equation}
	Next, by the definition of $s_{j}$, \eqref{f8.4} implies
	\begin{equation*}
		\int_{s_{j}}^{s_{j + 1}} \left\| \boldsymbol{\epsilon}(s) \right\|_{L^2_{x} l^2} \, \mathrm{d} s \lesssim 1,
	\end{equation*}
	and for any $1 < p < \infty$,
	\begin{equation*}
        \int_{s_{j}}^{s_{j + 1}}  \left\| \boldsymbol{\epsilon}(s) \right\|_{L^2_{x} l^2}^{p} \, \mathrm{d} s 
        \lesssim \eta_{\ast}^{p - 1} 2^{-j(p - 1)},
	\end{equation*}
	which shows that $ \left \| \boldsymbol{\epsilon}(s) \right\|_{L^2_{x} l^2}$ belongs to $L_{s}^{p}$ for any $p > 1$.
\end{proof}

\begin{remark}\label{remarkz}
Comparing with Theorem \ref{t2.3}, these theorems require an additional condition:
\begin{align*}
\left\|\boldsymbol{\epsilon}(0) \right\|_{L^2_x l^2}=\sup_{s\in[0,\infty)} \left\|\boldsymbol{\epsilon}(s) \right\|_{L^2_x l^2}=\eta_{\ast},
\end{align*}
which, one may check directly (using \eqref{f4.4}), that can be relaxed to
\begin{align}\label{weakerc}
\left\|\boldsymbol{\epsilon}(0) \right\|_{L^2_x l^2}\geq\frac{99}{100}\sup_{s\in[0,\infty)} \left\|\boldsymbol{\epsilon} (s)  \right\|_{L^2_x l^2}>0.
\end{align}
\eqref{weakerc} can be satisfied due to the time translation and rescaling if $\mathbf{u}$ is not a soliton. Similarly, Theorem \ref{t9.16} in the next section also holds under the weaker condition \eqref{weakerc}.
\end{remark}

Now, we show that $\lambda(s)$(or $\lambda(t))$ is an approximately monotone decreasing function.
\begin{theorem}\label{t9.16}
Let $\mathbf{u}$ be the solution in Theorem \ref{t2.3}, and suppose that
\begin{equation*}
	\left\| \boldsymbol{\epsilon}(0)  \right\|_{L^2_{x} l^2} = \eta_{\ast},\qquad \sup_{t\in I} \left\| \boldsymbol{\epsilon}(t)  \right\|_{L^2_{x} l^2}\leq \eta_{\ast}.
\end{equation*}
 For any $s \geq 0$, let
\begin{equation*} 
\tilde{\lambda}(s) = \inf_{\tau \in [0, s]} \lambda(\tau).
\end{equation*}
Then, for any $s \geq 0$,
\begin{equation*} 
1 \leq \frac{\lambda(s)}{\tilde{\lambda}(s)} <e.
\end{equation*}

\end{theorem}

\begin{proof}
Suppose there exist $0 \leq s_{-} \leq s_{+} < \infty$ satisfying
\begin{equation*} 
\frac{\lambda(s_{+})}{\lambda(s_{-})} = e.
\end{equation*}
Then we will show that $\mathbf{u}(t)$ is a soliton solution to \eqref{1.1}, which is a contradiction since $\lambda(s)$ is constant in that case.
As  computed in \cite{D2},
\begin{equation*}
\left\langle\nabla Q, |x|^2 Q \right\rangle_{L^2_{x}} =  \left\langle iQ, |x|^2 Q \right\rangle_{L^2_{x}} =  \left\langle i \nabla Q, |x|^2 Q \right\rangle_{L_x^2} = 0, \quad  \left\langle Q + x \cdot \nabla Q, |x|^{2} Q \right\rangle_{L^2_{x}} = -\| |x|Q \|_{L^2_{x}}^{2},
\end{equation*}
\begin{equation*}
 -L_{-} \left(|x|^2 \mathbf{Q} \right) = -4 \mathbf{Q} - 4 x \cdot \nabla \mathbf{Q}.
\end{equation*}
Let $\boldsymbol{\epsilon}=\mathbf{v}+i\mathbf{w}$,   and using \eqref{f4.26},  we compute
\begin{align*}
	& \frac{ \mathrm{d} }{ \mathrm{d} s}  \left\langle \boldsymbol{\epsilon}(s), |x|^{2} \mathbf{Q} \right\rangle_{L^2_x l^2} + \frac{N\lambda_{s}}{\lambda}  \| |x| Q \|_{L^2_{x}}^{2}
+ 4 \left\langle \mathbf{Q} + x \cdot \nabla \mathbf{Q}, \mathbf{w}(s) \right\rangle_{L^2_{x} l^2} \\
 = &\sum_{j\in\mathbb{Z}_N} O \left( \left|\gamma_{j,s} + 1 - \frac{x_{s}}{\lambda} \cdot \xi(s) - |\xi(s)|^{2}  \right| \| \epsilon_j \|_{L^2_{x}} \right)
+ O \left( \left|\xi_{s} - \frac{\lambda_{s}}{\lambda} \xi(s) \right|  \left\| \boldsymbol{\epsilon}  \right\|_{L^2_{x} l^2} \right)
+ O \left( \left|\frac{\lambda_{s}}{\lambda} \right|  \left\| \boldsymbol{\epsilon}  \right\|_{L^2_{x} l^2} \right) \\
& \quad
+ O \left( \left|\frac{x_{s}}{\lambda} + 2 \xi \right|  \left\| \boldsymbol{\epsilon}  \right\|_{L^2_{x} l^2} \right)	+ O \left( \left\| \boldsymbol{\epsilon}  \right\|^{2}_{L^2_{x} l^2} \right)
+ O \left( \left\| \boldsymbol{\epsilon}  \right\|_{L^2_{x} l^2}  \left\| \boldsymbol{\tilde{\epsilon}}  \right\|^{2}_{L^{4}_x l^2} \right).
\end{align*}
By Theorem \ref{t8.15}, the fundamental theorem of calculus, and \eqref{f4.45} as well as \eqref{f4.52}-\eqref{f4.55},
\begin{equation*} 
 \left\langle \boldsymbol{\epsilon}(s_+), |x|^2Q \right\rangle_{L^2_x l^2}
 - \left\langle \boldsymbol{\epsilon}(s_-), |x|^2Q \right\rangle_{L^2_x l^2} + \left\| |x| \mathbf{Q}  \right\|_{L^2_{x}}^{2}
 + 4 \int_{s_{-}}^{s_{+}}  \left\langle\mathbf{w}(s), \mathbf{Q} + x\cdot \nabla \mathbf{Q} \right\rangle_{L^2_{x}} \,\mathrm{d}s  = O(\eta_{\ast}).
\end{equation*}
Obviously,
\begin{equation*}
	\left\langle \boldsymbol{\epsilon}(s_+), |x|^2Q \right\rangle_{L^2_x l^2} ,  \left\langle \boldsymbol{\epsilon}(s_-), |x|^2Q \right\rangle_{L^2_x l^2}
\lesssim \max_{s\in[s_-,s_+]} \left\|\boldsymbol{\epsilon}(s) \right\|_{L^2_x l^2} \left\| |x|^2Q \right\|_{L_x^2}=O(\eta_{\ast}).
\end{equation*}
Thus, 
\begin{equation}\label{f9.8}
	\left\| |x| \mathbf{Q}  \right\|_{L^2_{x}}^{2} + 4 \int_{s_{-}}^{s_{+}}  \left\langle\mathbf{w}(s), \mathbf{Q} + x\cdot\nabla \mathbf{Q} \right\rangle_{L^2_{x}} \,\mathrm{d}s
= O (\eta_{\ast}).
\end{equation}
Therefore, there exists $s' \in [s_{-}, s_{+}]$ such that
\begin{equation}\label{f9.9}
\left\langle \mathbf{w}(s'),  \mathbf{Q} + x \cdot \nabla \mathbf{Q} \right\rangle_{L^2_{x} l^2} < 0.
\end{equation}
Let $T_{\ast}=\frac{1}{\eta_{\ast}}$, and let $s_j$ have the same meaning as in Theorem \ref{t8.15}. 
Since $s' \geq 0$, there exists a unique $j \geq 0$ such that $s_{j} \leq s'< s_{j + 1}$. Observing that $|s_{j+1}-s_j|\leq C 2^{j}T_{\ast}$, we have  $s_{j+1}\in \left[s',s'+C2^jT_{\ast} \right]$.
By \eqref{f8.4} and the integral mean value theorem, there exists $\tilde{s}_1\in  \left[0,2^jT_{\ast} \right]$ satisfying
\begin{equation}\label{f9.10}
	\boldsymbol{\epsilon} \left(s'+C2^jT_{\ast}+\tilde{s}_1 \right)\lesssim 2^{-j}\eta_{\ast},
\end{equation}
\begin{equation}\label{f9.11}
\int_{s'}^{s'+ C2^{j}T_{\ast}+\tilde{s}_1}  \left\|\boldsymbol{\epsilon}(s) \right\|^2_{L^2_x l^2} \, \mathrm{d} s
\lesssim 2^{-j}\eta_{\ast},\quad\int_{s'}^{s'+ C2^{j}T_{\ast}+\tilde{s}_1}  \left\|\boldsymbol{\epsilon}(s) \right\|_{L^2_x l^2}  \, \mathrm{d} s \lesssim 1,
\end{equation}
with implicit constants in \eqref{f9.10} and \eqref{f9.11}  independent of $\eta_{\ast}$. 
Similarly, for any $k\in \mathbb{N}_+$, there exists $s_k\in  \left[0,2^{j+k}T_{\ast} \right]$ such that
\begin{equation*}
s_{j+k+1}\in \Big[s'+C\sum_{i=j}^{j+k-1}2^iT_{\ast}+\sum_{m=1}^{k-1}\tilde{s}_m,\  s'+C\sum_{i=j}^{j+k-1}2^iT_{\ast}+\sum_{m=1}^{k-1}\tilde{s}_m+C2^{k+j}T_{\ast}+\tilde{s}_k\Big],
\end{equation*}
 \begin{equation}\label{f9.13}
	\bigg\|\boldsymbol{\epsilon} \Big(s'+C\sum_{i=j}^{j+k-1}2^iT_{\ast}+\sum_{m=1}^{k-1}\tilde{s}_m+C2^{k+j}T_{\ast}+\tilde{s}_k \Big)\bigg\|_{L_x^2l^2}
\lesssim 2^{-j-k}\eta_{\ast},
\end{equation}
\begin{equation}\label{f9.14}
	\int_{s'+C\sum\limits_{i=j}^{j+k-1}2^iT_{\ast}+\sum\limits_{m=1}^{k-1}\tilde{s}_m}^{s'+C\sum\limits_{i=j}^{j+k}2^iT_{\ast}+\sum\limits_{m=1}^{k}\tilde{s}_m}
\left\|\boldsymbol{\epsilon}(s) \right\|^2_{L^2_x l^2}\, \mathrm{ d} s
\lesssim 2^{-j-k}\eta_{\ast},\quad\int_{s'+C\sum\limits_{i=j}^{j+k-1}2^iT_{\ast}+\sum\limits_{m=1}^{k-1}\tilde{s}_m}^{s'+C\sum\limits_{i=j}^{j+k}2^iT_{\ast}+\sum\limits_{m=1}^{k}\tilde{s}_m} \left\|\boldsymbol{\epsilon}(s) \right\|_{L^2_x l^2} \, \mathrm{d} s \lesssim 1,
\end{equation}
with implicit constants in \eqref{f9.13} and \eqref{f9.14} independent of $\eta_{\ast}$ and $k$. 
Let $T_k=C\sum_{i=j}^{j+k}2^iT_{\ast}+\sum_{m=1}^{k}\tilde{s}_m$, and  suppose that for any $1\leq k \leq n$, there exists constant $C_0>0$ so that
\begin{equation}\label{qmm}
	\int_{s'}^{s'+T_k}  \left\|\boldsymbol{\epsilon}(s) \right\|^2_{L^2_x l^2} \, \mathrm{d} s
 \leq C_02^{-j-k}\eta_{\ast},\quad\int_{s'}^{s'+T_k}  \left\|\boldsymbol{\epsilon}(s) \right\|_{L^2_x l^2} \, \mathrm{d} s \leq C_0,
\end{equation}
then we claim that for $C_{0}$ sufficiently large,
\begin{equation}\label{induction10B}
	\int_{s'}^{s'+T_{k+1}}  \left\|\boldsymbol{\epsilon}(s) \right\|^2_{L^2_x l^2} \, \mathrm{d} s \leq C_02^{-j-k-1}\eta_{\ast}.
\end{equation}
In fact, by  \eqref{f8.4}, \eqref{f9.14} and \eqref{qmm}, we have
\begin{equation*} 
\int_{s'}^{s'+T_{k+1}}
\left\|\boldsymbol{\epsilon}(s) \right\|^2_{L^2_x l^2} \, \mathrm{d} s
\leq C_02^{-j-k}\eta_{\ast}+C_12^{-j-k}\eta_{\ast},\quad\int_{s'}^{s'+T_{k+1}}  \left\|\boldsymbol{\epsilon}(s) \right\|_{L^2_x l^2} \, \mathrm{d} s \leq C_0+C_1
\end{equation*}
for some constant $C_1$, which is also independent of $\eta$ and $k$. As in the proof of Theorem \ref{t7.14}, we can set
\begin{equation*}
	\frac{1}{\eta_1}\leq \lambda(s)\leq \frac{1}{\eta_1}T_{k+1}^{1/100},
\quad \forall \,  s\in \left[s',s'+T_{k+1} \right],\quad \xi \left(s' \right)=x \left(s'+T_{k+1} \right)=0,
\end{equation*}
and by \eqref{f4.54} and \eqref{f4.55},
\begin{equation*}
	\sup_{s\in \left[s',s'+T_{k+1} \right]}\frac{|\xi(s)|}{\lambda(s)}
\lesssim \eta_1 \left(\int_{s'}^{s'+T_{k+1}}  \left\| \boldsymbol{\epsilon}(s) \right\|_{L^2_{x} l^2}^{2} \, \mathrm{ d} s \right)^2\ll \eta_0,\quad \sup_{s\in \left[s',s'+T_{k+1} \right]}|x(s)|\ll T_{k+1}^{1/25}.
\end{equation*}
Noticing that $C2^{j+k+1}\leq T_{k+1}\leq 4C2^{j+k+1}T_{\ast}$, by Theorems \ref{t7.13} and \eqref{f9.9},
\begin{align}\label{8.45}
\aligned
& \int_{s'}^{s'+T_{k+1}}  \left\| \boldsymbol{\epsilon}(s) \right\|_{L^2_{x} l^2}^{2} \, \mathrm{d} s\\
\lesssim&   
\left\langle \mathbf{w}\left(s' \right),  \mathbf{Q} + x \cdot \nabla \mathbf{Q} \right\rangle_{L^2_{x} l^2} 
+ \left\|\boldsymbol{\epsilon} \left(s'+T_{k+1} \right) \right\|_{L^2_x l^2}
+T_{k+1}^{1/25} \left(\int_{s'}^{s'+T_{k+1}}  \left\| \boldsymbol{\epsilon}(s) \right\|_{L^2_{x} l^2}^{2}  \, \mathrm{d} s \right)^2+O \left( 
{T_{k+1}^{-8}} \right) \\
\lesssim& 2^{-j-k-1}\eta_{\ast}+T_{k+1}^{1/25} \left(\int_{s'}^{s'+T_{k+1}}  \left\| \boldsymbol{\epsilon}(s) \right\|_{L^2_{x} l^2}^{2}  \, \mathrm{d} s \right)^2+O \left( { \left|s_{j+1+J}-s' \right|^{-8}} \right) \\
\lesssim& 2^{-j-k-1}\eta_{\ast}+ \left(4C2^{j+k+1}T_{\ast} \right)^{1/25}\eta_{\ast}^2(C_0+C_1)^22^{-2j-2k}+\frac{1}{C^82^{8j+8k+8}T^8_{\ast}}.
\endaligned
\end{align}
 Moreover, the implicit constant in \eqref{8.45} only relies on that in  \eqref{f9.13} and Theorem \ref{t7.14}, which are both independent of $\eta_{\ast}$ and $k$. 
 Thus, by choosing $\eta_{\ast}$ sufficiently small and $C_0$ sufficiently large, \eqref{induction10B} holds.

Now we let $k \nearrow \infty$, 
 then \eqref{induction10B} implies 
\begin{equation*} 
\int_{s'}^{\infty}  \left\| \boldsymbol{\epsilon}(s) \right\|_{L^2_{x} l^2}^{2} \, \mathrm{d} s = 0,
\end{equation*}
which directly yields that $ \boldsymbol{\epsilon} (s) = 0$ for all $s \geq s'$. This is a contradiction since $\mathbf{u}$ is not a soliton.
\end{proof}

\subsection{Proof of Theorem \ref{t2.3} }
Finally, we give the proof of Theorem \ref{t2.3}. We begin with the infinite time blowup case:
\subsubsection{Part A: infinite time blowup}
\begin{theorem}\label{t9.17}
If $\mathbf{u}$ is the solution in Theorem \ref{t2.3} and
\begin{equation*} 
\sup(I) = \infty,
\end{equation*}
then $\mathbf{u}$ is a soliton solution of the form \eqref{soliton}.
\end{theorem}
\begin{proof}
If $\mathbf{u}$ is not a soliton of the form \eqref{soliton}, then $\sup\limits_{s\in[0,\infty)} \left\|\boldsymbol{\epsilon}(s) \right\|_{L^2_x l^2}>0$. Recalling Remark \ref{remarkz}, without loss of generality, we may assume 
\begin{align}\label{contra}
\left\|\boldsymbol{\epsilon}(0) \right\|_{L^2_x l^2}=\sup_{s\in[0,\infty)} \left\|\boldsymbol{\epsilon}(s) \right\|_{L^2_x l^2}=\eta_{\ast}.
\end{align}
Additionally, by rescaling, we can set $\tilde{\lambda}(s)\leq  1$ for all $s\geq0$.

\textbf{Case 1.} $\tilde{\lambda}(s)\searrow 0$. In this case, for any integer $k \geq 0$, let
\begin{equation}\label{f10.3}
I(k) =  \left\{ s \geq 0 : 2^{-k + 2} \leq \tilde{\lambda}(s) \leq 2^{-k + 3}  \right\}.
\end{equation}
Since $\sup(I)=\infty$, there exists a sequence $k_{n} \nearrow \infty$ such that
\begin{equation*} 
|I(k_{n})| 2^{-2k_{n}} \geq \frac{1}{k_{n}^{2}},
\end{equation*}
and that $|I(k_{n})| \geq |I(k)|$ for all $k \leq k_{n}$.

Next, we claim that for any sufficiently large $n$, there exists $s_{n} \in I(k_{n})$ such that
\begin{equation}\label{f10.5}
\left\| \boldsymbol{\epsilon}(s_{n})  \right\|_{L^2_{x} l^2} \lesssim \frac{1}{|I(k_n)|}.
\end{equation}
Indeed, by Theorem \ref{t8.15},
\begin{equation*}
\int_{I(k_{n})}  \left\| \boldsymbol{\epsilon}(s)  \right\|_{L^2_{x} l^2}^{2}  \, \mathrm{d} s \lesssim \eta_{\ast},
\end{equation*}
then by the virial identity in \eqref{f9.8}, we have
\begin{equation}\label{f10.7}
\int_{a_{n}}^{\frac{3 a_{n} + b_{n}}{4}}  \left\langle\mathbf{w}, \mathbf{Q} + x \cdot \nabla \mathbf{Q} \right\rangle_{L^2_{x} l^2} \, \mathrm{d} s
= O \left(\eta_{\ast} \right) + O(1).
\end{equation}
Therefore, by the integral mean value theorem, there exists $s_{n}^{-} \in  \left[a_{n}, \frac{3 a_{n} + b_{n}}{4} \right]$ such that
\begin{equation}\label{f10.8}
 \left| \left\langle\mathbf{w} \left(s_{n}^{-} \right), \mathbf{Q} + x \cdot \nabla \mathbf{Q} \right\rangle_{L^2_{x} l^2} \right| \lesssim \frac{1}{|I(k_{n})|}.
\end{equation}
Similarly, there exists $s_{n}^{+} \in  \left[\frac{a_{n} + 3 b_{n}}{4}, b_{n} \right]$ such that
\begin{equation}\label{f10.9}
 \left| \left\langle\mathbf{w} \left(s_{n}^{+} \right),\mathbf{Q} + x \cdot \nabla \mathbf{Q} \right\rangle_{L^2_{x} l^2} \right|
  \lesssim \frac{1}{|I(k_{n})|}.
\end{equation}
Splitting the interval $ \left[s_n^-,s_n^+ \right]$ into $z_n$ small intervals of length $T_{\ast}=\frac{1}{\eta_{\ast}}$, we proceed as follows. 
		
\textbf{Scenario I.}
If 
 $z_n \leq J=2^{j_{\ast}}$, then, as in \eqref{f8.5}-\eqref{f8.8}, we can similarly prove that
\begin{equation*}
\frac{\sup\limits_{s \in  \left[s_n^-, s_n^+ \right]} \lambda(s)}{\inf\limits_{s \in  \left[s_n^-, s_n^+ \right]} \lambda(s)}
\lesssim e^{\ln (T_{\ast})^{1/4}} \lesssim e^{\ln \left(T_{\ast}^{1/1000} \right)} \ll T_{\ast}^{1/100}.
\end{equation*}
Therefore, by Theorem \ref{t7.14}, \eqref{f10.8} and \eqref{f10.9} imply
\begin{equation*} 
\int_{s_{n}^{-}}^{s_{n}^{+}}  \left\| \boldsymbol{\epsilon}(s)  \right\|_{L^2_{x} l^2}^{2} \, \mathrm{d} s \lesssim \frac{1}{|I(k_{n})|}.
\end{equation*}

\textbf{Scenario II.}
If  $z_n > J=2^{j_{\ast}}$, then there exists $a_n\in\mathbb{N}$ such that $J^{a_n}<z_n\leq J^{a_{n}+1}$. As in \eqref{f8.17}, we can similarly argue by induction to deduce that
\begin{align*}
	\int_{s_{n}^{-}}^{s_{n}^{+}} \left\| \boldsymbol{\epsilon}(s)  \right\|_{L^2_{x} l^2} \, \mathrm{d} s\lesssim 1
\end{align*}
and
\begin{align*}
	\int_{s_{n}^{-}}^{s_{n}^{+}} \left\| \boldsymbol{\epsilon}(s) \right\|_{L^2_{x} l^2}^{2} \, \mathrm{d} s\lesssim J^{1-a_n}\eta_{\ast} .
\end{align*}
Following the proof of Theorem \ref{t7.14}, we see that Theorem \ref{t7.13} can be applied on the interval $ \left[s_n^-,s_n^+ \right]$.
Thus \eqref{f10.7} and  \eqref{f10.8} imply  
\begin{equation*}
\int_{s_{n}^{-}}^{s_{n}^{+}} \left\| \boldsymbol{\epsilon}(s)  \right\|_{L^2_{x} l^2}^{2}  \, \mathrm{d} s
\lesssim \frac{1}{|I(k_{n})|} + J^{a_n/25} T_{\ast}^{1/25}J^{1-a_n}\eta_{\ast}
 \left(\int_{s_{n}^{-}}^{s_{n}^{+}} \left\| \boldsymbol{\epsilon}(s)  \right\|_{L^2_{x} l^2}^{2}  \, \mathrm{d} s \right),
\end{equation*}
which further implies
\begin{align*}
\int_{s_{n}^{-}}^{s_{n}^{+}}  \left\| \boldsymbol{\epsilon}(s)  \right\|_{L^2_{x} l^2}^{2} \, \mathrm{d} s	\lesssim \frac{1}{|I(k_{n})|}, 
\end{align*}
provided $k_n$ is sufficient large. Thus \eqref{f10.5} holds.

Now, let $m$ be the smallest integer such that
 \begin{equation}\label{f10.12}
 	\frac{2^{2k_{n}}}{k_{n}^{2}} 2^{m} \geq |I(k_{n})|.
 \end{equation}
 Since $|I(k)| \leq |I(k_{n})|$ for all $0 \leq k \leq k_{n}$, \eqref{f10.12} implies that
 \begin{equation}\label{f10.13}
 	|s_{n}| \leq 2^{2k_{n} + m + 1}.
 \end{equation}
 Let $r_{n}$ be the smallest integer satisfying 
 \begin{equation*} 
 	2^{\frac{2k_{n} + m + 1}{3}} 2^{k_{n}}  \leq 2^{r_{n}}.
 \end{equation*}
Let $t_n=s^{-1}(s_n), \lambda_n=\frac{\eta_1}{4}2^{-k_n}$, and set $\mathbf{v}_n(t,x)=\lambda_n\mathbf{u} \left(\lambda_n^2t,\lambda_nx \right)$. 
Then, there exist real-valued parameters $\tilde{\xi}(t), \tilde{x}(t)$, and $\tilde{\gamma}_1(t),\cdots, \tilde{\gamma}_N(t)$ such that
	\begin{align} \label{yyz}
	&e^{i\tilde{\gamma}_{j}(t)}e^{ix\cdot\tilde{\xi}(t)}v_n \left(t,\tilde{\lambda}(t)x+\tilde{x}(t) \right)
=\tilde{r}_j(t)+ Q, \ 
	\tilde{\lambda}(t)=\lambda_n^{-1}{\lambda \left(\lambda_n^2t \right)},\  \forall \, t\in \left[0,\lambda_n^{-2}t_n \right], \  \forall \, j\in\mathbb{Z}_N,  \\\notag
	&\tilde{x}(0)=\tilde{\xi} \left(\lambda_n^{-2}t_n \right)=0,\  \frac{ \left|\tilde{\xi}(t) \right|}{\tilde{\lambda}(t)}
 \lesssim \eta_1\eta_{\ast},\    \left\|\boldsymbol{\tilde{\epsilon}}(t) \right\|_{L^2_x l^2}
 = \left\|\boldsymbol{\epsilon} \left(\lambda_n^2t \right) \right\|_{L^2_x l^2}, \  \frac{1}{\eta_1}\leq \tilde{\lambda}(t)\leq \frac{16}{\eta_1} 2^{k_n},
 \  \forall \,  t\in  \left[0,\lambda_n^{-2}t_n \right],\\
	&\int_{0}^{\lambda_n^{-2}t_n} \frac{1}{\tilde{\lambda}(t)^2} \,\mathrm{d}t =s_n  \ \mbox{and}\quad\boldsymbol{\tilde{\epsilon}}(t) \mbox{ satisfies }\eqref{orthod}. \notag
\end{align}
Similar to \eqref{bootstrap2}, we can prove that $\forall \, i\geq [s_n] $,
\begin{align*}
	\left\|P_{\geq i+3}\mathbf{v}_n \right\|_{U^2_{\Delta} \left([0,\lambda_{n}^{-2}t_n], L_x^2l^2 \right)}
&\lesssim   \left\|\boldsymbol{\tilde{\epsilon}} \left(\lambda_n^{-2}t_n \right) \right\|_{L_x^2 l^2}
+\eta_0^{\delta} \left\|P_{\geq i}\mathbf{v}_n \right\|_{U^2_{\Delta} \left([0,\lambda_{n}^{-2}t_n], L_x^2l^2 \right)}
+O \left(s_n^{-10} \right).
\end{align*}
Using \eqref{f10.5}, \eqref{f10.12}, \eqref{f10.13} and then making an induction on $i$, we obtain 
	\begin{equation}\label{f10.18}
	\left\| P_{\geq r_{n} + \frac{k_{n}}{8} + \frac{m}{4}} \mathbf{v}_n  \right\|_{U^2_{\Delta} \left( \left[0,\lambda_{n}^{-2}t_n \right],L_x^2l^2 \right)}
 \lesssim k_{n}^{2} 2^{-2k_{n}} 2^{-m},\quad \mbox{ if } k_n \mbox{ is  sufficiently large. }
\end{equation}
Furthermore, following the proof of Theorem \ref{t4.1}, \eqref{f10.5}, \eqref{yyz} and  \eqref{f10.18} imply
	\begin{equation}\label{f10.19}
			\sup_{t\in \left[0,\lambda_{n}^{-2}t_n \right]}E \left(P_{\leq r_{n} + \frac{k_{n}}{4} + \frac{m}{8}} \mathbf{v}_n(t) \right)
 \lesssim \left(k_{n}^{2} 2^{-2k_{n}} 2^{-m} 2^{r_{n} + \frac{k_{n}}{4} + \frac{m}{8}} \right)^{2} \ll 2^{-\frac{21}{10}k_n }.
	\end{equation}
Finally, inserting \eqref{yyz} and \eqref{f10.19} into \eqref{f6.33}, we deduce that
		\begin{equation}\label{f10.20}
			\left\|\boldsymbol{\epsilon}(0) \right\|^2_{L^2_x l^2}= \left\|\boldsymbol{\tilde{\epsilon}}_n(0) \right\|^2_{L^2_x l^2}
\leq \sup_{t\in \left[0,\lambda_{n}^{-2}t_n \right]}  \left\|\boldsymbol{\tilde{\epsilon}}_n(t) \right\|^2_{L^2_x l^2}\lesssim_{\eta_1} 2^{-\frac{k_n}{10}}.
		\end{equation}
Letting $k_n\nearrow \infty$, we find that 
$ \left\|\boldsymbol{\epsilon}(0) \right\|_{L^2_x l^2}=0$,  which contradicts \eqref{contra}.

\textbf{Case 2.} There exists $c>0$ such that $\tilde{\lambda}(s)\geq c$.
In this case, Theorem \ref{t9.16} implies that $\lambda(s)>\tilde{c}$ for some constant $\tilde{c}>0$.
By rescaling, we can set
\begin{align}\label{qdz}
\frac{1}{\eta_1}\leq \lambda(s)\leq \frac{1}{\tilde{c}\eta_1},\quad \forall\,  s>0.
\end{align}
Let $s'_n=\eta_{\ast}^{-1}2^{n}$, and let $\{s_j\}$ be the sequence defined in \eqref{f8.3}.
Since $|s_{j+1}-s_j|\leq C 2^{j}T_{\ast}$,
there exist $0<\theta<1$ and a sequence $ \left\{\tilde{s}_n \right\}$ so that
$$ \left\|\boldsymbol{\epsilon} \left(\tilde{s}_n \right) \right\|_{L^2_x l^2}\lesssim \eta_{\ast} 2^{-n},\quad\tilde{s}_n\in  \left[\theta s'_n,s'_n \right]. $$
Setting $\tilde{t}_n=s^{-1} \left(\tilde{s}_n \right)$, we can apply a Galilean transform so that
$\xi \left(\tilde{t}_n \right)=0$. Arguing as in \eqref{f10.18} and \eqref{f10.19}, we have
	\begin{equation*}
	\left\| P_{\geq \tilde{r}_{n} + \frac{n}{8}} \mathbf{u}(t)  \right\|_{U^2_{\Delta} \left( \left[0,\tilde{t}_n \right],L_x^2l^2 \right)}
 \lesssim_{\eta_{\ast}} 2^{-n},\quad \mbox{ if } n \mbox{ sufficiently large. }
\end{equation*}
\begin{equation}\label{f10.23}
	\sup_{t\in \left[0,\tilde{t}_n \right]}E \left(P_{\leq \tilde{r}_{n} + \frac{n}{4}}  \mathbf{u}(t) \right) \lesssim_{\eta_{\ast}} 2^{-\frac{5}{12}n},
\end{equation}
where $\tilde{r}_n= \left[\tilde{s}_n^{1/3} \right]$.
Then as in \eqref{f10.20}, inequalities \eqref{qdz} and \eqref{f10.23} imply  
\begin{equation*} 
	 \left\|\boldsymbol{\epsilon}(0) \right\|^2_{L^2_x l^2}
\leq \sup_{t\in \left[0,\lambda_{n}^{-2}t_n \right]} \left\|\boldsymbol{\epsilon}(t) \right\|^2_{L^2_x l^2}
\lesssim_{\eta_1,\tilde{c}} 2^{-\frac{5}{12}n}.
\end{equation*}
Letting $n\nearrow \infty$, we find that  $ \left\|\boldsymbol{\epsilon}(0) \right\|_{L^2_x l^2}=0$, which contradicts \eqref{contra}.
Thus, $\mathbf{u}$ must be a soliton of the form \eqref{soliton}. \end{proof}
\subsubsection{Part B: finite time blowup}
Next, we address the finite time blowup case. 
Heuristically, we can use the pseudo-conformal transform \eqref{pseudotransform} to reduce the problem to the infinite time blowup case.

	\begin{theorem}\label{t9.18}
		If $\mathbf{u}$ is the solution in Theorem \ref{t2.3} and
		\begin{equation*}
			\sup(I) <\infty,
		\end{equation*}
		then $\mathbf{u}$ is a pseudo-conformal transformation of the soliton solution (of the form \eqref{pseudosoliton}).
	\end{theorem}
    \begin{remark}
    In \cite{Su}, Y. Su gave the classification of the minimal-mass finite time blowup solutions for the nonlinear coupled Schr\"odinger system in $H^{1}_{x} l^2$ by using the arguments of \cite{HK,M}. His proof heavily relies on the assumption that $\mathbf{u}\in H_x^1l^2$ and therefore cannot be used to prove Thereom \ref{t9.18}.

\end{remark}
	\begin{proof}
	Without loss of generality, assume $-1\in I$, $\sup(I) = 0$, and
		\begin{align*}
			\sup_{\tilde{t}\in I} \left\|\mathbf{u} \left(\tilde{t} \right) \right\|_{L^2_x l^2}\leq \eta_{\ast}.
		\end{align*}
		Applying the pseudo-conformal transformation to $\mathbf{u}(t)$, we have 
        $$\mathbf{v}(t, x )=\frac{1}{t} \overline{\mathbf{u} \left(\frac{1}{t}, \frac{x}{t} \right)} e^{i |x|^{2}/4t}$$
		is an infinite time backward blowup solution of \eqref{1.1} on the interval $(-\infty,-1]$. By \eqref{decomp}, it admits the  decomposition:
\begin{equation*} 
v_j(t,x)
=\frac{1}{t} \frac{e^{i \gamma_j \left(1/t \right)} e^{i \frac{x}{t}
\cdot \frac{\xi \left(\frac{1}{t} \right)}{\lambda \left(\frac{1}{t} \right)}}}{\lambda \left(1/t \right)} Q \left(\frac{x - t x \left(\frac{1}{t} \right)}{t \lambda \left(1/t \right)} \right)
 e^{i |x|^{2}/4t}  + \frac{1}{t} \frac{e^{i \gamma_j \left(1/t \right)} e^{i \frac{x}{t} \cdot \frac{\xi \left(\frac{1}{t} \right)}{\lambda \left(\frac{1}{t} \right)}}}{\lambda \left(1/t \right)}
 \overline{ \epsilon_j \left(\frac{1}{t}, \frac{x - t x \left(\frac{1}{t} \right)}{t \lambda \left(1/t \right)} \right)} e^{i |x|^{2}/4t},\ 
 \forall \,  j\in\mathbb{Z}_N.
\end{equation*}
Our goal is to show that $\mathbf{v}(t)$ satisfies the conditions in Theorem \ref{t9.17} on some interval $(-\infty, t_0]$.
		Since the $L_x^{2}$ norm is preserved by the pseudo-conformal transformation, using \eqref{f8.32}, we have 
\begin{align*}
\lim_{t \searrow -\infty} \bigg\| \frac{1}{t} \frac{e^{i \gamma \left(1/t \right)} e^{i \frac{x}{t}
\cdot \frac{\xi \left(\frac{1}{t} \right)}{\lambda \left(\frac{1}{t} \right)}}}{\lambda \left(1/t \right)} \overline{ \epsilon_j \left(\frac{1}{t}, \frac{x - t x \left(\frac{1}{t} \right)}{t \lambda \left(1/t \right)} \right)} e^{i |x|^{2}/4t} \bigg\|_{L_x^{2}} = 0, 
\intertext{and} 
 \sup_{-\infty < t < -1}
\bigg\| \frac{1}{t} \frac{e^{i \gamma \left(1/t \right)} e^{i \frac{x}{t} \cdot \frac{\xi \left(\frac{1}{t} \right)}{\lambda \left(\frac{1}{t} \right)}}}{\lambda \left(1/t \right)} \overline{ \epsilon_j \left(\frac{1}{t}, \frac{x - t x \left(\frac{1}{t} \right)}{t \lambda \left(1/t \right)} \right)} e^{i|x|^{2}/4t}  \bigg\|_{L_x^{2}} \leq \eta_{\ast}.
\end{align*}
For any $j \in\mathbb{Z}_N$,
		\begin{equation*}
			\frac{1}{t} \frac{e^{i \gamma_j \left(1/t \right)} e^{i \frac{x}{t}
\cdot \frac{\xi \left(\frac{1}{t} \right)}{\lambda \left(\frac{1}{t} \right)}} e^{i x \cdot \frac{x \left(\frac{1}{t} \right)}{2}}
 e^{i \frac{t}{4}  \left|x \left(\frac{1}{t} \right) \right|^{2}}}{\lambda \left(1/t \right)}
 Q \left(\frac{x - t x \left(\frac{1}{t} \right)}{t \lambda \left( 1/t \right)} \right)
		\end{equation*}
		is of the form 
        $\frac{e^{i \tilde{\gamma}_j(t)} e^{ix \cdot \frac{\tilde{\xi}(t)}{\tilde{\lambda}(t)}}}{\tilde{\lambda}(t)}
 Q \left(\frac{x - \tilde{x}(t)}{\tilde{\lambda}(t)} \right)$. It remains to estimate
		\begin{align*}
		& 	\left(\sum_{j\in\mathbb{Z}_N} \left\| \frac{1}{t} \frac{e^{i \gamma_j \left(1/t \right)} e^{i\frac{x}{t}
\cdot \frac{\xi(t)}{\lambda(t)}}}{\lambda \left(1/t \right)} Q \left(\frac{x - t x \left(\frac{1}{t} \right)}{t \lambda \left(1/t \right)} \right)
 \left(e^{i  \left|x - t x \left(\frac{1}{t} \right) \right|^{2}/4t} - 1 \right)  \right\|^2_{L_x^{2}}\right)^{1/2}
\\
  =&N \left\| \frac{1}{t\lambda \left(1/t \right)} Q \left(\frac{x - t x \left(\frac{1}{t} \right)}{t \lambda \left(1/t \right)} \right)
 \left(e^{i  \left|x - t x \left( \frac{1}{t} \right) \right|^{2}/4t} - 1 \right) \right\|_{L_x^{2}} .
\end{align*}
Rescaling so that $\lambda(s)\leq 1, \forall \, s\geq0$, and similarly defining  $I(k)$ as in \eqref{f10.3}, we see that  $\lambda(s) \sim 2^{-k}$ for all $s \in I(k)$. Furthermore, by  \eqref{f4.45},
$ \left\| \boldsymbol{\epsilon}(t)  \right\|_{L_x^{2}l^2} \rightarrow 0$ as $t \nearrow 0$ implies that there exists a sequence $c_{k} \nearrow \infty$ such that
		\begin{equation*}
			|I(k)| \geq c_{k}, \qquad \text{for all} \qquad k \geq 0.
		\end{equation*}
By \eqref{f4.23}, there exists $r(t) \searrow 0$ as $t \nearrow 0$ such that
		\begin{equation*}
			\lambda(t) \leq t^{1/2} r(t),  \mbox{ and thus }  \lambda \left( \frac1t \right) \leq t^{-1/2} r \left( \frac1t \right).
		\end{equation*}
		Therefore, since $Q$ is rapidly decreasing,
		\begin{equation*}
			\lim_{t \searrow -\infty}  \left\| \frac{1}{t \lambda \left( \frac1t \right)} 
Q \left(\frac{x - t x \left(\frac{1}{t} \right)}{t \lambda \left( \frac1t \right)} \right) 
\frac{ \left|x - t x \left(\frac{1}{t} \right) \right|^{2}}{4t}  \right\|_{L_x^{2}} = 0,
		\end{equation*}
which further implies
		\begin{equation*}
			\lim_{t \searrow -\infty} \left\| \frac{1}{t\lambda \left( \frac1t \right)}
 Q \left(\frac{x - t x \left(\frac{1}{t} \right)}{t \lambda \left( \frac1t \right)} \right)
 \left(e^{ \frac{i  \left|x - t x \left(\frac{1}{t} \right) \right|^{2}}{4t}} - 1 \right) \right\|_{L_x^{2}} = 0.
		\end{equation*}
		Thus, $\mathbf{v}$ is a solution that blows up backward  in infinite time, and $\mathbf{v}$ satisfies the conditions of Theorem \ref{t2.3} on $(-\infty, t_{0}]$ for some $t_{0}\ll -1$. 
        By time reversal symmetry and Theorem \ref{t9.17}, $\mathbf{v}$ must be a soliton, which implies that $\mathbf{u}$ is the pseudo-conformal transformation of a soliton (of the form \eqref{pseudosoliton}).
	\end{proof}
Now, Theorems \ref{t9.17} and \ref{t9.18} imply Theorem \ref{t2.3}, and by the reduction in Section \ref{reduction},   Theorem \ref{main1} holds.




\section{Proof of Theorem \ref{th1.2}}\label{Sec:Thm1.2}
In this section, we prove Theorem \ref{th1.2} by contradiction. Using the same reduction as in \cite{CGYZ,CGZ,CGHY}, suppose that Theorem \ref{th1.2} fails, then we derive the existence of an almost periodic solution as follows:
\begin{theorem}[Existence of an almost periodic solution]\label{znb}
Suppose Theorem \ref{th1.2} fails. Then there exists a solution
 $\mathbf{v}\in C_t^0 L_x^2 h^1(I \times \mathbb{R}^2 \times \mathbb{Z}) \cap L_{t,x}^4 l^2 (I \times \mathbb{R}^2 \times \mathbb{Z})$
 to \eqref{1.1} with 
 $$ \| \mathbf{v}\|^2_{L_x^2 l^2}<\frac{M}{2M-1}\|Q_0\|^2_{L_x^2} \text{ and } 
 \| \mathbf{v}\|^2_{L_x^2 h^1}\leq C(M)\cdot\frac{2M-1}{2M}\| \mathbf{v}\|^2_{L_x^2 l^2},$$ 
 which is almost periodic in the sense that there exist $(x(t), \xi(t), N(t)) \in \mathbb{R}^2 \times \mathbb{R}^2  \times \mathbb{R}^+$ such that for any $\eta > 0$, there exists $C(\eta) > 0$ satisfying 
\begin{align}\label{eq2.10v20}
\int_{|x-x(t)|\ge \frac{C(\eta)}{N(t)}} \left\|\mathbf{v}(t,x)\right\|_{l^2}^2 \,\mathrm{d}x + \int_{|\xi- \xi(t)|\ge C(\eta) N(t)} \left\|\hat{\mathbf{v}}(t,\xi)\right\|_{l^2}^2 \,\mathrm{d}\xi < \eta
 \end{align}
for any $t\in I$. 
Here $I$ is the maximal lifespan interval. Additionally, we can take $N(0) = 1$, $x(0) = \xi(0) = 0$, $N(t) \le 1$ on $I$, and
\begin{align*}
|N'(t) | +  |\xi'(t)| \lesssim N(t)^3.
\end{align*}
\end{theorem}
To establish Theorem \ref{th1.2}, it suffices to exclude the existence of $\mathbf{v}$ in Theorem \ref{znb}. For this purpose, we will first prove a variational lemma: 
\begin{lemma} \label{GNweak}
For any positive integer $M$, we have
\begin{align*}
    \sup_{ \substack{\mathbf{u}\in H_x^1l^2 \cap L_x^2h^1, \\
    \|\mathbf{u}\|^2_{L_x^2\dot{h}^1}\leq C(M) \|\mathbf{u}\|^2_{L_x^2l^2}
    }} J(\mathbf{u})=J_M= \frac{2(2M-1)}{M}\|Q_0\|^{-2}_{L_x^2}.
\end{align*}

\end{lemma}
\begin{proof}
Let $\{\mathbf{u}_n\}$ be a sequence in $H_x^1l^2 \cap L_x^2h^1$ so that $\|\mathbf{u}_n\|^2_{L_x^2h^1}\leq C(M) \|\mathbf{u}_n\|^2_{L_x^2l^2}$ and  $J(\mathbf{u}_n)\nearrow J_M$ as $n \to \infty. $
Using Schwartz rearrangement, we can also assume that each component $u_{n,j}$ is non-negative, radial and decreasing. Recalling that $J(\mathbf{u})=J( \mathbf{u}^{\lambda,\mu})$, where $\mathbf{u}^{\lambda,\mu}(x)=\mu \mathbf{u}(\lambda x)$ for any $\mu, \lambda>0$, so we we can also set  
\begin{align*}
\|\mathbf{u}_n\|^2_{L_x^2l^2}=\|\nabla \mathbf{u}_n\|^2_{L_x^2l^2}=1, \ \|\mathbf{u}_n\|^2_{L_x^2h^1}\leq C(M) , \ \sum_{j\in\Z_N}\int_{\R^2}  F_j(\mathbf{u}_n)\cdot \bar{u}_{n,j} \, \mathrm{d} x =J(\mathbf{u}_{n})\nearrow J_M\ \mbox{ as } n\to\infty.
\end{align*}
Since the embedding 
$ H_{rad}^1l^2 \cap L_x^2h^1 \hookrightarrow L_x^4l^2 $ is compact, we see that there exists a radial vector-valued function $\mathbf{v}\in H_x^1l^2 \cap L_x^2h^1$ such that
\begin{align*}
\mathbf{u}_n  \rightharpoonup 
\mathbf{v} \mbox{ in } H_x^1l^2 \mbox{ and } L_x^2h^1, \quad\mathbf{u}_n\to \mathbf{v}  \mbox{  in } L_x^4l^2, \text{ as } n \to \infty .
\end{align*}
Using the basic property of weak convergence, we have 
\begin{align*}
\|\mathbf{v}\|^2_{L_x^2l^2}\leq 1,\quad\|\nabla \mathbf{v}\|^2_{L_x^2l^2}\leq 1, \quad \|\mathbf{v}\|^2_{L_x^2\dot{h}^1}\leq C(M) ,\quad \sum_{j\in\Z_N}\int_{\R^2}  F_j(\mathbf{v})\cdot \bar{v}_j \,  \mathrm{d} x=J_M.
\end{align*}
Therefore,
\begin{align*}
J_M\geq J(\mathbf{v})=\frac{\sum\limits_{j\in\Z}\int_{\R^2}  F_j(\mathbf{v})\cdot \bar{v}_j \, \mathrm{d} x}{\|\mathbf{v}\|^2_{L^2_x l^2}\|\nabla\mathbf{v}\|^2_{L^2_x l^2}}\geq\frac{J_M}{\|\mathbf{v}\|^2_{L^2_x l^2}\|\nabla\mathbf{v}\|^2_{L^2_x l^2}}\geq J_M,
\end{align*}
which immediately implies $J(\mathbf{v})=J_M$, $\|\mathbf{v}\|^2_{L_x^2l^2}=\|\nabla \mathbf{v}\|^2_{L_x^2l^2}=1$ and therefore,  
\begin{align*}
  \mathbf{u}_n\to \mathbf{v}  \mbox{ in  } H_x^1l^2, \text{ as } n \to \infty.
\end{align*}
Now that $\mathbf{v}$ is the maximizer of $J$, we have
\begin{align}\label{mini2}
		\frac{ \mathrm{d} }{ \mathrm{d}  {\varepsilon}} \bigg|_{\epsilon=0}J \left(\mathbf{v}+\varepsilon \boldsymbol{\phi} \right)= 0, \forall \, \boldsymbol{\phi}\in H_x^1l^2\cap L_x^2h^1.
	\end{align}
Therefore, for any $j\in\Z$, $v_j$ weakly solves  
\begin{align*}
  \Delta v_j-v_j=-F_j(\mathbf{v}).
\end{align*}
Using the standard argument in elliptic PDEs, we see that $v_j$ belongs to $C^1(\R^2)$. Finally, we can directly use the argument in \cite{WY} to prove that all non-zero components of $\mathbf{v}$ are identical. Given  $\|\mathbf{v}\|_{L_x^2l^2}=1$ and $\|\mathbf{v}\|_{L_x^2\dot{h}^1}\leq C(M)$, $\mathbf{v}$ has at most $M$ non-zero components. Thus, 
$$J_M=J(\mathbf{v})\leq \frac{2(2M-1)}{M}\|Q_0\|^{-2}_{L_x^2}.$$
Equality is achieved by 
$$\mathbf{h}=\bigg(\cdots,0, \sqrt{\frac{1}{2M-1}}Q_0,\cdots, \underset{[\frac{M}{2}]\text{-th}}{\sqrt{\frac{1}{2M-1}}Q_0},0,\cdots \bigg),$$
verifying $J_M=\frac{2(2M-1)}{M}\|Q_0\|^{-2}_{L_x^2}$.

\end{proof}
Since the $L_x^2h^1$ and $L_x^2l^2$ norms of $\mathbf{v}(t)$ are conserved, according to the discussion in the introduction, we have 

\begin{lemma}\label{GNZX}
  Let $\mathbf{v}$ be the solution in Theorem \ref{znb}. Then for all $ \chi\in C^{\infty}_0(\R^2)$ with $ |\chi(x)|\leq 1$, 
  \begin{align*}  E\left(\chi(x)e^{ix\cdot\xi}P_{\leq k}\mathbf{v}(t,x) \right)\gtrsim 
  \left\|\nabla \left(\chi(x)e^{ix\cdot\xi}P_{\leq k}\mathbf{v}(t,x) \right) \right\|_{L_x^2l^2}, 
\end{align*}
with the implicit constant independent  of $\chi$, $k$ and $\xi$.
\end{lemma}
Now we are in the position to exclude the existence of $\mathbf{v}$ and finish the proof of Theorem \ref{th1.2}:
\begin{proof}[Sketch of the proof of Theorem \ref{th1.2}]

We exclude the almost periodic solution $\mathbf{v}$ in two cases:

\textbf{Case 1.} $\int_0^{\infty}N(t)^3 \, \mathrm{d}t <\infty$. In this case, the argument in the proof of Theorem 5.4 in \cite{CGHY} can be directly applied to prove that $E(e^{ix\cdot\xi_{\infty}}\mathbf{v}_0)=0$ for some $\xi_{\infty}\in\R^2$. By  Lemma \ref{GNweak}, we derive that $\mathbf{v}=0$, which is  a contradiction.

\textbf{Case 2.} $\int_0^{\infty}N(t)^3 \, \mathrm{d}t =\infty$. In this case, we can direct follow the argument in the proof of Theorem 5.3 in \cite{CGHY}. For reader's convenience, we will use the notations in \cite{CGHY}. 

First, we can replace the frequency scale function $N(t)$ by a slowly varying frequency scale function of the almost periodic solution $\mathbf{v}$. Following the argument in \cite{D4}, we can use a smoothing algorithm and replace $N(t)$ with a slowly varying $\tilde{N}(t)$ and
 $\tilde{N}(t) \le N(t)$. Furthermore, by the construction, we can make sure
\begin{align*}
\frac{\big|\tilde{N}'(t)\big|}{ \tilde{N}(t)^3} \lesssim 1, \quad  t > 0 ,
\end{align*}
 and if $ \tilde{N}'(t)\ne 0$, then $ \tilde{N}(t) = N(t)$.

By applying the argument in \cite{D4}, we get
\begin{lemma}\label{le6.3v20}
For any $\delta > 0$, we can take a smoother $\tilde{N}(t)$ such that
\begin{align*}
\liminf\limits_{T\to \infty} \frac{ \int_0^T  \big| \tilde{N} '(t) \big| \,\mathrm{d}t}
{ \int_0^T  \tilde{N} (t)\, 
\int_{\mathbb{R}^2} \sum\limits_{j \in \mathbb{Z}} \big( \overline{P_{\le CK} {u}_j }
F_j (P_{\le CK} \mathbf{u})\big)(t,x) \,\mathrm{d}x
\,\mathrm{d}t } \le \delta .
\end{align*}
\end{lemma}
In the following, we still take $N(t)$ as $\tilde{N}(t)$. 
Let $\varphi$ be a $C_0^\infty$ radial function with
\begin{align*}
\varphi(x) =
\begin{cases}
1,\ |x|\le R- \sqrt{R},\\
0, \ |x|\ge R.
\end{cases}
\end{align*}
Let
\begin{align*}
\phi(x) = \frac1{2\pi R^2} \int_{\mathbb{R}^2} \varphi( |x- s|) \varphi(|s|) \,\mathrm{d}s,
\end{align*}
and let
\begin{align*}
\psi_{ R N(t)^{- 1 } } (r) = \frac1r \int_0^r \phi \left(\tfrac{ N(t)  s}R\right) \,\mathrm{d}s,
\end{align*}
then we define the frequency localized interaction Morawetz action
\begin{align*}
M(t) = \sum\limits_{j,j' \in \mathbb{Z}} \iint_{\mathbb{R}^2 \times \mathbb{R}^2}  \psi_{ R N(t)^{- 1 }} \left(|x-y|\right) N(t) (x-y)  \cdot \Im \big( \overline{ I  v_j(t,x)} \nabla_x I v_j(t,x) \big) \left| I v_{j'} (t,y)\right|^2 \,\mathrm{d}x \mathrm{d}y,
\end{align*}
where $I= P_{\le\log_2{[CK]}}$.

Following the calculation in \cite{CGHY}, we get 
\begin{align}
 \int_0^T M'(t)  \,\mathrm{d}t & \ge 4 \int_0^T \sum\limits_{j,j' \in \mathbb{Z}} \iint \frac1{ 2\pi  R^2} \int \varphi \left( \left|\tfrac{x N(t)}R - s \right| \right)
  \varphi\left( \left|\tfrac{y N(t)}R - s \right| \right) N(t) \notag \\
&\qquad\quad\quad \qquad\quad\quad 
\cdot \big| \nabla \big(e^{-i x \cdot \xi(s)} Iv_j(t,x) \big) \big|^2 \left|Iv_{j'}(t,y) \right|^2 \,\mathrm{d}x \mathrm{d}y \mathrm{d}t \mathrm{d}s
   \notag \\
&\quad -  \int_0^T \sum\limits_{j,j' \in \mathbb{Z}} \iint \Delta \left( \psi_{R N(t)^{-  1 }} \left( |x- y| \right) + \phi\left( \tfrac{ |x - y|N(t)}R \right) \right) N(t) \notag\\
& \qquad\quad\quad\qquad\quad\quad  
\cdot \left|Iv_j(t,x) \right|^2  \left|Iv_{j'}(t,y) \right|^2 \,\mathrm{d}x \mathrm{d}y  \mathrm{d}t
\notag  \\
&\quad - 2 \int_0^T
 \sum\limits_{j, j' \in \mathbb{Z}} \iint \psi_{RN(t)^{- 1 }} (|x- y|) N(t)\notag\\
 &\qquad\quad\quad\qquad\quad
 \quad  
 \cdot \sum\limits_{ (j_1,j_2,j_3) \in \mathcal{R}(j)} \left( \overline{Iv_j} Iv_{j_1}
 \overline{Iv_{j_2}} Iv_{j_3} \right)(t,x)  \left|Iv_{j'} (t,y) \right|^2 \,\mathrm{d}x \mathrm{d}y \mathrm{d}t
 \notag \\
&\quad - \int_0^T
 \sum\limits_{j,j' \in \mathbb{Z} } \iint \psi_{RN(t)^{-  1 }}' (|x- y|) N(t) |x- y| \notag\\
 &\qquad\quad
 \quad\qquad\quad\quad  
 \cdot \sum\limits_{(j_1,j_2,j_3) \in \mathcal{R}(j)} \left( \overline{Iv_j} Iv_{j_1} \overline{Iv_{j_2}} Iv_{j_3} \right)(t,x)  \left|Iv_{j'}(t,y) \right|^2 \,\mathrm{d}x \mathrm{d}y \mathrm{d}t
 \notag  \\
& \quad +  \int_0^T \sum\limits_{j,j' \in \mathbb{Z}} \iint \frac{d}{dt} \left( \psi_{RN(t)^{-  1 }} (|x- y|) N(t)(x- y) \right)\notag\\
&\qquad\quad\quad\qquad\quad\quad  
\cdot \Im \big( \overline{Iv_j(t,x) } \nabla_x Iv_j(t,x) \big)  \left|Iv_{j'}(t,y) \right|^2
\,\mathrm{d}x \mathrm{d}y \mathrm{d}t \label{eq5.17v89}\\
& \quad + \int_0^T \mathfrak{E}(t) \,\mathrm{d}t ,  \notag
\end{align}
 where  ${\xi}(s) \in \mathbb{R}^2$ is chosen such that
\begin{align*}
\sum_{j,j^{'}\in\Z_N}\int \chi^2  \left( \frac{ \tilde{\lambda}(t) (x- s)}{ R \lambda(t)}  \right) \Im  \left( \overline{\frac{e^{ix \cdot{\xi}(s)} }{ \lambda(t)} I v_j  \left(t, \frac{x}{\lambda(t)}  \right)} \cdot \nabla  \left(\frac{e^{ix \cdot{\xi}(s)} }{ \lambda(t)} I v_{j'} \left(t, \frac{x}{\lambda(t)}  \right) \right)  \right) \,\mathrm{d}x = 0
\end{align*}
and $\mathfrak{E}(t)$ is the contribution coming from the error terms $I \mathbf{F}(\mathbf{v}) -\mathbf{F}(I \mathbf{v})$ in the interaction Morawetz action that satisfies 
\begin{align*}
\int_0^T \mathfrak{E}(t) \,\mathrm{d}t \lesssim R^2 o(K).
\end{align*}
As in \cite{CGHY}, we have 
\begin{align*}
 &\int_0^T M'(t)\, \mathrm{d}t \\
  \ge&  4  \sum\limits_{j,j' \in \mathbb{Z}} \int_0^T \iint \frac1{ 2 \pi R^2} \int \varphi \left( \left |\tfrac{x N(t) }R - s \right| \right) \varphi \left( \left |\tfrac{y N(t) }R - s \right| \right)  N(t) \\
  & \qquad\qquad \qquad \qquad   
  \cdot \left|\nabla \left(e^{-i x\cdot \xi(s)} Iv_j(t,x) \right) \right|^2 \left|Iv_{j'}(t,y) \right|^2 \, \mathrm{d}s\mathrm{d}x \mathrm{d}y \mathrm{d}t \\
 & - 2 \sum\limits_{j,j' \in \mathbb{Z}}  \int_0^T \iint \phi \left( \tfrac{N(t) |x- y|}R \right) N(t)
 \sum\limits_{ (j_1,j_2,j_3) \in \mathcal{R}(j)}
 \left( \overline{Iv_j} Iv_{j_1} \overline{Iv_{j_2}} Iv_{j_3} \right)(t,x) \left|Iv_{j'}(t,y) \right|^2 \,\mathrm{d}x \mathrm{d}y  \mathrm{d}t \\
 \ge &\ \frac2{ \pi  R^2} \sum\limits_{j,j' \in \mathbb{Z}} \int_0^T \iiint \left|\nabla \left( \chi
 \left( \left|\tfrac{x N(t) }R - s \right| \right) e^{-i x\cdot \xi(s)} Iv_j(t,x) \right) \right|^2 \,\mathrm{d}x\\
 &\qquad\qquad\qquad\qquad 
 \cdot
\varphi \left( \left|\tfrac{y N(t) }R- s \right| \right) \left|Iv_{j'}(t,y) \right|^2 N(t) \,\mathrm{d}y \mathrm{d}s \mathrm{d}t  \\
 &\ -   \frac{1}{\pi R^2}  \sum\limits_{j,j' \in \mathbb{Z}}  
\int_0^T \iiint  \left|\chi \left( \left|\tfrac{x N(t) }R - s \right| \right) \right|^4   \varphi \left( \left|\tfrac{y N(t) }R - s \right| \right) N(t)  \,\mathrm{d}s \\
& \qquad\qquad\qquad\qquad
\cdot
 \sum\limits_{(j_1,j_2,j_3) \in \mathcal{R}(j)}  \left( \overline{Iv_j} Iv_{j_1} \overline{Iv_{j_2}} Iv_{j_3}  \right)(t,x) \left|Iv_{j'}(t,y) \right|^2  \,\mathrm{d}x \mathrm{d}y \mathrm{d}t \\
&\ - C \frac{K}{   R^4}   \left\|\mathbf{v} \right\|_{L_x^2 \ell^2}^4 - o_R(1) K .
\end{align*}
Then by Lemma \ref{GNZX}, one can find $\eta > 0$ such that
\begin{align*}
& \sum\limits_{j \in \mathbb{Z}} \int  \left|\nabla  \left(
 \chi \left(  \left|\tfrac{x N(t) }R - s \right| \right) e^{-i x\cdot \xi(s)} Iv_j(t,x)
\right)
\right|^2 \,\mathrm{d}x \\
& \qquad 
- \frac12 \sum\limits_{j \in \mathbb{Z}} \int  \left|\chi \left(  \left|\tfrac{x N(t) }R - s \right| \right)
\right|^4 \sum\limits_{(j_1,j_2,j_3) \in \mathcal{R}(j)}
\left( \overline{Iv_j} Iv_{j_1} \overline{Iv_{j_2}} Iv_{j_3}  \right)(t,x) \,\mathrm{d}x \\
\ge &\ \frac12  \left[ \frac{ \sqrt{\frac{M}{2M-1}} \|Q_0\|_{L_x^2} }{ \Big( \sum\limits_{j \in \mathbb{Z}}
 \left\|\chi  \left( \left|\frac{x N(t) }R - s \right|  \right) e^{-i x\cdot \xi(s)} Iv_j(t,x) \right\|_{L_x^2}^2 \Big)^\frac12 } \right]^2\\
&\qquad\qquad\qquad\qquad\times \sum\limits_{j \in \mathbb{Z}} \int  \left|\chi \left( \left|\tfrac{x N(t) }R - s \right| \right) \right|^4
\sum\limits_{(j_1,j_2,j_3) \in \mathcal{R}(j)} \left( \overline{Iv_j} Iv_{j_1} \overline{Iv_{j_2}} Iv_{j_3}  \right)(t,x) \,\mathrm{d}x \\
&\ - \frac12 \sum\limits_{j \in \mathbb{Z}} \int  \left|\chi \left( \left|\tfrac{x N(t) }R - s \right| \right) \right|^4
\sum\limits_{(j_1,j_2,j_3) \in \mathcal{R}(j)} \left( \overline{Iv_j} Iv_{j_1} \overline{Iv_{j_2}} Iv_{j_3}  \right)(t,x) \,\mathrm{d}x \\
\ge &\ \eta \sum\limits_{j \in \mathbb{Z}} \int  \left|\chi \left( \left|\tfrac{x N(t) }R - s \right| \right) \right|^4
\sum\limits_{(j_1,j_2,j_3) \in \mathcal{R}(j)} \left( \overline{Iv_j} Iv_{j_1} \overline{Iv_{j_2}} Iv_{j_3}  \right)(t,x) \,\mathrm{d}x.
\end{align*}
The remaining argument follows that in \cite{CGHY}. Then we can get 
\begin{align*}
\int_0^T M'(t) \,\mathrm{d}t \ge \frac{c  \eta}2 K- R^2 o(K) - o_R(1)K.
\end{align*}
On the other hand, by \eqref{eq2.10v20}, we have
\begin{align*}
\sup\limits_{t \in [0, T ]} |M(t)| \lesssim R^2 o(K).
\end{align*}
Choosing $R(\eta)$ sufficiently large, by  \eqref{eq2.10v20}, since we can take $T$ large enough and make $K$ be arbitrarily large, we conclude that $\mathbf{v} = 0$.
This is a contradiction. Therefore, this completes the proof of Theorem \ref{th1.2}. \end{proof}

\bigskip
\noindent \textbf{Acknowledgments.} 
We would like to thank Haewon Yoon and Zehua Zhao for their helpful comments. X. Cheng has been partially supported by the NSF of Jiangsu Province (Grant No.~BK20221497). 
J. Zheng was supported by National key R \&D program of China: 2021YFA1002500 and NSF grant of China (No. 12271051).

\appendix

\section{Proof of Theorem \ref{claim2.1}}

\renewcommand{\thesection}{\Alph{section}}

In this appendix, we provide the proof of Theorem \ref{claim2.1} for self-contained. By the scattering result in \cite{CGHY}, any non-scattering solution $\mathbf{u}$ to \eqref{1.1} with $\| \mathbf{u}_0 \|_{L_x^2 l^2} = \|\mathbf{Q} \|_{L_x^2 l^2}$ must be a minimal-mass blowup solution. This allows us to reduce Theorem \ref{claim2.1} to the analysis of the almost periodic solution.

Let $t_n \nearrow \sup I$ as $n \to \infty$ be a sequence of times. Applying Proposition \ref{pro3.9v23}, 
we obtain after passing to a subsequence, for all $J\geq 1$,
\begin{align}\label{eq2.1v26}
\mathbf{u} (t_n ) = \sum\limits_{j = 1}^J g_n^j  \left( e^{i t_n^j \Delta}
\boldsymbol{\phi}^j \right) +  \mathbf{w}_n^J,
\end{align}
where $g_n^j$ is the group action such  that for any $k\in\mathbb{Z}_N$,
\begin{align}\label{eq2.2v26}
\left(g_n^j \boldsymbol{\phi}^j \right)_k (x) = \lambda_{n}^j e^{i x \xi_{n}^j} {\phi}_k^j 
\left( \lambda_{n}^j x + x_{n}^j \right),
\end{align}
and
\begin{align*}
\lim\limits_{J \to \infty} \limsup\limits_{n \to \infty} \left\|e^{i t \Delta} \mathbf{w}_n^J  \right\|_{L_{t,x}^4 l^2} = 0.
\end{align*}
Since $\mathbf{u}$ is a minimal-mass blowup solution,
$\mathbf{\phi}^j  = 0$ for $j \ge 2$, and $\|\mathbf{\phi}^1 \|_{L_x^2 l^2} = \|\mathbf{Q} \|_{L_x^2 l^2}$ with 
$\| \mathbf{w}_n^1 \|_{L_x^2 l^2 } \to 0 $ as $n \to \infty$.
Thus, it is convenient to drop the $j$ notation and simply write
\begin{align*}
\mathbf{u} (t_n ) = g_n\boldsymbol{\phi} + \mathbf{w}_n.
\end{align*}
Let $\mathbf{v}$ be the solution to \eqref{1.1} with initial data $\mathbf{\phi}$, and let $\tilde{I}$ be its maximal interval of existence. Since
\begin{align*}
\lim\limits_{n \to \infty } \|\mathbf{u} \|_{L_{t,x}^4 l^2 ( (\inf I, t_n) \times \mathbb{R}^2\times\Z_N)} = \infty \mbox{ and } \| \mathbf{u} \|_{L_{t,x}^4 l^2 ((t_n, \sup I ) \times  \mathbb{R}^2\times\Z_N)} = \infty, \ \forall\, n\in \mathbb{N},
\end{align*}
we have
\begin{align}
\| \mathbf{v} \|_{L_{t,x}^4 l^2 \left( \left[0, \sup \tilde{I} \right) \times \mathbb{R}^2\times\Z_N \right)} = \|\mathbf{v} \|_{L_{t,x}^4 l^2 \left( \left( \inf \tilde{I}, 0 \right] \times \mathbb{R}^2\times\Z_N \right)} = \infty.
\label{eq2.6v26}
\end{align}
Now, since $\mathbf{v}(t)$ blows up in both time directions, \eqref{eq2.6v26} holds, and $\|\mathbf{v} \|_{L_x^2 l^2 } = \|\mathbf{Q} \|_{L_x^2 l^2 }$, we can use the result of \cite{CGHY}
 to prove that $\mathbf{v}$ is almost periodic. Specifically, for each $t \in \tilde{I}$, there exist $\lambda(t) > 0$, $\xi(t) \in \mathbb{R}^2$ and $x(t) \in \mathbb{R}^2$ such that
\begin{align}\label{eq2.14v26}
\frac1{ \lambda(t)} e^{i x\cdot \xi(t)}  \mathbf{v} \left( t, \frac{x - x(t)}{\lambda(t)}  \right) \in  \mathcal{K},
\end{align}
where $ \mathcal{K}$ is a precompact subset of $L_x^2l^2 $. Additionally, we have \begin{align*}\label{ppp}
\left| \lambda'(t) \right|+ \left| \xi'(t) \right|  \lesssim \lambda(t)^3.
\end{align*} 

\noindent To prove Theorem \ref{claim2.1}, we further reduce it to the following theorem:
\begin{theorem}\label{th5v26}
There exists a sequence $s_m \nearrow \sup I$ as $m\to\infty$ and group actions $g(s_m)$ of the form \eqref{eq2.2v26} such that
\begin{align*}
\| g(s_m) \mathbf{v}(s_m) - \mathbf{Q} \|_{L_x^2l^2} \to 0 ,\quad \mbox{as } m\to\infty.
\end{align*}
\end{theorem}

\begin{proof}[Proof of Theorem \ref{th5v26} implies Theorem \ref{claim2.1}]

Suppose $g(s_m ) \mathbf{v} (s_m ) \to \mathbf{Q}$ in $L^2_x l^2$.
For each $m$, let $s_m \in \tilde{I}$ be such that
\begin{equation}\label{qzaaa}
\| g(s_m ) \mathbf{v} (s_m ) - \mathbf{Q} \|_{L_x^2 l^2} \le 2^{-m }.
\end{equation}
Next, observe that \eqref{eq2.1v26} implies
\begin{align*}
\lambda_ne^{i x \cdot \xi_n}   \mathbf{u}( t_n, \lambda_n x + x_n ) \to \boldsymbol{\phi}(x) \text{ in } L_x^2l^2
\end{align*}
and by \eqref{sym5} and perturbation theory, for a fixed $m$, for $n$ sufficiently large,
\begin{align}\label{eq2.10v26}
\aligned
& \left\| e^{-i |\xi_n|^2 s_m } e^{i \xi_n \cdot x}  \lambda_n \mathbf{u} \left(t_n + \lambda_n^2 s_m, \lambda_n x + x_n - 2 \xi_n \lambda_n s_m \right) - \mathbf{v}(s_m) \right\|^2_{L^2_xl^2}\\
 \le& C(s_m)  \left\| e^{i \xi_n \cdot x} e^{i \gamma_{n,j} } \lambda_n \mathbf{u} \left(t_n , \lambda_n x + x_n  \right) -\boldsymbol{\phi}(x) \right\|^2_{L^2_xl^2} .
\endaligned
\end{align}
Define $\tilde{\mathbf{u}}(x) 
=\lambda_n e^{-i |\xi_n|^2 s_m} e^{i \xi_n\cdot x} \mathbf{u}  \left( t_n + \lambda_n^2 s_m, \lambda_n x + x_n - 2 \xi_n \lambda_n s_m  \right)$ for all $ j\in\Z_N$. 
By \eqref{qzaaa}, \eqref{eq2.10v26}, and the triangle inequality, we obtain 
\begin{align}\label{eq2.11v26}
\aligned
& \left\| \tilde{\boldsymbol{u}}(x)- \mathbf{Q}(x)  \right\|_{L_x^2 l^2}
\\
 \le& C(s_m)  \left\|e^{i \xi_n\cdot x} e^{i \gamma_n } \lambda_n \mathbf{u}( t_n, \lambda_n x + x_n ) -\boldsymbol{\phi} (x)  \right\|_{L_x^2 l^2} + 2^{-m }.
\endaligned
\end{align}
Since $g(s_m)$ has form \eqref{eq2.2v26}, there exists  $g_{n,m}$ of the form \eqref{eq2.2v26} such that
\begin{align*}
g(s_m )  \tilde{\boldsymbol{u}}(t_n,x)
= g_{n,m} \mathbf{u}  \left(t_n +  \lambda_n^2 s_m, x \right).
\end{align*}
From \eqref{eq2.11v26}, we have 
\begin{align*}
\lim\limits_{m,n \to \infty}  \left\|g_{n,m} \mathbf{u} \left(t_n + \lambda_n^2 s_m, x  \right) - \mathbf{Q} (x) \right\|_{L_x^2 l^2 } = 0.
\end{align*}
Since $t_n \nearrow \sup {I}$ as $n \to \infty$ and $s_m \ge 0$, $t_n + \lambda_n^2 s_m \nearrow \sup \tilde{I}$ as $n \to \infty$, which implies Theorem \ref{claim2.1}. 
\end{proof}

\begin{theorem}\label{th6v26}
For any sequence $T_n \in \tilde{I}$ with $T_n \nearrow \sup \tilde{I}$ as $n\to \infty$,
\begin{align*}
\lim\limits_{T_n \nearrow \sup \tilde{I}} \frac1{\sup\limits_{t \in [0, T_n]} \lambda(t)} \int_0^{T_n} \lambda(t)^3 \,\mathrm{d}t = \infty.
\end{align*}

\end{theorem}

\begin{proof}

Assume, for contradiction, that the statement fails. 
That is, there exists a constant $c_0 $ and a sequence $T_n \nearrow \sup \tilde{I}$ as $n \to \infty$ such that for all $n \in \mathbb{N}$,
\begin{align*}
\frac1{\sup\limits_{t \in [0, T_n]} \lambda(t)} \int_0^{T_n} \lambda(t)^3 \,\mathrm{d}t \le c_0.
\end{align*}
This corresponds to the rapid cascade scenario in \cite{CGHY}, where $\lambda(t)$ (analogous to $N(t)$ in those papers) can be chosen continuous.  If $ \lim\limits_{n\to\infty}\sup\limits_{t \in [0, T_n]} \lambda(t)=\infty$, then for each $T_n$, we can choose $t_n \in [0, T_n]$ such that
\begin{align*}
\lambda(t_n) = \sup\limits_{t \in [0, T_n]} \lambda(t)\ \mbox{ and }\ t_n \nearrow \sup\tilde{I}\ \mbox{ as } \ n\to\infty.
\end{align*}
Since $\tilde{I}$ is the maximal interval of existence of $\mathbf{v}$,
\begin{align*}
\lim\limits_{n \to \infty} \| \mathbf{v}(t,x) \|_{L_{t,x}^4 l^2  ([0 , t_n]\times \mathbb{R}^2\times\Z_N )} = \lim\limits_{n \to \infty} \| \mathbf{v}(t,x) \|_{L_{t,x}^4 l^2  ([t_n , \sup\tilde{I})\times \mathbb{R}^2\times\Z_N )} =\infty.
\end{align*}
Using the almost periodicity of $\mathbf{v}$ and \eqref{eq2.14v26}, there exist $x(t_n)\in\R^2$ and $\xi(t_n)\in\R^2$ such that if
\begin{align}\label{eq5.6v26}
\mathbf{v}_n(x):=  e^{i x\cdot \xi(t_n)} \frac{1}{\lambda(t_n)}\mathbf{v} \left(t_n,  \frac{x  - x(t_n)}{\lambda(t_n)} \right),
\end{align}
then $ \mathbf{v}_n(x)$ converges to some $\mathbf{w}_0$ in $L_x^2l^2$, where $ \mathbf{w}_0(x)$ is the initial data of a solution $ \mathbf{w}(t,x)$ to \eqref{1.1} that blows up in both time directions, $\tilde{\lambda}(t) \le 1$ for all $t \le 0$, and
\begin{align}\label{eq5.7v26}
\int_{-\infty}^0 \tilde{\lambda}(t)^3 \,\mathrm{d}t \le c_0.
\end{align}
Following the proof in \cite{CGHY}, we have 
\begin{align}\label{eq5.8v26}
\| \mathbf{w}(t,x) \|_{L_t^\infty \dot{H}_x^s l^2 ((- \infty, 0] \times \mathbb{R}^2 \times\Z_N)} \lesssim_s c_0^s,
\end{align}
for any $0 \le s < \infty $.

Combining \eqref{eq5.8v26} with \eqref{eq5.7v26} and $\left|\tilde{\lambda}'(t) \right| \lesssim \tilde{\lambda}(t)^3$, we obtain  
\begin{align}\label{eq5.9v26}
\lim\limits_{t \to  - \infty} \tilde{\lambda}(t) = 0.
\end{align}
From 
\begin{align*}
\left| \tilde{\xi}'(t) \right|  \lesssim \tilde{\lambda}(t)^3,
\end{align*}
and \eqref{eq5.7v26}, we deduce that $\tilde{\xi}(t)$ converges to some $\tilde{\xi}_-\in \mathbb{R}^2$ as $t \to - \infty$. Applying a Galilean transformation to set $\tilde{\xi}_- =0$. Then, by interpolation, \eqref{eq5.8v26} and \eqref{eq5.9v26}, we get 
\begin{align*}
\lim\limits_{t \to - \infty} E( \mathbf{w} (t)) = 0.
\end{align*}
Therefore, by conservation of energy and the convergence in $L_x^2l^2$ of \eqref{eq5.6v26}, we obtain 
\begin{align*}
E( \mathbf{w}_0) = 0 \text{ and } \| \mathbf{w}_0 \|_{L_x^2l^2} = \| \mathbf{Q}\|_{L_x^2l^2}.
\end{align*}
Then using the sharp Gagliardo-Nirenberg inequality, we get 
\begin{align*}
\forall\ j\in\mathbb{Z}_N, w_{0,j} (x) = \lambda  e^{i\gamma_j}Q ( \lambda(x - x_0) ), \ \mbox{for some } \gamma_j \in[0,2\pi], \lambda >0, x_0 \in \mathbb{R}^2.
\end{align*}
Assuming without loss of generality that $x_0 = 0$, $\lambda = 1$,  and $\gamma_j=0,\ \forall j\in\mathbb{Z}_N$,  the solution to \eqref{1.1} is given by
\begin{align*}
\mathbf{w} (t,x) = e^{it} \mathbf{Q}(x), \ t\in \mathbb{R}.
\end{align*}
However, such a solution does not satisfy \eqref{eq5.9v26}, which leads to a contradiction. 

If $ \lim\limits_{n\to\infty}\sup\limits_{t \in [0, T_n]} \lambda(t)<\infty$, then $T_n=\infty$ and we can rescale $\mathbf{v}$ so that $\lambda(t)\leq 1,\ \forall\ t>0$. Then, we can also rigorously prove   that up to symmetries of \eqref{1.1},
\begin{align*}
\mathbf{u} (t,x) = e^{it} \mathbf{Q}(x), \ t\in \mathbb{R},
\end{align*}
 which also leads to a contradiction.
\end{proof}
\noindent We now prove Theorem \ref{th5v26} under the assumption 
\begin{align*}
\lim\limits_{n \to \infty} \frac1{\sup\limits_{t \in [0, T_n]} \lambda(t) } \int_0^{T_n} \lambda(t)^3 \,\mathrm{d}t = \infty.
\end{align*}
After passing to a subsequence,  we suppose
\begin{align*}
\frac1{\sup\limits_{t \in [0, T_n]} \lambda(t) } \int_0^{T_n} \lambda(t)^3 \,\mathrm{d}t = 2^{2n}.
\end{align*}
Define \begin{align*}
M(t) = \sum\limits_{j, j' \in  \mathbb{Z}_N} \iint \left|Iv_{j'} (t,y)\right|^2 \Im  \left( \overline{Iu_j  } \nabla Iu_j  \right)(t, x) \tilde{\lambda}(t)(x- y) \psi \left( \tilde{\lambda}(t) (x- y) \right) \,\mathrm{d}x \mathrm{d}y,
\end{align*}
where $I$ is the Fourier truncation operator $ P_{\le\log_2{[ 2^{2n} \sup\limits_{t \in [0, T]}\lambda(t)}]} $, $T= 2^k$ for some $k \in \mathbb{N} $, $\tilde{\lambda}(t)$ is given by the smoothing algorithm, and $\psi( |x- y |)$ is a radial function defined to be
\begin{align}\label{eq3.2v26}
\psi(x) = \frac1{|x- y |} \int_0^{ |x- y |} \phi(s) \,\mathrm{d}s,
\end{align}
with $\phi ( |x|)$ being a radial function given by
\begin{align*}
\phi (|x- y |)= \frac1{R^2} \int \chi^2 \left( \frac{x- s}R \right) \chi^2  \left( \frac{y - s}R \right) \,\mathrm{d}s,
\end{align*}
where $\chi$ is a radial, smooth, compactly supported function with $\chi(x) = 1$ for $|x| \le 1$ and $\chi(x) $ is supported on $|x|\le 2$.
In addition, $\chi(|x|)$ is decreasing as a function of the radius. $R$ is a large, fixed constant that will be allowed to go to infinity as $T \nearrow \infty$. Then, 
\begin{align}\label{eq5.19v26}
\aligned
M'(t) = &  - 2\tilde{\lambda}(t)  \sum\limits_{j, j' \in \mathbb{Z}_N}
\iint \Im \left( \overline{Iv_j } \partial_k Iv_j  \right)(t,y) \Im \left( \overline{Iv_{j'} }\partial_l Iv_{j'}  \right)(t,x)
\\
& \qquad \cdot \left( \delta_{lk} \psi \left( \tilde{\lambda}(t) ( x- y)  \right) + \frac{ (x- y)_l  (x - y)_k}{|x- y |} \psi' \left( \tilde{\lambda}(t) ( x- y) \right) \right) \,\mathrm{d}x \mathrm{d}y
\\
& + \frac12 \tilde{\lambda}(t)^3 \sum\limits_{j, j' \in \mathbb{Z}_N}
\iint |Iv_{j'} (t,y)|^2 |Iv_j (t,x)|^2 \left( \Delta \phi \left( \tilde{\lambda}(t) (x - y ) \right) +   \Delta \psi  \left( \tilde{\lambda}(t) (x - y) \right)  \right) \,\mathrm{d}x \mathrm{d}y
 \\
& + 2 \tilde{\lambda}(t)   \sum\limits_{j, j' \in \mathbb{Z}_N}  \iint |Iv_{j'} (t,y)|^2 \Re \left( \overline{ \partial_k Iv_j } \partial_l Iv_j  \right)(t,x) \\
& \qquad \cdot \left( \delta_{l k}  \psi \left( \tilde{\lambda}(t) ( x- y)  \right) + \frac{ (x - y)_l  (x-y)_k}{|x - y |} \psi' \left( \tilde{\lambda}(t) ( x - y ) \right) \right) \,\mathrm{d}x \mathrm{d}y
  \\
& - \frac1{2} \tilde{\lambda}(t) \sum\limits_{j, j' \in \mathbb{Z}_N}  \iint \left|Iv_{j'} (t,y)\right|^2 \sum\limits_{j_1, j_2, j_3 \in \mathcal{R}_N(j)} ( \overline{I v}_j I v_{j_1} \overline{I v}_{j_2} I  v_{j_3} ) (t,x)  \\
 & \qquad \cdot \left( 2 \psi  \left( \tilde{\lambda}(t)(x- y) \right) + \psi'  \left( \tilde{\lambda}(t)(x - y)  \right) |x-y|  \right) \,\mathrm{d}x \mathrm{d}y
  \\
& + \tilde{\lambda}'(t) \sum\limits_{j, j' \in \mathbb{Z}_N} \iint |Iv_{j'} (t,y)|^2 \Im  \left( \overline{Iv_j } \nabla Iv_j   \right)(t,x) \phi  \left( \tilde{\lambda}(t) ( x- y)  \right) (x - y) \,\mathrm{d}x \mathrm{d}y\\
& + \mathcal{E},
\endaligned
\end{align}
where $\mathcal{E}$ are the error terms arising from $\mathcal{N}$,
\begin{align*}
i \partial_t I \mathbf{v} + \Delta I \mathbf{v}  +  \mathbf{F} ( I \mathbf{v}  ) =  \mathbf{F}(I \mathbf{v}  ) - I  \mathbf{F}( \mathbf{v}  ) = \mathcal{ N} .
\end{align*}
It is known from \cite{CGHY} that
\begin{align*}
\int_0^T \mathcal{ N}  \,\mathrm{d}t \lesssim R o(T).
\end{align*}
Therefore, choosing $R \to \infty$ sufficiently slowly,
\begin{align*}
\lim\limits_{T \to \infty } \frac{R o(T)}T= 0.
\end{align*}
By direct computation,
\begin{align*}
\phi (x) = \frac1{R^2} \int \chi^2  \left( \frac{x- s}R \right) \chi^2  \left( \frac{s}R \right) \,\mathrm{d}s \sim 1,
\end{align*}
for $|x| \le R$, $\phi(x)$ is supported on the set $|x| \le 4 R$, and $\phi(x)$ is a radially symmetric function that is decreasing as $|x| \to \infty$.
Therefore, \eqref{eq3.2v26} implies that
\begin{align}\label{eq3.10v26}
| \psi(x) |  \lesssim \frac{R}{ |x|}, \text{ for all } x \in \mathbb{R}^2.
\end{align}
Also, by direct computation,
\begin{align}\label{eq3.11v26}
\left|\Delta \phi(x) \right|= \left|\frac1{R^2} \int \Delta \chi^2 \left( \frac{x - s}R \right) \chi^2  \left( \frac{s}R \right) \,\mathrm{d}s\right| \lesssim \frac1{R^2}.
\end{align}
Next, by the same calculations that give \eqref{eq3.10v26}, we get
\begin{align*}
| \Delta \psi(x) | \lesssim \frac{R}{ |x|^3},
\end{align*}
So $|\Delta \psi(x)| \lesssim \frac1{R^2}$ for $|x| \ge R$. By the fundamental theorem of calculus, since $\phi'(0) = 0$, \eqref{eq3.2v26} implies
\begin{align*}
\psi (r) = \phi(0) + \frac1r \int_0^r \int_0^s (s- t) \phi''(t ) \,\mathrm{d}t \mathrm{d}s,
\end{align*}
so by \eqref{eq3.11v26},
\begin{align*}
| \Delta \psi(x ) | \lesssim \frac1{R^2} \text{ for } |x| \le R.
\end{align*}
Equation \eqref{eq3.10v26} implies
\begin{align*}
\sup\limits_{t \in [0, T_n ]} |M(t)| \lesssim R o \left(2^{2n} \right) \cdot \sup\limits_{t \in [0, T]} \lambda(t).
\end{align*}
Next, since the smoothing algorithm guarantees that $\tilde{\lambda}(t) \le \lambda(t)$, we have
\begin{align*}
& \int_0^{T_n} \frac12 \tilde{\lambda}(t)^3  \sum\limits_{j, j' \in  \mathbb{Z}_N} \iint |Iv_{j'} (t,y)|^2  |Iv_j (t,x)|^2
\left( \Delta \phi  \left( \tilde{\lambda}(t) (x - y)  \right) +   \Delta \psi  \left( \tilde{\lambda}(t) ( x- y)  \right)  \right) \,\mathrm{d}x \mathrm{d}y \mathrm{d}t
\\
& \lesssim \frac1{R^2} \| \mathbf{v} \|_{L^2_x l^2 }^4 \int_0^{T_n} \tilde{\lambda}(t) \lambda(t)^2 \,\mathrm{d}t
\lesssim \frac{2^{2n}}{R^2} \sup\limits_{ t \in [0, T]} \lambda(t). \notag
\end{align*}
Since $\tilde{\lambda}(t) \le \lambda(t)$, following the analysis in \cite{CGHY} and \cite{D3}, we have 
\begin{align}\label{eq5.23v26}
& 2 \int_0^{T_n} \tilde{\lambda}(t) \sum\limits_{j, j' \in \mathbb{Z}_N}
 \iint \Im  \left( \overline{Iv_{j'} } \partial_k Iv_{j'}   \right) (t,y) \Im  \left( \overline{Iv_j } \partial_l Iv_{j}  \right)(t,x)
 \notag \\
 & \qquad  \cdot \left( \delta_{l k} \psi \left( \tilde{\lambda}(t) ( x - y) \right) + \frac{ (x-y )_l (x - y)_k}{ |x - y |} \psi'  \left( \tilde{\lambda}(t) (x - y) \right) \right) \,\mathrm{d}x \mathrm{d}y \mathrm{d}t
\\
& + 2 \int_0^{T_n} \tilde{\lambda}(t) \sum\limits_{j, j' \in \mathbb{Z}_N} \iint  |Iv_{j'} (t,y)|^2 \Re  \left( \overline{\partial_l  Iv_j } \partial_k Iv_j   \right)(t,x) \notag \\
& \qquad \cdot \left( \delta_{lk} \psi  \left( \tilde{\lambda}(t) ( x- y) \right) + \frac{ (x -y )_l  ( x-y)_k }{|x- y|} \psi'  \left( \tilde{\lambda}(t) ( x-y ) \right) \right) \,\mathrm{d}x \mathrm{d}y \mathrm{d}t
\notag \\
& - \frac1{2} \int_0^{T_n} \tilde{\lambda}(t) \sum\limits_{j, j' \in \mathbb{Z}_N} \iint |Iv_{j'} (t,y)|^2 \sum\limits_{(j_1,j_2,j_3) \in \mathcal{R}_N(j)}
\left( Iv_{j_1} \overline{Iv_{j_2}} Iv_{j_3} \overline{Iv_j}  \right) (t,x) \notag \\
& \qquad  \cdot \left( 2 \psi \left( \tilde{\lambda}(t) ( x- y) \right) + \psi' \left( \tilde{\lambda}(t) (x-y ) \right) |x- y | \right) \,\mathrm{d}x \mathrm{d}y \mathrm{d}t
\notag\\
 = &  4 \int_0^{T_n} \frac{ \tilde{\lambda}(t) \lambda(t)^2 }R  \sum\limits_{  j' \in \mathbb{Z}_N}
\int  \left( \int \chi^2  \left( \frac{y - s}R  \right)  \left|\frac{e^{iy\cdot  {\xi}_{\ast}(s)} }{ \lambda(t)} Iv_{j'}  \left(t, \frac{y}{\lambda(t)}  \right) \right|^2 \,\mathrm{d}y  \right) \notag \\
&  \qquad \cdot \chi^2 \left( \frac{x_* - s}R \right) E \left( \chi^2  \left( \frac{x - x(t)}R \right) \frac{e^{ix \cdot{\xi}_{\ast}(s)} }{ \lambda(t)} I \mathbf{v}  \left(t, \frac{x}{\lambda(t)}  \right)  \right) \,\mathrm{d}s \mathrm{d}t \notag \\
&   + R^2 o \left(2^{2n}  \right) \cdot \sup\limits_{t \in [0, T]} \lambda(t)
+ O \left( \eta^4 \|\mathbf{v} \|_{L_t^\infty L_x^2l^2}^2 \int_0^{T_n} \tilde{\lambda}(t) \| \mathbf{v}(t) \|_{L_x^{4} l^2 }^{4} \,\mathrm{d}t  \right)
\notag \\
&  + O \left( \frac{C(\eta)}R \| \mathbf{v} \|_{L_t^\infty L_x^2 l^2 }^2 \int_0^{T_n} \tilde{\lambda} (t)  \| \mathbf{v} (t) \|_{L_x^{4} l^2 }^{4} \,\mathrm{d}t  \right),
\notag
\end{align}
where  ${\xi}_\ast(s) \in \mathbb{R}^2$ is chosen such that
\begin{align*}
\sum_{j,j^{'}\in\Z_N}\int \chi^2  \left( \frac{ \tilde{\lambda}(t) (x- s)}{ R \lambda(t)}  \right) \Im  \left( \overline{\frac{e^{ix \cdot{\xi}_{\ast}(s)} }{ \lambda(t)} I v_j  \left(t, \frac{x}{\lambda(t)}  \right)} \cdot \nabla  \left(\frac{e^{ix \cdot{\xi}_{\ast}(s)} }{ \lambda(t)} I v_{j'} \left(t, \frac{x}{\lambda(t)}  \right) \right)  \right) \,\mathrm{d}x = 0
\end{align*}
and $x_{\ast}$ is chosen such that
\begin{align*}
  \chi\left(\frac{x_{\ast}-s}{R}\right)=\inf_{|x_{\ast}-x(t)|\leq 4C(\eta)}\chi\left(\frac{x_{\ast}-s}{R}\right).
\end{align*}
Therefore, it only remains to consider the contribution of the term in \eqref{eq5.19v26} involving $\tilde{\lambda}'(t)$. By direct computation,
\begin{align*}
& \tilde{\lambda}'(t)  \sum\limits_{j, j' \in \mathbb{Z}_N} \iint |Iv_{j'} (t,y)|^2 \Im \left( \overline{Iv_j } \nabla Iv_j  \right)(t,x) \phi \left( \tilde{\lambda}(t) (x- y) \right) (x - y) \,\mathrm{d}x \mathrm{d}y \\
& = \frac{\tilde{\lambda}'(t) }{R^2 \tilde{\lambda}(t)}  \sum\limits_{j, j' \in \mathbb{Z}_N}
\int  \left(  \int \chi^2  \left( \frac{ \tilde{\lambda}(t) y - s}R  \right) |Iv_{j'} (t,y)|^2 \,\mathrm{d}y  \right)
\notag \\
& \qquad \cdot  \left( \int \chi^2  \left( \frac{ \tilde{\lambda}(t) x -s }R  \right) \Im  \left( \overline{Iv_j  } \nabla Iv_j  \right)(t,x) \cdot  \left( x \tilde{\lambda}(t) - s \right) \,\mathrm{d}x  \right) \,\mathrm{d}s
\notag \\
&  \quad - \frac{ \tilde{\lambda}'(t)}{R^2  \tilde{\lambda}(t) } \sum\limits_{j, j' \in \mathbb{Z}_N}
\int  \left( \int \chi^2   \left( \frac{ \tilde{\lambda}(t) y - s}R \right)
\left( y \tilde{\lambda}(t) - s \right) |Iv_{j'} (t,y)|^2 \,\mathrm{d}y  \right)  \notag \\
& \qquad
\cdot  \left( \int \chi^2  \left( \frac{\tilde{\lambda}(t) x - s }R  \right) \Im \left( \overline{Iv_j  }\nabla Iv_j   \right)(t,x) \,\mathrm{d}x  \right) \,\mathrm{d}s.
\notag
\end{align*}
Now, rescaling gives 
\begin{align*}
= &  \frac{ \tilde{\lambda}'(t)}{R^2 \tilde{\lambda}(t)} \lambda(t)\sum\limits_{j, j' \in \mathbb{Z}_N}  \int \left( \int \chi^2  \left( \frac{ \tilde{\lambda}(t) y - \lambda(t) s }{ R \lambda(t)}  \right)
\left| \frac1{\lambda(t)  } Iv_{j'}  \left(t, \frac{y}{ \lambda(t)}  \right) \right|^2 \,\mathrm{d}y \right)
\\
& \quad \cdot  \int \chi^2  \left( \frac{ \tilde{\lambda}(t) x - \lambda(t) s}{R \lambda(t)}  \right) \Im  \left( \frac1{\lambda(t)  } \overline{Iv_j
\left(t, \frac{x}{\lambda(t)}  \right)}
\nabla  \left(  \frac1{ \lambda(t) } Iv_j  \left(t, \frac{x}{\lambda(t)}  \right)  \right) \cdot \left( \frac{ x \tilde{\lambda}(t) - s \lambda(t) }{ \lambda(t)}  \right) \,\mathrm{d}x  \right) \,\mathrm{d}s
 \\
& - \frac{\tilde{\lambda}'(t)}{R^2 \tilde{\lambda}(t)} \lambda(t) \sum\limits_{j, j' \in \mathbb{Z}_N} \int \left( \int \chi^2  \left( \frac{ \tilde{\lambda}(t) y - \lambda(t) s }{R \lambda(t)}  \right)
\left( \frac{y \tilde{\lambda}(t)  - s \lambda(t)}{ \lambda(t)}  \right)  \left|\frac1{\lambda(t)  } Iv_{j'}   \left(t , \frac{y}{\lambda(t)}  \right)  \right|^2 \,\mathrm{d}y   \right)  \\
& \quad \cdot   \left( \int \chi^2  \left( \frac{ \tilde{\lambda}(t) x - s \lambda(t)}{ R \lambda(t)}  \right) \Im  \left( \frac1{\lambda(t)  } \overline{Iv_j  \left(t, \frac{x}{\lambda(t)}  \right) } \nabla  \left( \frac1{\lambda(t)  } Iv_j   \left(t, \frac{x}{\lambda(t)}  \right) \right) \right) \,\mathrm{d}x  \right) \,\mathrm{d}s .
\end{align*}
For any $\xi \in \mathbb{R}^2$,
\begin{align*}
= &  \frac{\tilde{\lambda}'(t) }{R^2 \tilde{\lambda}(t)} \lambda(t) \sum\limits_{j, j' \in \mathbb{Z}_N}
\int  \left( \int \chi^2  \left( \frac{\tilde{\lambda}(t) y - \lambda(t) s }{R \lambda(t)}  \right)
\left| \frac{e^{i x\cdot \xi} }{\lambda(t) } Iv_{j'}  \left(t, \frac{y }{\lambda(t)}  \right)  \right|^2 \,\mathrm{d} y  \right)
\\
& \quad \cdot    \int \chi^2  \left( \frac{ \tilde{\lambda}(t) x - \lambda(t) s}{R \lambda(t)}  \right)
\Im \left( \frac{e^{-i x\cdot \xi }}{ \lambda(t) } \overline{Iv_j  \left(t, \frac{x}{\lambda(t)}  \right) } \nabla \left(  \frac{e^{i x\cdot \xi}}{ \lambda(t)  } Iv_j  \left(t, \frac{x}{ \lambda(t)}  \right) \right)  \left( \frac{ x \tilde{\lambda}(t) - s \lambda(t)}{ \lambda(t)}  \right) \,\mathrm{d}x  \right) \,\mathrm{d}s  \\
& - \frac{\tilde{\lambda}'(t)}{ R^2 \tilde{\lambda} (t)} \lambda(t) \sum\limits_{j, j' \in \mathbb{Z}_N}
 \int  \left( \int \chi^2 \left( \frac{ \tilde{\lambda}(t) y - \lambda(t) s }{R \lambda(t)}  \right) \left( \frac{y \tilde{\lambda}(t) - s \lambda(t)}{\lambda(t)}  \right)
\left| \frac{e^{i x\cdot \xi}}{ \lambda(t)  } Iv_{j'}  \left(t, \frac{y}{ \lambda(t)}  \right) \right|^2 \,\mathrm{d}y  \right)  \\
& \quad \cdot  \left( \int \chi^2  \left( \frac{ \tilde{\lambda}(t) x - s \lambda(t)}{R \lambda(t)}  \right) \Im  \left( \frac{e^{-i x\cdot \xi} }{ \lambda(t) } \overline{Iv_j  \left(t, \frac{x}{\lambda(t)}  \right) }
\nabla  \left( \frac{e^{i x\cdot \xi}}{ \lambda(t)  } Iv_j  \left(t, \frac{x}{ \lambda(t)}  \right)  \right)  \right) \,\mathrm{d}x  \right) \,\mathrm{d}s.
\end{align*}
In particular, if we choose $\xi = {\xi}_{\ast}(s)$,
\begin{align*}
= & \frac{ \tilde{\lambda}'(t)}{ R^2 \tilde{ \lambda}(t)} \lambda(t)\sum\limits_{j, j' \in \mathbb{Z}_N}
 \int  \left( \int \chi^2  \left( \frac{ \tilde{\lambda}(t) y - \lambda(t) s}{ R \lambda(t)}  \right) \left|\frac{e^{iy \cdot{\xi}_{\ast}(s)} }{ \lambda(t)} Iv_{j'}
  \left(t, \frac{y}{\lambda(t)}  \right) \right|^2 \,\mathrm{d}y  \right) \\
& \quad \cdot  \left( \int \chi^2  \left( \frac{\tilde{\lambda} (t) x - \lambda(t) s}{R \lambda(t)}  \right) \Im  \left( \overline{\frac{e^{ix \cdot{\xi}_{\ast}(s)} }{ \lambda(t)} Iv_j  \left(t, \frac{x}{\lambda(t)}  \right) }     \nabla \left( \frac{e^{ix \cdot{\xi}_{\ast}(s)} }{ \lambda(t)} Iv_j  \left(t, \frac{x}{\lambda(t)}  \right)\right) \right)  \left( \frac{ x \tilde{\lambda}(t) - s \lambda(t) }{ \lambda(t)}  \right) \,\mathrm{d}x  \right) \,\mathrm{d}s  \\
= & \frac{ \tilde{\lambda}'(t)}{R}  \sum\limits_{j, j' \in \mathbb{Z}_N} \int  \left( \int \chi^2  \left( \frac{ \tilde{\lambda}(t) (y - s)}{R \lambda(t)}  \right)  \left|\frac{e^{i y \cdot{\xi}_{\ast}(s)} }{ \lambda(t)} Iv_{j'}  \left(t, \frac{ y }{\lambda(t)}  \right) \right|^2 \,\mathrm{d}y  \right)    \\
 & \cdot \left( \int \chi^2  \left( \frac{\tilde{\lambda}(t) (x - s) }{R \lambda(t)}  \right) \Im  \left( \overline{\frac{e^{ix \cdot{\xi}_{\ast}(s)} }{ \lambda(t)} Iv_j  \left(t, \frac{x}{\lambda(t)}  \right)}     \nabla \left(\frac{e^{ix \cdot {\xi}_{\ast}(s)} }{ \lambda(t)} Iv_j  \left(t, \frac{x}{\lambda(t)}  \right) \right) \right) \cdot  \left( \frac{\tilde{\lambda}(t) (x - s)}{ \lambda(t)}  \right) \,\mathrm{d}x  \right) \,\mathrm{d}s.
\end{align*}
Then, by the Cauchy-Schwarz inequality,
\begin{align}\label{eq5.30v26}
\lesssim &
 \frac{\eta^4}{R^2 } \lambda(t) \tilde{\lambda}(t)^2 \sum\limits_{j, j' \in \mathbb{Z}_N}
 \int  \left( \int \chi^2  \left( \frac{ \tilde{\lambda}(t)(y - s)}{R \lambda(t)}  \right)  \left|\frac{e^{iy  \cdot{\xi}_{\ast}(s)} }{ \lambda(t)} Iv_{j'}  \left(t, \frac{ y }{\lambda(t)}  \right) \right|^2 \,\mathrm{d}y  \right) \\
  & \quad \cdot \left( \int \chi^2  \left( \frac{ \tilde{\lambda}(t) ( x- s)}{ R\lambda(t)}  \right)  \left|\nabla_x  \left( \frac{e^{ix \cdot {\xi}_{\ast}(s)} }{ \lambda(t)} Iv_j \left(t, \frac{x}{\lambda(t)}  \right) \right) \right|^2 \,\mathrm{d}x  \right) \,\mathrm{d}s \notag
\\
& + \frac1{\eta^4}  \frac{  \left|\tilde{\lambda}'(t)  \right|^2 }{ \lambda(t) \tilde{\lambda}(t)^2 } \frac{ \tilde{\lambda} (t)}{ R^2 \lambda(t)}
\sum\limits_{j, j' \in \mathbb{Z}_N} \int  \left( \int \chi^2  \left( \frac{ \tilde{\lambda}(t) ( y - s)}{R \lambda(t)}  \right)  \left|\frac{e^{i y  \cdot{\xi}_{\ast}(s)} }{ \lambda(t)} Iv_{j'}  \left(t, \frac{y }{\lambda(t)}  \right) \right|^2 \,\mathrm{d}y  \right) \notag \\
 & \quad \cdot  \int \chi^2
 \left( \frac{ \tilde{\lambda} (t) ( x - s)}{R \lambda(t)}  \right)  \left|\frac{e^{ix \cdot{\xi}_{\ast}(s)} }{ \lambda(t)} Iv_j  \left(t, \frac{x}{\lambda(t)}  \right) \right|^2 \left| \frac{\tilde{\lambda}(t) ( x- s) }{ \lambda(t)}  \right|^2 \,\mathrm{d}x  \,\mathrm{d}s.
\notag
\end{align}
The first term in \eqref{eq5.30v26} can be absorbed into \eqref{eq5.23v26}. The second term in \eqref{eq5.30v26} is bounded by
\begin{align*}
\frac1{\eta^4} \frac{  \left|\tilde{\lambda}'(t)  \right|^2 } { \lambda(t) \tilde{\lambda}(t)^2 } R^2 \| \mathbf{v} \|_{L_t^\infty L_x^2 l^2 }^4.
\end{align*}
The smoothing algorithm introduced in  \cite{D6} is employed to control this term. Recall that after applying $n$ iterations of the smoothing algorithm over the interval $[0, T]$, the function $\tilde{\lambda}(t)$ exhibits the following properties:
\begin{enumerate}
\item $\tilde{\lambda}(t) \le \lambda(t)$,
\item If $\tilde{\lambda}'(t) \ne 0$, then $\lambda(t) = \tilde{\lambda}(t)$,
\item $\tilde{\lambda}(t) \ge 2^{-n } \lambda(t)$,
\item $\int_0^T  \left|\tilde{\lambda}'(t)  \right| \,\mathrm{d}t \le \frac1n\int_0^T  \left|\lambda'(t)  \right| \frac{ \tilde{\lambda}(t) }{ \lambda(t)} \,\mathrm{d}t$, where the implicit constant is independent of $n$ and $T$.
\end{enumerate}
Therefore, we have 
\begin{align*}
&  \int_0^{T_n } \frac1{\eta^4} \frac{  \left|\tilde{\lambda}'(t) \right|^2 }{ \lambda(t) \tilde{\lambda}(t)^2 } R^2  \left\| \mathbf{v}  \right\|_{L_t^\infty L_x^2 l^2 }^4 \,\mathrm{d}t 
\\
& \le \frac1{\eta^4} \| \mathbf{v}  \|_{L_t^\infty L_x^2 l^2 }^4 \int_0^{T_n} \frac{  \left|\lambda'(t) \right|}{ \lambda(t)^3 } R^2  \left|\tilde{\lambda}'(t)  \right| \,\mathrm{d}t  \lesssim \frac1n \frac{R^2}{\eta^4} \| \mathbf{v}  \|_{L_t^\infty L_x^2 l^2 }^4 \int_0^{T_n } \tilde{\lambda}(t) \lambda(t)^2 \,\mathrm{d}t.
\end{align*}
On the other hand, it is easy to derive that
\begin{align*}
\sup\limits_{t \in [0, T_n]} \lambda(t)  \le 2^{-2n} \int_0^{T_n} \lambda(t)^3 \,\mathrm{d}t,
\end{align*}
\begin{align*}
R_n \sup\limits_{t \in [0, T_n]} |M(t)| \lesssim R_n o(2^{2n}) \cdot \sup\limits_{t \in [0, T]} \lambda(t).
\end{align*}
Consequently, we can take a sequence $\eta_n \searrow 0$ and $R_n \nearrow \infty$ (probably very slowly) such that the following estimates hold:
\begin{align*}
\frac1{n} \frac{R_n^2}{\eta^4} \| \mathbf{v } \|_{L_t^\infty L_x^2 l^2 }^4 \int_0^{T_n}  \left|\tilde{\lambda} (t)  \right| \lambda(t)^2 \,\mathrm{d}t
= o_n(1) \int_0^{T_n} \tilde{\lambda}(t) \lambda(t)^2 \,\mathrm{d}t, \\
R_n \sup\limits_{t \in [0, T_n]} |M(t) | \lesssim o \left(2^{2n} \right) \cdot \sup\limits_{t \in [0, T]} \lambda(t),\\
O \left( \eta_n^4 \| \mathbf{v}  \|_{L_t^\infty L_x^2 l^2 }^2 \int_0^T \tilde{\lambda}(t) \| \mathbf{v} (t) \|_{L^{4}_x l^2 }^{4} \,\mathrm{d}t  \right) \lesssim o_n(1) \int_0^{T_n} \tilde{\lambda}(t) \lambda(t)^2 \,\mathrm{d}t,
\intertext{ and }
O \left( \frac{C(\eta_n) }{R_n} \| \mathbf{v}  \|_{L_t^\infty L_x^2 l^2 }^2 \int_0^{T_n} \tilde{\lambda}(t) \| \mathbf{v} (t) \|_{L_x^{4} l^2 }^{4} \,\mathrm{d}t  \right) \lesssim o_n(1) \int_0^{T_n} \tilde{\lambda}(t) \lambda(t)^2 \,\mathrm{d}t .
\end{align*}
These terms can thus be treated as error terms. Therefore, there exist a sequence of times $t_n \nearrow \sup \tilde{I}$ as $n \to \infty$ and a sequence $s_n$ such that
\begin{align}\label{eq5.38v26}
E \left( \chi \left( \frac{ (x - x(t_n)) \tilde{\lambda}(t_n) }{ R_n \lambda(t_n)}  \right) \frac{e^{ix \cdot{\xi}_{\ast}(s_n)} }{ \lambda(t_n )} I \mathbf{v}  \left(t_n , \frac{x}{\lambda(t_n )}  \right)  \right) \to 0, \\
\left\| \left( 1 - \chi \left( \frac{ (x - x(t_n)) \tilde{\lambda}(t_n) }{R_n \lambda(t_n)}  \right)  \right) \frac{e^{ix \cdot{\xi}_{\ast}(s_n)} }{ \lambda(t_n )} I \mathbf{v}  \left(t_n , \frac{x}{\lambda(t_n )}  \right)  \right\|_{L^2_x l^2 } \to 0, \label{eq5.39v26}\\
\left\| \frac{e^{ix \cdot{\xi}_{\ast}(s_n)} }{ \lambda(t_n )} I \mathbf{v} \left(t_n , \frac{x}{\lambda(t_n )}  \right)  \right\|_{L^2_x l^2 } \to \| \mathbf{Q}  \|_{L^2_x l^2 }, \label{eq5.40v26}
\end{align}
Let \begin{equation}
\mathbf{h}_n=\chi \left( \frac{ (x - x(t_n)) \tilde{\lambda}(t_n)}{ R_n \lambda(t_n)}  \right) \frac{e^{ix \cdot{\xi}_{\ast}(s_n)} }{ \lambda(t_n )} I \mathbf{v}  \left(t_n , \frac{x}{\lambda(t_n )}  \right)  ,\label{eq5.41v260}
\end{equation}
then \eqref{eq5.39v26}, \eqref{eq5.40v26} and the almost periodicity of $\mathbf{v}\neq 0$ guarantee that $\liminf\limits_{n\to\infty}\|\nabla \mathbf{h}_n\|_{L_x^2l^2}=c_0>0$. After passing to a subsequence, we may assume  $\|\nabla \mathbf{h}_n\|_{L_x^2l^2}>\frac{c_0}{2}$.

Now,  \eqref{eq5.38v26}-\eqref{eq5.41v260} are sufficient for us  to prove Theorem \ref{th5v26}:

\begin{proof}[Proof of Theorem \ref{th5v26}]
Now we define 
\begin{align*}
\tilde{\lambda}_n=\frac{\|\nabla\mathbf{Q}\|_{L_x^2l^2}}{\|\nabla \mathbf{h}_n\|_{L_x^2l^2}},\quad \tilde{\mu}_n=\frac{\|\mathbf{h}_n\|_{L_x^2l^2}}{\|\mathbf{Q}\|_{L_x^2l^2}}\cdot{\tilde{\lambda}_n}, \quad\mathbf{g}_n=\tilde{\lambda}_n\mathbf{h}(\tilde{\mu}_n x). 
 \end{align*}
 Then \eqref{eq5.38v26} and \eqref{eq5.40v26} yield
\begin{equation*}\label{profileq}
\|\mathbf{g}_n\|_{L_x^2l^2}=\|\mathbf{Q}\|_{L_x^2l^2},\quad \|\nabla \mathbf{g}_n\|_{L_x^2l^2}=\|\nabla\mathbf{Q}\|_{L_x^2l^2},\quad E(\mathbf{g}_n)\to0\ \mbox{ as } n\to\infty.
\end{equation*}
Using the variational characterization of the ground state $\mathbf{Q}$, there exist parameters $\tilde{x}_n\in\R^2$ and $\left(\tilde{\gamma}_{n,1},\cdots,\tilde{\gamma}_{n,N} \right)\in[0,2\pi]^N$ such that
\begin{align}\label{xpara}
\quad\mbox{ for any } j\in\mathbb{Z}_N,\quad\|g_{n,j}-e^{i\gamma_{n,j}}Q(x+\tilde{x}_n)\|_{H_x^1}\to 0 \quad\mbox{ as } n\to\infty.
\end{align}
Therefore, Theorem \ref{th5v26} follows immediately from \eqref{eq5.38v26}, \eqref{eq5.40v26},  \eqref{xpara}, and the definition of the operator $I$.
\end{proof}


\begin{thebibliography}{99}
\bibitem{AA} 
N. Akhmediev and A. Ankiewicz, \emph{Partially coherent solitons on a finite background}, Phys. Rev. Lett. {\bf 82} (1999), 2661-2664.



\bibitem{BV} 
P. B\'egout and A. Vargas, \emph{Mass concentration phenomena for the $L^2$-critical nonlinear Schr\"odinger
equation}, Trans. Amer. Math. Soc. {\bf359} (2007), no. 11, 5257-5282.

\bibitem{Bo1} 
J. Bourgain, \emph{Refinements of Strichartz inequality and applications to 2d-NLS with critical nonlinearity}, Int. Math. Res. Not. (1998), 253-283.


\bibitem{CK}
R. Carles and S. Keraani, \emph{On the role of quadratic oscillations in nonlinear Schr\"odinger equations. II. The $L^2$-critical case}, Trans. Amer. Math. Soc. {\bf 359} (2007), no. 1, 33-62.




\bibitem{CZou}
Z. Chen and W. Zou, Positive least energy solutions and phase separation for coupled Schr\"odinger equations with critical exponent, Arch. Ration. Mech. Anal. {\bf 205} (2012), no.~2, 515-551.

\bibitem{CGYZ}
X. Cheng, Z. Guo, K. Yang, and L. Zhao,
\emph{On scattering for the cubic defocusing nonlinear Schr\"odinger equation on the waveguide $\mathbb{R}^2 \times \mathbb{T}$}, 
Rev. Mat. Iberoam. {\bf 36 } (2020), no. 4, 985-1011.


\bibitem{CGZ}
X. Cheng, Z. Guo, and Z. Zhao, \emph{On scattering for the defocusing quintic nonlinear Schr\"odinger equation on the two-dimensional cylinder},
SIAM J. Math. Anal. {\bf 52} (2020), no. 5, 4185-4237.

\bibitem{CGHY}
X. Cheng, Z. Guo, G. Hwang, and H. Yoon, \emph{Global well-posedness and scattering of the two dimensional cubic focusing nonlinear Schr\"odinger system}, J. Differential Equations {\bf 433} (2025), Article 113225. 

\bibitem{DW}
E. N. Dancer and J. Wei,  \emph{Spike solutions in coupled nonlinear Schr\"odinger equations with attractive interaction}, Trans. Amer. Math. Soc. {\bf 361} (2009), no. 3, 1189-1208.

\bibitem{D7}
B. Dodson, \emph{Global well-posedness and scattering for the defocusing, $L^2_{x}$-critical nonlinear Schr\"odinger equation when $d\geq3$}, J. Amer. Math. Soc. {\bf 25} (2012), no. 2, 429-463.

\bibitem{D6}
B. Dodson, \emph{Global well-posedness and scattering for the mass critical nonlinear Schr\"odinger equation with mass below the mass of the ground state},
Adv. Math. {\bf 285} (2015), 1589-1618.

\bibitem{D66}
B. Dodson, \emph{Global well-posedness and scattering for the defocusing, $L^2$ critical, nonlinear Schr\"odinger equation when $d=1$}, Amer. J. Math. {\bf 138} (2016), no.~2, 531-569.

\bibitem{D5}
B. Dodson, \emph{Global well-posedness and scattering for the defocusing, $L^2$-critical, nonlinear Schr\"odinger equation when $d=2$},
Duke Math. J. {\bf 165} (2016), no. 18, 3435-3516.


\bibitem{D4}
B. Dodson, \emph{The $L^2$ sequential convergence of a solution to the one-dimensional, mass-critical NLS above the ground state}, SIAM J. Math. Anal. {\bf 53} (2021), no. 4, 4744-4764.



\bibitem{D3}
B. Dodson, \emph{The $L^2$ sequential convergence of a solution to the mass-critical NLS above the ground state},
Nonlinear Anal. {\bf 215} (2022), Paper No. 112612, 25 pp.


\bibitem{D2}
B. Dodson, \emph{A determination of the blowup solutions to the focusing NLS with mass equal to the mass of the soliton}, Ann. PDE {\bf 9} (2023), no. 1, Paper No. 3, 86 pp.


\bibitem{D1}
B. Dodson, \emph{A determination of the blowup solutions to the focusing, quintic NLS with mass equal to the mass of the soliton}, Anal. PDE {\bf 17} (2024), no. 5, 1693-1760.

\bibitem{E}
B. D. Esry, C. H. Greene, J. P. Burke Jr., and J. L. Bohn, \emph{Hartree-Fock theory for double condensates}, Phys. Rev. Lett. 78 (1997), 3594–3597.

\bibitem{F}
C. Fan, \emph{The $L^2$ weak sequential convergence of radial focusing mass critical NLS solutions with mass above the ground state},
Int. Math. Res. Not. IMRN (2021), no. 7, 4864-4906.

\bibitem{FHRY}
L. Farah, J. Holmer, S. Roudenko, and K. Yang, \emph{On instability of solitons in the 2d cubic Zakharov-Kuznetsov equation}, 
S\~ao Paulo J. Math. Sci. {\bf 13} (2019), no. 2, 435-446.





\bibitem{HHK}
M. Hadac, S. Herr, and H. Koch,
 \emph{Well-posedness and scattering for the KP-II equation in a critical space}. Ann. Inst. H. Poincar\'e Anal. Non Lin\'eaire {\bf 26} (2009), no. 3,  917-941.


\bibitem{HK}
T. Hmidi and S. Keraani, \emph{Blowup theory for the critical nonlinear Schr\"odinger equations revisited}, Int. Math. Res. Not. (2005), no. {\bf 46}, 2815-2828.


\bibitem{JLLLW}
Y. Jing, H. Liu, Y. Liu, Z. Liu and J. Wei, \emph{The number of positive solutions for $n$-coupled elliptic systems}, J. Lond. Math. Soc. (2) {\bf 110} (2024), no. 6, Paper No. e70040.




\bibitem{KLVZ}
R. Killip, D. Li, M. Visan, and X. Zhang, \emph{Characterization of minimal-mass blowup solutions to the focusing mass-critical NLS},
SIAM J. Math. Anal. {\bf 41} (2009), no. 1, 219-236.

\bibitem{KTV11}
R. Killip, T. Tao, and M. Visan, \emph{The cubic nonlinear Schr\"odinger equation in two dimensions with radial data}, J. Eur. Math. Soc. (JEMS) {\bf 11} (2009), no. 6, 1203-1258.


\bibitem{KVZ}
R. Killip, M. Visan and X. Zhang, \emph{The mass-critical nonlinear Schr\"odinger equation with radial data in dimensions three and higher}, Anal. PDE {\bf 1} (2008), no. 2, 229-266.


\bibitem{KTa}
H. Koch and D. Tataru, \emph{$L^p$ eigenfuction bounds for the Hermite operator}, Duke Math J. {\bf 128} (2005), no.2, 369-392.


\bibitem{KW}
M. K. Kwong, \emph{Uniqueness of positive solutions of $\Delta u - u + u^p = 0$ in $\mathbb{R}^n$}, Arch. Rational Mech. Anal. {\bf 105} (1989), no. 3, 243-266.

\bibitem{LZ28}
D. Li and X. Zhang, \emph{On the classification of minimal mass blowup solutions of the focusing mass-critical Hartree equation}, Adv. Math. {\bf 220} (2009), no. 4, 1171–1192.


\bibitem{LZ}
D. Li and X. Zhang, \emph{On the rigidity of minimal mass solutions to the focusing mass-critical NLS for rough initial data}, Electron. J. Differential Equations (2009), No. 78, 19 pp.

\bibitem{LZ0}
D. Li and X. Zhang, \emph{Regularity of almost periodic modulo scaling solutions for mass-critical NLS and applications}, Anal. PDE {\bf 3} (2010), no. 2, 175-195.


\bibitem{LZ2}
D. Li and X. Zhang, \emph{On the focusing mass critical problem in six dimensions with splitting spherically symmetric initial data}, Dyn. Partial Differ. Equ. {\bf 7} (2010), no. 4, 345-373.


\bibitem{LZ1}
D. Li and X. Zhang, \emph{On the rigidity of solitary waves for the focusing mass-critical NLS in dimensions $d \ge 2$},
Sci. China Math. {\bf 55} (2012), no. 2, 385-434.


\bibitem{LW}
T.-C. Lin and J. Wei, \emph{Ground state of $N$ coupled nonlinear Schr\"odinger equations in $\mathbb{R}^n$, $n\le 3$}, Comm. Math. Phys. {\bf 255 } (2005), no. 3, 629-653.

\bibitem{LZeng}
Z. Lin and C. Zeng, \emph{Instability, index theorem, and exponential trichotomy for linear Hamiltonian PDEs}, Mem. Amer. Math. Soc. {\bf 275} (2022), no. 1347, v+136 pp. ISBN: 978-1-4704-5044-1; 978-1-4704-7013-5.

\bibitem{MF1}
Y. Martel and F. Merle, \emph{Instability of solitons for the critical generalized Korteweg-de Vries equation}, Geom. Funct. Anal. {\bf 11} (2001), no. 1, 74-123.

\bibitem{MN}
N. Masmoudi and K. Nakanishi, \emph{From nonlinear Klein-Gordon equation to a system of coupled nonlinear Schr\"odinger equations}, Math. Ann. {\bf 324} (2002), no. 2, 359-389.


\bibitem{M}
F. Merle, \emph{Determination of blow-up solutions with minimal mass for nonlinear Schr\"odinger equations with critical power},
Duke Math. J. {\bf 69} (1993), no.2, 427-454.


\bibitem{MR}
F. Merle, \emph{Existence of blow-up solutions in the energy space for the critical generalized {K}d{V} equation}, J. Amer. Math. Soc., {\bf 14} (2001), no.3, 555-578.




\bibitem{MV} 
F. Merle and L. Vega, \emph{Compactness at blow-up time for $L^2$ solutions of the critical nonlinear Schr\"odinger equation in 2D}, Internat. Math. Res. Notices (1998), no. 8, 399-425.

\bibitem{Murphy}
J. Murphy, \emph{Threshold scattering for the 2D radial cubic-quintic NLS}, Comm. Partial Differential Equations {\bf 46} (2021), no.~11, 2213-2234.

\bibitem{NTDS}
N. V. Nguyen, R. Tian, B. Deconinck, and N. Sheils, \emph{Global existence for a coupled system of Schr\"odinger equations with power-type nonlinearities}, J. Math. Phys. {\bf 54} (2013), no. 1, 011503, 19 pp.

\bibitem{PWW}
S. Peng, Q. Wang and Z. Wang, \emph{On coupled nonlinear Schr\"odinger systems with mixed couplings}, Trans. Amer. Math. Soc. {\bf 371} (2019), no.~11, 7559-7583.


\bibitem{SWX}
Y. Sato, Z.-Q. Wang, and J. Xia, \emph{Symmetric nonradial solutions for nonlinear Schr\"odinger systems with mixed couplings on $\mathbb{R}^n$},
SIAM J. Math. Anal. {\bf 57 } (2025), no. 1, 306-335.


\bibitem{Si}
B. Sirakov, \emph{Least energy solitary waves for a system of nonlinear Schr\"odinger equations in $\mathbb{R}^n$}, Comm. Math. Phys. {\bf 271} (2007), no. 1, 199-221.


\bibitem{Su}
Y. Su, \emph{Uniqueness of minimal blow-up solutions to nonlinear Schr\"odinger system},
Nonlinear Anal. {\bf 155} (2017), 186-197.


\bibitem{T1} 
T. Tao, \emph{A sharp bilinear restriction estimate for paraboloids}, Geom. Funct. Anal. {\bf 13} (2003), 1359-1384.


\bibitem{TVZ1}
T. Tao, M. Visan and X. Zhang, \emph{Global well-posedness and scattering for the defocusing mass-critical nonlinear Schr\"odinger equation for radial data in high dimensions}, Duke Math. J. {\bf 140} (2007), no. 1, 165-202.


\bibitem{TVZ2}
T. Tao, M. Visan and X. Zhang, \emph{Minimal-mass blowup solutions of the mass-critical NLS}, Forum Math. {\bf 20} (2008), no. 5, 881-919.

\bibitem{Timmer}
E. Timmermans, \emph{Phase separation of Bose-Einstein condensates}, Phys. Rev. Lett. 81 (1998), 5718–5721.

\bibitem{WW}
J. Wei and Y. Wu, \emph{Ground states of nonlinear Schr\"odinger systems with mixed couplings}, J. Math. Pures Appl. (9) { 141 } (2020), 50-88.


\bibitem{WY}
{J. Wei and W. Yao, \emph{Uniqueness of positive solutions to some coupled nonlinear Schr\"odinger equations}, Commun. Pure Appl. Anal. { \bf 11} (2012), no. 3, 1003-1011.}

\bibitem{W}
M. I. Weinstein, \emph{Nonlinear Schr\"odinger equations and sharp interpolation estimates}, Comm. Math. Phys. {\bf 87} (1982/83), no. 4, 567-576.

\bibitem{weistein}
M. I. Weinstein, \emph{Modulational stability of ground states of nonlinear Schr\"odinger equations},  SIAM J. Math. Anal. {\bf 16} (1985), 472-491.

\bibitem{weinstein:charact}
M. I. Weinstein, \emph{On the structure and formation of singularities in solutions to nonlinear dispersive evolution equations,}
Comm. Partial Differential Equations {\bf 11} (1986), 545-565.


\bibitem{YZ}
K. Yang and L. Zhao, \emph{Global well-posedness and scattering for mass-critical, defocusing, infinite dimensional vector-valued resonant nonlinear Schr\"odinger system}, 
SIAM J. Math. Anal., {\bf 50} (2) (2018), pp. 1593-1655.



\bibitem{ZLQYY}
L.-C. Zhao, L. Ling, J.-W. Qi, Z.-Y. Yang, and W.-L. Yang, \emph{Dynamics of rogue wave excitation pattern on stripe phase backgrounds in a two-component Bose-Einstein condensate}, Commun. Nonlinear Sci. Numer. Simul.,
{\bf 49}, 2017, 39-47.


\bibitem{ZL}
 L. Zhao and J. Liu, \emph{Rogue-wave solutions of a three-component coupled nonlinear Schr\"odinger equation}, Phys. Rev. E {\bf 87} (2013), 013201.

\end{thebibliography}
\end{document}